\documentclass[a4paper, english, 9pt]{extreport}
\usepackage[T1]{fontenc}
\usepackage[latin1]{inputenc}
 \usepackage{amsmath,amssymb}
\usepackage{graphicx}
\usepackage{amsthm}
\usepackage{dirtytalk}
\usepackage{float}
\usepackage{enumerate}
\usepackage{caption}
\usepackage{framed}
\usepackage[all,cmtip]{xy}
\usepackage{multirow}
\usepackage{mathtools}
\usepackage{color}
\usepackage{amsfonts}
\usepackage{etex}
\usepackage{mathrsfs}
\usepackage{epigraph}
\usepackage{latexsym}
\usepackage{geometry}
\usepackage{makeidx}
\usepackage{shuffle}
\usepackage{perpage} 
\MakePerPage{footnote} 
\DeclareMathSymbol{\mlq}{\mathord}{operators}{``}
\DeclareMathSymbol{\mrq}{\mathord}{operators}{`'}

\usepackage{fix-cm}     
\makeatletter
\newcommand\HUGE{\@setfontsize\Huge{30}{40}}
\makeatother
\usepackage{titlesec}
\newcommand*{\justifyheading}{\raggedleft}
\titleformat{\chapter}[display]
  {\normalfont\huge\bfseries\justifyheading}{\chaptertitlename\ \thechapter}
  {20pt}{\HUGE}
\usepackage[english]{babel}
\usepackage{blindtext}
\usepackage[refpage]{nomencl}
\usepackage{xcolor}
\providecommand{\boldsymbol}[1]{\mbox{\boldmath $#1$}}

\newtheorem{theo}{Theorem}[section]

\newtheorem{defi}[theo]{Definition}
\newtheorem{coro}[theo]{Corollary}
\newtheorem{lemm}[theo]{Lemma}
\newtheorem{conj}[theo]{Conjecture}
\newtheorem*{theom}{Theorem}
\newtheorem*{conje}{Conjecture}

\newtheorem*{defin}{Definition}
\newtheorem*{corol}{Corollary}
\newtheorem*{lemme}{Lemma}
\makeindex\usepackage{babel}
\addto\extrasfrench{\providecommand{\fg}{\ifdim\lastskip>\z@\unskip\fi~\frqq}}

\newcommand\underrel[2]{\mathrel{\mathop{#2}\limits_{#1}}}
\newcommand{\twolines}[2][c]{
  \begin{tabular}[#1]{@{}c@{}}#2\end{tabular}}  
\makeatother
\makeatletter
\makeindex
\makenomenclature
 \def\makenomenclature{%
\newwrite\@nomenclaturefile
\immediate\openout\@nomenclaturefile=\jobname\@outputfileextension
\def\@nomenclature{%
 \@bsphack
\begingroup
\@sanitize
 \@ifnextchar[%
{\@@@nomenclature}{\@@@nomenclature[\nomprefix]}}%
\typeout{Writing nomenclature file \jobname\@outputfileextension}%
 \let\makenomenclature\@empty}

\begin{document}
{\large

\vspace*{1cm}

\begin{center}

{\large  UNIVERSIT\'E PIERRE ET MARIE CURIE}

\vspace*{0.5cm}

\'Ecole Doctorale de Sciences Math\'ematiques de Paris Centre

\end{center}
\vspace{1cm}

\begin{center}
 \textsc{\textbf{{\LARGE Th\`ese de doctorat}}}
\end{center}
\vspace{0.3cm}
\begin{center}
{\large Discipline: Math\'ematiques}
\end{center}
\vspace{1cm}

\noindent\makebox[\linewidth]{\rule{\paperwidth}{0.4pt}}
\begin{center}
{\Large \textbf{Periods of the motivic fundamental groupoid of $\boldsymbol{\mathbb{P}^{1} \diagdown \lbrace 0,  \mu_{N}, \infty \rbrace}$.}}  
\noindent\makebox[\linewidth]{\rule{\paperwidth}{0.4pt}}

\end{center}
\vspace{1cm}
\begin{center}
Pr\'esent\'ee par \\

\textbf{{\Large Claire GLANOIS}}
\end{center}
\vspace{0.5cm}

\begin{center}
Dirig\'ee par {\Large Francis BROWN.}
\end{center}

\vspace*{3cm} 
\begin{center}
Pr\'{e}sent\'{e}e et soutenue publiquement le 6 janvier 2016 devant le jury compos\'{e} de :
\end{center}
\begin{center}
  \begin{tabular}{lll}
Francis {\sc Brown} & Directeur de th\`{e}se & Oxford University\\
Don {\sc Zagier} & Rapporteur & Max Planck Institute, Bonn\\
Jianqiang {\sc Zhao} & Rapporteur & ICMAT, Madrid\\
Yves {\sc Andr\'{e}} & Examinateur & UPMC (IMJ), Paris\\ 
Pierre {\sc Cartier} & Examinateur & IH\'{E}S, Paris-Saclay\\
Herbert {\sc Gangl} & Examinateur  & Durham University\\
\end{tabular} 
\end{center}  

\newpage
\vspace*{3cm}
\begin{flushright}
\textit{Aux inconnues},\\
\textit{Aux variables},\\
\textit{A} Elle,
\end{flushright}
\epigraph{$\mlq$\textit{Prairie de tous mes instants, ils ne peuvent me fouler.}\\ \textit{Leur voyage est mon voyage et je reste obscurit\'e.}$\mrq$}{R. Char}

\newpage
\begin{center}
{\Huge \texttt{Abstract}}
\end{center}
In this thesis, following F. Brown's point of view, we look at the Hopf algebra structure of motivic cyclotomic multiple zeta values, which are motivic periods of the fundamental groupoid of $\mathbb{P}^{1} \diagdown \lbrace 0,  \mu_{N}, \infty \rbrace $. By application of a surjective \textit{period} map (which, under Grothendieck's period conjecture, is an isomorphism), we deduce results (such as generating families, identities, etc.) on cyclotomic multiple zeta values, which are complex numbers. The coaction of this Hopf algebra (explicitly given by a combinatorial formula from A. Goncharov and F. Brown's works) is the dual of the action of a so-called \textit{motivic} Galois group on these specific motivic periods. This entire study was actually motivated by the hope of a Galois theory for periods, which should extend the usual Galois theory for algebraic numbers.

In the first part, we focus on the case of motivic multiple zeta values ($N = 1$) and Euler sums ($N = 2$). In particular,  we present new bases for motivic multiple zeta values: one via motivic Euler sums, and another (depending on an analytic conjecture) which is known in the literature as the Hoffman $\star$ basis; under a general motivic identity that we conjecture, these bases are identical.

In the second part, we apply some Galois descents ideas to the study of these periods, and examine how periods of the fundamental groupoid of $\mathbb{P}^{1} \diagdown \lbrace 0,  \mu_{N'}, \infty \rbrace $ are embedded into periods of $\pi_{1}(\mathbb{P}^{1} \diagdown \lbrace 0,  \mu_{N}, \infty \rbrace )$, when $N'\mid N$. After giving some general criteria for any $N$, we focus on the cases $N=2,3,4,\mlq 6 \mrq, 8$, for which the motivic fundamental group generates the category of mixed Tate motives on $\mathcal{O}_{N}[\frac{1}{N}]$ (unramified if $N=6$). For those $N$, we are able to construct Galois descents explicitly, and extend P. Deligne's results.\\
\\
\texttt{Key words}:\textit{ Periods, Polylogarithms, multiple zeta values, Mixed Tate Motives, cyclotomic field, Hopf algebra, Motivic fundamental group, Galois Descent.}

\begin{center}
{\Huge \texttt{R\'{e}sum\'{e}.}}
\end{center}
A travers ce manuscrit, en s'inspirant du point de vue adopt\'{e} par F. Brown, nous examinons la structure d'alg\`{e}bre de Hopf des multiz\^{e}tas motiviques cyclotomiques, qui sont des p\'{e}riodes motiviques du groupo\"{i}de fondamental de $\mathbb{P}^{1} \diagdown \lbrace 0,  \mu_{N}, \infty \rbrace $. Par application d'un morphisme \textit{p\'{e}riode} surjectif (isomorphisme sous la conjecture de Grothendieck), nous pouvons d\'{e}duire des r\'{e}sultats (tels des familles g\'{e}n\'{e}ratrices, des identit\'{e}s, etc.) sur ces nombres complexes que sont les multiz\^{e}tas cyclotomiques. La coaction de cette alg\`{e}bre de Hopf (explicite par une formule combinatoire due aux travaux de A.B. Goncharov et F. Brown) est duale \`{a} l'action d'un d\'{e}nomm\'{e} \textit{groupe de Galois motivique} sur ces p\'{e}riodes motiviques. Ces recherches sont ainsi motiv\'{e}es par l'espoir d'une th\'{e}orie de Galois pour les p\'{e}riodes, \'{e}tendant la th\'{e}orie de Galois usuelle pour les nombres alg\'{e}briques.

Dans un premier temps, nous nous concentrons sur les multiz\^{e}tas ($N=1$) et les sommes d'Euler ($N=2$) motiviques. En particulier, de nouvelles bases pour les multizetas motiviques sont pr\'{e}sent\'{e}es: une via les sommes d'Euler motiviques, et une seconde (sous une conjecture analytique) qui est connue sous le nom de \textit{Hoffman} $\star$; soulignons que sous une identit\'{e} motivique g\'{e}n\'{e}rale que nous conjecturons \'{e}galement, ces bases sont identiques.

Dans un second temps, nous appliquons des id\'{e}es de descentes galoisiennes \`{a} l'\'{e}tude de ces p\'{e}riodes, en regardant notamment comment les p\'{e}riodes du groupo\"{i}de fondamental de $\mathbb{P}^{1} \diagdown \lbrace 0,  \mu_{N'}, \infty \rbrace $ se plongent dans les p\'{e}riodes de $\pi_{1}(\mathbb{P}^{1} \diagdown \lbrace 0,  \mu_{N}, \infty \rbrace )$, lorsque $N'\mid N$. Apr\`{e}s avoir fourni des crit\`{e}res g\'{e}n\'{e}raux (quel que soit $N$), nous nous tournons vers les cas $N=2,3,4,\mlq 6 \mrq, 8$, pour lesquels le groupo\"{i}de fondamental motivique engendre la cat\'{e}gorie des motifs de Tate mixtes sur $\mathcal{O}_{N}[\frac{1}{N}]$ (non ramifi\'{e} si $N=6$). Pour ces valeurs, nous sommes en mesure d'expliciter les descentes galoisiennes, et d'\'{e}tendre les r\'{e}sultats de P. Deligne.\\
\\
\texttt{Mots cl\'{e}s}: \textit{P\'{e}riodes, Polylogarithmes, multiz\^{e}tas, corps cyclotomiques, Motifs de Tate Mixtes, alg\`{e}bre de Hopf, groupe motivique fondamental, descente galoisienne.}
\newpage
\strut 
\newpage
\tableofcontents
\addtocontents{toc}{\par}

\newpage
\strut 
\newpage

\chapter{Introduction}

\epigraph{$\mlq$\textit{Qui est-ce ? Ah, tr\`{e}s bien, faites entrer l'infini.}$\mrq$}{Aragon}

\section{Motivation}

\subsection{Periods}
A \textbf{\textit{period}} \footnote{For an enlightening survey, see the reference article $\cite{KZ}$.} denotes a complex number that can be expressed as an integral of an algebraic function over an algebraic domain.\footnote{We can equivalently restrict to integral of rational functions over a domain in $\mathbb{R}^{n}$ given by polynomial inequalities with rational coefficients, by introducing more variables.} They form the algebra of periods $\mathcal{P}$\nomenclature{$\mathcal{P}$}{the algebra of periods, resp. $\widehat{\mathcal{P}}$ of extended periods}, fundamental class of numbers between algebraic numbers $\overline{\mathbb{Q}}$ and complex numbers.\\
The study of these integrals is behind a large part of algebraic geometry, and its connection with number theory, notably via L-functions \footnote{One can associate a L$-$ function to many arithmetic objects such as a number field, a modular form, an elliptic curve, or a Galois representation. It encodes its properties, and has wonderful (often conjectural) meromorphic continuation, functional equations, special values, and non-trivial zeros (Riemann hypothesis).}; and many of the constants which arise in mathematics, transcendental number theory or in physics turn out to be periods, which motivates the study of these particular numbers.\\
\\
\texttt{Examples:}
\begin{itemize}
\item[$\cdot$] The following numbers are periods:
$$\sqrt{2}= \int_{2x^{2} \leq 1} dx \text{  ,  }  \quad \pi= \int_{x^{2}+y^{2} \leq 1} dx dy \quad\text{  and  }  \quad \log(z)=\int_{1}^{z} \frac{dx}{x}, z>1, z\in \overline{\mathbb{Q}}.$$
\item[$\cdot$] Famous -alleged transcendental- numbers which conjecturally are not periods:
$$e =\lim_{n \rightarrow \infty} \left( 1+ \frac{1}{n} \right)^{n} \text{ , } \quad \gamma=\lim_{n \rightarrow \infty} \left( -ln(n) + \sum_{k=1}^{n} \frac{1}{k} \right)  \text{  or  } \quad \frac{1}{\pi}.$$
It can be more useful to consider the ring of extended periods, by inverting $\pi$:
$$\widehat{\mathcal{P}}\mathrel{\mathop:}=\mathcal{P} \left[  \frac{1}{\pi} \right].$$
\item[$\cdot$] Multiple polylogarithms at algebraic arguments (in particular cyclotomic multiple zeta values), by their representation as iterated integral given below, are periods. Similarly, special values of Dedekind zeta function $\zeta_{F}(s)$ of a number field, of L-functions, of hypergeometric series, modular forms, etc. are (conjecturally at least) periods or extended periods.
\item[$\cdot$] Periods also appear as Feynman integrals: Feynman amplitudes $I(D)$ can be written as a product of Gamma functions and meromorphic functions whose coefficients of its Laurent series expansion at any integer $D$ are periods (cf. $\cite{BB}$), where D is the dimension of spacetime.
\end{itemize}
Although most periods are transcendental, they are constructible; hence, the algebra $\mathcal{P}$ is countable, and any period contains only a finite amount of information. Conjecturally (by Grothendieck's conjecture), the only relations between periods comes from the following rules of elementary integral calculus\footnote{However, finding an algorithm to determine if a real number is a period, or if two periods are equal seems currently out of reach; whereas checking if a number is algebraic, or if two algebraic numbers are equal is rather \say{easy} (with \say{LLL}-type reduction algorithm, resp. by calculating the g.c.d of two vanishing polynomials associated to each).}:
\begin{itemize}
\item[$(i)$] Additivity (of the integrand and of the integration domain)
\item[$(ii)$] Invertible changes of variables
\item[$(iii)$] Stokes's formula.\\ 
\end{itemize} 

Another way of viewing a period $ \boldsymbol{\int_{\gamma} \omega }$ is via a comparison between two cohomology theories: the algebraic \textit{De Rham} cohomology, and the singular (\textit{Betti}) cohomology. More precisely, let $X$ a smooth algebraic variety defined over $\mathbb{Q}$ and $Y$ a closed subvariety over $\mathbb{Q}$.
\begin{itemize}
\item[$\cdot$] On the one hand, the \textit{algebraic} De Rham cohomology $H^{\bullet}_{dR}(X)$\nomenclature{$H^{\bullet}_{dR}$}{the \textit{algebraic} De Rham cohomology} is the hypercohomology of the sheaf of algebraic (K{\"a}hler) differentials on $X$. If $X$ is affine, it is defined from the de Rham complex $\Omega^{\bullet}(X)$ which is the cochain complex of global algebraic (K{\"a}hler) differential forms on X, with the exterior derivative as differential. Recall that the \textit{classical} $k^{\text{th}}$ de Rham cohomology group is the quotient of smooth closed $k$-forms on the manifold X$_{\diagup \mathbb{C}}$ modulo the exact $k$-forms on $X$.\\
Given $\omega$ a closed algebraic $n$-form on $X$ whose restriction on $Y$ is zero, it defines an equivalence class $[\omega]$ in the relative de Rham cohomology groups $H^{n}_{dR}(X,Y)$, which are finite-dimensional $\mathbb{Q}-$ vector space. 
\item[$\cdot$] On the other hand, the Betti homology $H_{\bullet}^{B}(X)$\nomenclature{$H_{\bullet}^{B}(X)$}{the Betti homology} is the homology of the chain complex induced by the boundary operation of singular chains on the manifold $X(\mathbb{C})$; Betti cohomology groups $H_{B}^{n}(X,Y)= H^{B}_{n}(X,Y)^{\vee}$ are the dual $\mathbb{Q}$ vector spaces (taking here coefficients in $\mathbb{Q}$, not $\mathbb{Z}$).\\
Given $\gamma$ a singular $n$ chain on $X(\mathbb{C}) $ with boundary in $ Y(\mathbb{C})$, it defines an equivalence class $[\gamma]$ in the relative Betti homology groups $ H_{n}^{B}(X,Y)=H^{n}_{B}(X,Y)^{\vee}$.\footnote{Relative homology can be calculated using the following long exact sequence:
 $$ \cdots \rightarrow H_{n}(Y) \rightarrow H_{n}(X) \rightarrow H_{n}(X,Y) \rightarrow H_{n-1} (Y) \rightarrow \cdots.$$
 }
\end{itemize}
Furthermore, there is a comparison isomorphism\nomenclature{$\text{comp}_{B,dR}$}{the comparison isomorphism between de Rham and Betti cohomology} between relative de Rham and relative Betti cohomology (due to Grothendieck, coming from the integration of algebraic differential forms on singular chains):
$$\text{comp}_{B,dR}: H^{\bullet}_{dR} (X,Y) \otimes_{\mathbb{Q}} \mathbb{C} \rightarrow H^{\bullet}_{B} (X,Y) \otimes_{\mathbb{Q}} \mathbb{C}.$$
By pairing a basis of Betti homology to a basis of de Rham cohomology, we obtain the \textit{matrix of periods}, which is a square matrix with entries in $\mathcal{P}$ and determinant in $\sqrt{\mathbb{Q}^{\ast}}(2i\pi)^{\mathbb{N}^{\ast}}$\nomenclature{$\mathbb{N}^{\ast}$}{the set of positive integers, $\mathbb{N}:=\mathbb{N}^{\ast}\cup\lbrace 0\rbrace$ }; i.e. its inverse matrix has its coefficients in $\widehat{\mathcal{P}}$. Then, up to the choice of these two basis:
\begin{framed}
The period $ \int_{\gamma} \omega $ is the coefficient of this pairing $\langle [\gamma], \text{comp}_{B,dR}([\omega]) \rangle$.
\end{framed}
\noindent
\texttt{Example}: Let $X=\mathbb{P}^{1}\diagdown \lbrace 0, \infty \rbrace$, $Y=\emptyset$ and $\gamma_{0}$ the counterclockwise loop around $0$:
 $$H^{B}_{i} (X)= \left\lbrace \begin{array}{ll}
\mathbb{Q} & \text{ if } i=0 \\
\mathbb{Q}\left[\gamma_{0}\right]  & \text{ if } i=1 \\
0 & \text{ else }.
\end{array} \right.  \quad \text{ and } \quad H^{i}_{dR} (X)= \left\lbrace \begin{array}{ll}
\mathbb{Q} & \text{ if } i=0 \\
\mathbb{Q}\left[ \frac{d}{dx}\right]  & \text{ if } i=1 \\
0 & \text{ else }.
\end{array} \right. $$
Since $\int_{\gamma_{0}} \frac{dx}{x}= 2i\pi$, $2i\pi$ is a period; as we will see below, it is a period of the Lefschetz motive $\mathbb{L}\mathrel{\mathop:}=\mathbb{Q}(-1)$.\\
\\

Viewing periods from this cohomological point of view naturally leads to the definition of \textbf{\textit{motivic periods}} given below \footnote{The definition of a motivic period is given in $\S 2.4$ in the context of a category of Mixed Tate Motives. In general, one can do with Hodge theory to define $ \mathcal{P}^{\mathfrak{m}}$, which is not strictly speaking \textit{motivic}, once we specify that the mixed Hodge structures considered come from the cohomology of algebraic varieties.}, which form an algebra $\mathcal{P}^{\mathfrak{m}}$, equipped with a period homomorphism:
$$\text{per}: \mathcal{P}^{\mathfrak{m}} \rightarrow \mathcal{P}.$$
A variant of Grothendieck's conjecture, which is a presently inaccessible conjecture in transcendental number theory, predicts that it is an isomorphism.\\
There is an action of a so-called motivic Galois group $\mathcal{G}$ on these motivic periods as we will see below in $\S 2.1$. If Grothendieck's period conjecture holds, this would hence extend the usual Galois theory for algebraic numbers to periods (cf. $\cite{An2}$). \\

In this thesis, we will focus on motivic (cyclotomic) multiple zeta values, defined in $\S 2.3$, which are motivic periods of the motivic (cyclotomic) fundamental group, defined in $\S 2.2$. Their images under this period morphism are the (cyclotomic) multiple zeta values; these are fascinating examples of periods, which are introduced in the next section (see also $\cite{An3}$). 
$$ \quad $$

\subsection{Multiple zeta values}

The Zeta function is known at least since Euler, and finds itself nowadays, in its various generalized forms (multiple zeta values, Polylogarithms, Dedekind zeta function, L-functions, etc), at the crossroad of many different fields as algebraic geometry (with periods and motives), number theory (notably with modular forms), topology, perturbative quantum field theory (with Feynman diagrams, cf. $\cite{Kr}$), string theory, etc. Zeta values at even integers are known since Euler to be rational multiples of even powers of $\pi$:

\begin{lemme}
$$\text{For } n \geq 1, \quad \zeta(2n)=\frac{\mid B_{2n}\mid (2\pi)^{2n}}{2(2n)!}, \text{ where } B_{2n}  \text{ is the } 2n^{\text{th}} \text { Bernoulli number.} $$
\end{lemme}
However, the zeta values at odd integers already turn out to be quite interesting periods:
\begin{conje}
$\pi, \zeta(3), \zeta(5), \zeta(7),\cdots $ are algebraically independent.
\end{conje}
This conjecture raises difficult transcendental questions, rather out of reach; currently we only know $\zeta(3)\notin \mathbb{Q}$ (Ap\'{e}ry), infinitely many odd zeta values are irrational (Rivoal), or other quite partial results (Zudilin, Rivoal, etc.); recently, F. Brown paved the way for a pursuit of these results, in $\cite{Br5}$.\\
\\
\textbf{Multiple zeta values relative to the $\boldsymbol{N^{\text{th}}}$ roots of unity} $\mu_{N}$, \nomenclature{$\mu_{N}$}{$N^{\text{th}}$ roots of unity}which we shall denote by \textbf{MZV}$\boldsymbol{_{\mu_{N}}}$\nomenclature{\textbf{MZV}$\boldsymbol{_{\mu_{N}}}$}{Multiple zeta values relative to the $\boldsymbol{N^{\text{th}}}$ roots of unity, denoted $\zeta()$}  are defined by: \footnote{Beware, there is no consensus on the order for the arguments of these MZV: sometimes the summation order is reversed.}
\begin{framed}
\begin{equation}\label{eq:mzv} \text{     }  \zeta\left(n_{1}, \ldots , n_{p} \atop  \epsilon_{1} , \ldots ,\epsilon_{p} \right)\mathrel{\mathop:}= \sum_{0<k_{1}<k_{2} \cdots <k_{p}} \frac{\epsilon_{1}^{k_{1}} \cdots \epsilon_{p}^{k_{p}}}{k_{1}^{n_{1}} \cdots k_{p}^{n_{p}}} \text{, } \epsilon_{i}\in \mu_{N} \text{, } n_{i}\in\mathbb{N}^{\ast}\text{, } (n_{p},\epsilon_{p})\neq (1,1).
\end{equation}
\end{framed}
The \textit{weight}, often denoted $w$ below, is defined as $\sum n_{i}$, the \textit{depth} is the length $p$, whereas the \textit{height}, usually denoted $h$, is the number of $n_{i}$ greater than $1$. The weight is conjecturally a grading, whereas the depth is only a filtration. Denote also by $\boldsymbol{\mathcal{Z}^{N}}$\nomenclature{$\mathcal{Z}^{N}$}{the $\mathbb{Q}$-vector space spanned by the multiple zeta values relative to $\mu_{N}$} the $\mathbb{Q}$-vector space spanned by these multiple zeta values relative to $\mu_{N}$.\\
These MZV$_{\mu_{N}}$ satisfy both \textit{shuffle} $\shuffle$ relation (coming from the integral representation below) and \textit{stuffle} $\ast$ relation (coming from this sum expression), which turns $\mathcal{Z}^{N}$ into an algebra. These relations, for $N=1$, are conjectured to generate all the relations between MZV if we add the so-called Hoffman (regularized double shuffle) relation; cf. \cite{Ca}, \cite{Wa} for a good introduction to this aspect. However, the literature is full of other relations among these (cyclotomic) multiple zeta values: cf. $\cite{AO},\cite{EF}, \cite{OW}, \cite{OZa}, \cite{O1}, \cite{Zh2}, \cite{Zh3}$. Among these, we shall require the so-called pentagon resp. \textit{hexagon} relations (for $N=1$, cf. $\cite{Fu}$), coming from the geometry of moduli space of genus $0$ curves with $5$ ordered marked points $X=\mathcal{M}_{0,5}$ resp. with $4$ marked points $X=\mathcal{M}_{0,4}=\mathbb{P}^{1}\diagdown\lbrace 0,1, \infty\rbrace$ and corresponding to a contractible path in X; hexagon relation (cf.  Figure $\ref{fig:hexagon}$) is turned into an \textit{octagon} relation (cf. Figure $\ref{fig:octagon}$) for $N>1$ (cf. $\cite{EF}$) and is used below in $\S 4.2$.\\
\\

One crucial point about multiple zeta values, is their \textit{integral representation}\footnote{Obtained by differentiating, considering there variables $z_{i}\in\mathbb{C}$, since:
$$\frac{d}{dz_{p}} \zeta \left( n_{1}, \ldots, n_{p} \atop z_{1}, \ldots, z_{p-1}, z_{p}\right)  = 
\left\lbrace \begin{array}{ll} 
\frac{1}{z_{p}} \zeta \left( n_{1}, \ldots, n_{p}-1 \atop z_{1}, \ldots, z_{p-1}, z_{p}\right)  & \text{ if }  n_{p}\neq 1\\
\frac{1}{1-z_{p}} \zeta \left( n_{1}, \ldots, n_{p-1} \atop z_{1}, \ldots, z_{p-1} z_{p}\right)  & \text{ if } n_{p}=1.
\end{array} \right. $$}, which makes them clearly \textit{periods} in the sense of Kontsevich-Zagier. Let us define first the following iterated integrals and differential forms, with $a_{i}\in \lbrace 0, \mu_{N} \rbrace$:\nomenclature{$I(0; a_{1}, \ldots , a_{n} ;1)$}{particular iterated integrals, with $a_{i}\in \lbrace 0, \mu_{N} \rbrace$}
\begin{equation}\label{eq:iterinteg}
\boldsymbol{I(0; a_{1}, \ldots , a_{n} ;1)}\mathrel{\mathop:}=  \int_{0< t_{1} < \cdots < t_{n} < 1} \frac{dt_{1} \cdots dt_{n}}{ (t_{1}-a_{1}) \cdots (t_{n}-a_{n}) }=\int_{0}^{1}  \omega_{a_{1}} \ldots \omega_{a_{n}} \text{, with } \omega_{a}\mathrel{\mathop:}=\frac{dt}{t-a}.
\end{equation}\nomenclature{$\omega_{a}$}{$\mathrel{\mathop:}=\frac{dt}{t-a}$, differential form.}
In this setting, with  $\eta_{i}\mathrel{\mathop:}=  (\epsilon_{i}\ldots \epsilon_{p})^{-1}\in\mu_{N}$, $n_{i}\in\mathbb{N}^{\ast}$\footnote{The use of bold in the iterated integral writing indicates a repetition of the corresponding number, as $0$ here.}:
\begin{framed}
\begin{equation}\label{eq:reprinteg}
\zeta \left({ n_{1}, \ldots , n_{p} \atop \epsilon_{1} , \ldots ,\epsilon_{p} }\right) = (-1)^{p} I(0; \eta_{1}, \boldsymbol{0}^{n_{1}-1}, \eta_{2},\boldsymbol{0}^{n_{2}-1}, \ldots, \eta_{p}, \boldsymbol{0}^{n_{p}-1} ;1).
\end{equation}
\end{framed}
\nomenclature{$\epsilon_{i}$, $\eta_{i}$}{\textit{(Usually)} The roots of unity corresponding to the MZV resp. to the iterated integral, i.e. $\eta_{i}\mathrel{\mathop:}= (\epsilon_{i}\cdots \epsilon_{p})^{-1}$.}   
\textsc{Remarks}:
\begin{itemize}
\item[$\cdot$] Multiple zeta values can be seen as special values of generalized multiple polylogarithms, when $\epsilon_{i}$ are considered in $\mathbb{C}$\footnote{The series is absolutely convergent for $ \mid \epsilon_{i}\mid <1$, converges also for $\mid \epsilon_{i}\mid =1$ if $n_{p} >1$. Cf. \cite{Os} for an introduction.}. First, notice that in weight $2$, $Li_{1}(z)\mathrel{\mathop:}=\zeta\left( 1\atop z\right) $ is the logarithm $-\log (1-z)$. Already the dilogarithm, in weight $2$, $Li_{2}(z)\mathrel{\mathop:}=\zeta \left( 2\atop z\right) =\sum_{k>0} \frac{z^{k}}{k^{2}}$, satisfies nice functional equations\footnote{As the functional equations with $Li_{2}\left( \frac{1}{z}\right) $ or $Li_{2}\left( 1-z \right) $ or the famous five terms relation, for its sibling, the Bloch Wigner function $D(z)\mathrel{\mathop:}= Im\left( Li_{2}(z) +\log(\mid z\mid) \log(1-z) \right)$:
$$D(x)+D(y)+ D\left( \frac{1-x}{1-xy}\right) + D\left( 1-xy\right)+ D\left( \frac{1-y}{1-xy}\right)=0. $$} and arises in many places such as in the Dedekind zeta value $\zeta_{F}(2)$ for F an imaginary quadratic field, in the Borel regulator in algebraic K-theory, in the volume of hyperbolic manifolds, etc.; cf. $\cite{GZ}$; some of these connections can be generalized to higher weights.
\item[$\cdot$] Recall that an iterated integral of closed (real or complex) differential $1-$forms $\omega_{i}$ along a path $\gamma$ on a 1-dimensional (real or complex) differential manifold $M$ is homotopy invariant, cf. $\cite{Ch}$. If $M=\mathbb{C}\diagdown \lbrace a_{1}, \ldots , a_{N}\rbrace$ \footnote{As for cyclotomic MZV, with $a_{i}\in \mu_{N}\cup\lbrace 0\rbrace$; such an $I=\int_{\gamma}\omega_{1}\ldots \omega_{n}$ is a multivalued function on $M$.} and $\omega_{i}$ are meromorphic closed $1-$forms, with at most simple poles in $a_{i}$, and $\gamma(0)=a_{1}$, the iterated integral $I=\int_{\gamma}\omega_{1}\cdots \omega_{n}$ is divergent. The divergence being polynomial in log $\epsilon$ ($\epsilon \ll 1$) \footnote{More precisely, we can prove that $\int_{\epsilon}^{1} \gamma^{\ast} (\omega_{1}) \cdots \gamma^{\ast} (\omega_{n}) = \sum_{i=0} \alpha_{i} (\epsilon) \log^{i} (\epsilon)$, with $\alpha_{i} (\epsilon)$ holomorphic in $\epsilon=0$; $\alpha_{0} (\epsilon) $ depends only on $\gamma'(0)$.}, we define the iterated integral I as the constant term, which only depends on $\gamma'(0)$. This process is called \textit{regularization}, we need to choose the tangential base points to properly define the integral. Later, we will consider the straight path $dch$ from $0$ to $1$, with tangential base point $\overrightarrow{1}$ at $0$ and $\overrightarrow{-1}$ at $1$, denoted also $\overrightarrow{1}_{0}, \overrightarrow{-1}_{1}$ or simply $\overrightarrow{01}$ for both.
\end{itemize}
\texttt{Notations}: In the case of \textit{multiple zeta values} (i.e. $N=1$) resp. of  \textit{Euler sums} (i.e. $N=2$), since $\epsilon_{i}\in \left\{\pm 1\right\}$, the notation is simplified, using $ z_{i}\in \mathbb{Z}^{\ast}$:\nomenclature{ES}{Euler sums, i.e. multiple zeta values associated to $\mu_{2}=\lbrace\pm 1\rbrace$}
 \begin{equation}\label{eq:notation2} \zeta\left(z_{1},  \ldots, z_{p} \right) \mathrel{\mathop:}= \zeta\left(n_{1}, \ldots , n_{p} \atop  \epsilon_{1} , \ldots ,\epsilon_{p} \right)\text{ with  } \left( n_{i}  \atop \epsilon_{i} \right)\mathrel{\mathop:}=\left( \mid z_{i} \mid \atop sign(z_{i} ) \right) .
 \end{equation}
Another common notation in the literature is the use of \textit{overlines} instead of negative arguments, i.e.: $ z_{i}\mathrel{\mathop:}=\left\lbrace \begin{array}{ll}
n_{i} &\text{ if } \epsilon_{i}=1\\
\overline{n_{i}} &\text{ if } \epsilon_{i}=-1
\end{array} \right. .$\nomenclature{$\overline{n}$}{another notation to denote a negative argument in the Euler sums: when the corresponding root of unity is $\epsilon=-1$.}

\section{Contents}

In this thesis, we mainly consider the \textit{\textbf{motivic}} versions of these multiple zeta values, denoted $\boldsymbol{\zeta^{\mathfrak{m}}}(\cdot)$ and shortened \textbf{MMZV}$\boldsymbol{_{\mu_{N}}}$ and defined in $\S 2.3$\nomenclature{\textbf{MMZV}$\boldsymbol{_{\mu_{N}}}$}{the motivic multiple zeta values relative to $\mu_{N}$, denoted $\zeta^{\mathfrak{m}}(\ldots)$, defined in $\S 2.3$}. They span a $\mathbb{Q}$-vector space $\boldsymbol{\mathcal{H}^{N}}$ of motivic multiple zetas relative to $\mu_{N}$\nomenclature{$\mathcal{H}^{N}$}{the $\mathbb{Q}$-vector space $\boldsymbol{\mathcal{H}^{N}}$ of motivic multiple zetas relative to $\mu_{N}$}. There is a surjective homomorphism, called the \textit{period map}, which is conjectured to be an isomorphism (this is a special case of the period conjecture)\nomenclature{per}{the surjective period map, conjectured to be an isomorphism.}:
\begin{equation}\label{eq:per} \textbf{per} : w:\quad \mathcal{H}^{N} \rightarrow \mathcal{Z}^{N} \text{ ,  } \quad \zeta^{\mathfrak{m}} (\cdot) \mapsto \zeta ( \cdot ).
\end{equation}

Working on the motivic side, besides being conjecturally identical to the complex numbers side, turns out to be somehow simpler, since motivic theory provides a Hopf Algebra structure as we will see throughout this thesis. Notably, each identity between motivic MZV$_{\mu_{N}}$ implies an identity for their periods; a motivic basis for MMZV$\mu_{N}$ is hence a generating family (conjecturally basis) for MZV$\mu_{N}$.

Indeed, on the side of motivic multiple zeta values, there is an action of a \textit{motivic Galois group} $\boldsymbol{\mathcal{G}}$\footnote{Later, we will define a category of Mixed Tate Motives, which will be a tannakian category: consequently equivalent to a category of representation of a group $\mathcal{G}$; cf. $\S 2.1$. }, which, passing to the dual, factorizes through a \textit{coaction} $\boldsymbol{\Delta}$ as we will see in $\S 2.4$. This coaction, which is given by an explicit combinatorial formula (Theorem $(\ref{eq:coaction})$, [Goncharov, Brown]), is the keystone of this PhD. In particular, it enables us to prove linear independence of MMZV, as in the theorem stated below (instead of adding yet another identity to the existing zoo of relations between MZV), and to study Galois descents. From this, we deduce results about numbers by applying the period map.
\\
\\
This thesis is structured as follows:
\begin{description}
\item[Chapter $2$] sketches the background necessary to understand this work, from Mixed Tate Motives to the Hopf algebra of motivic multiple zeta values at $\mu_{N}$, with some specifications according the values of $N$, and results used throughout the rest of this work. The combinatorial expression of the coaction (or of the weight graded \textit{derivation operators} $D_{r}$ extracted from it, $(\ref{eq:Der})$) is the cornerstone of this work. We shall also bear in mind Theorem $2.4.4$ stating which elements are in the kernel of these derivations), which sometimes allows to lift identities from MZV to motivic MZV, up to rational coefficients, as we will see throughout this work.\\
\texttt{Nota Bene}: A \textit{motivic relation} is indeed stronger; it may hence require several relations between MZV in order to lift an identity to motivic MZV. An example of such a behaviour occurs with some Hoffman $\star$ elements, in Lemma $\ref{lemmcoeff}$.\\
\item[Chapter $3$] explains the main results of this PhD, ending with a wider perspective and possible future paths. 
\item[Chapter $4$] focuses on the cases $N=1$, i.e. multiple zeta values and $N=2$, i.e. Euler sums, providing some new bases: 
\begin{itemize}
\item[$(i)$] First, we introduce \textit{Euler $\sharp$ sums}, variants of Euler sums, defined in  $\S 2.3$ as in $(\ref{eq:reprinteg})$, replacing each $\omega_{\pm 1}$ by $\omega_{\pm \sharp}\mathrel{\mathop:}=2 \omega_{\pm 1}-\omega_{0}$, except for the first one and prove:
\begin{theom}
Motivic Euler $\sharp$ sums with only positive $odd$ and negative $even$ integers as arguments are unramified: i.e. motivic multiple zeta values. 
\end{theom}
By application of the period map above:
\begin{corol}
Each Euler $\sharp$ sums with only positive $odd$ and negative $even$ integers as arguments is unramified, i.e. $\mathbb{Q}$ linear combination of multiple zeta values.
\end{corol}
Moreover, we can extract a basis from this family:
\begin{theom}
$ \lbrace \zeta^{\sharp, \mathfrak{m}} \left( 2a_{0}+1, 2a_{1}+3, \ldots, 2a_{p-1}+3 , -(2 a_{p}+2)\right) , a_{i} \geq 0  \rbrace$ is a graded basis of the space of motivic multiple zeta values.
\end{theom}
By application of the period map:
\begin{corol}
Each multiple zeta value is a $\mathbb{Q}$ linear combination of elements of the same weight in $\lbrace \zeta^{ \sharp} \left( 2a_{0}+1, 2a_{1}+3, \ldots, 2a_{p-1}+3 , -(2 a_{p}+2)\right) , a_{i} \geq 0 \rbrace$.
\end{corol}
\item[$(ii)$] We also prove the following, where Euler $\star$ sums are defined (cf. $\S 2.3$) as in $(\ref{eq:reprinteg})$, replacing each $\omega_{\pm 1}$ by $\omega_{\pm \star}\mathrel{\mathop:}=\omega_{\pm 1}-\omega_{0}$, except the first:
\begin{theom}
If the analytic conjecture ($\ref{conjcoeff}$) holds, then the motivic \textit{Hoffman} $\star$ family $\lbrace \zeta^{\star,\mathfrak{m}} (\lbrace 2,3 \rbrace^{\times})\rbrace$ is a basis of $\mathcal{H}^{1}$, the space of MMZV.
\end{theom}
\item[$(iii)$] Conjecturally, the two previous basis, namely the Hoffman $^{\star}$ family and the Euler$^{\sharp}$ family, are the same. Indeed, we conjecture a generalized motivic Linebarger-Zhao equality (Conjecture $\ref{lzg}$) which expresses each motivic multiple zeta $\star$ value as a motivic Euler $\sharp$ sum. It extends the Two One formula [Ohno-Zudilin], the Three One Formula [Zagier], and Linebarger Zhao formula, and applies to motivic MZV. If this conjecture holds, then $(i)$ implies that the Hoffman$^{\star}$ family is a basis.
\end{itemize}
Such results on linear independence of a family of motivic MZV are proved recursively, once we have found the \textit{appropriate level} filtration on the elements; ideally, the family considered is stable under the derivations \footnote{If the family is not \textit{a priori} \textit{stable} under the coaction, we need to incorporate in the recursion an hypothesis on the coefficients which appear when we express the right side with the elements of the family.}; the filtration, as we will see below, should correspond to the \textit{motivic depth} defined in $\S 2.4.3$, and decrease under the derivations \footnote{In the case of Hoffman basis ($\cite{Br2}$), or Hoffman $\star$ basis (Theorem $4.4.1$) it is the number of $3$, whereas in the case of Euler $\sharp$ sums basis (Theorems $4.3.2$), it is the depth minus one; for the \textit{Deligne} basis given in Chapter $5$ for $N=2,3,4, \mlq  6 \mrq ,8$, it is the usual depth. The filtration by the level has to be stable under the coaction, and more precisely, the derivations $D_{r}$ decrease the level on the elements of the conjectured basis, which allows a recursion.}; if the derivations, modulo some spaces, act as a deconcatenation on these elements, linear independence follows naturally from this recursion. Nevertheless, to start this procedure, we need an analytic identity\footnote{Where F. Brown in $\cite{Br2}$, for the Hoffman basis, used the analytic identity proved by Zagier in $\cite{Za}$, or $\cite{Li}$.}, which is left here as a conjecture in the case of the Hoffman $\star $ basis. This conjecture is of an entirely different nature from the techniques developed in this thesis. We expect that it could be proved using analytic methods along the lines of $\cite{Za}, \cite{Li}$.\\
\item[Chapter $5$] applies ideas of Galois descents on the motivic side. Originally, the notion of Galois descent was inspired by the question: which linear combinations of Euler sums are \textit{unramified}, i.e. multiple zeta values?\footnote{This was already a established question, studied by Broadhurst (which uses the terminology \textit{honorary}) among others. Notice that this issue surfaces also for motivic Euler sums in some results in Chapter $3$ and $5$.}  More generally, looking at the motivic side, one can ask which linear combinations of MMZV$_{\mu_{N}}$ lie in MMZV$_{\mu_{N'}}$ for $N'$ dividing $N$. This is what we call \textit{descent} (the first level of a descent) and can be answered by exploiting the motivic Galois group action. General descent criteria are given; in the particular case of $N=2,3,4,\mlq 6 \mrq,8$\footnote{\texttt{Nota Bene}: $N=\mlq 6 \mrq$ is a special case; the quotation marks indicate here that we restrict to \textit{unramified} MMZV cf. $\S 2.1.1$.}, Galois descents are made explicit and our results lead to new bases of MMZV relative to $\mu_{N'}$ in terms of a basis of MMZV relative to $\mu_{N}$, and in particular, a new proof of P. Deligne's results $\cite{De}$.\\
Going further, we define ramification spaces which constitute a tower of intermediate spaces between the elements in MMZV$_{\mu_{N}}$ and the whole space of MMZV$_{\mu_{N'}}$. This is summed up in $\S 3.2$ and studied in detail Chapter $5$ or article $\cite{Gl1}$.\\
Moreover, as we will see below, these methods enable us to construct the motivic periods of categories of mixed Tate motives which cannot be reached by standard methods: i.e. are not simply generated by a motivic fundamental group.
\item[Chapter $6$] gathers some applications of the coaction, from maximal depth terms, to motivic identities, via unramified motivic Euler sums; other potential applications of these Galois ideas to the study of these periods are still waiting to be further investigated. \\
\\
\\
\end{description}
\texttt{\textbf{Consistency:}}\\
Chapter $2$ is fundamental to understand the tools and the proofs of both Chapter $4$, $5$ and $6$ (which are independent between them), but could be skimmed through before the reading of the main results in Chapter $3$. The proofs of Chapter $4$ are based on the results of Annexe $A.1$, but could be read independently.

\chapter{Background} 

\section{Motives and Periods}

Here we sketch the motivic background where the motivic iterated integrals (and hence this work) mainly take place; although most of it can be taken as a black box. Nevertheless, some of the results coming from this rich theory are fundamental to our proofs. 

\subsection{Mixed Tate Motives}

\paragraph{Motives in a nutshell.}
Motives are supposed to play the role of a universal (and algebraic) cohomology theory (see $\cite{An}$). This hope is partly nourished by the fact that, between all the classical cohomology theories (de Rham, Betti, $l$-adique, crystalline), we have comparison isomorphisms in characteristic $0$ \footnote{Even in positive characteristic, $\dim H^{i}(X)$ does not depend on the cohomology chosen among these.}. More precisely, the hope is that there should exist a tannakian (in particular abelian, tensor) category of motives $\mathcal{M}(k)$, and a functor $\text{Var}_{k} \xrightarrow{h} \mathcal{M}(k) $ such that:\\ 
\textit{For each Weil cohomology}\footnote{This functor should verify some properties, such as Kunneth formula, Poincare duality, etc. as the classic cohomology theories. \\
If we restrict to smooth projective varieties, $\text{SmProj}_{k}$, we can construct such a category, the category of pure motives $ \mathcal{M}^{pure}(k)$ starting from the category of correspondence of degree $0$. For more details, cf. $\cite{Ka}$.}: $\text{Var}_{k} \xrightarrow{H} \text{Vec}_{k}$, \textit{there exists a realization map $w_{H}$ such that the following commutes}:\nomenclature{$\text{Var}_{k}$}{the category of varieties over k}\nomenclature{$\text{Vec}_{k}$}{the category of $k$-vector space of finite dimension}\nomenclature{$\text{SmProj}_{k}$}{the category of smooth projective varieties over k.}
$$\xymatrix{
\text{Var}_{k} \ar[d]^{\forall H} \ar[r]^{ h} &  \mathcal{M}(k) \ar[dl]^{\exists  w_{H}}\\
\text{Vec}_{K}   &   },$$
\textit{where $h$ satisfy properties such as }$h(X\times Y)=h(X)\oplus h(Y)$, $h(X \coprod Y)= h(X)\otimes h(Y)$. The realizations functors are conjectured to be full and faithful (conjecture of periods of Grothendieck, Hodge conjecture, Tate conjecture, etc.)\footnote{In the case of Mixed Tate Motives over number fields as seen below, Goncharov proved it for Hodge and l-adique Tate realizations, from results of Borel and Soule.}.\\
To this end, Voedvosky (cf. $\cite{Vo}$) constructed a triangulated category of Mixed Motives $DM^{\text{eff}}(k)_{\mathbb{Q}}$, with rational coefficients, equipped with tensor product and a functor:
\begin{center}
$M_{gm}: \text{Sch}_{\diagup k} \rightarrow DM^{\text{eff}}$   satisfying some properties such as:
\end{center}
\begin{description}
\item[Kunneth] $M_{gm}(X \times Y)=M_{gm}(X)\otimes M_{gm}(Y)$.
\item[$\mathbb{A}^{1}$-invariance]  $M_{gm}(X \times \mathbb{A}^{1})= M_{gm}(X)$.
\item[Mayer Vietoris] $M_{gm}(U\cap V)\rightarrow M_{gm}(U) \otimes M_{gm}(V) \rightarrow M_{gm}(U\cup V)\rightarrow M_{gm}(U\cap V)[1] $, $U,V$ open, is a distinguished triangle.\footnote{Distinguished triangles in $DMT^{\text{eff}}(k)$, i.e. of type Tate, become exact sequences in $\mathcal{MT}(k)$.}
\item[Gysin] $M_{gm}(X\diagdown Z)\rightarrow M_{gm}(X)  \rightarrow M_{gm}(Z)(c)[2c]\rightarrow M_{gm}(X\diagdown Z)[1] $, $X$ smooth, $Z$ smooth, closed, of codimension $c$, is a distinguished triangle.
\end{description}
We would like to extract from the triangulated category $DM^{\text{eff}}(k)_{\mathbb{Q}}$ an abelian category of Mixed Motives over k\footnote{ A way would be to define a $t$ structure on this category, and the heart of the t-structure, by Bernstein, Beilinson, Deligne theorem is a full admissible abelian sub-category.}. However, we still are not able to do it in the general case, but it is possible for some triangulated tensor subcategory of type Tate, generated by $\mathbb{Q}(n)$ with some properties.\\
\\
\textsc{Remark}: $\mathbb{L}\mathrel{\mathop:}=\mathbb{Q}(-1)= H^{1}(\mathbb{G}_{\mathfrak{m}})=H^{1}(\mathbb{P}^{1} \diagdown \lbrace 0, \infty\rbrace)$\nomenclature{$\mathbb{G}_{\mathfrak{m}}$}{the multiplicative group} which is referred to as the \textit{Lefschetz motive}, is a pure motive, and has period $(2i\pi)$. Its dual is the so-called \textit{Tate motive} $\mathbb{T}\mathrel{\mathop:}=\mathbb{Q}(1)=\mathbb{L}^{\vee}$. More generally, let us define $\mathbb{Q}(-n)\mathrel{\mathop:}= \mathbb{Q}(-1)^{\otimes n}$ resp. $\mathbb{Q}(n)\mathrel{\mathop:}= \mathbb{Q}(1)^{\otimes n}$ whose periods are in $(2i\pi)^{n} \mathbb{Q}$ resp. $(\frac{1}{2i\pi})^{n} \mathbb{Q}$, hence extended periods in $\widehat{P}$; we have the decomposition of the motive of the projective line: $h(\mathbb{P}^{n})= \oplus _{k=0}^{n}\mathbb{Q}(-k)$.  

\paragraph{Mixed Tate Motives over a number field.}
Let's first define, for $k$ a number field, the category $\mathcal{DM}(k)_{\mathbb{Q}}$ from $\mathcal{DM}^{\text{eff}}(k)_{\mathbb{Q}}$ by formally \say{inverting} the Tate motive $\mathbb{Q}(1)$, and then $\mathcal{DMT}(k)_{\mathbb{Q}}$ as the smallest triangulated full subcategory of $\mathcal{DM}(k)_{\mathbb{Q}}$ containing $\mathbb{Q}(n), n\in\mathbb{Z}$ and stable by extension.\\
By the vanishing theorem of Beilinson-Soule, and results from Levine (cf. $\cite{Le}$), there exists:\footnote{A \textit{tannakian} category is abelian, $k$-linear, tensor rigid (autoduality), has an exact faithful fiber functor, compatible with $\otimes$ structures, etc. Cf. $\cite{DM}$ about Tannakian categories.}
\begin{framed}
A tannakian \textit{category of Mixed Tate motives} over k with rational coefficients, $\mathcal{MT}(k)_{\mathbb{Q}}$\nomenclature{$\mathcal{MT}(k)$}{category of Mixed Tate Motives over $k$} and equipped with a weight filtration $W_{r}$ indexed by even integers such that $gr_{-2r}^{W}(M)$ is a sum of copies of $\mathbb{Q}(r)$ for $M\in \mathcal{MT}(k)$, i.e.,\\
 Every object $M\in \mathcal{MT}(k)_{\mathbb{Q}}$ is an iterated extension of Tate motives $\mathbb{Q}(n), n\in\mathbb{Z}$.
\end{framed}
such that (by the works of Voedvodsky, Levine $\cite{Le}$, Bloch, Borel (and K-theory), cf. $\cite{DG}$):
$$\begin{array}{ll}
\text{Ext}^{1}_{\mathcal{MT}(k)}(\mathbb{Q}(0),\mathbb{Q}(n) )\cong  K_{2n-1}(k)_{\mathbb{Q}} \otimes \mathbb{Q} \cong  & \left\lbrace  \begin{array}{ll}
k^{\ast}\otimes_{\mathbb{Z}} \mathbb{Q}  & \text{ if } n=1 .\\
 \mathbb{Q}^{r_{1}+r_{2}}  & \text{ if } n>1 \text{ odd }\\
\mathbb{Q}^{r_{2}}  & \text{ if } n>1 \text{ even }
\end{array} \right. . \\
\text{Ext}^{i}_{\mathcal{MT}(k)}(\mathbb{Q}(0),\mathbb{Q}(n) )\cong 0 & \quad \text{ if } i>1 \text{ or } n\leq 0.\\
\end{array}$$
Here, $r_{1}$ resp $r_{2}$ stand for the number of real resp. complex (and non real, up to conjugate) embeddings from $k$ to $\mathbb{C}$.\\
In particular, the weight defines a canonical fiber functor\nomenclature{$\omega$}{the canonical fiber functor}:
$$\begin{array}{lll}
\omega: & \mathcal{MT}(k) \rightarrow \text{Vec}_{\mathbb{Q}} &    \\
 &M \mapsto \oplus \omega_{r}(M) & \quad  \quad \text{ with  } \left\lbrace  \begin{array}{l}
  \omega_{r}(M)\mathrel{\mathop:}= \text{Hom}_{\mathcal{MT}(k)}(\mathbb{Q}(r), gr_{-2r}^{W}(M))\\
  \text{ i.e. }  \quad  gr_{-2r}^{W}(M)= \mathbb{Q}(r)\otimes \omega_{r}(M).
\end{array} \right. 
\end{array} $$
\\
The category of Mixed Tate Motives over $k$, since tannakian, is equivalent to the category of representations of the so-called \textit{\textbf{motivic Galois group}} $\boldsymbol{\mathcal{G}^{\mathcal{MT}}}$\nomenclature{$\mathcal{G}^{\mathcal{M}}$}{the motivic Galois group of the category of Mixed Tate Motives $\mathcal{M}$ } of $\mathcal{MT}(k)$ \footnote{With the equivalence of category between $A$ Comodules and Representations of the affine group scheme $\text{Spec}(A)$, for $A$ a Hopf algebra. Note that $\text{Rep}(\mathbb{G}_{m})$ is the category of $k$-vector space $\mathbb{Z}$-graded of finite dimension.}:\nomenclature{$\text{Rep}(\cdot)$}{the category of finite representations}
\begin{framed}
\begin{equation}\label{eq:catrep}
\mathcal{MT}(k)_{\mathbb{Q}}\cong \text{Rep}_{k} \mathcal{G}^{\mathcal{MT}} \cong \text{Comod } (\mathcal{O}(\mathcal{G}^{\mathcal{MT}})) \quad \text{ where } \mathcal{G}^{\mathcal{MT}}\mathrel{\mathop:}=\text{Aut}^{\otimes } \omega . 
\end{equation}
\end{framed}
The motivic Galois group $\mathcal{G}^{\mathcal{MT}}$ decomposes as, since $\omega$ is graded:
\begin{center}
$\mathcal{G}^{\mathcal{MT}}= \mathbb{G}_{m} \ltimes \mathcal{U}^{\mathcal{MT}}$, $\quad  \text{i.e. }  \quad 1 \rightarrow \mathcal{U}^{\mathcal{MT}} \rightarrow \mathcal{G}^{\mathcal{MT}} \leftrightarrows \mathbb{G}_{m} \rightarrow 1 \quad \textit{ is an exact sequence, }$
\end{center}
\begin{center}
\textit{where} $\mathcal{U}^{\mathcal{MT}}$\nomenclature{$\mathcal{U}^{\mathcal{M}}$}{the pro-unipotent part of the motivic Galois group $\mathcal{G}^{\mathcal{M}}$.} \textit{is a pro-unipotent group scheme defined over  }$\mathbb{Q}$.
\end{center}
The action of $\mathbb{G}_{m}$ is a grading, and $\mathcal{U}^{\mathcal{MT}}$ acts trivially on the graded pieces $\omega(\mathbb{Q}(n))$.\\
\\
Let $\mathfrak{u}$ denote the completion of the pro-nilpotent graded Lie algebra of $\mathcal{U}^{\mathcal{MT}}$ (defined by a limit); $\mathfrak{u}$ is free and graded with negative degrees from the $\mathbb{G}_{m}$-action. Furthermore\footnote{Since $\text{Ext}^{2}_{\mathcal{MT}} (\mathbb{Q}(0), \mathbb{Q}(n))=0$, which implies $\forall M$, $H^{2}(\mathfrak{u},M)=0$, hence $\mathfrak{u}$ free. Moreover, $(\mathfrak{u}^{ab})= \left( \mathfrak{u} \diagup [\mathfrak{u}, \mathfrak{u}]\right)= H_{1}(\mathfrak{u}; \mathbb{Q})$, then, for $\mathcal{U}$ unipotent: $$\left( \mathfrak{u}^{ab}\right)^{\vee}_{m-n} \cong \text{Ext}^{1}_{\text{Rep} _{\mathbb{Q}}} (\mathbb{Q}(n), \mathbb{Q}(m)).$$}:
\begin{equation}
\label{eq:uab}
\mathfrak{u}^{ab} \cong \bigoplus \text{Ext}^{1}_{\mathcal{MT}} (\mathbb{Q}(0), \mathbb{Q}(n))^{\vee} \text{ in degree } n.
\end{equation}
Hence the \textit{fundamental Hopf algebra} is\nomenclature{$\mathcal{A}^{\mathcal{M}}$}{the fundamental Hopf algebra of $\mathcal{M}$} \footnote{Recall the anti-equivalence of Category, between Hopf Algebra and Affine Group Schemes:
$$\xymatrix@R-1pc{ k-Alg^{op} \ar[r]^{\sim}   &  k-\text{AffSch }   \\
k-\text{HopfAlg}^{op} \ar@{^{(}->}[u]  \ar[r]^{\sim} & k-\text{ AffGpSch } \ar@{^{(}->}[u]\\
A  \ar@{|->}[r] & \text{ Spec } A \\
\mathcal{O}(G)  & G: R \mapsto \text{Hom}_{k}(\mathcal{O}(G),R)  \ar@{|->}[l] }. $$
It comes from the fully faithful Yoneda functor $C^{op} \rightarrow \text{Fonct}(C, \text{Set})$, leading to an equivalence of Category if we restrict to Representable Functors: $k-\text{AffGpSch } \cong \text{RepFonct } (\text{Alg }^{op},Gp)$. Properties for Hopf algebra are obtained from Affine Group Scheme properties by 'reversing the arrows' in each diagram.\\
Remark that $G$ is unipotent if and only if $A$ is commutative, finite type, connected and filtered.}:
\begin{equation}
\label{eq:Amt}
\mathcal{A}^{\mathcal{MT}}\mathrel{\mathop:}=\mathcal{O}(\mathcal{U}^{\mathcal{MT}})\cong (U^{\wedge} (\mathfrak{u}))^{\vee} \cong T(\oplus_{n\geq 1} \text{Ext}^{1}_{\mathcal{MT}_{N}} (\mathbb{Q}(0), \mathbb{Q}(n))^{\vee} ).
\end{equation}
Hence, by the Tannakian dictionary $(\ref{eq:catrep})$: $\mathcal{MT}(k)_{\mathbb{Q}}\cong \text{Rep}_{k}^{gr} \mathcal{U}^{\mathcal{MT}} \cong \text{Comod}^{gr} \mathcal{A}^{\mathcal{MT}} .$ \\
\\
Once an embedding $\sigma: k \hookrightarrow \mathbb{C}$ is fixed, Betti cohomology leads to a functor \textit{Betti realization}:\nomenclature{$\omega_{B_{\sigma}}$}{Betti realization functor}
$$\omega_{B_{\sigma}}: \mathcal{MT}(k) \rightarrow \text{Vec}_{\mathbb{Q}} ,\quad  M \mapsto M_{\sigma}.$$
De Rham cohomology leads similarly to the functor \textit{de Rham realization}\nomenclature{$\omega_{dR}$}{de Rham realization functor}:
$$\omega_{dR}: \mathcal{MT}(k) \rightarrow \text{Vec}_{k} , \quad M \mapsto M_{dR} \text{ , } \quad M_{dR} \text{ weight graded}.$$
Beware, the de Rham functor $\omega_{dR}$ here is not defined over $\mathbb{Q}$ but over $k$ and $\omega_{dR}=\omega \otimes_{\mathbb{Q}} k$, so the de Rham realization of an object $M$ is $M_{dR}=\omega(M)\otimes_{\mathbb{Q}} k$.\\
Between all these realizations, we have comparison isomorphisms, such as: 
$$ M_{\sigma}\otimes_{\mathbb{Q}} \mathbb{C} \xrightarrow[\sim]{\text{comp}_{dR, \sigma}}  M_{dR} \otimes_{k,\sigma} \mathbb{C} \text{ with its inverse } \text{comp}_{\sigma,dR}.$$
$$ M_{\sigma}\otimes_{\mathbb{Q}} \mathbb{C} \xrightarrow[\sim]{\text{comp}_{\omega, B_{\sigma}}}  M_{\omega} \otimes_{\mathbb{Q}} \mathbb{C} \text{ with its inverse } \text{comp}_{B_{\sigma}, \omega}.$$
Define also, looking at tensor-preserving isomorphisms:
$$\begin{array}{lll}
\mathcal{G}_{B}\mathrel{\mathop:}=\text{Aut}^{\otimes}(\omega_{B}), & \text{resp. } \mathcal{G}_{dR}\mathrel{\mathop:}=\text{Aut}^{\otimes}(\omega_{dR})  & \\
P_{\omega, B}\mathrel{\mathop:}=\text{Isom}^{\otimes}(\omega_{B},\omega), &  \text{resp. } P_{B,\omega}\mathrel{\mathop:}=\text{Isom}^{\otimes}(\omega,\omega_{B}), & (\mathcal{G}^{\mathcal{MT}}, \mathcal{G}_{B}) \text{ resp. } (\mathcal{G}_{B}, \mathcal{G}^{\mathcal{MT}}) \text{ bitorsors }.
\end{array}$$
Comparison isomorphisms above define $\mathbb{C}$ points of these schemes: $\text{comp}_{\omega,B}\in P_{B,\omega} (\mathbb{C})$.\\
\\
\textsc{Remarks}: By $(\ref{eq:catrep})$:\footnote{The different cohomologies should be viewed as interchangeable realizations. Etale chomology, with the action of the absolute Galois group $\text{Gal}(\overline{\mathbb{Q}}\diagup\mathbb{Q})$ (cf $\cite{An3}$) is related to the number $N_{p}$ of points of reduction modulo $p$. For Mixed Tate Motives (and conjecturally only for those) $N_{p}$ are polynomials modulo $p$, which is quite restrictive.}
\begin{framed}
A Mixed Tate motive over a number field is uniquely defined by its de Rham realization, a vector space $M_{dR}$, with an action of the motivic Galois group $\mathcal{G}^{\mathcal{MT}}$.
\end{framed}

\texttt{Example:} For instance $\mathbb{Q}(n)$, as a Tate motive, can be seen as the vector space $\mathbb{Q}$ with the action $\lambda\cdot x\mathrel{\mathop:}= \lambda^{n} x$, for $\lambda\in\mathbb{Q}^{\ast}=\text{Aut}(\mathbb{Q})=\mathbb{G}_{m}(\mathbb{Q})$.

\paragraph{Mixed Tate Motives over $\mathcal{O}_{S}$.}

Before, let's recall for $k$ a number field and $\mathcal{O}$\nomenclature{$\mathcal{O}_{k}$}{ring of integers of $k$} its ring of integers, \textit{archimedian values} of $k$ are associated to an embedding $k \xhookrightarrow{\sigma} \mathbb{C}$, such that:
$$\mid x \mid \mathrel{\mathop:}= \mid\sigma(x)\mid_{\infty}\text{ , where } \mid\cdot \mid_{\infty} \text{ is the usual absolute value},$$ 
and \textit{non archimedian values} are associated to non-zero prime ideals of $\mathcal{O}$\footnote{$ \mathcal{O}$ is a Dedekind domain, $ \mathcal{O}_{\mathfrak{p}}$ a discrete valuation ring whose prime ideals are prime ideals of $\mathcal{O}$ which are included in $(\mathfrak{p})\mathcal{O}_{\mathfrak{p}}$.}:
$$v_{\mathfrak{p}}: k^{\times} \rightarrow \mathbb{Z}, \quad v_{\mathfrak{p}}(x) \text{ is the integer such that  }  x \mathcal{O}_{\mathfrak{p}} = \mathfrak{p}^{v_{\mathfrak{p}}(x)} \mathcal{O}_{\mathfrak{p}} \text{ for } x\in k^{\times}.$$
For $S$ a finite set of absolute values in $k$ containing all archimedian values, the \textit{ring of S-integers}\nomenclature{$\mathcal{O}_{S}$}{ring of $S$-integers}:
$$\mathcal{O}_{S}\mathrel{\mathop:}= \left\lbrace x\in k \mid v(x)\geq 0 \text{  for all valuations } v\notin S \right\rbrace. $$
Dirichlet unit's theorem generalizes for $\mathcal{O}^{\times}_{S}$, abelian group of type finite:\nomenclature{$\mu(K)$}{the finite cyclic group of roots of unity in $K$} \footnote{ It will be used below, for dimensions, in $\ref{dimensionk}$. Here, $\text{ card } (S)= r_{1}+r_{2}+\text{ card }(\text{non-archimedian places}) $; as usual, $r_{1}, r_{2}$ standing for the number of real resp. complex (and non real, and up to conjugate) embeddings from $k$ to $\mathbb{C}$; $\mu(K) $ is the finite cyclic group of roots of unity in $K$.}
$$\mathcal{O}_{S}^{\times} \cong \mu(K) \times \mathbb{Z}^{\text{ card }(S)-1}.$$
\texttt{Examples}:
\begin{itemize}
\item[$\cdot$] Taking S as the set of the archimedian values leads to the usual ring of integers $\mathcal{O}$, and would lead to the unramified category of motives $\mathcal{MT}(\mathcal{O})$ below.
\item[$\cdot$] For $k=\mathbb{Q}$, $p$ prime, with $S=\lbrace v_{p}, \mid \cdot \mid_{\infty}\rbrace$, we obtain $\mathbb{Z}\left[  \frac{1}{p} \right] $. Note that the definition does not allow to choose $S= \lbrace \cup_{q \text{prime} \atop q\neq p } v_{q}, \mid \cdot \mid_{\infty}\rbrace$, which would lead to the localization $\mathbb{Z}_{(p)}\mathrel{\mathop:}=\lbrace x\in\mathbb{Q} \mid v_{p}(x) \geq 0\rbrace$.
\end{itemize}
Now, let us define the categories of Mixed Tate Motives which interest us here:\nomenclature{$\mathcal{MT}_{\Gamma}$}{a tannakian category associated to $\Gamma$ a sub-vector space of $\text{Ext}^{1}_{\mathcal{MT}(k)}( \mathbb{Q}(0), \mathbb{Q}(1))$}
\begin{defin}\label{defimtcat}
\begin{description}
\item[$\boldsymbol{\mathcal{MT}_{\Gamma}}$:] For $\Gamma$ a sub-vector space of $\text{Ext}^{1}_{\mathcal{MT}(k)}( \mathbb{Q}(0), \mathbb{Q}(1)) \cong k^{\ast}\otimes \mathbb{Q}$:
\begin{center}
$\mathcal{MT}_{\Gamma}:$ the tannakian subcategory formed by objects $M$ such that each subquotient $E$ of $M$:
$$0 \rightarrow \mathbb{Q}(n+1) \rightarrow E \rightarrow \mathbb{Q}(n) \rightarrow 0  \quad \Rightarrow \quad [E]\in \Gamma \subset  \text{Ext}^{1}_{\mathcal{MT}(k)}( \mathbb{Q}(0), \mathbb{Q}(1)) \footnote{$\text{Ext}^{1}_{\mathcal{MT}(k)}( \mathbb{Q}(0), \mathbb{Q}(1)) \cong \text{Ext}^{1}_{\mathcal{MT}(k)}( \mathbb{Q}(n), \mathbb{Q}(n+1)).$}.$$
\end{center}

\item[$\boldsymbol{\mathcal{MT}(\mathcal{O}_{S})}$:] The category of mixed Tate motives unramified in each finite place $v\notin S$:
\begin{center}
$\mathcal{MT}(\mathcal{O}_{S})\mathrel{\mathop:}=\mathcal{MT}_{\Gamma}, \quad $ for $\Gamma=\mathcal{O}_{S}^{\ast}\otimes \mathbb{Q}$. 
\end{center}
\end{description}
\end{defin}
Extension groups for these categories are then identical to those of $\mathcal{MT}(k)$ except:

\begin{equation} \label{eq:extension}
\text{Ext}^{1}_{\mathcal{MT}_{\Gamma}}( \mathbb{Q}(0), \mathbb{Q}(1))= \Gamma, \quad  \text{   resp.    } \quad 
\text{Ext}^{1}_{\mathcal{MT}(\mathcal{O}_{S})}( \mathbb{Q}(0), \mathbb{Q}(1))= K_{1}(\mathcal{O}_{S})\otimes \mathbb{Q}=\mathcal{O}^{\ast}_{S}\otimes \mathbb{Q}.
\end{equation}

\paragraph{Cyclotomic Mixed Tate Motives.}
In this thesis, we focus on the cyclotomic case and consider the following categories, and sub-categories, for $k_{N}$\nomenclature{$k_{N}$}{the $N^{\text{th}}$ cyclotomic field} the $N^{\text{th}}$ cyclotomic field, $\mathcal{O}_{N}\mathrel{\mathop:}= \mathbb{Z}[\xi_{N}]$\nomenclature{$\mathcal{O}_{N}$}{the ring of integers of $k_{N}$ i.e. $ \mathbb{Z}[\xi_{N}]$ } its ring of integers, with $\xi_{N}$\nomenclature{$\xi_{N}$}{a primitive $N^{\text{th}}$ root of unity} a primitive $N^{\text{th}}$ root of unity:\nomenclature{$\mathcal{MT}_{N,M}$}{the tannakian Mixed Tate subcategory of $\mathcal{MT}(k_{N})$ ramified in $M$}
\begin{framed}
$$\begin{array}{ll}
\boldsymbol{\mathcal{MT}_{N,M}} &  \mathrel{\mathop:}= \mathcal{MT} \left( \mathcal{O}_{N} \left[ \frac{1}{M}\right] \right).\\
\boldsymbol{\mathcal{MT}_{\Gamma_{N}}}, & \text{ with $\Gamma_{N}$ the $\mathbb{Q}$-sub vector space of } \left( \mathcal{O}\left[ \frac{1}{N} \right] \right) ^{\ast} \otimes \mathbb{Q}\\
&   \text{ generated by $\lbrace 1-\xi^{a}_{N}\rbrace_{0< a < N}$ (modulo torsion). }
\end{array}$$
\end{framed}
\noindent
Hence:
$$ \mathcal{MT} \left( \mathcal{O}_{N}  \right) \subsetneq \mathcal{MT}_{\Gamma_{N}}  \subset \mathcal{MT} \left( \mathcal{O}_{N} \left[ \frac{1}{N}\right] \right)$$
The second inclusion is an equality if and only if $N$ has all its prime factors inert\footnote{I.e. each prime $p$ dividing $N$, generates $(\mathbb{Z} /m \mathbb{Z})^{\ast}$, for $m$ such as $N=p^{v_{p}(N)} m$.\nomenclature{$v_{p}(\cdot)$}{$p$-adic valuation} It could occur only in the following cases: $N= p^{s}, 2p^{s}, 4p^{s}, p^{s}q^{k}$, with extra conditions in most of these cases such as: $2$ is a primitive root modulo $p^{s}$ etc.}, since:
\begin{equation}\label{eq:gamma}
\Gamma_{N}= \left\lbrace \begin{array}{ll}
\left( \mathcal{O}\left[ \frac{1}{p} \right] \right) ^{\ast} \otimes \mathbb{Q} & \text{ if } N = p^{r} \\
(\mathcal{O} ^{\ast} \otimes \mathbb{Q} )\oplus \left(  \oplus_{ p \text{ prime } \atop p\mid N} \langle p \rangle\otimes \mathbb{Q} \right)   &\text{ else .}
\end{array} \right. .
\end{equation}
The motivic cyclotomic MZVs lie in the subcategory $\mathcal{MT}_{\Gamma_{N}}$, as we will see more precisely in $\S 2.3$.\\
\\
\texttt{Notations:} We may sometimes drop the $M$ (or even $N$), to lighten the notations:\footnote{For instance, $\mathcal{MT}_{3}$ is the category $\mathcal{MT} \left( \mathcal{O}_{3} \left[  \frac{1}{3} \right]  \right) $.}:\nomenclature{$\mathcal{MT}_{N}$}{a tannakian Mixed Tate subcategory of $\mathcal{MT}(k_{N})$}\\
$$\mathcal{MT}_{N}\mathrel{\mathop:}= \left\lbrace \begin{array}{ll}
\mathcal{MT}_{N,N} & \text{ if } N=2,3,4,8\\
\mathcal{MT}_{6,1} & \text{ if } N=\mlq 6 \mrq. \\
\end{array} \right. $$

\subsection{Motivic periods}
Let $\mathcal{M}$ a tannakian category of mixed Tate motives. Its \textit{algebra of motivic periods} is defined as (cf. $\cite{D1}$, $\cite{Br6}$, and $\cite{Br4}$, $\S 2$):\nomenclature{$\mathcal{P}_{\mathcal{M}}^{\mathfrak{m}}$}{the algebra of motivic period of a tannakian category of MTM $\mathcal{M}$}
$$\boldsymbol{\mathcal{P}_{\mathcal{M}}^{\mathfrak{m}}}\mathrel{\mathop:}=\mathcal{O}(\text{Isom}^{\otimes}_{\mathcal{M}}(\omega, \omega_{B}))=\mathcal{O}(P_{B,\omega}).$$
\begin{framed}
A \textbf{\textit{motivic period}} denoted as a triplet  $\boldsymbol{\left[M,v,\sigma \right]^{\mathfrak{m}}}$,\nomenclature{$\left[M,v,\sigma \right]^{\mathfrak{m}}$}{a motivic period} element of $\mathcal{P}_{\mathcal{M}}^{\mathfrak{m}}$,  is constructed from a motive $M\in \text{ Ind } (\mathcal{M})$, and classes $v\in\omega(M)$, $\sigma\in\omega_{B}(M)^{\vee}$. It is a function $P_{B,\omega} \rightarrow \mathbb{A}^{1}$, which, on its rational points, is given by:
\begin{equation}\label{eq:mper}  \quad  P_{B,\omega} (\mathbb{Q}) \rightarrow \mathbb{Q}\text{  ,  } \quad  \alpha \mapsto \langle \alpha(v), \sigma\rangle .
\end{equation} 
Its \textit{period} is obtained by the evaluation on the complex point $\text{comp}_{B, dR}$:
\begin{equation}\label{eq:perm} 
\begin{array}{lll}
\mathcal{P}_{\mathcal{M}}^{\mathfrak{m}} & \rightarrow & \mathbb{C} \\
\left[M,v,\sigma \right]^{\mathfrak{m}} & \mapsto &  \langle \text{comp}_{B,dR} (v\otimes 1), \sigma \rangle .
\end{array}
\end{equation}
\end{framed}
\noindent
\texttt{Example}: The first example is the \textit{Lefschetz motivic period}: $\mathbb{L}^{\mathfrak{m}}\mathrel{\mathop:}=[H^{1}(\mathbb{G}_{m}), [\frac{dx}{x}], [\gamma_{0}]]^{\mathfrak{m}}$, period  of the Lefschetz motive $\mathbb{L}$; it can be seen as the \textit{motivic} $(2 i\pi)^{\mathfrak{m}}$; this notation appears below.\nomenclature{$\mathbb{L}$}{Lefschetz motive, $\mathbb{L}^{\mathfrak{m}}$ the Lefschetz motivic period}\\
\\
This construction can be generalized for any pair of fiber functors $\omega_{1}$, $\omega_{2}$ leading to:
\begin{center}
\textit{Motivic periods of type} $(\omega_{1},\omega_{2})$, which are in the following algebra of motivic periods: 
$$\boldsymbol{\mathcal{P}_{\mathcal{M}}^{\omega_{1},\omega_{2}}}\mathrel{\mathop:}= \mathcal{O}\left( P_{\omega_{1},\omega_{2}}\right) = \mathcal{O}\left( \text{Isom}^{\otimes}(\omega_{2}, \omega_{1})\right).$$
\end{center}
\textsc{Remarks}:
\begin{itemize}
\item[$\cdot$]  The groupoid structure (composition) on the isomorphisms of fiber functors on $\mathcal{M}$, by dualizing, leads to a coalgebroid structure on the spaces of motivic periods:
$$ \mathcal{P}_{\mathcal{M}}^{\omega_{1},\omega_{3}} \rightarrow \mathcal{P}_{\mathcal{M}}^{\omega_{1},\omega_{2}} \otimes \mathcal{P}_{\mathcal{M}}^{\omega_{2},\omega_{3}}.$$
\item[$\cdot$] Any structure carried by these fiber functors (weight grading on $\omega_{dR}$, complex conjugation on $\omega_{B}$, etc.) is transmitted to the corresponding ring of periods.
\end{itemize}
\texttt{Examples}: 
\begin{itemize}
\item[$\cdot$ ] For $(\omega,\omega_{B})$, it comes down to (our main interest) $\mathcal{P}_{\mathcal{M}}^{\mathfrak{m}}$ as defined in $(\ref{eq:mper})$. By the last remark, $\mathcal{P}_{\mathcal{M}}^{\mathfrak{m}}$ inherits a weight grading and we can define (cf. $\cite{Br4}$, $\S 2.6$):\nomenclature{$\mathcal{P}_{\mathcal{M}}^{\mathfrak{m},+}$}{the ring of geometric motivic periods of $\mathcal{M}$}
\begin{framed}
\begin{center}
$\boldsymbol{\mathcal{P}_{\mathcal{M}}^{\mathfrak{m},+}} \subset \mathcal{P}_{\mathcal{M}}^{\mathfrak{m}}$, the ring of \textit{geometric periods}, is generated by periods of motives with non-negative weights:  $\left\lbrace  \left[M,v,\sigma \right]^{\mathfrak{m}}\in  \mathcal{P}_{\mathcal{M}}^{\mathfrak{m}} \mid W_{-1} M=0 \right\rbrace $.
\end{center}
\end{framed}
\item[$\cdot$] The ring of periods of type $(\omega,\omega)$ is $\mathcal{P}_{\mathcal{M}}^{\omega}\mathrel{\mathop:}= \mathcal{O} \left( \text{Aut}^{\otimes}(\omega)\right)= \mathcal{O} \left(\mathcal{G}^{\mathcal{MT}}\right)$.\footnote{ In the case of a mixed Tate category over $\mathbb{Q}$, as $\mathcal{MT}(\mathbb{Z})$, this is equivalent to the \textit{De Rham periods} in $\mathcal{P}_{\mathcal{M}}^{\mathfrak{dR}}\mathrel{\mathop:}= \mathcal{O} \left( \text{Aut}^{\otimes}(\omega_{dR})\right)$, defined in $\cite{Br4}$; however, for other cyclotomic fields $k$ considered later ($N>2$), we have to consider the canonical fiber functor, since it is defined over $k$.} \\
\textit{Unipotent variants} of these periods are defined when restricting to the unipotent part $\mathcal{U}^{\mathcal{MT}}$ of $\mathcal{G}^{\mathcal{MT}}$, and appear below (in $\ref{eq:intitdr}$):\nomenclature{$\mathcal{P}_{\mathcal{M}}^{\mathfrak{a}}$}{the ring of unipotent periods}
$$\boldsymbol{\mathcal{P}_{\mathcal{M}}^{\mathfrak{a}}}\mathrel{\mathop:}=\mathcal{O} \left( \mathcal{U}^{\mathcal{MT}}\right)= \mathcal{A}^{\mathcal{MT}}, \quad \text{ the fundamental Hopf algebra}.$$
They correspond to the notion of framed objects in mixed Tate categories, cf. $\cite{Go1}$. By restriction, there is a map:
$$ \mathcal{P}_{\mathcal{M}}^{\omega} \rightarrow  \mathcal{P}_{\mathcal{M}}^{\mathfrak{a}}.$$

\end{itemize} 
 By the remark above, there is a coaction:
$$\boldsymbol{\Delta^{\mathfrak{m}, \omega}}:\mathcal{P}_{\mathcal{M}}^{\mathfrak{m}}  \rightarrow \mathcal{P}_{\mathcal{M}}^{\omega} \otimes \mathcal{P}_{\mathcal{M}}^{\mathfrak{m}}.$$
Moreover, composing this coaction by the augmentation map $\epsilon: \mathcal{P}_{\mathcal{M}}^{\mathfrak{m},+} \rightarrow (\mathcal{P}_{\mathcal{M}}^{\mathfrak{m},+})_{0} \cong \mathbb{Q}$, leads to the morphism (details in $\cite{Br4}$, $\S 2.6$):
\begin{equation}\label{eq:projpiam}
\boldsymbol{\pi_{\mathfrak{a},\mathfrak{m}}}: \quad \mathcal{P}_{\mathcal{M}}^{\mathfrak{m},+}  \rightarrow \mathcal{P}_{\mathcal{M}}^{\mathfrak{a}}, 
\end{equation}
which is, on periods of a motive $M$ such that $W_{-1} M=0$:  $\quad \left[M,v,\sigma \right]^{\mathfrak{m}}  \rightarrow \left[M,v,^{t}c(\sigma) \right]^{\mathfrak{a}},  $
$$ \text{ where } c \text{ is defined as the composition}: \quad M_{\omega} \twoheadrightarrow gr^{W}_{0}M_{\omega}= W_{0} M_{\omega} \xrightarrow{\text{comp}_{B,\omega}}   W_{0} M_{B}  \hookrightarrow M_{B} .$$
Bear in mind also the non-canonical isomorphisms, compatible with weight and coaction ($\cite{Br4}$, Corollary $2.11$) between those $\mathbb{Q}$ algebras:
\begin{equation} \label{eq:periodgeom}
 \mathcal{P}_{\mathcal{M}}^{\mathfrak{m}} \cong \mathcal{P}_{\mathcal{M}}^{\mathfrak{a}} \otimes_{\mathbb{Q}} \mathbb{Q} \left[ (\mathbb{L}^{\mathfrak{m}} )^{-1} ,\mathbb{L}^{\mathfrak{m}} \right], \quad \text{and} \quad \mathcal{P}_{\mathcal{M}}^{\mathfrak{m},+} \cong \mathcal{P}_{\mathcal{M}}^{\mathfrak{a}} \otimes_{\mathbb{Q}} \mathbb{Q} \left[ \mathbb{L}^{\mathfrak{m}} \right]. 
\end{equation}
In particular, $\pi_{\mathfrak{a},\mathfrak{m}}$ is obtained by sending $\mathbb{L}^{\mathfrak{m}}$ to $0$.\\
\\
In the case of a category of mixed Tate motive $\mathcal{M}$ defined over $\mathbb{Q}$, \footnote{As, in our concerns, $\mathcal{MT}_{N}$ above with $N=1,2$; in these exceptional (real) cases, we want to keep track of only even Tate twists.} the complex conjugation defines the \textit{real Frobenius} $\mathcal{F}_{\infty}: M_{B} \rightarrow M_{B}$, and induces an involution on motivic periods $\mathcal{F}_{\infty}: \mathcal{P}_{\mathcal{M}}^{\mathfrak{m}} \rightarrow\mathcal{P}_{\mathcal{M}}^{\mathfrak{m}}$. Furthermore, $\mathbb{L}^{\mathfrak{m}}$ is anti invariant by $\mathcal{F}_{\infty}$ (i.e.  $\mathcal{F}_{\infty} (\mathbb{L}^{\mathfrak{m}})=-\mathbb{L}^{\mathfrak{m}}$). Then, let us define:\nomenclature{$\mathcal{F}_{\infty}$}{real Frobenius}
\begin{framed}
\begin{center}
$\boldsymbol{\mathcal{P}_{\mathcal{M}, \mathbb{R} }^{\mathfrak{m},+}}$ the subset of $\mathcal{P}_{\mathcal{M}}^{\mathfrak{m},+}$ invariant under the real Frobenius $\mathcal{F}_{\infty}$, \text{ which, by }  $(\ref{eq:periodgeom})$ \text{ satisfies }:\nomenclature{$\mathcal{P}_{\mathcal{M}, \mathbb{R} }^{\mathfrak{m},+}$}{the ring of the Frobenius-invariant geometric periods}
\end{center}
\begin{equation}\label{eq:periodgeomr}
\mathcal{P}_{\mathcal{M}}^{\mathfrak{m},+}\cong  \mathcal{P}_{\mathcal{M}, \mathbb{R}}^{\mathfrak{m},+} \oplus  \mathcal{P}_{\mathcal{M}, \mathbb{R}}^{\mathfrak{m},+}. \mathbb{L}^{\mathfrak{m}} \quad \text{and} \quad  \mathcal{P}_{\mathcal{M}, \mathbb{R}}^{\mathfrak{m},+}\cong  \mathcal{P}_{\mathcal{M}}^{\mathfrak{a}} \otimes_{\mathbb{Q}} \mathbb{Q}\left[ (\mathbb{L}^{\mathfrak{m}})^{2} \right] .
\end{equation}
\end{framed}

\paragraph{Motivic Galois theory.}
The ring of motivic periods $\mathcal{P}_{\mathcal{M}}^{\mathfrak{m}} $ is a bitorsor under Tannaka groups $(\mathcal{G}^{\mathcal{MT}}, \mathcal{G}_{B})$. If Grothendieck conjecture holds, via the period isomorphism, there is therefore a (left) action of the motivic Galois group $\mathcal{G}^{\mathcal{MT}}$ on periods. \\
More precisely, for each period $p$ there would exist:
\begin{itemize}
\item[$(i)$] well defined conjugates: elements in the orbit of $\mathcal{G}^{\mathcal{MT}}(\mathbb{Q})$. 
\item[$(ii)$] an algebraic group over $\mathbb{Q}$, $\mathcal{G}_{p}= \mathcal{G}^{\mathcal{MT}} \diagup Stab(p)$, where $Stab(p)$ is the stabilizer of $p$; $\mathcal{G}_{p}$, the Galois group of $p$, transitively permutes the conjugates.
\end{itemize}
\texttt{Examples}:
\begin{itemize}
\item[$\cdot$] For $\pi$ for instance, the Galois group corresponds to $\mathbb{G}_{m}$. Conjugates of $\pi$ are in fact $\mathbb{Q}^{\ast} \pi$, and the associated motive would be the Lefschetz motive $\mathbb{L}$, motive of $\mathbb{G}_{m}=\mathbb{P}^{1}\diagdown \lbrace0,\infty\rbrace$, as seen above.
\item[$\cdot$] For $\log t$, $t>0$, $t\in\mathbb{Q} \diagdown \lbrace -1, 0, 1\rbrace$, this is a period of the Kummer motive in degree $1$:\footnote{Remark the short exact sequence: $ 0 \rightarrow \mathbb{Q}(1) \rightarrow H_{1}(X, \lbrace 1,t \rbrace) \rightarrow \mathbb{Q}(0)  \rightarrow 0 .$}
 $$K_{t}\mathrel{\mathop:}=M_{gm}(X, \lbrace 1,t \rbrace)\in \text{Ext}^{1}_{\mathcal{MT}(\mathbb{Q})}(\mathbb{Q}(0),\mathbb{Q}(1)) \text{ , where } X\mathrel{\mathop:}=\mathbb{P}^{1}\diagdown \lbrace 0, \infty \rbrace.$$ 
 Since a basis of $H^{B}_{1}(X, \lbrace 1,t \rbrace)$ is $[\gamma_{0}]$, $[\gamma_{1,t}]$ with $\gamma_{1,t}$ the straight path from $1$ to $t$, and a basis of $H^{1}_{dR}(X, \lbrace 1,t \rbrace) $ is $[dx], \left[ \frac{dx}{x} \right] $, the period matrix is:
 $$ \left( \begin{array}{ll}
\mathbb{Q} & 0\\
\mathbb{Q} \log(t) & 2i\pi \mathbb{Q} \\
 \end{array} \right). $$
The conjugates of $\log t$ are $\mathbb{Q}^{\ast}\log t+\mathbb{Q}$, and its Galois group is $\mathbb{Q}^{\ast}  \ltimes \mathbb{Q}$. 
\item[$\cdot$] Similarly for zeta values $\zeta(n)$, $n$ odd in $\mathbb{N}^{\ast}\diagdown\lbrace 1 \rbrace$ which are periods of a mixed Tate motive over $\mathbb{Z}$ (cf. below): its conjugates are $\mathbb{Q}^{\ast}\zeta(n)+\mathbb{Q}$, and its Galois group is $\mathbb{Q}^{\ast}  \ltimes \mathbb{Q}$. Grothendieck's conjecture implies that $\pi,\zeta(3), \zeta(5), \ldots$ are algebraically independent.\\
More precisely, $\zeta(n)$ is a period of $E_{n}\in \mathcal{MT}(\mathbb{Q})$, where:
$$ 0\rightarrow \mathbb{Q}(n) \rightarrow E_{n} \rightarrow \mathbb{Q}(0) \rightarrow 0.$$
Notice that for even $n$, by Borel's result, $\text{Ext}_{\mathcal{MT}(\mathbb{Q})}^{1}(\mathbb{Q}(0),\mathbb{Q}(n))=0$, which implies $E_{n}= \mathbb{Q}(0)\oplus \mathbb{Q}(n)$, and hence $\zeta(n)\in (2i\pi)^{n}\mathbb{Q}$.
\item[$\cdot$] More generally, multiple zeta values at roots of unity $\mu_{N}$ occur as periods of mixed Tate motives over $\mathbb{Z}[\xi_{N}]\left[ \frac{1}{N}\right] $, $\xi_{N}$ primitive $N^{\text{th}}$ root of unity. The motivic Galois group associated to the algebra $\mathcal{H}^{N}$  generated by MMZV$_{\mu_{N}}$ is conjectured to be a quotient of the motivic Galois group $\mathcal{G}^{\mathcal{MT}_{N}}$, equal for some values of $N$: $N=1,2,3,4,8$ for instance, as seen below. We expect MZV to be simple examples in the conjectural Galois theory for transcendental numbers.
\end{itemize}
\textsc{Remark}: By K-theory results above, non-zero Ext groups for $\mathcal{MT}(\mathbb{Q})$  are:
$$\text{Ext}^{1}_{\mathcal{MT}(\mathbb{Q})} (\mathbb{Q}(0), \mathbb{Q}(n))\cong \left\lbrace  \begin{array}{ll}
\mathbb{Q}^{\ast}\otimes_{\mathbb{Z}} \mathbb{Q} \cong \oplus_{p \text{ prime} } \mathbb{Q} & \text{ if } n=1\\
\mathbb{Q} & \text{ if } n \text{ odd} >1.\\
\end{array}\right. $$
Generators of these extension groups correspond exactly to periods $\log(p)$, $p$ prime in degree 1 and $\zeta(odd)$ in degree odd $>1$, which are periods of $\mathcal{MT}(\mathbb{Q})$.

\section{Motivic fundamental group}

\paragraph{Prounipotent completion.}

Let $\Pi$ be the group freely generated by $\gamma_{0}, \ldots, \gamma_{N}$. The completed Hopf algebra $\widehat{\Pi}$ is defined by:
$$\widehat{\Pi}\mathrel{\mathop:}= \varprojlim \mathbb{Q}[\Pi] \diagup I^{n} , \quad \text{ where }I\mathrel{\mathop:}= \langle \gamma-1 , \gamma\in \Pi \rangle \text{ is the augmentation ideal} .$$
Equipped with the completed coproduct $\Delta$ such that the elements of $\Pi$ are primitive, it is isomorphic to the Hopf algebra of non commutative formal series:\footnote{Well defined inverse since the log converges in $\widehat{\Pi}$; $exp(e_{i})$ are then group-like for $\Delta$. Notice that the Lie Algebra of the group of group-like elements is formed by the primitive elements and conversely; besides, the universal enveloping algebra of primitive elements is the whole Hopf algebra.} 
$$\widehat{\Pi} \xrightarrow[\gamma_{i}\mapsto \exp(e_{i}) ]{\sim} \mathbb{Q} \langle\langle e_{0}, \ldots, e_{N}\rangle\rangle. $$
\\
The \textit{prounipotent completion} of  $\Pi$ is an affine group scheme $\Pi^{un}$:\nomenclature{$\Pi^{un}$}{prounipotent completion of $\Pi$}
\begin{equation}\label{eq:prounipcompletion}
\boldsymbol{\Pi^{un}}(R)=\lbrace x\in \widehat{\Pi} \widehat{\otimes} R \mid \Delta x=x\otimes x\rbrace \cong \lbrace S\in R\langle\langle e_{0}, \ldots, e_{N} \rangle\rangle^{\times}\mid \Delta S=S\otimes S, \epsilon(S)=1\rbrace ,
\end{equation}
i.e. the set of non-commutative formal series with $N+1$ generators which are group-like for the completed coproduct for which $e_{i}$ are primitive. \\
It is dual to the shuffle $\shuffle$ relation between the coefficients of the series $S$\footnote{It is a straightforward verification that the relation $\Delta S= S\otimes S$ implies the shuffle $\shuffle$ relation between the coefficients of S.}. Its affine ring of regular function is the Hopf algebra (filtered, connected) for the shuffle product, and deconcatenation coproduct:
\begin{equation}
\boldsymbol{\mathcal{O}(\Pi^{un})}= \varinjlim \left(  \mathbb{Q}[\Pi] \diagup I^{n+1} \right) ^{\vee} \cong \mathbb{Q} \left\langle e^{0}, \ldots, e^{N} \right\rangle .
\end{equation}
$$\boldsymbol{\mathcal{O}(\Pi^{\mathfrak{m}}(X_{N},x,y))}\in\mathcal{MT}(k_{N}).$$

\paragraph{Motivic Fundamental pro-unipotent groupoid.}\footnote{ \say{\textit{Esquisse d'un programme}}$\cite{Gr}$, by Grothendieck, vaguely suggests to study the action of the absolute Galois group of the rational numbers $ \text{Gal}(\overline{\mathbb{Q}} \diagup \mathbb{Q} ) $ on the \'{e}tale fundamental group $\pi_{1}^{et}(\mathcal{M}_{g,n})$, where $\mathcal{M}_{g,n}$ is the moduli space of curves of genus $g$ and $n$ ordered marked points. In the case of $\mathcal{M}_{0,4}= \mathbb{P}^{1}\diagdown \lbrace 0, 1, \infty \rbrace$, Deligne proposed to look instead (analogous) at the pro-unipotent fundamental group $\pi_{1}^{un}(\mathbb{P}^{1}\diagdown \lbrace 0, 1, \infty \rbrace)$. This motivates also the study of multiple zeta values, which arose as periods of this fundamental group. } The previous construction can be applied to $\pi_{1}(X,x)$, resp. $\pi_{1}(X,x,y)$, if assumed free, the fundamental group resp. groupoid of $X$ with base point $x$, resp. $x,y$, rational points of $X$, an algebraic variety over $\mathbb{Q}$; the groupoid $\pi_{1}(X,x,y)$, is a bitorsor formed by the homotopy classes of path from $x$ to $y$. \\

From now, let's turn to the case $X_{N}\mathrel{\mathop:}=\mathbb{P}^{1}\diagdown \lbrace 0, \infty, \mu_{N} \rbrace$. There, the group $\pi_{1}(X_{N}, x)$ is freely generated by $\gamma_{0}$ and $(\gamma_{\eta})_{\eta\in\mu_{N}}$, the loops around $0$ resp. $\eta\in\mu_{N}$.\footnote{Beware, since $\pi_{1}(X,x,y)$ is not a group, we have to pass first to the dual in the previous construction:
$$ \pi^{un}_{1}(X,x,y)\mathrel{\mathop:}= \text{Spec} \left( \varinjlim \left(  \mathbb{Q}[\pi_{1}] \diagup I^{n+1} \right) ^{\vee} \right) .$$}\\ 
Chen's theorem implies here that we have a perfect pairing: 
\begin{equation}\label{eq:chenpairing}
\mathbb{C}[\pi_{1} (X_{N},x,y)] \diagup I^{n+1} \otimes \mathbb{C}\langle \omega_{0}, (\omega_{\eta})_{\eta\in\mu_{N}} \rangle_{\leq n}  \rightarrow   \mathbb{C}  .
\end{equation}
In order to define the motivic $\pi^{un}_{1}$, let us introduce (cf. $\cite{Go2}$, Theorem $4.1$):
\begin{equation}\label{eq:y(n)}
Y^{(n)}\mathrel{\mathop:}=\cup_{i} Y_{i}, \text{ where } \quad  \begin{array}{ll}
Y_{0}\mathrel{\mathop:}= \lbrace x\rbrace \times X^{n-1}& \\
Y_{i}\mathrel{\mathop:}= X^{i-1}\times \Delta \times X^{n-i-1}, & \Delta \subset X \times X \text{ the diagonal} \\
 Y_{n}\mathrel{\mathop:}=   X^{n-1} \times \lbrace y\rbrace &
\end{array}.
\end{equation}
Then, by Beilinson theorem ($\cite{Go2}$, Theorem $4.1$), coming from $\gamma \mapsto [\gamma(\Delta_{n})]$:
$$H_{k}(X^{n},Y^{(n)}) \cong \left\lbrace \begin{array}{ll}
 \mathbb{Q}[\pi_{1}(X,x,y)] \diagup I^{n+1}& \text{ for } k=n \\
0  & \text{ for } k<n
\end{array} \right.  .$$
The left side defines a mixed Tate motive and:
\begin{equation} \label{eq:opiunvarinjlim}
 \mathcal{O}(\pi_{1}^{un}(X,x,y)) \xrightarrow{\sim} \varinjlim_{n} H^{n}(X^{n}, Y^{(n)}).
\end{equation}
By $(\ref{eq:opiunvarinjlim})$, $\mathcal{O}\left( \pi_{1}^{un}(X,x,y)\right)$ defines an Ind object \footnote{Ind objects of a category $\mathcal{C}$ are inductive filtered limit of objects in $\mathcal{C}$.} in the category of Mixed Tate Motives over $k$, since $Y_{I}^{(n)}\mathrel{\mathop:}=\cap Y_{i}^{(n)}$ is the complement of hyperplanes, hence of type Tate:
\begin{framed}
\begin{equation}\label{eq:pi1unTate0}
\mathcal{O}\left( \pi_{1}^{un}(\mathbb{P}^{1}\diagdown \lbrace 0, \infty, \mu_{N} \rbrace,x,y)\right)  \in  \text{Ind } \mathcal{MT}(k).
\end{equation}
\end{framed}
We denote it $\boldsymbol{\mathcal{O}\left( \pi_{1}^{\mathfrak{m}}(X,x,y)\right) }$, and $\mathcal{O}\left( \pi_{1}^{\omega}(X,x,y)\right)$, $\mathcal{O}\left( \pi_{1}^{dR}(X,x,y)\right)$, $\mathcal{O}\left( \pi_{1}^{B}(X,x,y)\right)$ its realizations, 
resp. $\boldsymbol{\pi_{1}^{\mathfrak{m}}(X)}$ for the corresponding $\mathcal{MT}(k)$-groupoid scheme, called the \textit{\textbf{motivic fundamental groupoid}}, with the composition of path. \\
\\
\textsc{Remark: } The pairing $(\ref{eq:chenpairing})$ can be thought in terms of a perfect pairing between homology and de Rham cohomology, since (Wojtkowiak $\cite{Wo2}$):
$$H_{dR}^{n}(X^{n},Y^{(n)}) \cong k_{N}\langle \omega_{0}, \ldots, \omega_{N} \rangle_{\leq n}.$$ 
\\
The construction of the prounipotent completion and then the motivic fundamental groupoid would still work for the case of \textit{tangential base points }, cf. $\cite{DG}$, $\S 3$\footnote{I.e. here non-zero tangent vectors in a point of $\lbrace 0, \mu_{N}, \infty\rbrace $ are seen as \say{base points at infinite}. Deligne explained how to replace ordinary base points with tangential base points.}. Let us denote $\lambda_{N}$ the straight path between $0$ and $\xi_{N}$, a primitive root of unity. In the following, we will particularly consider the tangential base points $\overrightarrow{0\xi_{N}}\mathrel{\mathop:}=(\overrightarrow{1}_{0}, \overrightarrow{-1}_{\xi_{N}})$, defined as $(\lambda_{N}'(0), -\lambda_{N}'(1))$; but similarly for each $x,y\in \mu_{N}\cup \lbrace 0, \infty\rbrace $, such that $_{x}\lambda_{y}$ the straight path between $x,y$ in in $\mathbb{P}^{1} (\mathbb{C} \diagdown \lbrace 0, \mu_{N}, \infty \rbrace)$, we associate the tangential base points $\overrightarrow{xy}\mathrel{\mathop:}= (_{x}\lambda_{y}'(0), - _{x}\lambda_{y}'(1))$\footnote{In order that the path does not pass by $0$, we have to exclude the case where $x=-y$ if $N$ even.}. Since the motivic torsor of path associated to such tangential basepoints depends only on $x,y$ (cf. $\cite{DG}$, $\S 5$) we will denote it $_{x}\Pi^{\mathfrak{m}}_{y}$. This leads to a groupoid structure via $_{x}\Pi^{\mathfrak{m}}_{y} \times _{y}\Pi^{\mathfrak{m}}_{z} \rightarrow _{x}\Pi^{\mathfrak{m}}_{z}$: cf. Figure $\ref{fig:Pi}$ and $\cite{DG}$.\\
In fact, by Goncharov's theorem, in case of these tangential base points, the motivic torsor of path corresponding has good reduction outside N and (cf. $\cite{DG}, \S 4.11$):
\begin{equation}\label{eq:pi1unTate}
\mathcal{O}\left( _{x}\Pi^{\mathfrak{m}}_{y} \right)  \in  \text{ Ind }  \mathcal{MT}_{\Gamma_{N}} \subset \text{ Ind }  \mathcal{MT}\left( \mathcal{O}_{N}\left[ \frac{1}{N} \right] \right) .
\end{equation}
The case of ordinary base points, lying in $\text{ Ind }  \mathcal{MT}(k)$, has no such good reduction.\\
In summary, from now, we consider, for $x,y\in \mu_{N}\cup\lbrace 0\rbrace$\footnote{$_{x}\Pi^{\mathfrak{m}}_{y}$ is a bitorsor under $(_{x}\Pi^{\mathfrak{m}}_{x}, _{y}\Pi^{\mathfrak{m}}_{y})$.}:\nomenclature{$_{x}\Pi^{\mathfrak{m}}_{y}$}{motivic bitorsor of path, and $_{x}\Pi_{y}$, $_{x}\Pi_{y}^{dR}$, $_{x}\Pi_{y}^{B}$, its  $\omega$, resp. de Rham resp. Betti realizations}
\begin{framed}
\textit{The motivic bitorsors of path} $_{x}\Pi^{\mathfrak{m}}_{y}\mathrel{\mathop:}=\pi_{1}^{\mathfrak{m}} (X_{N}, \overrightarrow{xy})$ on $X_{N}\mathrel{\mathop:}=\mathbb{P}^{1} -\left\{0,\mu_{N},\infty\right\}$ with tangential basepoints given by $\overrightarrow{xy}\mathrel{\mathop:}= (\lambda'(0), -\lambda'(1))$ where $\lambda$ is the straight path from $x$ to $y$, $x\neq -y$.\\
\end{framed}
Let us denote $_{x}\Pi_{y}\mathrel{\mathop:}=_{x}\Pi^{\omega}_{y}$, resp. $_{x}\Pi_{y}^{dR}$, $_{x}\Pi_{y}^{B}$ its $\omega$, resp. de Rham resp. Betti realizations. In particular, Chen's theorem implies that we have an isomorphism:
$$_{0}\Pi_{1}^{B}\otimes\mathbb{C}\xrightarrow{\sim} {} _{0}\Pi_{1}\otimes \mathbb{C}.$$\\
Therefore, the motivic fundamental group above boils down to:
\begin{itemize}
\item[$(i)$] The affine group schemes $_{x}\Pi_{y}^{B}$, $x,y\in\mu_{N}\cup \lbrace0, \infty\rbrace $, with a groupoid structure. The Betti fundamental groupoid is the pro-unipotent
completion of the ordinary topological fundamental groupoid, i.e. corresponds to $\pi_{1}^{un}(X,x,y)$ above.
\item[$(ii)$]  $\Pi(X)=\pi^{\omega}_{1}(X)$, the affine group scheme over $\mathbb{Q}$. It does not depend on $x,y$ since the existence of a canonical de Rham path between x and y implies a canonical isomorphism $\Pi(X)\cong _{x}\Pi(X)_{y}$; however, the action of the motivic Galois group $\mathcal{G}$ is sensitive to the tangential base points $x,y$.
\item[$(iii)$] a canonical comparison isomorphism of schemes over $\mathbb{C}$, $\text{comp}_{B,\omega}$.
\end{itemize}
 
\begin{figure}[H]
\centering
\includegraphics[]{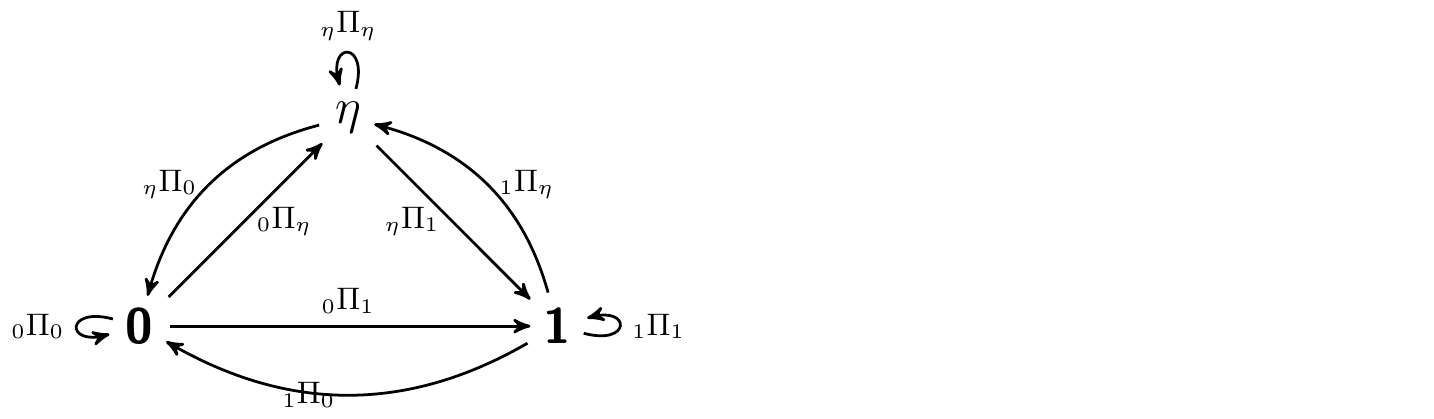}

\caption{Part of the Fundamental groupoid $\Pi$.\\
This picture however does not represent accurately the tangential base points.} \label{fig:Pi}
\end{figure}

Moreover, the dihedral group\footnote{Symmetry group of a regular polygon with $N$ sides.} $Di_{N}= \mathbb{Z}\diagup 2 \mathbb{Z} \ltimes \mu_{N}$\nomenclature{$Di_{N}$}{dihedral group of order $2n$} acts on $X_{N}=\mathbb{P}^{1}\diagdown \lbrace 0, \mu_{N},\infty\rbrace$\nomenclature{$X_{N}$}{defined as $\mathbb{P}^{1}\diagdown \lbrace 0, \mu_{N},\infty\rbrace$}: the group with two elements corresponding to the action $x \mapsto x^{-1}$ and the cyclic group $\mu_{N}$ acting by $x\mapsto \eta x$. Notice that for $N=1,2,4$, the group of projective transformations $X_{N}\rightarrow X_{N}$ is larger than $Di_{N}$, because of special symmetries, and detailed in $A.3$. \footnote{Each homography $\phi$ defines isomorphisms: 
$$\begin{array}{lll}
 _{a}\Pi_{b} & \xrightarrow[\sim]{\phi}& _{\phi(a)}\Pi_{\phi(b)} \\
f(e_{0}, e_{1}, \ldots, e_{n}) &\mapsto &f(e_{\phi(0)}, e_{\phi(1)}, \ldots, e_{\phi(n)})
\end{array} \text{ and, passing to the dual }  \mathcal{O}(_{\phi(a)}\Pi_{\phi(b)}) \xrightarrow[\sim]{\phi^{\vee}} \mathcal{O}( _{a}\Pi_{b}) .$$} \\ The dihedral group $Di_{N}$ acts then on the motivic fundamental groupoid $\pi^{\mathfrak{m}}_{1}(X,x,y)$, $x,y \in \lbrace 0 \rbrace \cup \mu_{N}$ by permuting the tangential base points (and its action is respected by the motivic Galois group): 
$$\text{For } \quad \sigma\in Di_{N}, \quad  _{x}\Pi_{y} \rightarrow _{\sigma.x}\Pi_{\sigma.y} $$

The group scheme $\mathcal{V}$ of automorphisms on these groupoids $_{x}\Pi_{y}$, respecting their structure, i.e.:
\begin{itemize}
\item[$\cdot$] groupoid structure, i.e. the compositions $_{x}\Pi_{y}\times _{y}\Pi_{z} \rightarrow _{x}\Pi_{z}$,
\item[$\cdot$] $\mu_{N}$-equivariance as above,
\item[$\cdot$] inertia: the action fixes $\exp(e_{x})\in _{x}\Pi_{x}(\mathbb{Q})$,
\end{itemize}
is isomorphic to (cf. $\cite{DG}$, $\S 5$ for the detailed version):
\begin{equation}\label{eq:gpaut}
\begin{array}{ll}
\mathcal{V}\cong _{0}\Pi_{x} \\
a\mapsto a\cdot _{0}1_{x}
\end{array}.
\end{equation}
In particular, the \textit{Ihara action} defined in  $(\ref{eq:iharaaction})$ corresponds via this identification to the composition law for these automorphisms, and then can be computed explicitly. Its dual would be the combinatorial coaction $\Delta$ used through all this work.\\
\\
In consequence of these equivariances, we can restrict our attention to:
\begin{framed}
$$_{0}\Pi^{\mathfrak{m}}_{ \xi_{N}}\mathrel{\mathop:}=\pi_{1}^{\mathfrak{m}}(X_{N}, \overrightarrow{0\xi_{N}} ) \text{ or equivalently  at }  _{0}\Pi^{\mathfrak{m}}_{1}.$$
\end{framed}
\noindent
Keep in mind, for the following, that $_{0}\Pi_{1}$ is the functor:\nomenclature{$R\langle X \rangle$ resp. $R\langle\langle X \rangle\rangle$}{the ring of non commutative polynomials, resp. of non commutative formal series in elements of X}
\begin{framed}
\begin{equation}\label{eq:pi}_{0}\Pi_{1}: R \text{ a } \mathbb{Q}-\text{algebra } \mapsto \left\{S\in R\langle\langle e_{0}, (e_{\eta})_{\eta\in\mu_{N}}\rangle\rangle^{\times} | \Delta S= S\otimes S  \text{ and } \epsilon(S)= 1 \right\} ,\end{equation}
whose affine ring of regular functions is the graded (Hopf) algebra for the shuffle product:
\begin{equation}\label{eq:opi}
\mathcal{O}(_{0}\Pi_{1})\cong \mathbb{Q} \left\langle e^{0}, (e^{\eta})_{\eta\in\mu_{N}} \right\rangle.
\end{equation}
\end{framed}
\noindent
The Lie algebra of $_{0}\Pi_{1}(R)$ would naturally be the primitive series ($\Delta S= 1 \otimes S+ S\otimes 1$). \\
\\
Let us denote $dch_{0,1}^{B}=_{0}1^{B}_{1}$,\nomenclature{$dch_{0,1}^{B}$}{the image of the straight path} the image of the straight path (\textit{droit chemin}) in $_{0}\Pi_{1}^{B}(\mathbb{Q})$, and $dch_{0,1}^{dR}$  or $\Phi_{KZ_{N}}$ the corresponding element in $_{0}\Pi_{1}(\mathbb{C})$ via the Betti-De Rham comparison isomorphism:
\begin{equation}\label{eq:kz}
 \boldsymbol{\Phi_{KZ_{N}}}\mathrel{\mathop:}= dch_{0,1}^{dR}\mathrel{\mathop:}= \text{comp}_{dR,B}(_{0}1^{B}_{1})= \sum_{W\in \lbrace e_{0}, (e_{\eta})_{\eta\in\mu_{N}}  \rbrace^{\times}} \zeta_{\shuffle}(w) w \quad \in \mathbb{C} \langle\langle e_{0}, (e_{\eta})_{\eta\in\mu_{N}} \rangle\rangle ,
\end{equation}
where the correspondence between MZV and words in $e_{0},e_{\eta}$ is similar to the iterated integral representation $(\ref{eq:reprinteg})$, with $\eta_{i}$. It is known as the \textit{Drinfeld associator} and arises also from the monodromy of the famous Knizhnik$-$Zamolodchikov differential equation.\footnote{Indeed, for $N=1$, Drinfeld associator is equal to $G_{1}^{-1}G_{0}$, where $G_{0},G_{1}$ are solutions, with certain asymptotic behavior at $0$ and $1$ of the Knizhnik$-$Zamolodchikov differential equation:$$ \frac{d}{dz}G(z)= \left(\frac{e_{0}}{z}+ \frac{e_{1}}{1-z} \right) G(z) .$$}

\paragraph{Category generated by $\boldsymbol{\pi_{1}^{\mathfrak{m}}}$. }
Denote by:\nomenclature{$\mathcal{MT}'_{N}$}{the full Tannakian subcategory of $\mathcal{MT}_{N}$ generated by the fundamental groupoid}
\begin{framed}
  $\boldsymbol{\mathcal{MT}'_{N}}$ the full Tannakian subcategory of $\mathcal{MT}_{N}$ generated by the fundamental groupoid,
 \end{framed} 
(i.e. generated by $\mathcal{O}(\pi_{1}^{\mathfrak{m}} (X_{N},\overrightarrow{01}))$ by sub-objects, quotients, $\otimes$, $\oplus$, duals) and let:\nomenclature{$\mathcal{G}^{N}$}{the motivic Galois group of $\boldsymbol{\mathcal{MT}'_{N}}$ }\nomenclature{$\mathcal{A}^{N}$}{the fundamental Hopf algebra of $\boldsymbol{\mathcal{MT}'_{N}}$ }\nomenclature{$\mathcal{L}^{N}$}{the motivic coalgebra associated to $\boldsymbol{\mathcal{MT}'_{N}}$ }
\begin{itemize}
\item[$\cdot$] $\mathcal{G}^{N}=\mathbb{G}_{m} \ltimes \mathcal{U}^{N} $ its motivic \textit{Galois group} defined over $\mathbb{Q}$,
\item[$\cdot$] $\mathcal{A}^{N}=\mathcal{O}(\mathcal{U}^{N})$ its \textit{fundamental Hopf algebra},
\item[$\cdot$] $\mathcal{L}^{N}\mathrel{\mathop:}= \mathcal{A}^{N}_{>0} / \mathcal{A}^{N}_{>0} \cdot\mathcal{A}^{N}_{>0}$ the Lie \textit{coalgebra of indecomposable elements}.
\end{itemize}
\texttt{Nota Bene}: $\mathcal{U}^{N}$ is the quotient of $\mathcal{U}^{\mathcal{MT}} $ by the kernel of the action on $_{0}\Pi_{1}$: i.e. $\mathcal{U}^{N}$ acts faithfully on $_{0}\Pi_{1}$.\nomenclature{$\mathcal{U}^{N}$}{the motivic prounipotent group associated to $\boldsymbol{\mathcal{MT}'_{N}}$ }\\
\\
\textsc{Remark:} In the case of $N=1$ (by F. Brown in \cite{Br2}), or $N=2,3,4,\mlq 6 \mrq,8 $ (by P. Deligne, in \cite{De}, proven in a dual point of view in Chapter $5$), these categories $\mathcal{MT}'_{N}$ and $\mathcal{MT}(\mathcal{O}_{N}\left[ \frac{1}{N} \right] )$ are equal. More precisely, for $\xi_{N}\in\mu_{N}$ a fixed primitive root, the following motivic torsors of path are sufficient to generate the category:
\begin{description}
\item[$\boldsymbol{N=2,3,4}$:] $\Pi^{\mathfrak{m}} (\mathbb{P}^{1} \diagdown \lbrace 0, 1, \infty \rbrace, \overrightarrow{0 \xi_{N}})$ generates $\mathcal{MT}(\mathcal{O}_{N}\left[ \frac{1}{N} \right] )$.
\item[$\boldsymbol{N=\mlq 6\mrq}$:]\footnote{The quotation marks around $6$ underlines that we consider the unramified category in this case.} $\Pi^{\mathfrak{m}} (\mathbb{P}^{1} \diagdown \lbrace 0, 1, \infty \rbrace, \overrightarrow{0 \xi_{6}})$  generates $\mathcal{MT}(\mathcal{O}_{6})$.
\item[$\boldsymbol{N=8}$:] $\Pi^{\mathfrak{m}} (\mathbb{P}^{1} \diagdown \lbrace 0, \pm 1, \infty \rbrace, \overrightarrow{0 \xi_{8}})$  generates $\mathcal{MT}(\mathcal{O}_{8}\left[ \frac{1}{2}\right] )$.\\
\end{description} 
However, if $N$ has a prime factor which is non inert, the motivic fundamental group is in the proper subcategory $\mathcal{MT}_{\Gamma_{N}}$ and hence can not generate $\mathcal{MT}(\mathcal{O}_{N}\left[ \frac{1}{N}\right] )$.

\section{Motivic Iterated Integrals}

Taking from now $\mathcal{M}=\mathcal{MT}'_{N}$, $M=\mathcal{O}(\pi^{\mathfrak{m}}_{1}(\mathbb{P}^{1}-\lbrace 0,\mu_{N},\infty\rbrace ,\overrightarrow{xy} ))$, the definition of motivic periods $(\ref{eq:mper})$ leads to motivic iterated integrals relative to $\mu_{N}$. Indeed:\nomenclature{$I^{\mathfrak{m}}(x;w;y)$}{motivic iterated integral}
\begin{framed}\label{mii}
A \textit{\textbf{motivic iterated integral}} is the triplet $I^{\mathfrak{m}}(x;w;y)\mathrel{\mathop:}= \left[\mathcal{O} \left( \Pi^{\mathfrak{m}} \left( X_{N}, \overrightarrow{xy}\right) \right) ,w,_{x}dch_{y}^{B}\right]^{\mathfrak{m}}$ where $w\in \omega(M)$, $_{x}dch_{y}^{B}$ is the image of the straight path from $x$ to $y$ in $\omega_{B}(M)^{\vee}$ and whose period is:
\begin{equation}\label{eq:peri} \text{per}(I^{\mathfrak{m}}(x;w;y))= I(x;w;y) =\int_{x}^{y}w= \langle \text{comp}_{B,dR}(w\otimes 1),_{x}dch_{y}^{B} \rangle \in\mathbb{C}.
\end{equation}
\end{framed}
\noindent
\\
\\
\textsc{Remarks: } 
\begin{itemize}
\item[$\cdot$] There, $w\in \omega(\mathcal{O}(_{x}\Pi^{\mathfrak{m}}_{y}))\cong \mathbb{Q}  \left\langle  \omega_{0}, (\omega_{\eta})_{\eta\in\mu_{N}}  \right\rangle $ where $\omega_{\eta}\mathrel{\mathop:}= \frac{dt}{t-\eta}$. Similarly to $\ref{eq:iterinteg}$, let:
\begin{equation}\label{eq:iterintegw}
I^{\mathfrak{m}} (a_{0}; a_{1}, \ldots, a_{n}; a_{n+1})\mathrel{\mathop:}= I^{\mathfrak{m}} (a_{0}; \omega_{\boldsymbol{a}}; a_{n+1}), \quad \text{ where }  \omega_{\boldsymbol{a}}\mathrel{\mathop:}=\omega_{a_{1}} \cdots \omega_{a_{n}}, \text{ for } a_{i}\in \lbrace 0\rbrace \cup \mu_{N} 
\end{equation}
\item[$\cdot$] The Betti realization functor $\omega_{B}$ depends on the embedding $\sigma: k \hookrightarrow \mathbb{C}$. Here, by choosing a root of unity, we fixed the embedding $\sigma$.
\end{itemize}
For $\mathcal{M}$ a category of Mixed Tate Motives among $\mathcal{MT}_{N},  \mathcal{MT}_{\Gamma_{N}}$ resp. $\mathcal{MT}'_{N}$, let introduce the graded $\mathcal{A}^{\mathcal{M}}$-comodule, with trivial coaction on $\mathbb{L}^{\mathfrak{m}}$ (degree $1$): 
 \begin{equation}\label{eq:hn}
 \boldsymbol{\mathcal{H}^{\mathcal{M}}} \mathrel{\mathop:}= \mathcal{A}^{\mathcal{M}} \otimes \left\{
\begin{array}{ll} 
\mathbb{Q}\left[  (\mathbb{L}^{\mathfrak{m}})^{2} \right] & \text{ for } N=1,2  \\
\mathbb{Q}\left[ \mathbb{L}^{\mathfrak{m}} \right] & \text{ for } N>2 .
\end{array}
\right. \subset \mathcal{O}(\mathcal{G}^{\mathcal{M}})=  \mathcal{A}^{\mathcal{M}}\otimes  \mathbb{Q}[\mathbb{L}^{\mathfrak{m}}, (\mathbb{L}^{\mathfrak{m}})^{-1}].
\end{equation}
\texttt{Nota Bene}: For $N>2$, it corresponds to the geometric motivic periods, $\mathcal{P}_{\mathcal{M}}^{\mathfrak{m},+}$ whereas for $N=1,2$, it is the subset $\mathcal{P}_{\mathcal{M},\mathbb{R}}^{\mathfrak{m},+}$ invariant by the real Frobenius; cf. $(\ref{eq:periodgeom}), (\ref{eq:periodgeomr})$.\\
For $\mathcal{M}=\mathcal{MT}'_{N}$, we will simply denote it $\mathcal{H}^{N}\mathrel{\mathop:}=\mathcal{H}^{\mathcal{MT}'_{N}}$. Moreover:
$$\mathcal{H}^{N}\subset \mathcal{H}^{\mathcal{MT}_{\Gamma_{N}}} \subset \mathcal{H}^{\mathcal{MT}_{N}} .$$
\\
Cyclotomic iterated integrals of weight $n$ are periods of $\pi^{un}_{1}$ (of $X^{n}$ relative to $Y^{(n)}$): \footnote{Notations of $(\ref{eq:y(n)})$. Cf. also $(\ref{eq:pi1unTate})$. The case of tangential base points requires blowing up to get rid of singularities. Most interesting periods are often those whose integration domain meets the singularities of the differential form.}
\begin{framed}
 Any motivic iterated integral $I^{\mathfrak{m}}$ relative to $\mu_{N}$ is an element of $\mathcal{H}^{N}$, which is the graded $\mathcal{A}^{N}-$ comodule generated by these motivic iterated integrals relative to $\mu_{N}$.
\end{framed}

In a similar vein, define:\nomenclature{$ I^{\mathfrak{a}}$ resp. $I^{\mathfrak{l}}$}{versions of motivic iterated integrals in $\mathcal{A}$ resp. in $\mathcal{L}$}
\begin{itemize}
\item[$\cdot \boldsymbol{I^{\omega}}$: ] A motivic period of type $(\omega,\omega)$, in $\mathcal{O}(\mathcal{G})$:
\begin{equation} \label{eq:intitdr}
I^{\omega}(x;w;y)=\left[\mathcal{O} \left( _{x}\Pi^{\mathfrak{m}}_{y}\right)  ,w,_{x}1^{\omega}_{y} \right]^{\omega}, \quad \text{ where } \left\lbrace \begin{array}{l}
w\in\omega(\mathcal{O} \left( _{x}\Pi^{\mathfrak{m}}_{y}\right) )\\
_{x}1^{\omega}_{y}\in \omega(M)^{\vee}=\mathcal{O}\left( _{x}\Pi_{y}\right)^{\vee} 
\end{array}\right. .
\end{equation}
where $_{x}1^{\omega}_{y}\in \mathcal{O}\left( _{x}\Pi_{y}\right)^{\vee}$ is defined by the augmentation map $\epsilon:\mathbb{Q}\langle e^{0}, (e^{\eta})_{\eta\in\mu_{N}}\rangle \rightarrow \mathbb{Q}$, corresponding to the unit element in $_{x}\Pi_{y}$. This defines a function on $\mathcal{G}=\text{Aut}^{\otimes}(\omega)$, given on the rational points by $g\in\mathcal{G}(\mathbb{Q})  \mapsto \langle g\omega, \epsilon\rangle \in \mathbb{Q}$.
\item[$\cdot \boldsymbol{I^{\mathfrak{a}}}$: ] the image of $I^{\omega}$ in $\mathcal{A}= \mathcal{O}(\mathcal{U})$, by the projection $\mathcal{O}(\mathcal{G})\twoheadrightarrow \mathcal{O}(\mathcal{U})$. These \textit{unipotent} motivic periods are the objects studied by Goncharov, which he called motivic iterated integrals; for instance, $\zeta^{\mathfrak{a}}(2)=0$.
\item[$\cdot \boldsymbol{I^{\mathfrak{l}}}$: ] the image of $I^{\mathfrak{a}}$ in the coalgebra of indecomposables $\mathcal{L}\mathrel{\mathop:}=\mathcal{A}_{>0} \diagup \mathcal{A}_{>0}. \mathcal{A}_{>0}$.\footnote{Well defined since $\mathcal{A}= \mathcal{O} (\mathcal{U})$ is graded with positive degrees.}
\end{itemize}

\textsc{Remark:} It is similar (cf. $\cite{Br2}$) to define $\mathcal{H}^{N}$, as  $\mathcal{O}(_{0}\Pi_{1}) \diagup J \quad$, with:
 \begin{itemize}
 \item[$\cdot$]  $J\subset \mathcal{O}(_{0}\Pi_{1})$ is the biggest graded ideal $\subset \ker  per$ closed by the coaction $\Delta^{c}$, corresponding to the ideal of motivic relations, i.e.:
$$\Delta^{c}(J)\subset \mathcal{A}\otimes J + J\mathcal{A} \otimes \mathcal{O}(_{0}\Pi_{1}).$$
\item[$\cdot$] the $\shuffle$-homomorphism: $\text{per}: \mathcal{O}(_{0}\Pi_{1}) \rightarrow \mathbb{C} \text{  ,  }  e^{a_{1}} \cdots e^{a_{n}} \mapsto  I(0; a_{1}, \ldots, a_{n} ; 1)\mathrel{\mathop:}= \int_{dch} \omega.$
\\
 \end{itemize}
Once the motivic iterated integrals are defined, motivic cyclotomic multiple zeta values follow, as usual (cf. $\ref{eq:iterinteg}$):
\begin{center}
\textit{Motivic multiple zeta values} relative to $\mu_{N}$ are defined by, for $\epsilon_{i}\in\mu_{N}, k\geq 0, n_{i}>0$
\begin{equation}\label{mmzv}
\boldsymbol{\zeta_{k}^{\mathfrak{m}} \left({ n_{1}, \ldots , n_{p} \atop \epsilon_{1} , \ldots ,\epsilon_{p} }\right) }\mathrel{\mathop:}= (-1)^{p} I^{\mathfrak{m}} \left(0;\boldsymbol{0}^{k}, (\epsilon_{1}\cdots \epsilon_{p})^{-1}, \boldsymbol{0}^{n_{1}-1} ,\cdots, (\epsilon_{i}\cdots \epsilon_{p})^{-1}, \boldsymbol{0}^{n_{i}-1} ,\cdots, \epsilon_{p}^{-1}, \boldsymbol{0}^{n_{p}-1} ;1 \right)
\end{equation}
\end{center}
 An \textit{admissible} (motivic) MZV is such that $\left( n_{p}, \epsilon_{p}\right) \neq \left(  1, 1 \right) $; otherwise, they are defined by shuffle regularization, cf. ($\ref{eq:shufflereg}$) below; the versions $\boldsymbol{\zeta_{k}^{\mathfrak{a}}} (\cdots)$, or $\boldsymbol{\zeta_{k}^{\mathfrak{l}}} (\cdots)$ are defined similarly, from $I^{\mathfrak{a}}$ resp. $I^{\mathfrak{l}}$ above. The roots of unity in the iterated integral will often be denoted by $\eta_{i}\mathrel{\mathop:}= (\epsilon_{i}\cdots \epsilon_{p})^{-1}$\\ 
\\
From $(\ref{eq:projpiam})$, there is a surjective homomorphism called the \textbf{\textit{period map}}, conjectured to be isomorphism:
 \begin{equation}\label{eq:period}\text{per}:\mathcal{H} \rightarrow \mathcal{Z} \text{ ,  } \zeta^{\mathfrak{m}} \left(n_{1}, \ldots , n_{p} \atop \epsilon_{1} , \ldots ,\epsilon_{p} \right)\mapsto  \zeta\left(n_{1}, \ldots , n_{p} \atop \epsilon_{1} , \ldots ,\epsilon_{p} \right).
 \end{equation}
\texttt{Nota Bene:} Each identity between motivic cyclotomic multiple zeta values is then true for cyclotomic multiple zeta values and in particular each result about a basis with motivic MZV implies the corresponding result about a generating family of MZV by application of the period map.\\
Conversely, we can sometimes \textit{lift} an identity between MZV to an identity between motivic MZV, via the coaction (as in $\cite{Br2}$, Theorem $3.3$); this is discussed below, and illustrated throughout this work in different examples or counterexamples, as in Lemma $\ref{lemmcoeff}$. It is similar in the case of motivic Euler sums ($N=2$). We will see (Theorem $2.4.4$) that for other roots of unity there are several rational coefficients which appear at each step (of the coaction calculus) and prevent us from concluding by identification.\\

\paragraph{Properties.}\label{propii}
Motivic iterated integrals satisfy the following properties:\nomenclature{$\mathfrak{S}_{p}$}{set of permutations of $\lbrace 1, \ldots, p\rbrace$.}
\begin{itemize}
	\item[(i)] $I^{\mathfrak{m}}(a_{0}; a_{1})=1$.
	\item[(ii)] $I^{\mathfrak{m}}(a_{0}; a_{1}, \cdots a_{n}; a_{n+1})=0$ if $a_{0}=a_{n+1}$.
	\item[(iii)] Shuffle product:\footnote{Product rule for iterated integral in general is:
$$ \int_{\gamma} \phi_{1} \cdots \phi_{r}  \cdot \int_{\gamma} \phi_{r+1} \cdots \phi_{r+s} = \sum_{\sigma\in Sh_{r,s}} \int_{\gamma} \phi_{\sigma^{-1}(1)} \cdots \phi_{\sigma^{-1}(r+s)} , $$	
	where $Sh_{r,s}\subset \mathfrak{S}_{r+s}$ is the subset of permutations which respect the order of $\lbrace 1 , \ldots, r\rbrace $ and $\lbrace r+1 , \ldots, r+s\rbrace$. Here, to define the non convergent case, $(iii)$ is sufficient, paired with the other rules.} 
	\begin{multline}\label{eq:shufflereg}
\zeta_{k}^{\mathfrak{m}} \left( {n_{1}, \ldots , n_{p} \atop \epsilon_{1}, \ldots ,\epsilon_{p} }\right)= \\
(-1)^{k}\sum_{i_{1}+ \cdots + i_{p}=k} \binom {n_{1}+i_{1}-1} {i_{1}} \cdots \binom {n_{p}+i_{p}-1} {i_{p}} \zeta^{\mathfrak{m}} \left( {n_{1}+i_{1}, \ldots , n_{p}+i_{p} \atop \epsilon_{1}, \ldots ,\epsilon_{p} }\right).
 \end{multline}
	\item[(iv)] Path composition: 
	$$ \forall x\in \mu_{N} \cup \left\{0\right\} ,  I^{\mathfrak{m}}(a_{0}; a_{1}, \ldots, a_{n}; a_{n+1})=\sum_{i=1}^{n} I^{\mathfrak{m}}(a_{0}; a_{1}, \ldots, a_{i}; x) I^{\mathfrak{m}}(x; a_{i+1}, \ldots, a_{n}; a_{n+1}) .$$
	\item[(v)] Path reversal: $I^{\mathfrak{m}}(a_{0}; a_{1}, \ldots, a_{n}; a_{n+1})= (-1)^n I^{\mathfrak{m}}(a_{n+1}; a_{n}, \ldots, a_{1}; a_{0}).$
	\item[(vi)] Homothety: $\forall \alpha \in \mu_{N}, I^{\mathfrak{m}}(0; \alpha a_{1}, \ldots, \alpha a_{n}; \alpha a_{n+1})  = I^{\mathfrak{m}}(0; a_{1}, \ldots, a_{n}; a_{n+1})$.
\end{itemize}
\vspace{0,5cm}
\textsc{Remark}: These relations, for the multiple zeta values relative to $\mu_{N}$, and for the iterated integrals  $I(a_{0}; a_{1}, \cdots ,a_{n}; a_{n+1})$ ($\ref{eq:reprinteg}$), are obviously all easily checked.\\
\\
It has been proven that motivic iterated integrals verify stuffle $\ast$ relations, but also pentagon, and hexagon (resp. octagon for $N>1$) ones, as iterated integral at $\mu_{N}$. In depth $1$, by Deligne and Goncharov, the only relations satisfied by the motivic iterated integrals are distributions and conjugation relations, stated in $\S 2.4.3$.

\paragraph{Motivic Euler $\star$, $\boldsymbol{\sharp}$ sums.}
Here, assume that $N=1$ or $2$.\footnote{Detailed definitions of these $\star$ and $\sharp$ versions are given in $\S 4.1$.} In the motivic iterated integrals above, $I^{\mathfrak{m}}(\cdots, a_{i}, \cdots)$, $a_{i}$ were in $\lbrace 0, \pm 1 \rbrace$. We can extend by linearity to $a_{i}\in \lbrace \pm \star, \pm \sharp\rbrace$, which corresponds to a $\omega_{\pm \star}$, resp. $\omega_{ \pm\sharp}$ in the iterated integral, with the differential forms:\nomenclature{$\omega_{\pm\star}$, $\omega_{\pm\sharp}$}{specific differential forms}
$$\boldsymbol{\omega_{\pm\star}}\mathrel{\mathop:}= \omega_{\pm 1}- \omega_{0}=\frac{dt}{t(\pm t -1)} \quad \text{ and }  \quad \boldsymbol{\omega_{\pm\sharp}}\mathrel{\mathop:}=2 \omega_{\pm 1}-\omega_{0}=\frac{(t \pm 1)dt}{t(t\mp 1)}.$$
It means that, by linearity, for $A,B$ sequences in $\lbrace 0, \pm 1, \pm \star, \pm \sharp \rbrace$:
\begin{equation} \label{eq:miistarsharp}
 I^{\mathfrak{m}}(A, \pm \star, B)= I^{\mathfrak{m}}(A, \pm 1, B) - I^{\mathfrak{m}}(A, 0, B),  \text{ and }  I^{\mathfrak{m}}(A, \pm \sharp, B)=  2 I^{\mathfrak{m}}(A, \pm 1, B) - I^{\mathfrak{m}}(A, 0, B).
\end{equation} 
\nomenclature{$\zeta^{\star, \mathfrak{m}}$, resp. $\zeta^{\sharp, \mathfrak{m}}$}{Motivic Euler $\star$ Sums, resp. Motivic Euler $\sharp$ Sums}
\begin{itemize}
\item[$\boldsymbol{\zeta^{\star, \mathfrak{m}}}$: ]\textit{Motivic Euler $\star$ Sums} are defined by a similar integral representation as MES ($\ref{eq:reprinteg}$), with $\omega_{\pm \star}$ replacing the $\omega_{\pm 1}$, except the first one, which stays a $\omega_{\pm 1}$. \\
Their periods, Euler $\star$ sums, which are already common in the literature, can be written as a summation similar than for Euler sums replacing strict inequalities by large ones:
$$ \zeta^{\star}\left(n_{1}, \ldots , n_{p} \right) =  \sum_{0 < k_{1}\leq k_{2} \leq \cdots \leq k_{p}} \frac{\epsilon_{1}^{k_{1}} \cdots \epsilon_{p}^{k_{p}}}{k_{1}^{\mid n_{1}\mid} \cdots k_{p}^{\mid n_{p}\mid}}, \quad \epsilon_{i}\mathrel{\mathop:}=sign(n_{i}), \quad n_{i}\in\mathbb{Z}^{\ast}, n_{p}\neq 1. $$
\item[$\boldsymbol{\zeta^{\sharp, \mathfrak{m}}}$: ] \textit{Motivic Euler $\sharp$ Sums} are defined by a similar integral representation as MES ($\ref{eq:reprinteg}$), with $\omega_{\pm \sharp}$ replacing the $\omega_{\pm 1}$, except the first one, which stays a $\omega_{\pm 1}$.
\end{itemize}
They are both $\mathbb{Q}$-linear combinations of multiple Euler sums, and appear in Chapter $4$, via new bases for motivic MZV (Hoffman $\star$, or with Euler $\sharp$ sums) and in the Conjecture $\ref{lzg}$.\\

\paragraph{Dimensions.}
Algebraic $K$-theory provides an \textit{upper bound} for the dimensions of motivic cyclotomic iterated integrals, since: 
\begin{equation}
\begin{array}{ll}
  \text{Ext}_{\mathcal{MT}_{N,M}}^{1} (\mathbb{Q}(0), \mathbb{Q}(1)) =  (\mathcal{O}_{k_{N}}[\frac{1}{M}])^{\ast} \otimes \mathbb{Q}& \\
    \text{Ext}_{\mathcal{MT}_{\Gamma_{N}}}^{1} (\mathbb{Q}(0), \mathbb{Q}(1)) =  \Gamma_{N} & \\
  \text{Ext}_{\mathcal{MT}_{N,M}}^{1} (\mathbb{Q}(0), \mathbb{Q}(n)) =  \text{Ext}_{\mathcal{MT}_{\Gamma}}^{1} (\mathbb{Q}(0), \mathbb{Q}(n)) = K_{2n-1}(k_{N}) \otimes \mathbb{Q}  & \text{ for } n >1 .\\
  \text{Ext}_{\mathcal{MT}_{N,M}}^{i} (\mathbb{Q}(0), \mathbb{Q}(n))= \text{Ext}_{\mathcal{MT}_{\Gamma}}^{i} (\mathbb{Q}(0), \mathbb{Q}(n)) =0 & \text{ for } i>1 \text{ or } n\leq 0 . 
  \end{array}
\end{equation}  
Let $ n_{\mathfrak{p}_{M}}$\nomenclature{$ n_{\mathfrak{p}_{M}}$}{ the number of different prime ideals above the primes dividing $M$} denote the number of different prime ideals above the primes dividing $M$, $\nu_{N}$\nomenclature{$\nu_{N}$}{the number of primes dividing $N$} the number of primes dividing $N$ and $\varphi$ Euler's indicator function\nomenclature{$\varphi$}{Euler's indicator function}. For $M|N$ (cf. $\cite{Bo}$), using Dirichlet $S$-unit theorem when $n=1$:
\begin{equation}\label{dimensionk}\dim K_{2n-1}(\mathcal{O}_{k_{N}} [1/M]) \otimes \mathbb{Q} = \left\{
\begin{array}{ll}
  1 & \text{ if } N =1 \text{ or } 2 , \text{ and } n \text{ odd }, (n,N) \neq (1,1) .\\
  0 & \text{ if } N =1 \text{ or } 2 , \text{ and } n \text{ even } .\\
 \frac{\varphi(N)}{2}+ n_{\mathfrak{p}_{M}}-1& \text{ if } N >2, n=1 . \\
  \frac{\varphi(N)}{2} & \text{ if } N >2 , n>1  .
\end{array}
\right.
\end{equation}
The numbers of generators in each degree, corresponding to the categories $\mathcal{MT}_{N,M}$ resp. $\mathcal{MT}_{\Gamma_{N}}$, differ only in degree $1$:
\begin{equation}\label{eq:agamma}
\begin{array}{ll}
 \text{In degree } >1 : & b_{N}\mathrel{\mathop:}=b_{N,M}= b_{\Gamma_{N}}= \frac{\varphi(N)}{2} \\
  \text{In degree } 1 : & a_{N,M}\mathrel{\mathop:}=\frac{\varphi(N)}{2}+ n_{\mathfrak{p}_{M}}-1 \quad \text{ whereas } \quad a_{\Gamma_{N}}\mathrel{\mathop:}= \frac{\varphi(N)}{2}+\nu(N)-1.
\end{array}
\end{equation}
\texttt{Nota Bene}: The following formulas in this paragraph can be applied for the categories $\mathcal{MT}_{N,M}$ resp. $\mathcal{MT}_{\Gamma_{N}}$, replacing $a_{N}$ by $a_{N,M}$ resp. $a_{\Gamma_{N}}$.\\
\\
In degree $1$, for $\mathcal{MT}_{M,N}$, only the units modulo torsion matter whereas for the category $\mathcal{MT}_{\Gamma_{N}}$, only the cyclotomic units modulo torsion matter in degree $1$, cf. $\S 2.4.3$. Recall that cyclotomic units form a subgroup of finite index in the group of units, and generating families for cyclotomic units modulo torsion are (cf. $\cite{Ba}$)\footnote{If we consider cyclotomic units in $\mathbb{Z}[\xi_{N}]\left[ \frac{1}{M}\right] $, with $M=\prod r_{i}$, $r_{i}$ prime power, we have to add $\lbrace 1- \xi_{r_{i}}\rbrace$.}:\nomenclature{ $a\wedge b$}{$gcd(a,b)$}
$$\begin{array}{ll}
\text{ For } N=p^{s} : &\left\lbrace \frac{1-\xi_{N}^{a}}{1-\xi_{N}} , a\wedge p=1 \right\rbrace, \quad \text{ where } a\wedge b\mathrel{\mathop:}= gcd(a,b).\\
\text{ For } N=\prod_{i} p_{i}^{s_{i}}= \prod q_{i} : &\left\lbrace \frac{1-\xi_{q_{i}}^{a}}{1-\xi_{q_{i}}} , a\wedge p_{i}=1 \right\rbrace \cup  \left\lbrace 1-\xi_{d}^{a}, \quad  a\wedge d=1, d\mid N, d\neq q_{i} \right\rbrace \\
\end{array}. $$
Results on cyclotomic units determine depth $1$ weight $1$ results for MMZV$_{\mu_{N}}$ (cf. $\S. 2.4.3$).\\
\\
\\
Knowing dimensions, we lift $(\ref{eq:uab})$ to a non-canonical isomorphism with the free Lie algebra:
\begin{equation}
\label{eq:LieAlg}
\mathfrak{u}^{\mathcal{MT} } \underrel{n.c}{\cong} L\mathrel{\mathop:}= \mathbb{L}_{\mathbb{Q}} \left\langle  \left(  \sigma^{j}_{1}\right)_{1 \leq j \leq a_{N}},  \left(  \sigma^{j}_{i}\right)_{1 \leq j \leq b_{N}}, i>1 \right\rangle \quad \sigma_{i} \text{ in degree } -i.
\end{equation}
The generators $\sigma^{j}_{i}$\nomenclature{$\sigma^{j}_{i}$}{generators of the graded Lie algebra $\mathfrak{u}$} of the graded Lie algebra $\mathfrak{u}$ are indeed non-canonical, only their classes in the abelianization are.\footnote{In other terms, this means:
\begin{equation}
H_{1}(\mathfrak{u}^{\mathcal{MT}}; \mathbb{Q}) \cong \bigoplus_{i,j \text{ as above }}  [\sigma^{j}_{i}]\mathbb{Q} , \quad H^{B}_{i}(\mathfrak{u}^{\mathcal{MT} }; \mathbb{Q}) =0 \text{ for } i>1.
\end{equation}}
For the fundamental Hopf algebra, with $f^{j}_{i}=(\sigma^{j}_{i})^{\vee}$\nomenclature{$f^{j}_{i}$}{are defined as $(\sigma^{j}_{i})^{\vee}$} in degree $j$:
\begin{equation}
\label{HopfAlg}
\mathcal{A}^{\mathcal{MT}} \underrel{n.c}{\cong} A\mathrel{\mathop:}= \mathbb{Q} \left\langle  \left(  f^{j}_{1}\right)_{1 \leq j \leq a_{N}},  \left(  f^{j}_{i}\right)_{1 \leq j \leq b_{N}}, i>1 \right\rangle .
\end{equation}
\begin{framed}
$\mathcal{A}^{\mathcal{MT}}$ is a cofree commutative graded Hopf algebra cogenerated by $a_{N}$ elements $f^{\bullet}_{1}$ in degree 1, and $b_{N}$ elements $f^{\bullet}_{r}$ in degree $r>1$.
\end{framed}

The comodule $\mathcal{H}^{N}\subseteq \mathcal{O}(_{0}\Pi_{1})$ embeds, non-canonically\footnote{We can fix the image of algebraically independent elements with trivial coaction.\\
For instance, for $N=3$, we can choose to send: $ \zeta^{\mathfrak{m}}\left( r \atop j \right)  \xmapsto{\phi} f_{r} , \quad \text{ and } \quad \left( 2i \pi \right)^{\mathfrak{m}}  \xmapsto{\phi} g_{1}$.}, into $\mathcal{H}^{\mathcal{MT}_{N}}$ and hence:\nomenclature{$\phi^{N}$}{the Hopf algebra morphism $\mathcal{H}^{N}  \hookrightarrow H^{N}$}\nomenclature{$H^{N}$}{the Hopf algebra $\mathbb{Q} \left\langle \left(f^{j}_{1}\right) _{1\leq j \leq a_{N}}, \left( f^{j}_{r}\right)_{r>1\atop 1\leq j \leq b_{N}} \right\rangle  \otimes \mathbb{Q}\left[ g_{1} \right]$}
  \begin{framed}
\begin{equation}\label{eq:phih}
\mathcal{H}^{N}    \xhookrightarrow[n.c.]{\quad\phi^{N}\quad}  H^{N}\mathrel{\mathop:}=  \mathbb{Q} \left\langle \left(f^{j}_{1}\right) _{1\leq j \leq a_{N}}, \left( f^{j}_{r}\right)_{r>1\atop 1\leq j \leq b_{N}} \right\rangle  \otimes \mathbb{Q}\left[ g_{1} \right].
\end{equation}
  \end{framed}
  \noindent
  \texttt{Nota Bene:} This comodule embedding is an isomorphism for $N=1,2,3,4,\mlq 6\mrq,8$ (by F. Brown $\cite{Br2}$ for $N=1$, by Deligne $\cite{De}$ for the other cases; new proof in Chapter $5$), since the categories $\mathcal{MT}'_{N}$, $\mathcal{MT}(\mathcal{O}\left[ \frac{1}{N}\right] )$ and $\mathcal{MT}_{\Gamma_{N}}$ are equivalent. However, for some other $N$, such as $N$ prime greater than $5$, it is not an isomorphism.\\
\noindent
Looking at the dimensions $d^{N}_{n}\mathrel{\mathop:}= \dim \mathcal{H}^{\mathcal{MT}_{N}}_{n}$:\nomenclature{$d^{N}_{n}$}{the dimension of the $\mathbb{Q}$vector space $\mathcal{H}^{\mathcal{MT}_{N}}_{n}$}
\begin{lemm}
For $N>2$, $d^{N}_{n}$ satisfies two (equivalent) recursive formulas\footnote{Those two recursive formulas, although equivalent, leads to two different perspective for counting dimensions.}:
$$\begin{array}{lll}
d^{N}_{n} = & 1 + a_{N} d_{n-1}+ b_{N}\sum_{i=2}^{n} d_{n -i} & \\
d^{N}_{n} = & (a_{N}+1)d_{n-1}+ (b_{N}-a_{N})d_{n -2} & \text{ with } \left\lbrace  \begin{array}{l}
 d_{0}=1\\
d_{1}=a_{N}+1
\end{array}\right. 
\end{array} .$$
Hence the Hilbert series for the dimensions of $\mathcal{H}^{\mathcal{MT}}$ is:
$$h_{N}(t)\mathrel{\mathop:}=\sum_{k} d_{k}^{N} t^{k}=\frac{1}{1-(a_{N}+1)t+ (a_{N}-b_{N})t^{2}}. $$
\end{lemm}
In particular, these dimensions (for $\mathcal{H}^{\mathcal{MT}_{\Gamma_{N}}}$) are an upper bound for the dimensions of motivic MZV$_{\mu_{N}}$ (i.e. of $\mathcal{H}^{N}$), and hence of MZV$_{\mu_{N}}$ by the period map. In the case $N=p^{r}$, $p\geq 5$ prime, this upper bound is conjectured to be not reached; for other $N$ however, this bound is still conjectured to be sharp (cf. $\S 3.4$). \\
\\
\texttt{Examples:}
\begin{itemize}
\item[$\cdot$] For the unramified category $\mathcal{MT}(\mathcal{O}_{N})$:
$$d_{n}= \frac{\varphi(N)}{2}d_{n-1}+ d_{n-2}.$$
\item[$\cdot$]  For $M \mid N$ such that all primes dividing $M$ are inert, $ n_{\mathfrak{p}_{M}}=\nu(N)$. In particular, it is the case if $N=p^{r}$:
$$\text{ For } \mathcal{MT}\left( \mathcal{O}_{p^{r}}\left[  \frac{1}{p} \right] \right)  \text{  ,  } \quad d_{n}= \left( \frac{\varphi(N)}{2}+1\right) ^{n}.$$
Let us detail the cases $N=2,3,4,\mlq 6\mrq,8$ considered in Chapter $5$:\\
\end{itemize}

\begin{tabular}{|c|c|c|c|}
    \hline
    & & & \\
   $N \backslash$ $d_{n}^{N}$& $A$ & Dimension relation $d_{n}^{N}$ & Hilbert series \\
  \hline
  $N=1$\footnotemark[2] & \twolines{$1$ generator in each odd degree $>1$\\
  $\mathbb{Q} \langle f_{3}, f_{5}, f_{7}, \cdots \rangle$ }  &\twolines{$d_{n}=d_{n-3} +d_{n-2}$,\\ $d_{2}=1$, $d_{1}=0$} & $ \frac{1}{1-t^{2}-t^{3}}$  \\
      \hline
   $N=2$\footnotemark[3] & \twolines{$1$ generator in each odd degree $\geq 1$\\
   $\mathbb{Q} \langle f_{1}, f_{3}, f_{5}, \cdots \rangle$ } &  \twolines{$d_{n}=d_{n-1} +d_{n-2}$\\ $d_{0}=d_{1}=1$}  & $ \frac{1}{1-t-t^{2}}$  \\
       \hline
  $N=3,4$ & \twolines{$1$ generator in each degree $\geq 1$\\
  $\mathbb{Q} \langle f_{1}, f_{2}, f_{3}, \cdots \rangle$ } & $d_{k}=2d_{k-1} = 2^{k}$ & $\frac{1}{1-2t}$  \\
      \hline
  $N=8$ & \twolines{$2$ generators in each degree $\geq 1$ \\ $\mathbb{Q} \langle f^{1}_{1}, f^{2}_{1}, f^{1}_{2}, f^{2}_{2}, \cdots \rangle$ } & $d_{k}= 3 d_{k -1}=3^{k}$ & $\frac{1}{1-3t}$  \\
      \hline
   \twolines{$N=6$ \\ $\mathcal{MT}(\mathcal{O}_{6}\left[\frac{1}{6}\right])$} & \twolines{$1$ in each degree $> 1$, $2$ in degree $1$\\
   $\mathbb{Q} \langle f^{1}_{1}, f^{2}_{1}, f_{2}, f_{3}, \cdots \rangle$ } & \twolines{$d_{k}= 3 d_{k -1} -d_{k-2}$, \\$d_{1}=3$} & $\frac{1}{1-3t+t^{2}}$ \\
       \hline
   \twolines{$N=6$ \\ $\mathcal{MT}(\mathcal{O}_{6})$} & \twolines{$1$ generator in each degree $>1$\\
   $\mathbb{Q} \langle f_{2}, f_{3}, f_{4}, \cdots \rangle$} & \twolines{$d_{k}= 1+ \sum_{i\geq 2} d_{k-i}$\\$=d_{k -1} +d_{k-2}$} & $ \frac{1}{1-t-t^{2}}$ \\
   \hline
\end{tabular}
\footnotetext[2]{For $N=1$, Broadhurst and Kreimer made a more precise conjecture for dimensions of multiple zeta values graded by the depth, which transposes to motivic ones:
\begin{equation}\label{eq:bkdepth}
 \sum \dim (gr^{\mathfrak{D}}_{d} \mathcal{H}^{1}_{n})s^{n}t^{d} = \frac{1+\mathbb{E}(s)t}{1- \mathbb{O}(s)t+\mathbb{S}(s)t^{2}-\mathbb{S}(s) t^{4}} , \quad \text{ where } \begin{array}{l}
 \mathbb{E}(s)\mathrel{\mathop:}= \frac{s^{2}}{1-s^{2}}\\ 
 \mathbb{O}(s)\mathrel{\mathop:}= \frac{s^{3}}{1-s^{2}}\\
 \mathbb{S}(s)\mathrel{\mathop:}= \frac{s^{12}}{(1-s^{4})(1-s^{6})}
 \end{array}
\end{equation}
where $ \mathbb{E}(s)$, resp. $ \mathbb{O}(s)$, resp. $ \mathbb{S}(s)$ are the generating series of even resp. odd simple zeta values resp. of the space of cusp forms for the full modular group $PSL_{2}(\mathbb{Z})$. The coefficient $\mathbb{S}(s)$ of $t^{2}$ can be understood via the relation between double zetas and cusp forms in $\cite{GKZ}$; The coefficient $\mathbb{S}(s)$ of $t^{4}$, underlying exceptional generators in depth $4$, is now also understood by the recent work of F. Brown $\cite{Br3}$, who gave an interpretation of this conjecture via the homology of an explicit Lie algebra.}
\footnotetext[3]{For $N=2$, the dimensions are Fibonacci numbers.}

\section{Motivic Hopf algebra}

\subsection{Motivic Lie algebra.}

Let $\mathfrak{g}$\nomenclature{$\mathfrak{g}$}{ the free graded Lie algebra generated by $e_{0},(e_{\eta})_{\eta\in\mu_{N}}$ in degree $-1$} the free graded Lie algebra generated by $e_{0},(e_{\eta})_{\eta\in\mu_{N}}$ in degree $-1$. Then, the completed Lie algebra $\mathfrak{g}^{\wedge}$ is the Lie algebra of $_{0}\Pi_{1}(\mathbb{Q})$ and the universal enveloping algebra $ U\mathfrak{g}$ is the cocommutative Hopf algebra which is the graded dual of $O(_{0}\Pi_{1})$:
\begin{equation} \label{eq:ug}  (U\mathfrak{g})_{n}=\left( \mathbb{Q}e_{0} \oplus \left( \oplus_{\eta\in\mu_{N}} \mathbb{Q}e_{\eta}\right) \right) ^{\otimes n}= (O(_{0}\Pi_{1})^{\vee})_{n}.
\end{equation}
The product is the concatenation, and the coproduct is such that $e_{0},e_{\eta}$ are primitive.\\
\\
Considering the motivic version of the Drinfeld associator:\nomenclature{$\Phi^{\mathfrak{m}}$}{the motivic Drinfeld associator}
\begin{equation} \label{eq:associator}
\Phi^{\mathfrak{m}}\mathrel{\mathop:}= \sum_{w} \zeta^{\mathfrak{m}} (w) w \in \mathcal{H}\left\langle \left\langle e_{0},e_{\eta} \right\rangle \right\rangle \text{, where :}
\end{equation} 
$$ \zeta^{\mathfrak{m}} (e_{0}^{n}e_{\eta_{1}}e_{0}^{n_{1}-1}\cdots e_{\eta_{p}}e_{0}^{n_{p}-1})  =\zeta^{\mathfrak{m}}_{n}\left( n_{1}, \ldots, n_{p} \atop \epsilon_{1} , \ldots, \epsilon_{p}\right)  \text{ with } \begin{array}{l}
\epsilon_{p}\mathrel{\mathop:}=\eta_{p}^{-1}\\
 \epsilon_{i}\mathrel{\mathop:}=\eta_{i}^{-1}\eta_{i+1}
\end{array}.$$
\texttt{Nota Bene:} This motivic Drinfeld associator satisfies the double shuffle relations, and, for $N=1$, the associator equations defined by Drinfeld (pentagon and hexagon), replacing $2\pi i$ by the Lefschetz motivic period $\mathbb{L}^{\mathfrak{m}}$; for $N>1$, an octagon relation generalizes this hexagon relation, as we will see in $\S 4.2.2$.\\
Moreover, it defines a map:
$$\oplus \mathcal{H}_{n}^{\vee} \rightarrow U \mathfrak{g} \quad \text{ which induces a map: } \oplus \mathcal{L}_{n}^{\vee} \rightarrow U \mathfrak{g}.$$ 
Define $\boldsymbol{\mathfrak{g}^{\mathfrak{m}}}$, the \textit{Lie algebra of motivic elements} as the image of $\oplus \mathcal{L}_{n}^{\vee}$ in $U \mathfrak{g}$:\footnote{The action of the Galois group $\mathcal{U}^{\mathcal{MT}}$ turns $\mathcal{L}$ into a coalgebra, and hence $\mathfrak{g}^{\mathfrak{m}}$ into a Lie algebra.}
\begin{equation} \label{eq:motivicliealgebra} \oplus \mathcal{L}_{n}^{\vee} \xrightarrow{\sim} \mathfrak{g}^{\mathfrak{m}} \hookrightarrow U \mathfrak{g}.
\end{equation}
The Lie algebra $\mathfrak{g}^{\mathfrak{m}}$ is equipped with the Ihara bracket given precisely below.
Notice that for the cases $N=1,2,3,4,\mlq 6\mrq,8$, $\mathfrak{g}^{\mathfrak{m}}$ is non-canonically isomorphic to the free Lie algebra $L$ defined in $(\ref{eq:LieAlg})$, generated by $(\sigma_{i})'s$. 

\paragraph{Ihara action.}
As said above, the group scheme $\mathcal{V}$ of automorphisms of $_{x}\Pi_{y}, x,y\in\lbrace 0, \mu_{N} \rbrace$ is isomorphic to $_{0}\Pi_{1}$ ($\ref{eq:gpaut}$), and the group law of automorphisms leads to the Ihara action. More precisely, for $a\in _{0}\Pi_{1}$ (cf. $\cite{DG}$): 
\begin{equation} \label{eq:actionpi01}
\begin{array}{lllll}
\text{ The action on } _{0}\Pi_{0} : \quad \quad &\langle a\rangle_{0} : &  _{0}\Pi_{0} & \rightarrow &_{0} \Pi_{0} \\
&& \exp(e_{0}) &\mapsto & \exp(e_{0}) \\
&&\exp(e_{\eta}) &\mapsto &([\eta]\cdot a) \exp(e_{\eta}) ([\eta]\cdot a)^{-1} \\
\text{ Then, the action on } _{0}\Pi_{1} :\quad \quad & \langle a\rangle :   & _{0}\Pi_{1}  &\rightarrow &_{0}\Pi_{1} \\
& &  b &\mapsto & \langle a\rangle _{0} (b)\cdot a  
\end{array}
\end{equation}
This action is called the \textbf{\textit{Ihara action}}:\nomenclature{$\circ$}{Ihara action}
\begin{equation} \label{eq:iharaaction}
\begin{array}{llll}
\circ : &  _{0}\Pi_{1}  \times _{0}\Pi_{1}  &  \rightarrow & _{0}\Pi_{1} \\
& (a,b) & \mapsto &   a\circ b \mathrel{\mathop:}= \langle a\rangle _{0} (b)\cdot a.
\end{array}
\end{equation}
At the Lie algebra level, it defines the \textit{Ihara bracket} on $Lie(_{0}\Pi_{1})$:
\begin{equation}
\lbrace a, b\rbrace \mathrel{\mathop:}= a \circ b - b\circ a.
\end{equation}

\texttt{Nota Bene:} The dual point of view leads to a combinatorial coaction $\Delta^{c}$, which is the keystone of this work.

\subsection{Coaction}

The motivic Galois group $\mathcal{G}^{\mathcal{MT}_{N}}$ and hence $\mathcal{U}^{\mathcal{MT}}$ acts on the de Rham realization $_{0}\Pi_{1}$ of the motivic fundamental groupoid (cf. $\cite{DG}, \S 4.12$). It is fundamental, since the action of $\mathcal{U}^{\mathcal{MT}}$ is compatible with the structure of $_{x}\Pi_{y}$ (groupoid, $\mu_{N}$ equivariance and inertia), that this action factorizes through the Ihara action, using the isomorphism $\mathcal{V}\cong _{0}\Pi_{1}$ ($\ref{eq:gpaut}$):
$$ \xymatrix{
\mathcal{U}^{\mathcal{MT}}\times _{0}\Pi_{1}  \ar[r] \ar[d] &_{0}\Pi_{1}  \ar[d]^{\sim}\\
_{0}\Pi_{1}  \times _{0}\Pi_{1}  \ar[r]^{\circ}  & _{0}\Pi_{1} \\
}$$
Since  $\mathcal{A}^{\mathcal{MT}}= \mathcal{O}(\mathcal{U}^{\mathcal{MT}})$, this action gives rise by duality to a coaction: $\Delta^{\mathcal{MT}}$, compatible with the grading, represented below. By the previous diagram, the combinatorial coaction $\Delta^{c}$ (on words on $0, \eta\in\mu_{N}$), which is explicit (the formula being given below), factorizes through $\Delta^{\mathcal{MT}}$. Remark that $\Delta^{\mathcal{MT}}$ factorizes through $\mathcal{A}$, since $\mathcal{U}$ is the quotient of $\mathcal{U}^{\mathcal{MT}}$ by the kernel of its action on $_{0}\Pi_{1}$. By passing to the quotient, it induces a coaction $\Delta$ on $\mathcal{H}$:

$$ \label{Coaction} \xymatrix{
\mathcal{O}(_{0}\Pi_{1}) \ar[r]^{\Delta^{c}} \ar[d]^{\sim} & \mathcal{A} \otimes_{\mathbb{Q}} \mathcal{O} (_{0}\Pi_{1})  \ar[d] \\
\mathcal{O}(_{0}\Pi_{1}) \ar[d]\ar[r]^{\Delta^{\mathcal{MT}}} & \mathcal{A}^{\mathcal{MT}} \otimes_{\mathbb{Q}} \mathcal{O} (_{0}\Pi_{1})  \ar[d]\\
\mathcal{H} \ar[r]^{\Delta}  & \mathcal{A} \otimes \mathcal{H}. \\
}$$
\\
The coaction for motivic iterated integrals is given by the following formula, due to A. B. Goncharov (cf. $\cite{Go1}$) for $\mathcal{A}$ and extended by F. Brown to $\mathcal{H}$ (cf. $\cite{Br2}$):\nomenclature{$\Delta$}{Goncharov coaction}
\begin{theom} \label{eq:coaction}
The coaction $\Delta: \mathcal{H} \rightarrow \mathcal{A} \otimes_{\mathbb{Q}} \mathcal{H}$ is given by the combinatorial coaction $\Delta^{c}$:
$$\Delta^{c} I^{\mathfrak{m}}(a_{0}; a_{1}, \cdots a_{n}; a_{n+1}) =$$
$$\sum_{k ;i_{0}= 0<i_{1}< \cdots < i_{k}<i_{k+1}=n+1} \left( \prod_{p=0}^{k} I^{\mathfrak{a}}(a_{i_{p}}; a_{i_{p}+1}, \cdots a_{i_{p+1}-1}; a_{i_{p+1}}) \right) \otimes I^{\mathfrak{m}}(a_{0}; a_{i_{1}}, \cdots a_{i_{k}}; a_{n+1}) .$$
\end{theom}
\noindent
\textsc{Remark:} It has a nice geometric formulation, considering the $a_{i}$ as vertices on a half-circle: 
$$\Delta^{c} I^{\mathfrak{m}}(a_{0}; a_{1}, \cdots a_{n}; a_{n+1})=\sum_{\text{ polygons on circle  } \atop \text{ with vertices } (a_{i_{p}})}     \prod_{p}  I^{\mathfrak{a}}\left(  \text{ arc between consecutive vertices  } \atop \text{ from } a_{i_{p}} \text{ to } a_{i_{p+1}} \right)  \otimes I^{\mathfrak{m}}(\text{  vertices } ).$$
\texttt{Example}: In the reduced coaction\footnote{$\Delta'(x):=\Delta(x)-1\otimes x-x\otimes 1$} of $\zeta^{\mathfrak{m}}(-1,3)=I^{\mathfrak{m}}(0; -1,1,0,0;1)$, there are $3$ non zero cuts: \includegraphics[]{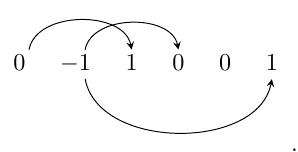}. Hence:
\begin{multline}\nonumber
\Delta'(I^{\mathfrak{m}}(0; -1,1,0,0;1))\\
= I^{\mathfrak{a}}(0; -1;1) \otimes I^{\mathfrak{m}}(0; 1,0,0;1)+ I^{\mathfrak{a}}(-1; 1;0) \otimes I^{\mathfrak{m}}(0; -1,0,0;1)+ I^{\mathfrak{a}}(-1; 1,0,0;1) \otimes I^{\mathfrak{m}}(0; -1;1)
\end{multline}   
I.e, in terms of motivic Euler sums, using the properties of motivic iterated integrals ($\S \ref{propii}$):
$$\Delta'(\zeta^{\mathfrak{m}}(-1,3))= \zeta^{\mathfrak{a}}(-1)\otimes \zeta^{\mathfrak{m}}(3)-\zeta^{\mathfrak{a}}(-1)\otimes \zeta^{\mathfrak{m}}(-3)+ (\zeta^{\mathfrak{a}}(3)-\zeta^{\mathfrak{a}}(-3) )\otimes \zeta^{\mathfrak{m}}(-1).$$
\\
Define for $r\geq 1$, the \textit{derivation operators}:
\begin{equation}\label{eq:dr}
\boldsymbol{D_{r}}: \mathcal{H} \rightarrow \mathcal{L}_{r} \otimes_{\mathbb{Q}} \mathcal{H},
\end{equation}
 composite of $\Delta'= \Delta^{c}- 1\otimes id$ with $\pi_{r} \otimes id$, where $\pi_{r}$ is the projection $\mathcal{A} \rightarrow \mathcal{L} \rightarrow \mathcal{L}_{r}$.\\
 \\
\texttt{Nota Bene:} It is sufficient to consider these weight-graded derivation operators to keep track of all the information of the coaction.\\
\\
According to the previous theorem, the action of $D_{r}$ on $I^{\mathfrak{m}}(a_{0}; a_{1}, \cdots a_{n}; a_{n+1})$ is:\nomenclature{$D_{r}$}{the $r$-weight-graded part of the coaction $\Delta$ }
\begin{framed}
\begin{equation}
\label{eq:Der}
D_{r}I^{\mathfrak{m}}(a_{0}; a_{1}, \cdots, a_{n}; a_{n+1})= 
\end{equation}
$$\sum_{p=0}^{n-1} I^{\mathfrak{l}}(a_{p}; a_{p+1}, \cdots, a_{p+r}; a_{p+r+1}) \otimes I^{\mathfrak{m}}(a_{0}; a_{1}, \cdots, a_{p}, a_{p+r+1}, \cdots, a_{n}; a_{n+1}) .$$
\end{framed}
\textsc{Remarks}
\begin{itemize}
\item[$\cdot$] Geometrically, it is equivalent to keep in the previous coaction only the polygons corresponding to a unique cut of (interior) length $r$ between two elements of the iterated integral.
\item[$\cdot$] These maps $D_{r}$ are derivations:
$$D_{r} (XY)= (1\otimes X) D_{r}(Y) + (1\otimes Y) D_{r}(X).$$
\item[$\cdot$] This formula is linked with the differential equation satisfied by the iterated integral $I(a_{0}; \cdots; a_{n+1})$ when the $a_{i} 's$ vary (cf. \cite{Go1})\footnote{Since $I(a_{i-1};a_{i};a_{i+1})= \log(a_{i+1}-a_{i})-\log(a_{i-1}-a_{i})$.}:
$$dI(a_{0}; \cdots; a_{n+1})= \sum dI(a_{i-1};a_{i};a_{i+1})  I(a_{0}; \cdots, \widehat{a_{i}}, \cdots; a_{n+1}).$$
\end{itemize}
\texttt{Example}: By the previous example:\\
$$D_{3}(\zeta^{\mathfrak{m}}(-1,3))=(\zeta^{\mathfrak{a}}(3)-\zeta^{\mathfrak{a}}(-3) )\otimes \zeta^{\mathfrak{m}}(-1)$$
$$D_{1}(\zeta^{\mathfrak{m}}(-1,3))= \zeta^{\mathfrak{a}}(-1)\otimes ( \zeta^{\mathfrak{m}}(3)- \zeta^{\mathfrak{m}}(-3)) $$

\subsection{Depth filtration}

The inclusion of $\mathbb{P}^{1}\diagdown \lbrace 0, \mu_{N},\infty\rbrace \subset \mathbb{P}^{1}\diagdown \lbrace 0,\infty\rbrace$ implies the surjection for the de Rham realizations of fundamental groupoid:
\begin{equation} \label{eq:drsurj}
  _{0}\Pi_{1} \rightarrow  \pi_{1}^{dR}(\mathbb{G}_{m}, \overrightarrow{01}).
  \end{equation}
Looking at the dual, it corresponds to the inclusion of:
\begin{equation} \label{eq:drsurjdual}
 \mathcal{O}  \left( \pi_{1}^{dR}(\mathbb{G}_{m}, \overrightarrow{01} ) \right)  \cong \mathbb{Q} \left\langle e^{0} \right\rangle  \xhookrightarrow[\quad \quad]{} \mathcal{O} \left(  _{0}\Pi_{1} \right) \cong \mathbb{Q} \left\langle e^{0}, (e^{\eta})_{\eta} \right\rangle  .
   \end{equation}
This leads to the definition of  an increasing \textit{depth filtration} $\mathcal{F}^{\mathfrak{D}}$ on $\mathcal{O}(_{0}\Pi_{1})$\footnote{It is the filtration dual to the filtration given by the descending central series of the kernel of the map  $\ref{eq:drsurj}$; it can be defined also from the cokernel of $\ref{eq:drsurjdual}$, via the decontatenation coproduct.} such that:\nomenclature{$\mathcal{F}_{\bullet}^{\mathfrak{D}}$}{the depth filtration}
 \begin{equation}\label{eq:filtprofw} \boldsymbol{\mathcal{F}_{p}^{\mathfrak{D}}\mathcal{O}(_{0}\Pi_{1})} \mathrel{\mathop:}= \left\langle  \text{ words } w \text{ in }e^{0},e^{\eta}, \eta\in\mu_{N} \text{ such that } \sum_{\eta\in\mu_{N}} deg _{e^{\eta}}w \leq p \right\rangle _{\mathbb{Q}}.
 \end{equation}
This filtration is preserved by the coaction and thus descends to $\mathcal{H}$ (cf. $\cite{Br3}$), on which:
 \begin{equation}\label{eq:filtprofh}  \mathcal{F}_{p}^{\mathfrak{D}}\mathcal{H}\mathrel{\mathop:}= \left\langle  \zeta^{\mathfrak{m}}\left( n_{1}, \ldots, n_{r} \atop \epsilon_{1}, \ldots, \epsilon_{r} \right) , r\leq p \right\rangle _{\mathbb{Q}}.
 \end{equation}
In the same way, we define $ \mathcal{F}_{p}^{\mathfrak{D}}\mathcal{A}$ and  $\mathcal{F}_{p}^{\mathfrak{D}}\mathcal{L}$. Beware, the corresponding grading on $\mathcal{O}(_{0}\Pi_{1})$ is not motivic and the depth is not a grading on $\mathcal{H}$\footnote{ For instance: $\zeta^{\mathfrak{m}}(3)=\zeta^{\mathfrak{m}}(1,2)$. }. The graded spaces $gr^{\mathfrak{D}}_{p}$ are defined as the quotient $\mathcal{F}_{p}^{\mathfrak{D}}/\mathcal{F}_{p-1}^{\mathfrak{D}}$.\\
Similarly, there is an increasing depth filtration on $U\mathfrak{g}$, considering the degree in $\lbrace e_{\eta}\rbrace_{\eta\in\mu_{N}}$, which passes to the motivic Lie algebra $\mathfrak{g}^{\mathfrak{m}}$($\ref{eq:motivicliealgebra}$) such that the graded pieces $gr^{r}_{\mathfrak{D}} \mathfrak{g}^{\mathfrak{m}}$ are dual to $gr^{\mathfrak{D}}_{r} \mathcal{L}$.\\ 
In depth $1$, there are canonical elements:\footnote{For $N=1$, there are only the $\overline{\sigma}_{2i+1}\mathrel{\mathop:}=(\text{ad} e_{0})^{2i} (e_{1}) \in gr^{1}_{\mathfrak{D}} \mathfrak{g}^{\mathfrak{m}}$, $i>0$ and the subLie algebra generated by them is not free, which means also there are other \say{exceptional} generators in higher depth, cf. \cite{Br2}.\\
For $N=2,3,4,\mlq 6\mrq,8$, when keeping $\eta_{i}$ as in Lemma $5.2.1$,  $(\overline{\sigma}^{(\eta_{i})}_{i})$ then generate a free Lie algebra in $gr_{\mathfrak{D}} \mathfrak{g}$.}
\begin{equation}\label{eq:oversigma}
\overline{\sigma}^{(\eta)}_{i}\mathrel{\mathop:}=(\text{ad } e_{0})^{i-1} (e_{\eta}) \in gr^{1}_{\mathfrak{D}} \mathfrak{g}^{\mathfrak{m}}.
\end{equation}
They satisfy the distribution and conjugation relations stated below.\\

\paragraph{Depth $\boldsymbol{1}$.}
In depth $1$, it is known for $\mathcal{A}$ (cf. $\cite{DG}$ Theorem $6.8$):
\begin{lemm}[Deligne, Goncharov]
The elements $\zeta^{\mathfrak{a}} \left( r; \eta \right)$ are subject only to the following relations in $\mathcal{A}$:
\begin{description}
\item[Distribution]
$$\forall d|N \text{  ,  }  \forall \eta\in\mu_{\frac{N}{d}} \text{  ,  } (\eta,r)\neq(1,1)\text{    ,    } \zeta^{\mathfrak{a}} \left({r \atop \eta}\right)= d^{r-1} \sum_{\epsilon^{d}=\eta} \zeta^{\mathfrak{a}} \left({r \atop \epsilon}\right).$$ 
\item[Conjugation]
$$\zeta^{\mathfrak{a}} \left({r \atop \eta}\right)= (-1)^{r-1} \zeta^{\mathfrak{a}} \left({r \atop \eta^{-1}}\right).$$
\end{description}
\end{lemm}
\textsc{Remark}: More generally, distribution relations for MZV relative to $\mu_{N}$ are:
$$\forall d| N, \quad \forall \epsilon_{i}\in\mu_{\frac{N}{d}} \text{  ,  } \quad \zeta\left( { n_{1}, \ldots , n_{p} \atop  \epsilon_{1} , \ldots ,\epsilon_{p} } \right) = d^{\sum n_{i} - p} \sum_{\eta_{1}^{d}=\epsilon_{1}} \cdots  \sum_{\eta_{p}^{d}=\epsilon_{p}} \zeta \left( {n_{1}, \ldots , n_{p} \atop  \eta_{1} , \ldots ,\eta_{p} } \right) .$$
They are deduced from the following identity:  
$$\text{ For } d|N , \epsilon\in\mu_{\frac{N}{d}} \text {   ,  } \sum_{\eta^{d}=\epsilon} \eta^{n}=  \left\{
\begin{array}{ll}
  d \epsilon ^{\frac{n}{d}}& \text{ if } d|n \\
  0 & \text{ else }.\\
\end{array}
\right. $$
These relations are obviously analogous of those satisfied by the cyclotomic units modulo torsion. \\
\\
In weight $r>1$, a basis for $gr_{1}^{\mathfrak{D}} \mathcal{A}$ is formed by depth $1$ MMZV at primitive roots up to conjugation. However, MMZV$_{\mu_{N}}$ of weight $1$, $\zeta^{\mathfrak{m}} \left( 1 \atop \xi^{a}_{N}\right)  = -\log(1-\xi^{a}_{N})$, are more subtle. For instance (already in $\cite{CZ})$:
\begin{lemme}
A $\mathbb{Z}$-basis for $\mathcal{A}_{1}$ is hence:
\begin{description}
\item[$\boldsymbol{N=p^{r}}$:]  $ \quad \quad \left\lbrace \zeta^{\mathfrak{a}}\left( 1 \atop \xi^{k}\right) \quad a\wedge p=1 \quad 1 \leq a \leq \frac{p-1}{2} \right\rbrace$.
\item[$\boldsymbol{N=pq}$:] With $p<q$ primes:    $$ \quad\left\lbrace  \left\lbrace \zeta^{\mathfrak{a}}\left( 1 \atop \xi^{k}\right) \quad a\wedge p=1 \quad 1 \leq a \leq \frac{p-1}{2} \right\rbrace \bigcup_{a\in (\mathbb{Z}/q\mathbb{Z})^{\ast}\diagup \langle -1, p\rangle } \left\lbrace \zeta^{\mathfrak{a}}\left( 1 \atop \xi^{ap}\right)\right\rbrace \diagdown \left\lbrace\zeta^{\mathfrak{a}}\left( 1 \atop \xi^{a}\right) \right\rbrace \right. $$
$$\left.  \bigcup_{a\in (\mathbb{Z}/p\mathbb{Z})^{\ast}\diagup \langle -1, q\rangle}\left\lbrace \zeta^{\mathfrak{a}}\left( 1 \atop \xi^{aq}\right)\right\rbrace \diagdown \left\lbrace\zeta^{\mathfrak{a}}\left( 1 \atop \xi^{a}\right) \right\rbrace \right\rbrace $$
\end{description}
\end{lemme}
\textsc{Remarks}: 
\begin{itemize}
\item[$\cdot$] Indeed, for $N=pq$, a phenomenon of loops occurs: orbits via the action of $p$ and $-1$ on $(\mathbb{Z}/q \mathbb{Z})^{\ast}$, resp. of $q$ and $-1$ on $(\mathbb{Z}/p \mathbb{Z})^{\ast}$. Consequently, for each loop we have to remove a primitive root $\zeta\left( 1 \atop \xi^{a}\right) $ and add the non primitive $\zeta\left( 1 \atop \xi^{ap}\right) $ to the basis.\footnote{Cardinal of an orbit $ \lbrace \pm ap^{i} \mod N\rbrace$ is either the order of $p$ modulo $q$, if odd, or half of the order of $p$ modulo $q$, if even.} The situation for $N$ a product of primes would be analogous, considering different orbits associated to each prime; we just have to pay more attention when orbits intersect, for the choice of the representatives $a$: avoid to withdraw or add an element already chosen for previous orbits.
\item[$\cdot$] Depth $1$ results also highlight a nice behavior in the cases $N=2,3,4,\mlq 6\mrq,8$: primitive roots of unity modulo conjugation form a basis (as in the case of prime powers) and if we restrict (for dimension reasons) for non primitive roots to $1$ (or $\pm 1$ for $N=8$), it is annihilated in weight $1$ and in weight $>1$ modulo $p$.
\item[$\cdot$]  In weight $1$, there always exists a $\mathbb{Z}$- basis.\footnote{Conrad and Zhao conjectured (\cite{CZ}) there exists a basis of MZV$_{\mu_{N}}$ for the $\mathbb{Z}$-module spanned by MZV$_{\mu_{N}}$ for each $N$ and fixed weight $w$, except $N=1$, $w=6,7$.}
\end{itemize} 
\texttt{Example}:
For $N=34$, relations in depth $1$, weight $1$ lead to two orbits, with $(a)\mathrel{\mathop:}=\zeta^{\mathfrak{a}} \left( 1 \atop \xi^{a}_{N}\right) $:
$$\begin{array}{ll}
(2)= (16) +(1) & \quad \quad (6)=(3)+(14) \\
(16)=(8) +(9) & \quad \quad (14)=(7)+(10) \\ 
(8)= (4) +(13) & \quad \quad (10)=(5)+(12) \\
(4)= (2) +(15) & \quad \quad (12)=(11)+(6) \\
\end{array},$$
Hence a basis could be chosen as: 
$$\left\lbrace \zeta^{\mathfrak{a}}\left( 1 \atop \xi_{34}^{k}\right), k\in \lbrace 5,7,9,11,13,15,\boldsymbol{2,6} \rbrace \right\rbrace .$$

\paragraph{Motivic depth.} The \textit{motivic depth} of an element in $\mathcal{H}^{\mathcal{MT}_{N}}$ is defined, via the correspondence ($\ref{eq:phih}$), as the degree of the polynomial in the $(f_{i})$. \footnote{Beware, $\phi$ is non-canonical, but the degree is well defined.}  It can also be defined recursively as, for $\mathfrak{Z}\in \mathcal{H}^{N}$:
$$\begin{array}{lll}
\mathfrak{Z} \text{ of motivic depth } 1 & \text{ if and only if } & \mathfrak{Z}\in  \mathcal{F}_{1}^{\mathfrak{D}}\mathcal{H}^{N}.\\
\mathfrak{Z} \text{ of motivic depth } \leq p & \text{ if and only if } &  \left( \forall r< n, D_{r}(\mathfrak{Z})  \text{ of motivic depth } \leq p-1 \right).
\end{array}.$$
For $\mathfrak{Z}= \zeta^{\mathfrak{m}} \left( n_{1}, \ldots, n_{p} \atop \epsilon_{1}, \ldots, \epsilon_{p} \right) \in \mathcal{H}^{N}$ of motivic depth $p_{\mathfrak{m}}$, we clearly have the inequalities:
$$ \text {depth } p \geq p_{c} \geq p_{\mathfrak{m}} \text{ motivic depth}, \quad \text{ where $ p_{c}$ is the smallest $i$ such that $\mathfrak{Z}\in \mathcal{F}_{i}^{\mathfrak{D}}\mathcal{H}^{N}$}. $$
\texttt{Nota Bene:} For $N=2,3,4, \mlq 6\mrq, 8$, $p_{\mathfrak{m}}$ always coincides with $p_{c}$, whereas for $N=1$, they may differ.

\subsection{Derivation space}

Translating ($\ref{eq:dr}$) for cyclotomic MZV:
\begin{lemm}
\label{drz}
$$D_{r}: \mathcal{H}_{n} \rightarrow  \mathcal{L}_{r} \otimes  \mathcal{H}_{n-r} $$
\begin{multline}
 D_{r} \left(\zeta^{\mathfrak{m}} \left({n_{1}, \ldots , n_{p} \atop \epsilon_{1}, \ldots ,\epsilon_{p}} \right)\right) =   \delta_{r=n_{1}+ \cdots +n_{i}} \zeta^{\mathfrak{l}} \left({n_{1}, \cdots, n_{i} \atop  \epsilon_{1}, \ldots, \epsilon_{i}}\right) \otimes \zeta^{\mathfrak{m}} \left(  { n_{i+1},\cdots, n_{p} \atop \epsilon_{i+1}, \ldots, \epsilon_{p} }\right) \\
+ \sum_{1 \leq i<j\leq p \atop \lbrace r \leq \sum_{k=i}^{j} n_{k} -1\rbrace}  \left[ \delta_{ \sum_{k=i+1}^{j} n_{k} \leq r }  \zeta^{\mathfrak{l}}_{r-  \sum_{k=i+1}^{j}n_{k}} \left({ n_{i+1}, \ldots , n_{j} \atop  \epsilon_{i+1}, \ldots, \epsilon_{j}}\right) +(-1)^{r} \delta_{ \sum_{k=i}^{j-1} n_{k} \leq r}  \zeta^{\mathfrak{l}}_{r-  \sum_{k=i}^{j-1}n_{k}} \left({ n_{j-1}, \cdots, n_{i} \atop  \epsilon_{j-1}^{-1}, \ldots, \epsilon_{i}^{-1}}\right) \right]  \\
 \otimes \zeta^{\mathfrak{m}} \left( {\cdots, \sum_{k=i}^{j} n_{k}-r,\cdots \atop  \cdots , \prod_{k=i}^{j}\epsilon_{k}, \cdots}\right)  
\end{multline}
\end{lemm}
\begin{proof}
Straightforward from $(\ref{eq:dr})$, passing to MZV$_{\mu_{N}}$ notation.
\end{proof}
A key point is that the Galois action and hence the coaction respects the weight grading and the depth filtration\footnote{Notice that $(\mathcal{F}_{0}^{\mathfrak{D}} \mathcal{L}=0$.}:
$$D_{r} (\mathcal{H}_{n}) \subset \mathcal{L}_{r} \otimes_{\mathbb{Q}} \mathcal{H}_{n-r}.$$
$$ D_{r} (\mathcal{F}_{p}^{\mathfrak{D}}  \mathcal{H}_{n}) \subset \mathcal{L}_{r} \otimes_{\mathbb{Q}} \mathcal{F}_{p-1}^{\mathfrak{D}} \mathcal{H}_{n-r}.$$
Indeed, the depth filtration is motivic, i.e.:
$$\Delta (\mathcal{F}^{\mathfrak{D}}_{n}\mathcal{H}) \subset \sum_{p+q=n} \mathcal{F}^{\mathfrak{D}}_{p}\mathcal{A} \otimes \mathcal{F}^{\mathfrak{D}}_{q}\mathcal{H}.$$
Furthermore, $\mathcal{F}^{\mathfrak{D}}_{0}\mathcal{A}=\mathcal{F}^{\mathfrak{D}}_{0}\mathcal{L}=0$. Therefore, the right side of $\Delta(\bullet)$ is in $\mathcal{F}^{\mathfrak{D}}_{q}\mathcal{H}$, with $q<n$. This feature of the derivations $D_{r}$ (decreasing the depth) will enable us to do some recursion on depth through this work.\\
\\
Passing to the depth-graded, define:
$$gr^{\mathfrak{D}}_{p} D_{r}: gr_{p}^{\mathfrak{D}} \mathcal{H} \rightarrow \mathcal{L}_{r} \otimes gr_{p-1}^{\mathfrak{D}} \mathcal{H} \text{, as the composition } (id\otimes gr_{p-1}^{\mathfrak{D}}) \circ D_{r |gr_{p}^{\mathfrak{D}}\mathcal{H}}.$$
By Lemma $\ref{drz}$, all the terms appearing in the left side of $gr^{\mathfrak{D}}_{p} D_{2r+1}$ have depth $1$. Hence, let's consider from now the derivations $D_{r,p}$:\nomenclature{$D_{r,p}$}{depth graded derivations}
\begin{lemm}
\label{Drp}
$$\boldsymbol{D_{r,p}}: gr_{p}^{\mathfrak{D}} \mathcal{H} \rightarrow gr_{1}^{\mathfrak{D}} \mathcal{L}_{r} \otimes gr_{p-1}^{\mathfrak{D}} \mathcal{H} $$
$$ D_{r,p} \left(\zeta^{\mathfrak{m}} \left({n_{1}, \ldots , n_{p} \atop \epsilon_{1}, \ldots ,\epsilon_{p}} \right)\right) = \textsc{(a0)      }  \delta_{r=n_{1}}\ \zeta^{\mathfrak{l}} \left({r \atop  \epsilon_{1}}\right) \otimes \zeta^{\mathfrak{m}} \left(  { n_{2},\cdots \atop \epsilon_{2}, \cdots }\right) $$
$$\textsc{(a)      }  + \sum_{i=2}^{p-1} \delta_{n_{i}\leq r < n_{i}+ n_{i-1}-1} (-1)^{r-n_{i}} \binom {r-1}{r-n_{i}} \zeta^{\mathfrak{l}} \left({ r \atop  \epsilon_{i}}\right) \otimes \zeta^{\mathfrak{m}} \left( {\cdots, n_{i}+n_{i-1}-r,\cdots \atop  \cdots , \epsilon_{i-1}\epsilon_{i}, \cdots}\right)  $$
$$ \textsc{(b)  } -\sum_{i=1}^{p-1} \delta_{n_{i}\leq r < n_{i}+ n_{i+1}-1} (-1)^{n_{i}} \binom{r-1}{r-n_{i}} \zeta^{\mathfrak{l}}  \left( {r \atop   \epsilon_{i}^{-1}}\right) \otimes \zeta^{\mathfrak{m}} \left( {\cdots, n_{i}+n_{i+1}-r, \cdots \atop  \cdots , \epsilon_{i+1}\epsilon_{i}, \cdots}\right)  $$
$$\textsc{(c)  } +\sum_{i=2}^{p-1} \delta_{ r = n_{i}+ n_{i-1}-1 \atop \epsilon_{i-1}\epsilon_{i}\neq 1}  \left( (-1)^{n_{i}} \binom{r-1}{n_{i}-1} \zeta^{\mathfrak{l}}  \left( {r \atop \epsilon_{i-1}^{-1}}  \right) + (-1)^{n_{i-1}-1} \binom{r-1}{n_{i-1}-1} \zeta^{\mathfrak{l}}  \left( {r \atop \epsilon_{i}} \right)  \right)$$
$$\otimes \zeta^{\mathfrak{m}} \left( {\cdots, 1, \cdots \atop \cdots, \epsilon_{i-1} \epsilon_{i}, \cdots}\right) $$

$$ \textsc{(d)  } +\delta_{ n_{p} \leq r < n_{p}+ n_{p-1}-1} (-1)^{r-n_{p}} \binom{r-1}{r-n_{p}} \zeta^{\mathfrak{l}}  \left({r \atop \epsilon_{p}} \right) \otimes \zeta^{\mathfrak{m}} \left( {\cdots, n_{p-1}+n_{p}-r\atop  \cdots, \epsilon_{p-1}\epsilon_{p}}\right)  $$

$$\textsc{(d') } +\delta_{ r = n_{p}+ n_{p-1}-1 \atop \epsilon_{p-1}\epsilon_{p}\neq 1}   (-1)^{n_{p-1}}\left( \binom{r-1}{n_{p}-1} \zeta^{\mathfrak{l}}   \left( {r \atop  \epsilon_{p-1}^{-1}}  \right) - \binom{r-1}{n_{p-1}-1} \zeta^{\mathfrak{l}}  \left( {r \atop \epsilon_{p}}  \right) \right)  \otimes \zeta^{\mathfrak{m}} \left( { \cdots,  1 \atop  \cdots \epsilon_{p-1}\epsilon_{p}}\right) .$$ 
\end{lemm}
\textsc{Remarks}:
\begin{itemize}
\item[$\cdot$]
The terms of type \textsc{(d, d')}, corresponding to a \textit{deconcatenation}, play a particular role since modulo some congruences (using depth $1$ result for the left side of the coaction), we will get rid of the other terms in the cases $N=2,3,4,\mlq 6\mrq,8$ for the elements in the basis. In the dual point of view of Lie algebra, like in Deligne's article \cite{De} or Wojtkowiak \cite{Wo}, this corresponds to showing that the Ihara bracket $\lbrace,\rbrace$ on these elements modulo some vector space reduces to the usual bracket $[,]$. More generally, for other bases, like Hoffman's one for $N=1$, the idea is still to find an appropriate filtration on the conjectural basis, such that the coaction in the graded space acts on this family, modulo some space, as the deconcatenation, as for the $f_{i}$ alphabet. Indeed, on $H$ ($\ref{eq:phih}$), the weight graded part of the coaction, $D_{r}$ is defined by:
\begin{equation}\label{eq:derf}
D_{r} : \quad H_{n} \quad \longrightarrow \quad L_{r} \otimes H_{n-r} \quad \quad\text{ such that :}
\end{equation}
$$ f^{j_{1}}_{i_{1}} \cdots f^{j_{k}}_{i_{k}}\longmapsto\left\{
\begin{array}{ll}
  f^{j_{1}}_{i_{1}} \otimes f^{j_{2}}_{i_{2}} \ldots f^{j_{k}}_{i_{k}} & \text{ if } i_{1}=r .\\
  0 & \text{ else }.\\
\end{array}
\right.$$ 

\item[$\cdot$] One fundamental feature for a family of motivic multiple zeta values (which makes it \say{natural} and simple) is the \textit{stability} under the coaction. For instance, if we look at the following family which appears in Chapter $5$:
  $$\zeta^{\mathfrak{m}}\left(n_{1}, \cdots, n_{p-1}, n_{p} \atop   \epsilon_{1}, \ldots , \epsilon_{p-1},\epsilon_{p}\right) \quad \text{ with }  \epsilon_{p}\in\mu_{N} \quad \text{primitive} \quad \text{ and } (\epsilon_{i})_{i<p} \quad \text{non primitive}.$$
  If N is a power of a prime, this family is stable via the coaction. \footnote{Since in this case, $(\text{non primitive}) \cdot  (\text{ non primitive})=$ non primitive and non primitive $\cdot$ primitive $=$ primitive root. Note also, for dimensions reasons, if we are looking for a basis in this form, we should have $N-\varphi(N)\geq \frac{\varphi(N)}{2}$, which comes down here to the case where $N$ is a power of $2$ or $3$.} It is also stable via the Galois action if we only need to take $1$ as a non primitive ($1$-dimensional case), as for $\mathcal{MT} (\mathcal{O}_{6})$.\\ 

\end{itemize}

\begin{proof}
Straightforward from $\ref{drz}$, using the properties of motivic iterated integrals previously listed ($\S \ref{propii}$). Terms of type \textsc{(a)} correspond to cuts from a $0$ (possibly the very first one) to a root of unity, $\textsc{(b)}$ terms from a root of unity to a $0$, $\textsc{(c)}$ terms between two roots of unity and $\textsc{(d,d')}$ terms are the cuts ending in the last $1$, called \textit{deconcatenation terms}.
\end{proof}

\paragraph{Derivation space.}
By Lemma $2.4.1$ (depth $1$ results), once we have chosen a basis for $gr_{1}^{\mathfrak{D}} \mathcal{L}_{r}$, composed by some $\zeta^{\mathfrak{a}}(r_{i};\eta_{i})$,  we can well define: \footnote{Without passing to the depth-graded, we could also define $D^{\eta}_{r}$ as $D_{r}: \mathcal{H}\rightarrow gr^{\mathfrak{D}}_{1}\mathcal{L}_{r} \otimes \mathcal{H}$ followed by $\pi^{\eta}_{r}\otimes id$ where $\pi^{\eta}:gr^{\mathfrak{D}}_{1}\mathcal{L}_{r} \rightarrow \mathbb{Q}$ is the projection on $\zeta^{\mathfrak{m}}\left(  r \atop \eta \right) $, once we have fixed a basis for $gr^{\mathfrak{D}}_{1}\mathcal{L}_{r}$; and define as above $\mathscr{D}_{r}$ as the set of the $D^{\eta}_{r,p}$, for $\zeta^{\mathfrak{m}}(r,\eta)$ in the basis of $gr_{1}^{\mathfrak{D}} \mathcal{A}_{r}$.}
\begin{itemize}
\item[$(i)$] For each $(r_{i}, \eta_{i})$:\nomenclature{$D^{\eta}_{r,p}$}{defined from $D_{r,p}$ followed by a projection}
\begin{equation}
\label{eq:derivnp}
\boldsymbol{D^{\eta_{i}}_{r_{i},p}}: gr_{p}^{\mathfrak{D}}\mathcal{H} \rightarrow  gr_{p-1}^{\mathfrak{D}} \mathcal{H},
\end{equation}
 as the composition of $D_{r_{i},p}$ followed by the projection:
$$\pi^{\eta}: gr_{1}^{\mathfrak{D}} \mathcal{L}_{r}\otimes gr_{p-1}^{\mathfrak{D}} \mathcal{H}\rightarrow  gr_{p-1}^{\mathfrak{D}} \mathcal{H},  \quad \quad\zeta^{\mathfrak{m}}(r;  \epsilon) \otimes X \mapsto c_{\eta, \epsilon, r} X , $$ 
with $c_{\eta, \epsilon, r}\in \mathbb{Q}$ the coefficient of $\zeta^{\mathfrak{m}}(r;  \eta)$ in the decomposition of $\zeta^{\mathfrak{m}}(r;  \epsilon)$ in the basis.
\item[$(ii)$] \begin{equation}\label{eq:setdrp}
 \boldsymbol{\mathscr{D}_{r,p}} \text{ as the set of }  D^{\eta_{i}}_{r_{i},p}   \text{ for } \zeta^{\mathfrak{m}}(r_{i},\eta_{i}) \text{ in the chosen basis of } gr_{1}^{\mathfrak{D}} \mathcal{A}_{r}.
 \end{equation}
\item[$(iii)$] The \textit{derivation set} $\boldsymbol{\mathscr{D}}$ as the (disjoint) union:  $\boldsymbol{\mathscr{D}}  \mathrel{\mathop:}= \sqcup_{r>0} \left\lbrace \mathscr{D}_{r} \right\rbrace $.
\end{itemize}\nomenclature{$\mathscr{D}$}{the derivation set}

\textsc{Remarks}: 
\begin{itemize}
\item[$\cdot$]  In the case $N=2,3,4,\mlq 6\mrq$, the cardinal of $\mathscr{D}_{r,p}$ is one (or $0$ if $r$ even and $N=2$, or if $(r,N)=(1,6)$), whereas for $N=8$ the space generated by these derivations is $2$-dimensional, generated by  $D^{\xi_{8}}_{r}$ and $D^{-\xi_{8}}_{r}$ for instance.
\item[$\cdot$] Following the same procedure for the non-canonical Hopf comodule $H$ defined in $(\ref{eq:phih})$, isomorphic to $\mathcal{H}^{\mathcal{MT}_{N}}$, since the coproduct on $H$ is the deconcatenation $(\ref{eq:derf})$, leads to the following derivations operators:
$$\begin{array}{llll}
D^{j}_{r} : &  H_{n} &  \rightarrow & H_{n-r} \\
& f^{j_{1}}_{i_{1}} \cdots f^{j_{k}}_{i_{k}} & \mapsto & \left\{
\begin{array}{ll}
  f^{j_{2}}_{i_{2}} \ldots f^{j_{k}}_{i_{k}} & \text{ if } j_{1}=j \text{ and } i_{1}=r .\\
  0 & \text{ else }.\\
\end{array}
\right.
\end{array}.$$
\end{itemize}
Now, consider the following application, depth graded version of the derivations above, fundamental for several linear independence results in $\S 4.3$ and Chapter $5$:\nomenclature{$\partial _{n,p}$}{a depth graded version of the infinitesimal coactions}
\begin{equation}
\label{eq:pderivnp}
\boldsymbol{\partial _{n,p}} \mathrel{\mathop:}=\oplus_{r<n\atop D\in \mathscr{D}_{r,p}} D : gr_{p}^{\mathfrak{D}}\mathcal{H}_{n}  \rightarrow \oplus_{r<n } \left( gr_{p-1}^{\mathfrak{D}}\mathcal{H}_{n-r}\right) ^{\oplus \text{ card } \mathscr{D}_{r,p}}
\end{equation}

\paragraph{Kernel of $\boldsymbol{D_{<n}}$. }

A key point for the use of these derivations is the ability to prove some relations (and possibly lift some from MZV to motivic MZV) up to rational coefficients. This comes from the following theorem, looking at primitive elements:\nomenclature{$D_{<n}$ }{is defined as $\oplus_{r<n} D_{r}$}
\begin{theo}
Let $D_{<n}\mathrel{\mathop:}= \oplus_{r<n} D_{r}$, and fix a basis $\lbrace \zeta^{\mathfrak{a}}\left( n \atop \eta_{j} \right) \rbrace$ of $gr_{1}^{\mathfrak{D}} \mathcal{A}_{n}$. Then:
$$\ker D_{<n}\cap \mathcal{H}^{N}_{n} =  
\left\lbrace  \begin{array}{ll} 
\mathbb{Q}\zeta^{\mathfrak{m}}\left( n \atop 1 \right)  & \text{ for } N=1,2 \text{ and } n\neq 1.\\
\oplus  \mathbb{Q} \pi^{\mathfrak{m}} \bigoplus_{1 \leq j \leq a_{N}} \mathbb{Q}  \zeta^{\mathfrak{m}}\left( 1 \atop \eta_{j} \right). & \text{ for } N>2, n=1.\\
\oplus  \mathbb{Q} (\pi^{\mathfrak{m}})^{n} \bigoplus_{1 \leq j \leq b_{N}} \mathbb{Q}  \zeta^{\mathfrak{m}}\left( n \atop \eta_{j} \right). & \text{ for } N>2, n>1.
\end{array}\right. .$$
\end{theo}
\begin{proof} It comes from the injective morphism of graded Hopf comodules $(\ref{eq:phih})$, which is an isomorphism for $N=1,2,3,4,\mlq 6\mrq,8$:
$$\phi: \mathcal{H}^{N} \xrightarrow[\sim]{n.c} H^{N} \mathrel{\mathop:}= \mathbb{Q}\left\langle \left( f^{j}_{i}\right)  \right\rangle  \otimes_{\mathbb{Q}} \mathbb{Q} \left[  g_{1}\right] .$$
Indeed, for $H^{N}$, the analogue statement is obviously true, for $\Delta'=1\otimes \Delta+ \Delta\otimes 1$:
$$\ker \Delta' \cap H_{n} = \oplus_{j} f^{j}_{n} \oplus g_{1}^{n} .$$
\end{proof}
\begin{coro}\label{kerdn}
Let $D_{<n}\mathrel{\mathop:}= \oplus_{r<n} D_{r}$.\footnote{For $N=1$, we restrict to $r$ odd $>1$; for $N=2$ we restrict to r odd; for $N=\mlq 6\mrq$ we restrict to $r>1$.} Then:
$$\ker D_{<n}\cap \mathcal{H}^{N}_{n} = \left\lbrace  \begin{array}{ll}
\mathbb{Q}\zeta^{\mathfrak{m}}\left( n \atop 1 \right)   & \text{ for } N=1,2.\\
\mathbb{Q} (\pi^{\mathfrak{m}})^{n} \oplus \mathbb{Q}  \zeta^{\mathfrak{m}}\left( n \atop \xi_{N} \right)  & \text{ for } N=3,4,\mlq 6\mrq.\\
\mathbb{Q} (\pi^{\mathfrak{m}})^{n} \oplus \mathbb{Q}  \zeta^{\mathfrak{m}}\left( n \atop \xi_{8} \right) \oplus \mathbb{Q}  \zeta^{\mathfrak{m}}\left( n \atop -\xi_{8} \right) & \text{ for } N=8.\\
\end{array}\right. .$$
\end{coro}
In particular, by this result (for $N=1,2$), proving an identity between motivic MZV (resp. motivic Euler sums), amounts to:
\begin{enumerate}
\item Prove that the coaction is identical on both sides, computing $D_{r}$ for $r>0$ smaller than the weight. If the families are not stable under the coaction, this step would require other identities.
\item Use the corresponding analytic result for MZV (resp. Euler sums) to deduce the remaining rational coefficient; if the analytic equivalent is unknown, we can at least evaluate numerically this rational coefficient.
\end{enumerate}
Some examples are given in $\S 6.3 $ and $\S 4.4.3$.\\
Another important use of this corollary, is the decomposition of (motivic) multiple zeta values into a conjectured basis, which has been explained by F. Brown in \cite{Br1}.\footnote{He gave an exact numerical algorithm for this decomposition, where, at each step, a rational coefficient has to be evaluated; hence, for other roots of unity, the generalization, albeit easily stated, is harder for numerical experiments.}\\ 
However, for greater $N$, several rational coefficients appear at each step, and we would need linear independence results before concluding.\\

\chapter{Results} 
\section{Euler $\star,\sharp$ sums \textit{[Chapter 4]}}

In Chapter $4$, we focus on motivic Euler sums ($N=2$), shortened ES, and motivic multiple zeta values ($N=1$), with in particular some new bases for the vector space of MMZV: one with Euler $\sharp$ sums and, under an analytic conjecture, the so-called \textit{Hoffman $\star$ family}. These two variants of Euler sums are (cf. Definition $4.1.1$):
\begin{description}
\item[Euler $\star$ sums] corresponds to the analogue multiple sums of ES with $\leq$ instead of strict inequalities. It verifies:
\begin{equation} \label{eq:esstar}\zeta ^{\star}(n_{1}, \ldots, n_{p})= \sum_{\circ=\mlq + \mrq \text{ or } ,} \zeta (n_{1}\circ \cdots \circ n_{p}).
\end{equation} 
\texttt{Notation:} This $\mlq + \mrq$ operation on $n_{i}\in\mathbb{Z}$, is a summation of absolute values, while signs are multiplied.\\
These have already been studied in many papers: $\cite{BBB}, \cite{IKOO}, \cite{KST}, \cite{LZ}, \cite{OZ}, \cite{Zh3}$.
\item[Euler $\sharp$ sums] are, similarly, linear combinations of MZV but with $2$-power coefficients:
\begin{equation} \label{eq:essharp} \zeta^{\sharp}(n_{1}, \ldots, n_{p})= \sum_{\circ=\mlq + \mrq \text{ or } ,} 2^{p-n_{+}} \zeta(n_{1}\circ \cdots \circ n_{p}), \quad \text{   with } n_{+} \text{ the number of  } +.
\end{equation}
\end{description} 
We also pave the way for a motivic version of a generalization of a Linebarger and Zhao's equality (Conjecture $\ref{lzg}$) which expresses each motivic multiple zeta $\star$ as a motivic Euler $\sharp$ sums; under this conjecture, Hoffman $\star$ family is a basis, identical to the one presented with Euler sums $\sharp$.\\
\\
The first (naive) idea, when looking for a basis for the space of multiple zeta values, is to choose:
$$\lbrace \zeta\left( 2n_{1}+1,2n_{2}+1, \ldots, 2n_{p}+1 \right) (2 i \pi)^{2s}, n_{i}\in\mathbb{N}^{\ast}, s\in \mathbb{N} \rbrace .$$
However, considering Broadhurst-Kreimer conjecture $(\ref{eq:bkdepth})$, the depth filtration clearly does \textit{not} behave so nicely in the case of MZV \footnote{Remark, as we will see in Chapter $5$, or as we can see in $\cite{De}$ that for $N=2,3,4,\mlq 6\mrq,8$, the depth filtration is dual of the descending central series of $\mathcal{U}$, and, in that sense, does \textit{behave well}. For instance, the following family is indeed a basis of motivic Euler sums:
$$\lbrace \zeta^{\mathfrak{m}}\left( 2n_{1}+1,2n_{2}+1, \ldots, 2n_{p-1}+1,-(2n_{p}+1) \right) (\mathbb{L}^{\mathfrak{m}})^{2s}, n_{i}\in\mathbb{N}, s\in \mathbb{N} \rbrace .$$ } and already in weight $12$, they are not linearly independent: 
$$28\zeta(9,3)+150\zeta(7,5)+168\zeta(5,7) = \frac{5197}{691}\zeta(12).$$
Consequently, in order to find a basis of motivic MZV, we have to:
\begin{itemize}
\item[\texttt{Either}: ]  Allow \textit{higher} depths, as the Hoffman basis (proved by F Brown in $\cite{Br2}$), or the $\star$ analogue version:
$$\texttt{ Hoffman } \star \quad : \left\lbrace  \zeta^{\star, \mathfrak{m}} \left( \boldsymbol{2}^{a_{0}},3, \boldsymbol{2}^{a_{1}}, \ldots, 3, \boldsymbol{2}^{a_{p}} \right), a_{i}\geq 0 \right\rbrace .$$
The analogous real family Hoffman $\star$ was also conjectured (in $\cite{IKOO}$, Conjecture $1$) to be a basis of the space of MZV.  Up to an analytic conjecture ($\ref{conjcoeff}$), we prove (in $\S 4.4$) that the motivic Hoffman $\star$ family is a basis of $\mathcal{H}^{1}$, the space of motivic MZV\footnote{Up to this analytic statement, $\ref{conjcoeff}$, the Hoffman $\star$ family is then a generating family for MZV.}. In this case, the notion of \textit{motivic depth} (explained in $\S 2.4.3$) is the number of $3$, and is here in general much smaller than the depth.
\item[\texttt{Or}: ] Pass by motivic Euler sums, as the Euler $\sharp$ basis given below; it is also another illustration of the descent idea of Chapter $5$: roughly, it enables to reach motivic periods in $\mathcal{H}^{N'}$  coming from above, i.e. via motivic periods in $\mathcal{H}^{N}$, for $N' \mid N$.
\end{itemize}
More precisely, let look at the following motivic Euler $\sharp$ sums:
\begin{theom}
The motivic Euler sums $\zeta^{\sharp,  \mathfrak{m}} \left( \lbrace \overline{\text{even }},  \text{odd } \rbrace^{\times} \right) $ are motivic geometric periods of $\mathcal{MT}(\mathbb{Z})$. Hence, they are $\mathbb{Q}$ linear combinations of motivic multiple zeta values.\footnote{Since, by $\cite{Br2}$, we know that Frobenius invariant geometric motivic periods of $\mathcal{MT}(\mathbb{Z})$ are $\mathbb{Q}$ linear combinations of motivic multiple zeta values.}
\end{theom}
\texttt{Notations}: Recall that an overline $\overline{x}$ corresponds to a negative sign, i.e. $-x$ in the argument. Here, the family considered is a family of Euler $\sharp$ sums with only positive odd and negative even integers for arguments.\\
This motivic family is even a generating family of motivic MZV from which we are able to extract a basis:
\begin{theom}
A basis of $\boldsymbol{\mathcal{P}_{\mathcal{MT}(\mathbb{Z}), \mathbb{R}}^{\mathfrak{m},+}}=\mathcal{H}^{1}$, the space of motivic multiple zeta values is:
$$\lbrace\zeta^{\sharp,\mathfrak{m}} \left( 2a_{0}+1,2a_{1}+3,\cdots, 2 a_{p-1}+3, \overline{2a_{p}+2}\right) \text{ , } a_{i}\geq 0 \rbrace .$$
\end{theom}
The proof is based on the good behaviour of this family with respect to the coaction and the depth filtration; the suitable filtration corresponding to the \textit{motivic depth} for this family is the usual depth minus $1$.\\
By application of the period map, combining these results:
\begin{corol}
Each Euler sum $\zeta^{\sharp} \left( \lbrace \overline{\text{even }},  \text{odd } \rbrace^{\times} \right) $ (i.e. with positive odd and negative even integers for arguments) is a $\mathbb{Q}$ linear combination of multiple zeta values of the same weight.\\
Conversely, each multiple zeta value of depth $<d$ is a 
$\mathbb{Q}$ linear combination of elements $\zeta^{\sharp} \left( 2a_{0}+1,2a_{1}+3,\cdots, 2 a_{p-1}+3, \overline{2a_{p}+2}\right) $, of the same weight with $a_{i}\geq 0$, $p\leq d$.
\end{corol} 

\textsc{Remarks}:
\begin{itemize}
\item[$\cdot$] Finding a \textit{good} basis for the space of motivic multiple zeta values is a fundamental question. Hoffman basis may be unsatisfactory for various reasons, while this basis with Euler sums (linear combinations with $2$ power coefficients) may appear slightly more natural, in particular since the motivic depth is here the depth minus $1$. However, both of those two baess are not bases of the $\mathbb{Z}$ module and the primes appearing in the determinant of the \textit{passage matrix}\footnote{The inverse of the matrix expressing the considered basis in term of a $\mathbb{Z}$ basis.} are growing rather fast.\footnote{Don Zagier has checked this for small weights with high precision; he suggested that the primes involved in the case of this basis could have some predictable features, such as being divisor of $2^{n}-1$.}
\item[$\cdot$] Looking at how periods of $\mathcal{MT}(\mathbb{Z})$ embed into periods of $\mathcal{MT}(\mathbb{Z}[\frac{1}{2}])$, is a fragment of the Galois descent ideas of Chapter $5$.\\ Euler sums which belong to the $\mathbb{Q}$-vector space of multiple zeta values, sometimes called \textit{honorary}, have been studied notably by D. Broadhurst (cf. $\cite{BBB1}$) among others. We define then \textit{unramified} motivic Euler sums as motivic ES which are $\mathbb{Q}$-linear combinations of motivic MZVs, i.e. in $\mathcal{H}^{1}$. Being unramified for a motivic period implies that its period is unramified, i.e. honorary; some examples of unramified motivic ES are given in $\S 6.2$, or with the family above. In Chapter 5, we give a criterion for motivic Euler sums to be unramified \ref{criterehonoraire}, which generalizes for some other roots of unity; by the period map, this criterion also applies to Euler sums. 
\item[$\cdot$] For these two theorems, in order to simplify the coaction, we crucially need a motivic identity in the coalgebra $\mathcal{L}$, proved in $\S 4.2$, coming from the octagon relation pictured in Figure $\ref{fig:octagon2}$. More precisely, we need to consider the linearized version of the anti-invariant part by the Frobenius at infinity of this relation, in order to prove this hybrid relation (Theorem $\ref{hybrid}$), for $n_{i}\in\mathbb{N}^{\ast}$,  $\epsilon_{i}\in\pm 1$:
$$\zeta^{\mathfrak{l}}_{k}\left(n_{0},\cdots, n_{p} \atop \epsilon_{0} , \ldots, \epsilon_{p} \right) + \zeta^{\mathfrak{l}}_{n_{0}+k}\left( n_{1}, \ldots, n_{p}  \atop \epsilon_{1} , \ldots, \epsilon_{p} \right) \equiv (-1)^{w+1}\left(  \zeta^{\mathfrak{l}}_{k}\left( n_{p}, \ldots, n_{0} \atop \epsilon_{p} , \ldots, \epsilon_{0}\right) + \zeta^{\mathfrak{l}}_{k+n_{p}}\left( n_{p-1}, \ldots,n_{0} \atop \epsilon_{p-1}, \ldots, \epsilon_{0}\right)  \right).$$
Thanks to this hybrid relation, and the antipodal relations presented in $\S 4.2.1$, the coaction expression is considerably simplified in Appendix $A.1$.
\end{itemize}

\begin{theom}
If the analytic conjecture  ($\ref{conjcoeff}$) holds, then the motivic \textit{Hoffman} $\star$ family $\lbrace \zeta^{\star,\mathfrak{m}} (\lbrace 2,3 \rbrace^{\times})\rbrace$ is a basis of $\mathcal{H}^{1}$, the space of MMZV.
\end{theom} 
\texttt{Nota Bene:} A MMZV $\star$, in the depth graded, is obviously equal to the corresponding MMZV. However, the motivic Hoffman (i.e. with only $2$ and $3$) multiple zeta $(\star)$ values are almost all zero in the depth graded (the \textit{motivic depth} there being the number of $3$). Hence, the analogous result for the non $\star$ case\footnote{I.e. that the motivic Hoffman family is a basis of the space of MMZV, cf $\cite{Br1}$.}, proved by F. Brown, does not make the result in the $\star$ case anyhow simpler.\\
\\
Denote by $\mathcal{H}^{2,3}$\nomenclature{$\mathcal{H}^{2,3}$}{the $\mathbb{Q}$-vector space spanned by the motivic Hoffman $\star$ family} the $\mathbb{Q}$-vector space spanned by the motivic Hoffman $\star$ family. The idea of the proof is similar as in the non-star case done by Francis Brown. We define an increasing filtration $\mathcal{F}^{L}_{\bullet}$ on $\mathcal{H}^{2,3}$, called the \textit{level}, such that:\footnote{Beware, this notion of level is different than the level associated to a descent in Chapter $5$. It is similar as the level notion for the Hoffman basis, in F. Brown paper's $\cite{Br2}$. It corresponds to the motivic depth, as we will see through the proof.}\nomenclature{$\mathcal{F}^{L}_{l}$}{level filtration on $\mathcal{H}^{2,3}$}
\begin{center}
$\mathcal{F}^{L}_{l}\mathcal{H}^{2,3}$ is spanned by $\zeta^{\star,\mathfrak{m}} (\boldsymbol{2}^{a_{0}},3,\cdots,3, \boldsymbol{2}^{a_{l}}) $, with less than \say{l} $3$.
\end{center}
One key feature is that the vector space $\mathcal{F}^{L}_{l}\mathcal{H}^{2,3}$ is stable under the action of $\mathcal{G}$.\\
The linear independence is then proved thanks to a recursion on the level and on the weight, using the injectivity of a map $\partial$ where $\partial$ came out of the level and weight-graded part of the coaction $\Delta$ (cf. $\S 4.4.1$). The injectivity is proved via $2$-adic properties of some coefficients with Conjecture $\ref{conjcoeff}$.\\
One noteworthy difference is that, when computing the coaction on the motivic MZV$^{\star}$, some motivic MZV$^{\star\star}$ arise, which are a non convergent analogue of MZV$^{\star}$ and have to be renormalized. Therefore, where F. Brown in the non-star case needed an analytic formula proven by Don Zagier ($\cite{Za}$), we need some slightly more complicated identities (in Lemma $\ref{lemmcoeff}$) because the elements involved, such as $\zeta^{\star \star,\mathfrak{m}} (\boldsymbol{2}^{a},3, \boldsymbol{2}^{b}) $ for instance, are not of depth $1$ but are linear combinations of products of depth $1$ motivic MZV times a power of $\pi$.\\
\\
\\
These two bases for motivic multiple zeta values turn to be identical, when considering this conjectural motivic identity, more generally:
\begin{conje}
For $a_{i},c_{i} \in \mathbb{N}^{\ast}$, $c_{i}\neq 2$,
\begin{equation}  \zeta^{\star, \mathfrak{m}} \left( \boldsymbol{2}^{a_{0}},c_{1},\cdots,c_{p}, \boldsymbol{2}^{a_{p}}\right)  =
\end{equation}
$$(-1)^{1+\delta_{c_{1}}} \zeta^{\sharp,  \mathfrak{m}} \left(  \pm (2a_{0}+1-\delta_{c_{1}}),\boldsymbol{1}^{ c_{1}-3},\cdots,\boldsymbol{1}^{ c_{i}-3  },\pm(2a_{i}+3-\delta_{c_{i}}-\delta_{c_{i+1}}), \ldots, \pm ( 2 a_{p}+2-\delta_{c_{p}}) \right) . $$
where the sign $\pm$ is always $-$ for an even argument, $+$ for an odd one, $\delta_{c}=\delta_{c=1}$, Kronecker symbol, and $\boldsymbol{1}^{n}:=\boldsymbol{1}^{min(0,n)}$ is a sequence of $n$ 1 if $n\in\mathbb{N}$, an empty sequence else.
\end{conje}
This conjecture expresses each motivic MZV$^{\star}$ as a linear combination of motivic Euler sums, which gives another illustration of the  Galois descent between the Hopf algebra of motivic MZV and the Hopf algebra of motivic Euler sums.\\
\\
\texttt{Nota Bene}: Such a \textit{motivic relation} between MMZV$_{\mu_{N}}$ is stronger than its analogue between MZV$_{\mu_{N}}$ since it contains more information; it implies many other relations because of its Galois conjugates. This explain why its is not always simple to lift an identity from MZV to MMZV from the Theorem $\ref{kerdn}$. If the family concerned is not stable via the coaction, such as $(iv)$ in Lemma $\ref{lemmcoeff}$, we may need other analytic equalities before concluding.\\
\\
This conjecture implies in particular the following motivic identities, whose analogue for real Euler sums are proved as indicated in the brackets\footnote{Beware, only the identity for real Euler sums is proved; the motivic analogue stays a conjecture.}:
\begin{description}
\item[Two-One] [For $c_{i}=1$, Ohno Zudilin: $\cite{OZ}$]:
\begin{equation}
\zeta^{\star, \mathfrak{m}} (\boldsymbol{2}^{a_{0}},1,\cdots,1, \boldsymbol{2}^{a_{p}})= - \zeta^{\sharp,  \mathfrak{m}} \left( \overline{2a_{0}}, 2a_{1}+1, \ldots, 2a_{p-1}+1, 2 a_{p}+1\right) . 
\end{equation}
\item[Three-One] [For $c_{i}$ alternatively $1$ and $3$, Zagier conjecture, proved in $\cite{BBB}$]
\begin{equation} \zeta^{\star, \mathfrak{m}} (\boldsymbol{2}^{a_{0}},1,\boldsymbol{2}^{a_{1}},3 \cdots,1, \boldsymbol{2}^{a_{p-1}}, 3, \boldsymbol{2}^{a_{p}}) = -\zeta^{\sharp,  \mathfrak{m}} \left( \overline{2a_{0}}, \overline{2a_{1}+2}, \ldots, \overline{2a_{p-1}+2}, \overline{2 a_{p}+2} \right) .
\end{equation}
\item[Linebarger-Zhao $\star$] [With $c_{i}\geq 3$, Linebarger Zhao in $\cite{LZ}$]:
\begin{equation}
\zeta^{\star, \mathfrak{m}} \left( \boldsymbol{2}^{a_{0}},c_{1},\cdots,c_{p}, \boldsymbol{2}^{a_{p}}\right)  =
-\zeta^{\sharp,  \mathfrak{m}} \left( 2a_{0}+1,\boldsymbol{1}^{ c_{1}-3  },\cdots,\boldsymbol{1}^{ c_{i}-3 },2a_{i}+3, \ldots, \overline{ 2 a_{p}+2} \right) 
\end{equation}
In particular, restricting to all $c_{i}=3$:
\begin{equation}\label{eq:LZhoffman}
\zeta^{\star, \mathfrak{m}} \left( \boldsymbol{2}^{a_{0}},3,\cdots,3,\boldsymbol{ 2}^{a_{p}}\right)  = - \zeta^{\sharp,  \mathfrak{m}} \left( 2a_{0}+1, 2a_{1}+3, \ldots, 2a_{p-1}+3, \overline{2 a_{p}+2}\right) . 
\end{equation}
\end{description}
\texttt{Nota Bene}: Hence the previous conjecture $(\ref{eq:LZhoffman})$ implies that the motivic Hoffman $\star$ is a basis, since we proved the right side of $(\ref{eq:LZhoffman})$ is a basis:
$$ \text{ Conjecture } \ref{lzg}  \quad \Longrightarrow  \quad  \text{ Hoffman } \star \text{ is a basis of MMZV} . $$
\\
\texttt{Examples:} The previous conjecture would give such relations:\\
$$\begin{array}{ll}
\zeta^{\star, \mathfrak{m}} (2,2,3,3,2)  =-\zeta^{\sharp, \mathfrak{m}} (5,3,-4) &\zeta^{\star, \mathfrak{m}} (5,6,2)  =-\zeta^{\sharp, \mathfrak{m}} (1,1,1,3,1,1,1,-4) \\
\zeta^{\star, \mathfrak{m}} (1,6)  =\zeta^{\sharp, \mathfrak{m}} (-2,1,1,1,-2) &\zeta^{\star, \mathfrak{m}} (2,4, 1, 2,2,3)  =-\zeta^{\sharp, \mathfrak{m}} (3,1, -2, -6,-2) .
\end{array}$$

\section{Galois Descents \textit{[Chapter 5]}}

There, we study Galois descents for categories of mixed Tate motives $\mathcal{MT}_{\Gamma_{N}}$, and how periods of $\pi_{1}^{un}(X_{N'})$ are embedded into periods of $\pi_{1}^{un}(X_{N})$ for $N'\mid N$. Indeed, for each $N, N'$ with $N'|N$ there are the motivic Galois group $\mathcal{G}^{\mathcal{MT}_{N}}$ acting on $\mathcal{H}^{\mathcal{MT}_{N}} $ and a Galois descent between $\mathcal{H}^{\mathcal{MT}_{N'}}$ and $\mathcal{H}^{\mathcal{MT}_{N}}$, such that:
$$(\mathcal{H}^{\mathcal{MT}_{N}})^{\mathcal{G}^{N/N'}}=\mathcal{H}^{\mathcal{MT}_{N'}}.$$
Since for $N=2,3,4,\mlq 6\mrq,8$, the categories $\mathcal{MT}_{N}$ and $\mathcal{MT}'_{N}$ are equal, this Galois descent has a parallel for the motivic fundamental group side; we will mostly neglect the difference in this chapter:
\begin{figure}[H]
$$\xymatrixcolsep{5pc}\xymatrix{
\mathcal{H}^{N} \ar@{^{(}->}[r]
^{\sim}_{n.c} & \mathcal{H}^{\mathcal{MT}_{N}}  \\
\mathcal{H}^{N'}\ar[u]_{\mathcal{G}^{N/N'}} \ar@{^{(}->}[r]
_{n.c}^{\sim} &\mathcal{H}^{\mathcal{MT}_{N'}} \ar[u]^{\mathcal{G}^{\mathcal{MT}}_{N/N'}}    \\
\mathbb{Q}[i\pi^{\mathfrak{m}}] \ar[u]_{\mathcal{U}^{N'}} \ar@{^{(}->}[r]^{\sim} & \mathbb{Q}[i\pi^{\mathfrak{m}}] \ar[u]^{\mathcal{U}^{\mathcal{MT}_{N'}}} \\
\mathbb{Q}  \ar[u]_{\mathbb{G}_{m}} \ar@/^2pc/[uuu]^{\mathcal{G}^{N}} & \mathbb{Q}  \ar[u]^{\mathbb{G}_{m}} \ar@/_2pc/[uuu]_{\mathcal{G}^{\mathcal{MT}_{N}}}
}$$
\caption{Galois descents, $N=2,3,4,\mlq 6\mrq,8$ (level $0$).\protect\footnotemark }\label{fig:paralleldescent}
\end{figure}
\footnotetext{The (non-canonical) horizontal isomorphisms have to be chosen in a compatible way.} 
\texttt{Nota Bene:}  For $N'=1$  or $2$, $i\pi^{\mathfrak{m}}$ has to be replaced by $\zeta^{\mathfrak{m}}(2)$ or $(\pi^{\mathfrak{m}})^{2}$, since we consider, in $\mathcal{H}^{N'}$ only periods invariant by the Frobenius $\mathcal{F}_{\infty}$. In the descent between $\mathcal{H}^{N}$ and $\mathcal{H}^{N'}$, we require hence invariance by the Frobenius in order to keep only those periods; this condition get rid of odd powers of $i\pi^{\mathfrak{m}}$.\\

The first section of Chapter $5$ gives an overview for the Galois descents valid for any $N$: a criterion for the descent between MMZV$_{\mu_{N'}}$ and MMZV$_{\mu_{N}}$ (Theorem $5.1.1$), a criterion for being unramified (Theorem $5.1.2$), and their corollaries. The conditions are expressed in terms of the derivations $D_{r}$, since they reflect the Galois action. Indeed, looking at the descent between  $\mathcal{MT}_{N,M}$ and $\mathcal{MT}_{N',M'}$, sometimes denoted $(\mathcal{d})=(k_{N}/k_{N'}, M/M')$, it has possibly two components:\nomenclature{$\mathcal{d}$}{a specific Galois descent between $\mathcal{MT}_{N,M}$ and $\mathcal{MT}_{N',M'}$}
\begin{itemize}
\item[$\cdot$] The change of cyclotomic fields  $k_{N}/k_{N'}$; there, the criterion has to be formulated in the depth graded.
\item[$\cdot$] The change of ramification $M/M'$, which is measured by the $1$ graded part of the coaction i.e. $D_{1}$ with the notations of $\S 2.4$.\\
\end{itemize}

The second section specifies the descents for $N\in \left\{2, 3, 4, \mlq 6\mrq, 8\right\}$ \footnote{As above, the quotation marks underline that we consider the unramified category for $N=\mlq 6\mrq$.} represented in Figure $\ref{fig:d248}$, and $\ref{fig:d36}$. In particular, this gives a basis of motivic multiple zeta values relative to $\mu_{N'}$ via motivic multiple zeta values relative to $\mu_{N}$, for these descents considered, $N'\mid N$. It also gives a new proof of Deligne's results ($\cite{De}$): the category of mixed Tate motives over $\mathcal{O}_{k_{N}}[1/N]$, for $N\in \left\{2, 3, 4,\mlq 6\mrq, 8\right\}$ is spanned by the motivic fundamental groupoid of $\mathbb{P}^{1}\setminus\left\{0,\mu_{N},\infty \right\}$ with an explicit basis; as claimed in $\S 2.2$, we can even restrict to a smaller fundamental groupoid.\\
Let us present our results further and fix a descent $(\mathcal{d})=(k_{N}/k_{N'}, M/M')$ among these considered (in Figures $\ref{fig:d248}$, $\ref{fig:d36}$), between the category of mixed Tate motives of $\mathcal{O}_{N}[1/M]$ and $\mathcal{O}_{N'}[1/M']$.\footnote{Usually, the indication of the descent (in the exponent) is omitted when we look at a specific descent.} Each descent $(\mathcal{d})$ is associated to a subset $\boldsymbol{\mathscr{D}^{\mathcal{d}}} \subset \mathscr{D}$ of derivations, which represents the action of the Galois group $\mathcal{G}^{N/N'}$. It defines, recursively on $i$, an increasing motivic filtration $\mathcal{F}^{\mathcal{d}}_{i}$ on $\mathcal{H}^{N}$ called \textit{motivic level}, stable under the action of $\mathcal{G}^{\mathcal{MT}_{N}}$:\nomenclature{$\mathcal{F}^{\mathcal{d}}_{\bullet}$}{the increasing filtration by the motivic level associated to $\mathcal{d}$}
$$\texttt{Motivic level:}  \left\lbrace  \begin{array}{l}
 \mathcal{F}^{\mathcal{d}} _{-1} \mathcal{H}^{N} \mathrel{\mathop:}=0\\
\boldsymbol{\mathcal{F}^{\mathcal{d}}_{i}} \text{ the largest submodule of } \mathcal{H}^{N} \text{ such that } \mathcal{F}^{\mathcal{d}}_{i}\mathcal{H}^{N}/\mathcal{F}^{\mathcal{d}} _{i-1}\mathcal{H}^{N} \text{ is killed by } \mathscr{D}^{\mathcal{d}}.
\end{array} \right.  . $$
The $0^{\text{th}}$ level $\mathcal{F}^{\mathcal{d}}_{0}\mathcal{H}^{N}$, corresponds to invariants under the group $\mathcal{G}^{N/N'}$ while the $i^{\text{th}}$ level $\mathcal{F}^{\mathcal{d}}_{i}$, can be seen as the $i^{\text{th}}$ \textit{ramification space} in generalized Galois descents. Indeed, they correspond to a decreasing filtration of $i^{\text{th}}$ ramification Galois groups $\mathcal{G}_{i}$, which are the subgroups of $\mathcal{G}^{N/N'}$ which acts trivially on $\mathcal{F}^{i}\mathcal{H}^{N}$.\footnote{\textit{On \textbf{ramification groups} in usual Galois theory}: let $L/K$ a Galois extension of local fields. By Hensel's lemma, $\mathcal{O}_{L}=\mathcal{O}_{K}[\alpha]$ and the i$^{\text{th}}$ ramification group is defined as:
\begin{equation}\label{eq:ramifgroupi}
G_{i}\mathrel{\mathop:}=\left\lbrace g\in \text{Gal}(L/K) \mid v(g(\alpha)-\alpha) >i  \right\rbrace , \quad \text{ where } \left\lbrace \begin{array}{l}
 v \text{ is the valuation on } L \\
 \mathfrak{p}= \lbrace x\in L \mid v(x) >0 \rbrace \text{ maximal ideal for } L
\end{array}\right.  .
\end{equation} 
Equivalently, this condition means $g$ acts trivially on $\mathcal{O}_{L}\diagup \mathfrak{p}^{i+1}$, i.e. $g(x)\equiv x \pmod{\mathfrak{p}^{i+1}}$. This decreasing filtration of normal subgroups corresponds, by the Galois fundamental theorem, to an increasing filtration of Galois extensions:
$$G_{0}=\text{ Gal}(L/K) \supset G_{1} \supset G_{2} \supset \cdots \supset G_{i} \cdots$$
$$K=K_{0}  \subset K_{1} \subset K_{2} \subset \cdots \subset K_{i} \cdots$$
$G_{1}$, the inertia subgroup, corresponds to the subextension of minimal ramification.}
   \begin{figure}[H]
\centering
\begin{equation}\label{eq:descent} \xymatrix{
\mathcal{H}^{N}  \\
\mathcal{F}_{i}\mathcal{H}^{N} \ar[u]^{\mathcal{G}_{i}}\\
\mathcal{F}_{0}\mathcal{H}^{N} =\mathcal{H}^{N'} \ar@/^2pc/[uu]^{\mathcal{G}_{0}=\mathcal{G}^{N/N'}}  \ar[u] \\
\mathbb{Q} \ar[u]^{\mathcal{G}^{N'}} \ar@/_2pc/[uuu]_{\mathcal{G}^{N}} }  \quad \quad 
\begin{array}{l}
 (\mathcal{H}^{N})^{\mathcal{G}_{i}}=\mathcal{F}_{i}\mathcal{H}^{N} \\
\\
 \begin{array}{llll}
 \mathcal{G}^{N/N'}=\mathcal{G}_{0} & \supset \mathcal{G}_{1} & \supset \cdots  &\supset \mathcal{G}_{i} \cdots\\
 & & & \\
 \mathcal{H}^{N'}= \mathcal{F}_{0}\mathcal{H}^{N} &  \subset  \mathcal{F}_{1 }\mathcal{H}^{N} & \subset \cdots  & \subset \mathcal{F}_{i}\mathcal{H}^{N}  \cdots.
 \end{array}
\end{array}
\end{equation}
\caption{Representation of a Galois descent.}\label{fig:descent}
\end{figure}
\noindent
Those ramification spaces constitute a tower of intermediate spaces between the elements in MMZV$_{\mu_{N}}$ and the whole space of MMZV$_{\mu_{N'}}$.\\
\\
Let define the quotients associated to the motivic level: 
$$\boldsymbol{\mathcal{H}^{\geq i}} \mathrel{\mathop:}=  \mathcal{H}/ \mathcal{F}_{i-1}\mathcal{H}\text{  ,  } \quad\mathcal{H}^{\geq 0}=\mathcal{H}.$$
\newpage
\noindent
The descents considered are illustrated by the following diagrams:\\
   \begin{figure}[H]
\centering
$$\xymatrixcolsep{5pc}\xymatrix{
\mathcal{H}^{\mathcal{MT}(\mathcal{O}_{8}\left[ \frac{1}{2}\right] )} &   \\
\mathcal{H}^{\mathcal{MT}(\mathcal{O}_{4}\left[ \frac{1}{2}\right] )} \ar[u]^{\mathcal{F}^{k_{8}/k_{4},2/2}_{0}}   &  \text{\framebox[1.1\width]{$\mathcal{H}^{\mathcal{MT}(\mathcal{O}_{4})}$}} \ar[l]_{\mathcal{F}^{k_{4}/k_{4},2/1}_{0}} \\
\mathcal{H}^{\mathcal{MT}\left( \mathbb{Z}\left[ \frac{1}{2}\right] \right) } \ar[u]^{\mathcal{F}^{k_{4}/\mathbb{Q},2/2}_{0}} &  \mathcal{H}^{\mathcal{MT}(\mathbb{Z}),} \ar[l]^{\mathcal{F}^{\mathbb{Q}/\mathbb{Q},2/1}_{0}} \ar[lu]_{\mathcal{F}^{k_{4}/\mathbb{Q},2/1}_{0}} \ar@{.>}@/_1pc/[u] \ar@/_7pc/[uul] ^{\mathcal{F}^{k_{8}/\mathbb{Q},2/1}_{0}}
}
$$
\caption{\textsc{The cases $N=1,2,4,8$}. }\label{fig:d248}
\end{figure}
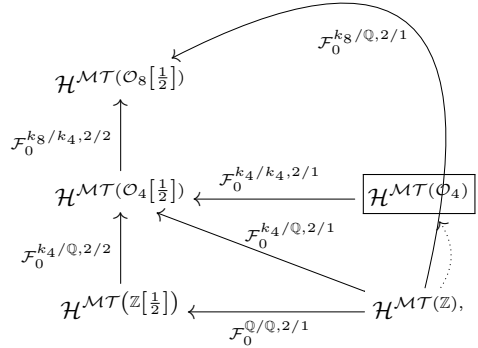

\begin{figure}[H]
$$\xymatrixcolsep{5pc}\xymatrix{
& \mathcal{H}^{\mathcal{MT}(\mathcal{O}_{6})}   \\
\mathcal{H}^{\mathcal{MT}\left( \mathcal{O}_{3}\left[ \frac{1}{3}\right] \right) } & \text{\framebox[1.1\width]{$\mathcal{H}^{\mathcal{MT}(\mathcal{O}_{3})}$}} \ar[l]_{\mathcal{F}^{k_{3}/k_{3},3/1}_{0}} \\
\text{\framebox[1.1\width]{$\mathcal{H}^{\mathcal{MT}\left( \mathbb{Z}\left[ \frac{1}{3}\right] \right) }$}}  \ar[u]^{\mathcal{F}^{k_{3}/\mathbb{Q},3/3}_{0}} &  \mathcal{H}^{\mathcal{MT}(\mathbb{Z})}  \ar[lu]_{\mathcal{F}^{k_{3}/\mathbb{Q},3/1}_{0}} \ar@{.>}@/_1pc/[u] \ar@/_3pc/[uu]_{\mathcal{F}^{k_{6}/\mathbb{Q},1/1}_{0}}
}
$$
\caption{\textsc{The cases $N=1,3,\mlq 6\mrq$}. }\label{fig:d36}
\end{figure}

\textsc{Remarks:}
\begin{enumerate}
\item[$\cdot$] The vertical arrows represent the change of field and the horizontal arrows the change of ramification. The full arrows are the descents made explicit in Chapter $5$.\\
More precisely, for each arrow $A \stackrel{\mathcal{F}_{0}}{\leftarrow}B$ in the above diagrams, we give a basis $\mathcal{B}^{A}_{n}$ of $\mathcal{H}_{n}^{A}$, and a basis of $\mathcal{H}_{n}^{B}= \mathcal{F}_{0} \mathcal{H}_{n}^{A}$ in terms of the elements of $\mathcal{B}_{n}^{A}$; similarly for the higher level of these filtrations.
 \item[$\cdot$] The framed spaces $\mathcal{H}^{\cdots}$ appearing in these diagrams are not known to be associated to a fundamental group and there is presently no other known way to reach these (motivic) periods. For instance, we obtain by descent, a basis for $\mathcal{H}_{n}^{\mathcal{MT}(\mathbb{Z}\left[ \frac{1}{3}\right] )}$ in terms of the basis of $\mathcal{H}_{n}^{\mathcal{MT}\left( \mathcal{O}_{3}\left[ \frac{1}{3}\right] \right) }$.\\
\end{enumerate}
\texttt{Example: Descent between Euler sums and MZV.} The comodule $\mathcal{H}^{1}$ embeds, non-canonically, into $\mathcal{H}^{2}$. Let first point out that:\footnote{Since all the motivic iterated integrals with only $0,1$ of length $1$ are zero by properties stated in $\S \ref{propii}$, hence the left side of $D_{1}$, defined in $(\ref{eq:Der})$, would always cancel.} $D_{1}(\mathcal{H}^{1})=0$; the Galois descent between $\mathcal{H}^{2}$ and $\mathcal{H}^{1}$ is precisely measured by $D_{1}$:
\begin{theom}
Let $\mathfrak{Z}\in\mathcal{H}^{2}$, a motivic Euler sum. Then:
$$\mathfrak{Z}\in\mathcal{H}^{1}, \text{ i.e. is a motivic MZV } \Longleftrightarrow  D_{1}(\mathfrak{Z})=0 \textrm{  and  } D_{2r+1}(\mathfrak{Z})\in\mathcal{H}^{1}.$$ 
\end{theom}
This is a useful recursive criterion to determine if a (motivic) Euler sum is in fact a (motivic) multiple zeta value. It can be generalized for other roots of unity, as we state more precisely in $\S 5.1$. These unramified motivic Euler sums are the $0^{\text{th}}$-level of the filtration by the motivic level here defined as:
\begin{center}
$\mathcal{F}_{i}\mathcal{H}^{2}$ is the largest sub-module such that $\mathcal{F}_{i}/ \mathcal{F}_{i-1}$ is killed by $D_{1}$.
\end{center}

\paragraph{Results.}
More precisely, for $N\in \left\{2, 3, 4, \mlq 6\mrq, 8\right\}$, we define a particular family $\mathcal{B}^{N}$ of motivic multiple zeta values relative to $\mu_{N}$ with different notions of \textbf{\textit{level}} on the basis elements, one for each Galois descent considered above:\nomenclature{$\mathcal{B}^{N}$}{basis of $\mathcal{H}^{N}$}
\begin{equation}\label{eq:base} \mathcal{B}^{N}\mathrel{\mathop:}=\left\{ \zeta^{\mathfrak{m}}\left(x_{1}, \cdots x_{p-1}, x_{p} \atop \epsilon_{1}, \ldots, \epsilon_{p-1}, \epsilon_{p}\xi_{N}\right) (2\pi i)^{s ,\mathfrak{m}} \text{ , }  x_{i}\in\mathbb{N}^{\ast} , s\in\mathbb{N}^{\ast} ,  \left\{
\begin{array}{ll}
 x_{i} \text{ odd, } s \text{ even}, \epsilon_{i}=1 &\text{if } N=2 \\
 \epsilon_{i}=1 &\text{if } N=3,4\\
   x_{i} >1 \text{,  } \epsilon_{i}=1 &\text{if } N=6\\
   \epsilon_{i}\in\lbrace\pm 1\rbrace &\text{if } N=8
   \end{array}
\right. \right\}
\end{equation}
  Denote by $\mathcal{B}_{n,p,i}$ the subset of elements with weight $n$, depth $p$ and level $i$. \\
  \\
  \texttt{Examples}: 
\begin{description}
  \item[$\cdot N=2$: ]   The basis for motivic Euler sums: $\mathcal{B}^{2}\mathrel{\mathop:}=\left\{ \zeta^{\mathfrak{m}}\left(2y_{1}+1, \ldots , 2 y_{p}+1 \atop 1, 1, \ldots, 1, -1\right) \zeta^{\mathfrak{m}} (2)^{s}, y_{i} \geq 0, s\geq 0 \right\}$ .
  The level for the descent from $\mathcal{H}^{2}$ to $\mathcal{H}^{1}$ is defined as the number of $y_{i}'s$ equal to $0$.
  \item[$\cdot N=4$: ]  The basis is: $\mathcal{B}^{4}\mathrel{\mathop:}= \left\{   \zeta^{\mathfrak{m}}\left(x_{1}, \ldots , x_{p} \atop 1,1, \ldots, 1, \sqrt{-1}\right)  (2\pi i)^{s ,\mathfrak{m}}, s\geq 0, x_{i} >0 \right\} $. 
$$\text{The level is:} \begin{array}{l}
\cdot \text{  the number of even $x_{i}'s$ for  the descent from $\mathcal{H}^{4}$ to $\mathcal{H}^{2}$ }\\
\cdot  \text{ the number of even $x_{i}'s + $ number of $x_{i}'s$ equal to 1  for  the descent from $\mathcal{H}^{4}$ to $\mathcal{H}^{1}$ }
\end{array}$$  
  \item[$\cdot N=8$: ] the level includes the number of $\epsilon_{i}'s$ equal to $-1$, etc.\\
  \end{description}
The quotients $\mathcal{H}^{\geq i}$, respectively filtrations $\mathcal{F}_{i}$ associated to the descent $\mathcal{d}$, will match with the sub-families (level restricted) $\mathcal{B}_{n,p, \geq i}$, respectively $\mathcal{B}_{n,p, \leq i}$. Indeed, we prove: \footnote{Cf. Theorem $5.2.4$ slightly more precise.}\nomenclature{$\mathbb{Z}_{1[P]}$}{subring of $\mathbb{Z}$}
\begin{theom}
With $ \mathbb{Z}_{1[P]} \mathrel{\mathop:}=\left\{ \frac{a}{1+b P}, a,b\in\mathbb{Z} \right\}$ where $ P \mathrel{\mathop:}= \left\lbrace \begin{array}{ll}
2 & \text{ for } N=2,4,8 \\
3 & \text{ for } N=3,\mlq 6\mrq
\end{array} \right. $.
\begin{enumerate}
\item[$\cdot$] $\mathcal{B}_{n,\leq p, \geq i}$ is a basis of $\mathcal{F}_{p}^{\mathfrak{D}} \mathcal{H}_{n}^{\geq i}$  and  $\mathcal{B}_{n,\cdot, \geq i} $ a basis of $\mathcal{H}^{\geq i}_{n}$.
\item[$\cdot$] $\mathcal{B}_{n, p, \geq i}$ is a basis of $gr_{p}^{\mathfrak{D}} \mathcal{H}_{n}^{\geq i}$ on which it defines a $\mathbb{Z}_{1[P]}$-structure:
\begin{center}
Each $\zeta^{\mathfrak{m}}\left( z_{1}, \ldots , z_{p} \atop \epsilon_{1}, \ldots, \epsilon_{p}\right)$ decomposes in $gr_{p}^{\mathfrak{D}} \mathcal{H}_{n}^{\geq i}$ as a $\mathbb{Z}_{1[P]}$-linear combination of $\mathcal{B}_{n, p, \geq i}$ elements.
\end{center}
	\item[$\cdot$] We have the two split exact sequences in bijection:
$$ 0\longrightarrow \mathcal{F}_{i}\mathcal{H}_{n} \longrightarrow \mathcal{H}_{n} \stackrel{\pi_{0,i+1}} {\rightarrow}\mathcal{H}_{n}^{\geq i+1} \longrightarrow 0$$
$$ 0 \rightarrow \langle \mathcal{B}_{n, \cdot, \leq i} \rangle_{\mathbb{Q}} \rightarrow \langle\mathcal{B}_{n} \rangle_{\mathbb{Q}} \rightarrow \langle \mathcal{B}_{n, \cdot, \geq i+1} \rangle_{\mathbb{Q}} \rightarrow 0 .$$
		\item[$\cdot$] A basis for the filtration spaces $\mathcal{F}_{i} \mathcal{H}_{n}$ is:
$$\cup_{p} \left\{ \mathfrak{Z}+ cl_{n, \leq p, \geq i+1}(\mathfrak{Z}), \mathfrak{Z}\in \mathcal{B}_{n, p, \leq i} \right\},$$
	$$\text{ where } cl_{n,\leq p,\geq i}: \langle\mathcal{B}_{n, p, \leq i-1}\rangle_{\mathbb{Q}} \rightarrow \langle\mathcal{B}_{n, \leq p, \geq i}\rangle_{\mathbb{Q}} \text{ such that } \mathfrak{Z}+cl_{n,\leq p,\geq i}(\mathfrak{Z})\in \mathcal{F}_{i-1}\mathcal{H}_{n}.$$
\item[$\cdot$] A basis for the graded space $gr_{i} \mathcal{H}_{n}$:
$$\cup_{p} \left\{ \mathfrak{Z}+ cl_{n, \leq p, \geq i+1}(\mathfrak{Z}), \mathfrak{Z}\in \mathcal{B}_{n, p, i} \right\}.$$
\end{enumerate}
\end{theom}
\noindent
\texttt{Nota Bene}: The morphism $cl_{n, \leq p, \geq i+1}$ satisfying those conditions is unique.\\
\\
The linear independence is obtained first in the depth graded, and the proof relies on the bijectivity of the following map $\partial^{i, \mathcal{d}}_{n,p}$ by an argument involving $2$ or $3$ adic properties:\footnote{The first $ c ^{\mathcal{d}}_{r}$ components of $\partial^{i, \mathcal{d}}_{n,p}$ correspond to the derivations in $\mathscr{D}^{\mathcal{d}}$ associated to the descent, which hence decrease the motivic level.}
\begin{equation}\label{eq:derivintro}
\partial^{i, \mathcal{d}}_{n,p}: gr_{p}^{\mathfrak{D}} \mathcal{H}_{n}^{\geq i} \rightarrow  \oplus_{r<n} \left( gr_{p-1}^{\mathfrak{D}} \mathcal{H}_{n-r}^{\geq i-1}\right) ^{\oplus c ^{\mathcal{d}}_{r}} \oplus_{r<n} \left( gr_{p-1}^{\mathfrak{D}} \mathcal{H}_{n-r}^{\geq i}\right) ^{\oplus c^{\backslash\mathcal{d}}_{r}} \text{, } c ^{\mathcal{d}}_{r}, c^{\backslash\mathcal{d}}_{r}\in\mathbb{N},
\end{equation}
which is obtained from the depth and weight graded part of the coaction, followed by a projection for the left side (by depth $1$ results), and by passing to the level quotients ($(\ref{eq:pderivinp})$). Once the freeness obtained, the generating property is obtained from counting dimensions, since K-theory gives an upper bound for the dimensions.\\
\\
This main theorem generalizes in particular a result of P. Deligne ($\cite{De}$), which we could formulate by different ways: \footnote{The basis $\mathcal{B}$, in the cases where $N\in \left\{3, 4,8\right\}$ is identical to P. Deligne's in $\cite{De}$. For $N=2$ (resp. $N=\mlq 6\mrq$ unramified) it is a linear basis analogous to his algebraic basis which is formed by Lyndon words in the odd (resp. $\geq 2$) positive integers (with $\ldots 5 \leq 3 \leq 1$); a Lyndon word being strictly smaller in lexicographic order than all of the words formed by permutation of its letters. Deligne's method is roughly dual to this point of view, working in Lie algebras, showing the action is faithful and that the descending central series of $\mathcal{U}$ is dual to the depth filtration.}
\begin{corol} 
\begin{itemize}
\item[$\cdot$] The map $\mathcal{G}^{\mathcal{MT}} \rightarrow \mathcal{G}^{\mathcal{MT}'}$ is an isomorphism.
\item[$\cdot$] The motivic fundamental group $\pi_{1}^{\mathfrak{m}} \left( \mathbb{P}^{1}\diagdown \lbrace 0, \mu_{N}, \infty \rbrace, \overline{0 \xi_{N}} \right)$ generates the category of mixed Tate motives $\mathcal{MT}_{N}$. 
\item[$\cdot$] $\mathcal{B}_{n}$ is a basis of $ \mathcal{H}^{N}_{n}$, the space of motivic MZV relative to $\mu_{N}$.
\item[$\cdot$] The geometric (and Frobenius invariant if $N=2$) motivic periods of $\mathcal{MT}_{N}$ are $\mathbb{Q}$-linear combinations of motivic MZV relative to $\mu_{N}$ (unramified for $N=\mlq 6\mrq$).
\end{itemize}
\end{corol}
\textsc{Remarks: }
\begin{itemize}
\item[$\cdot$] For $N=\mlq 6\mrq$ the result remains true if we restrict to iterated integrals relative not to all $6^{\text{th}}$ roots of unity but only to these relative to primitive roots.
\item[$\cdot$] We could even restrict to: $\begin{array}{ll}
\pi_{1}^{\mathfrak{m}} \left( \mathbb{P}^{1}\diagdown \lbrace 0, 1, \infty \rbrace, \overline{0 \xi_{N}} \right)  & \text{ for $N=2,3,4,\mlq 6\mrq$} \\
\pi_{1}^{\mathfrak{m}} \left( \mathbb{P}^{1}\diagdown \lbrace 0, \pm 1, \infty \rbrace, \overline{0 \xi_{N}} \right)  & \text{ for } N=8
\end{array}$ .
\end{itemize}

The previous theorem also provides the Galois descent from $\mathcal{H}^{\mathcal{MT}_{N}}$ to $\mathcal{H}^{\mathcal{MT}_{N'}}$:
\begin{corol}
A basis for MMZV$_{\boldsymbol{\mu_{N'}}}$ is formed by MMZV$_{\boldsymbol{\mu_{N}}}$ $\in \mathcal{B}^{N}$ of level $0$ each corrected by a $\mathbb{Q}$-linear combination of MMZV $_{\boldsymbol{\mu_{N}}}$ of level greater than or equal to $1$:  
$$\text{ Basis of } \mathcal{H}^{N'}_{n} : \quad \left\{ \mathfrak{Z}+ cl_{n,\cdot, \geq 1}(\mathfrak{Z}), \mathfrak{Z}\in \mathcal{B}^{N}_{n, \cdot, 0} \right\}.$$
\end{corol}
\textsc{Remark}:
Descent can be calculated explicitly in small depth, less than or equal to $3$, as we explain in the Appendix $A.2$. In the general case, we could make the part of maximal depth of $cl(\mathfrak{Z})$ explicit (by inverting a matrix with binomial coefficients) but motivic methods do not enable us to describe the other coefficients for terms of lower depth.\\
\\
\texttt{Example, $N=2$:} A basis for motivic multiple zeta values is formed by:
$$ \left\lbrace  \zeta^{\mathfrak{m}}\left( 2x_{1}+1, \ldots, \overline{2x_{p}+1}\right) \zeta^{\mathfrak{m}}(2)^{s} + \sum_{\exists i, y_{i}=0 \atop q\leq p} \alpha_{\textbf{y}}^{\textbf{x}} \zeta^{\mathfrak{m}}(2y_{1}+1, \ldots, \overline{2y_{q}+1})\zeta^{\mathfrak{m}}(2)^{s} \text{  ,  } x_{i}>0, \alpha^{\textbf{x}}_{\textbf{y}}\in\mathbb{Q} \right\rbrace .$$
Starting from a motivic Euler sum with odd numbers greater than $1$, we add some \textit{correction terms}, in order to get an element in $\mathcal{H}^{1}$, the space of MMZV. At this level, correction terms are motivic Euler sums with odds, and at least one $1$ in the arguments; i.e. they are of level $\geq 1$ with the previous terminology. For instance, the following linear combination is a motivic MZV:
$$\zeta^{\mathfrak{m}}(3,3,\overline{3})+ \frac{774}{191} \zeta^{\mathfrak{m}}(1,5, \overline{3})  - \frac{804}{191} \zeta^{\mathfrak{m}}(1,3, \overline{5})   -6 \zeta^{\mathfrak{m}}(3,1,\overline{5}) +\frac{450}{191}\zeta^{\mathfrak{m}}(1,1, \overline{7}).$$

\section{Miscellaneous Results \textit{[Chapter 6]}}

Chapter $6$ is devoted on the Hopf algebra structure of motivic multiple zeta values relative to $\mu_{N}$, particularly for $N=1,2$, presenting various uses of the coaction, and divided into sections as follows:\\
\begin{enumerate}
\item An important use of the coaction, is the decomposition of (motivic) multiple zeta values into a conjectured basis, as explained in $\cite{Br1}$. It is noteworthy to point out that the coaction always enables us to determine the coefficients of the maximal depth terms. We consider in $\S 6.1$ two simple cases, in which the space $gr^{\mathfrak{D}}_{max}\mathcal{H}_{n}$ is $1$ dimensional:
\begin{itemize}
\item[$(i)$] For $N=1$, when the weight is a multiple of $3$ ($w=3d$), such that the depth $p>d$:\footnote{This was a question asked for by D. Broadhurst: an algorithm, or a formula for the coefficient of $\zeta(3)^{d}$ of such a MZV, when decomposed in Deligne basis.}
$$gr^{\mathfrak{D}}_{p}\mathcal{H}_{3d} =\mathbb{Q} \zeta^{\mathfrak{m}}(3)^{d}.$$
\item[$(ii)$] For $N=2,3,4$, when weight equals depth:
$$gr^{\mathfrak{D}}_{p}\mathcal{H}_{p} =\mathbb{Q} \zeta^{\mathfrak{m}}\left( 1 \atop \xi_{N}\right) ^{p}.$$
The corresponding Lie algebra, called the \textit{diagonal Lie algebra}, has been studied by Goncharov in $ \cite{Go2}, \cite{Go3} $.
\end{itemize}
In these cases, we are able to determine the projection:
$$\vartheta : gr_{max}^{\mathfrak{D}} \mathcal{H}_{n}^{N} \rightarrow \mathbb{Q},$$
either via the linearized Ihara action $\underline{\circ}$, or via the dual point of view of infinitesimal derivations $D_{r}$. For instance, for $(i)$ ($N=1$, $w=3d$), it boils down to look at:
$$\frac{D_{3}^{\circ d }}{d! } \quad \text{ or } \quad   \exp_{\circ} ( \overline{\sigma}_{3}), \text{ where } \overline{\sigma}_{2i+1}= (-1)^{i}(\text{ad} e_{0} )^{2i}(e_{1}) \footnote{ These $\overline{\sigma}_{2i+1}$ are the generators of $gr_{1}^{\mathfrak{D}} \mathfrak{g}^{\mathfrak{m}}$, the depth $1$ graded part of the motivic Lie algebra; cf. $(\ref{eq:oversigma})$.}.$$
In general, the space $gr^{\mathfrak{D}}_{max}\mathcal{H}^{N}_{n}$ is more than $1$-dimensional; nevertheless, these methods could be generalized. 
\item Using criterion $\ref{criterehonoraire}$, we provide in the second section infinite families of honorary motivic multiple zeta values up to depth 5, with specified alternating odd or even integers. It was inspired by some isolated examples of honorary multiple zeta values found by D. Broadhurst\footnote{ Those emerged when looking at the \textit{depth drop phenomena},  cf. $\cite{Bro2}$.}, such as $\zeta( \overline{8}, 5, \overline{8}), \zeta( \overline{8}, 3, \overline{10}), \zeta(3, \overline{6}, 3, \overline{6}, 3)$, where we already could observe some patterns of even and odd. Investigating this trail in a systematic way, looking for any general families of unramified (motivic) Euler sums (without linear combinations first), we arrive at the families presented in $\S 6.2$, which unfortunately, stop in depth $5$. However, this investigation does not cover the unramified $\mathbb{Q}$-linear combinations of motivic Euler sums, such as those presented in Chapter $4$, Theorem $\ref{ESsharphonorary}$ (motivic Euler $\sharp$ sums with positive odds and negative even integers).
\item By Corollary $\ref{kerdn}$, we can lift some identities between MZV to motivic MZV (as in $\cite{Br2}$, Theorem $4.4$), and similarly in the case of Euler sums. Remark that, as we will see for depth $1$ Hoffman $\star$ elements (Lemma $\ref{lemmcoeff}$), the lifting may not be straightforward, if the family is not stable under the coaction.  In this section $\S 6.3$, we list some identities that we are able to lift to motivic versions, in particular some \textit{Galois trivial} elements\footnote{Galois trivial here means that the unipotent part of the Galois group acts trivially, not $\mathbb{G}_{m}$; hence not strictly speaking Galois trivial.} or product of simple zetas, and sum identities.
\end{enumerate}

\textsc{Remark}: The stability of a family on the coaction is a precious feature that allows to prove easily (by recursion) properties such as linear independence\footnote{If we find an appropriate filtration respected by the coaction, and such as the $0$ level elements are Galois-trivial, it corresponds then to the motivic depth filtration; for the Hoffman ($\star$) basis it is the number of $3$; for the Euler $\sharp$ sums basis, it is the number of odds, also equal to the depth minus $1$; for Deligne basis relative to $\mu_{N}$, $N=2,3,4,\mlq 6\mrq,8$, it is the usual depth.}, Galois descent features (unramified for instance), identities ($\S 6.3$), etc. 
$$\quad $$

\section{And Beyond?}

For most values of $N$, the situation concerning the periods of $\mathcal{MT}_{\Gamma_{N}} \subset \mathcal{MT} ( \mathcal{O}_{N}[\frac{1}{N}])$ is still hazy, although it has been studied in several articles, notably by Goncharov (\cite{Go2},\cite{Go3}, \cite{Go4}\footnote{Goncharov studied the structure of the fundamental group of $\mathbb{G}_{m} \diagdown \mu_{N}$ and made some parallels with the topology of some modular variety for $GL_{m, \diagup\mathbb{Q}}$, $m>1$ notably. He also proved, for $N=p\geq 5$, that the following morphism, given by the Ihara bracket, is not injective:
$$\beta:  \bigwedge^{2} gr^{\mathfrak{D}}_{1} \mathfrak{g}^{\mathfrak{m}}_{1}  \rightarrow  gr^{\mathfrak{D}}_{2} \mathfrak{g}^{\mathfrak{m}}_{2} \quad  \text{ and } \quad \begin{array}{ll}
\dim  \bigwedge^{2} gr^{\mathfrak{D}}_{1} \mathfrak{g}^{\mathfrak{m}}_{1} & = \frac{(p-1)(p-3)}{8}\\
\dim \ker \beta & = \frac{p^{2}-1}{24}\\
\dim Im \beta & = \dim gr^{\mathfrak{D}} \mathfrak{g}_{2}^{\mathfrak{m}}= \frac{(p-1)(p-5)}{12} \\
\dim gr^{\mathfrak{D}} \mathfrak{g}_{3}^{\mathfrak{m}}& \geq \frac{(p-5)(p^{2}-2p-11)}{48}.
\end{array}$$
Note that $gr^{\mathfrak{D}}_{2} \mathfrak{g}^{\mathfrak{m}}$ corresponds to the space generated by $\zeta^{\mathfrak{m}}\left(1,1 \atop \epsilon_{1}, \epsilon_{2}\right)$ quotiented by dilogarithms $\zeta^{\mathfrak{m}}\left(2 \atop \epsilon \right)$, modulo torsion.}) and Zhao: some bounds on dimensions, tables in small weight, and other results and thoughts on cyclotomic MZV can be seen in $\cite{Zh2}$, $\cite{Zh1}$,  $\cite{CZ}$. \\
\\
\texttt{Nota Bene}: As already pointed out, as soon as $N$ has a non inert prime factor $p$\footnote{In particular, as soon as $N\neq p^{s}, 2p^{s}, 4p^{s}, p^{s}q^{k}$ for $p,q$ odd prime since $(\mathbb{Z} \diagup m\mathbb{Z})^{\ast}$ is cyclic $\Leftrightarrow m=2,4,p^{k}, 2p^{k}$.}, $ \mathcal{MT}_{\Gamma_{N}} \subsetneq \mathcal{MT}(\mathcal{O}_{N}\left[ \frac{1}{N} \right] )$. Hence, some motivic periods of $\mathcal{MT}(\mathcal{O}_{N}\left[ \frac{1}{N}\right] )$ are not motivic iterated integrals on $\mathcal{P}^{1}\diagdown \lbrace 0, \mu_{N}, \infty\rbrace$ as considered above; already in weight $1$, there are more generators than the logarithms of cyclotomic units $\log^{\mathfrak{m}} (1-\xi_{N}^{a})$.\\
\\
Nevertheless, we can \textit{a priori} split the situation (of $\mathcal{MT}_{\Gamma_{N}}$) into two main schemes:
\begin{itemize}
\item[$(i)$] As soon as $N$ has two distinct prime factors, or $N$ power of $2$ or $3$, it is commonly believed that the motivic fundamental group $\pi_{1}^{\mathfrak{m}}(\mathbb{P}^{1}\diagdown \lbrace 0, \infty, \mu_{N}\rbrace, \overrightarrow{01})$ generates $\mathcal{MT}_{\Gamma_{N}}$, even though no suitable basis has been found. Also, in these cases, Zhao conjectured there were non standard relations\footnote{Non standard relations are these which do not come from distribution, conjugation, and regularised double shuffle relation, cf. $\cite{Zh1}$}. Nevertheless, in the case of $N$ power of $2$ or power of $3$, there seems to be a candidate for a basis ($\ref{eq:firstidea}$) and some linearly independent families were exhibited:
\begin{equation}\label{eq:firstidea}
\zeta^{\mathfrak{m}}\left(n_{1}, \cdots n_{p-1}, n_{p} \atop   \epsilon_{1}, \ldots , \epsilon_{p-1},\epsilon_{p}\right) \quad \text{ with }  \epsilon_{p}\in\mu_{N} \quad \text{primitive} \quad \text{ and } (\epsilon_{i})_{i<p} \quad \text{non primitive}.
\end{equation}
Indeed, when $N$ is a power of $2$ or $3$, linearly independent subfamilies of $\ref{eq:firstidea}$, keeping $\frac{3}{4}$ resp. $\frac{2}{3}$ generators in degree $1$, and all generators in degree $r>1$ are presented in $\cite{Wo}$ (in a dual point of view of the one developed here).\\
\\
\texttt{Nota Bene}: Some subfamilies of $\ref{eq:firstidea}$, restricting to $\lbrace \epsilon_{i}=1, x_{i} \geq 2\rbrace$ (here $\epsilon_{p}$ still as above) can be easily proven (via the coaction, by recursion on depth) to be linearly independent for any $N$; if N is a prime power, we can widen to $x_{i}\geq 1$, and for $N$ even to $\epsilon_{i}\in \lbrace \pm 1 \rbrace$; nevertheless, these families are considerably \textit{small}.\\

\item[$(ii)$] For $N=p^{s}$, $p$ prime greater than $5$, there are \textit{missing periods}: i.e. it is conjectured that the motivic fundamental group $\pi_{1}^{\mathfrak{m}}(\mathbb{P}^{1}\diagdown \lbrace 0, \infty, \mu_{N}\rbrace, \overrightarrow{01})$ does not generate $\mathcal{MT}_{\Gamma_{N}}$. For $N=p \geq 5$, it can already be seen in weight $2$, depth $2$. More precisely, (taking the dual point of view of Goncharov in $\cite{Go3}$), the following map is not surjective:
\begin{equation}\label{eq:d1prof2}
\begin{array}{llll}
D_{1}: & gr^{\mathfrak{D}}_{2} \mathcal{A}_{2}  &  \rightarrow &  \mathcal{A}_{1} \otimes \mathcal{A}_{1}\\
 &\zeta^{\mathfrak{a}} \left( 1,1 \atop \xi^{a}, \xi^{b}\right)& \mapsto &(a) \otimes (b) + \delta_{a+b \neq 0} ((b)-(a))\otimes (a+b)
\end{array}, \text{ where } \boldsymbol{(a)}\mathrel{\mathop:}= \zeta^{\mathfrak{a}} \left( 1 \atop \xi^{a}\right).
\end{equation}
These missing periods were a motivation for instance to introduce Aomoto polylogarithms  (in \cite{Du})\footnote{Aomoto polylogarithms generalize the previous iterated integrals, with notably differential forms such as $\frac{dt_{i}}{t_{i}-t_{i+1}-a_{i}}$; there is also a coaction acting on them.}.\\
Another idea, in order to reach these missing periods would be to use Galois descents: coming from a category above, in order to arrive at the category underneath, in the manner of Chapter $5$. For instance, missing periods for $N=p$ prime $>5$, could be reached via a Galois descent from the category $\mathcal{MT}_{\Gamma_{2p}}$ \footnote{This category is equal to $\mathcal{MT}(\mathcal{O}_{2p} ([\frac{1}{2p}]))$ iff $2$ is a primitive root modulo $p$. Some conditions on $p$ necessary or sufficient are known: this implies that $p\equiv 3,5 ± \mod 8$; besides, if $p\equiv 3,11 \mod 16$, it is true, etc.}. First, let point out that this category has the same dimensions than $\mathcal{MT}_{p}$ in degree $>1$, and has one more generator in degree $1$, corresponding to $\zeta^{\mathfrak{a}} \left( 1 \atop \xi^{p} \right) $. Furthermore, for $p$ prime, the descent between $\mathcal{H}^{2p}$ and $\mathcal{H}^{p}$ is measured by $D_{1}^{p}$, the component of $D_{1}$ associated to $\zeta^{\mathfrak{a}} \left( 1 \atop \xi^{p}\right) $:\\
$$\text{Let } \mathfrak{Z} \in \mathcal{H}^{2p}, \text{ then } \mathfrak{Z} \in  \mathcal{H}^{p} \Leftrightarrow \left\lbrace \begin{array}{l}
D^{p}_{1}(\mathfrak{Z})=0\\
D_{r} (\mathfrak{Z}) \in \mathcal{H}^{p}
\end{array}\right.$$
The situation is pictured by:
\begin{equation}\label{eq:descent2p} \xymatrix{
\mathcal{H}^{2p}:=\mathcal{H}^{\Gamma_{2p}}  \ar@{^{(}->}[r] &  \mathcal{H}^{\mathcal{MT}\left( \mathcal{O}_{2p}\left[ \frac{1}{2p}\right]\right)  }\\
\mathcal{H}^{p}:= \mathcal{H}^{\Gamma_{p}}=\mathcal{H}^{\mathcal{MT}\left( \mathcal{O}_{p}\left[ \frac{1}{p}\right]\right)}  \ar[u]^{D^{p}_{1}} \ar@{=}[r]  &  \mathcal{H}^{\mathcal{MT}\left( \mathcal{O}_{2p}\left[ \frac{1}{p}\right]\right) } \\
\mathcal{H}^{\mathcal{MT}\left( \mathcal{O}_{p}\right) }\ar[u]^{\lbrace D^{2a}_{1}-D^{a}_{1}\rbrace_{2 \leq a \leq \frac{p-1}{2}}}  \ar@{=}[r] &  \mathcal{H}^{\mathcal{MT}\left( \mathcal{O}_{2p}\right) } }.
\end{equation}
\texttt{Example, for N=5}:\nomenclature{$\text{Vec}_{k} \left\langle X \right\rangle$}{the $k$ vector space generated by elements in $X$ }
A basis of $gr^{\mathfrak{D}}_{p}\mathcal{A}_{1}$ corresponds to the logarithms of the roots of unity $\xi^{1}, \xi^{2}$; here, $\xi=\xi_{5}$ is a primitive fifth root of unity. Moreover, the image of $D_{1}: \mathcal{A}_{2} \rightarrow \mathcal{A}_{1} \otimes  \mathcal{A}_{1}$ on $\zeta^{\mathfrak{a}} \left( 1,1 \atop \xi_{5}^{a}, \xi_{5}^{b}\right) $ is (cf. $\ref{eq:d1prof2}$):
$$\text{Vec}_{\mathbb{Q}} \left\langle (1)\otimes (1),  (2)\otimes (2),(1)\otimes (2)+ (2)\otimes (1)  \right\rangle .$$
We notice that one dimension is missing ($3$ instead of $4$). Allowing the use of tenth roots of unity, adding for instance here in depth $2$, $\zeta^{\mathfrak{a}} \left( 1,1 \atop \xi_{10}^{1}, \xi_{10}^{2}\right)$ recovers the surjection for $D_{1}$. Since we have at our disposal criterion to determine if a MMZV$_{\mu_{10}}$ is in $\mathcal{H}^{5}$, we could imagine constructing a base of $\mathcal{H}^{5}$ from tenth roots of unity. \\
\texttt{Nota Bene}: More precisely, we have the following spaces, descents and dimensions:
\begin{equation}\label{eq:descent10} \xymatrix{
\mathcal{H}^{\mathcal{MT}\left( \mathcal{O}_{10}\left[ \frac{1}{10}\right]\right)  }= \mathcal{H}^{\Gamma_{10}}  &  \\
\mathcal{H}^{\mathcal{MT}\left( \mathcal{O}_{10}\left[ \frac{1}{5}\right]\right)  }=\mathcal{H}^{\mathcal{MT}\left( \mathcal{O}_{5}\left[ \frac{1}{5}\right] \right) }= \mathcal{H}^{\Gamma_{5}} \ar[u]^{D^{5}_{1}} &   d_{n}= 2d_{n-1}+3d_{n-2}= 3d_{n-1}\\
\mathcal{H}^{\mathcal{MT}\left( \mathcal{O}_{5}\right) }=\mathcal{H}^{\mathcal{MT}\left( \mathcal{O}_{10}\right)  } \ar[u]^{D^{4}_{1}+ D^{2}_{1}}  & d'_{n}=  2d'_{n-1}+d'_{n-2} }\\
\end{equation}

\end{itemize}
\noindent
\textsc{Remarks}:
\begin{itemize}
\item[$\cdot$]
Recently (in $\cite{Bro3}$), Broadhurst made some conjectures about \textit{multiple Landen values}, i.e. periods associated to the ring of integers of the real subfield of $\mathbb{Q}(\xi_{5})$, i.e. $\mathbb{Z} \left[\rho \right]$, with $\rho\mathrel{\mathop:}= \frac{1+\sqrt{5}}{2} $, the golden ratio\footnote{He also looked at the case of the real subfield of $\mathbb{Q}(\xi_{7})$ in his latest article: $\cite{Bro4}$}. Methods presented through this thesis could be transposed in such context.\\
\item[$\cdot$] It also worth noticing that, for $N=p>5$, modular forms obstruct the freeness of the Lie algebra $gr_{\mathfrak{D}} \mathfrak{g}^{\mathfrak{m}}$\footnote{Goncharov proved that the subspace of cuspidal forms of weight 2 on the modular curve $X_{1}(p)$ (associated to $\Gamma_{1}(p)$), of dimension $\frac{(p-5)(p-7)}{24}$ embeds into $\ker \beta$, for $N=p \geq 11$ which leaves another part of dimension $\frac{p-3}{2}$.}, as in the case of $N=1$ (cf. $\cite{Br3}$). Indeed, for $N=1$ one can associate, to each cuspidal form of weight $n$, a relation between weight $n$ double and simple multiple zeta values, cf. $\cite{GKZ}$. Notice that, on the contrary, for $N=2,3,4,8$, $gr_{\mathfrak{D}} \mathfrak{g}^{\mathfrak{m}}$ is free. This fascinating connection with modular forms still waits to be explored for cyclotomic MZV. \footnote{We could hope also for an interpretation, in these cyclotomic cases, of exceptional generators and relations in the Lie algebra, in the way of $\cite{Br3}$ for $N=1$.}\\
\item[$\cdot$] In these cases where $gr_{\mathfrak{D}} \mathfrak{g}^{\mathfrak{m}}$ is not free, since we have to turn towards other basis (than $\ref{eq:firstidea}$), we may remember the Hoffman basis (of $\mathcal{H}^{1}_{n}$, cf $\cite{Br2}$): $\left\lbrace \zeta^{\mathfrak{m}} \left( \lbrace 2, 3\rbrace^{\times}\right)\right\rbrace_{\text{ weight } n}$, whose dimensions verify $d_{n}=d_{n-2}+d_{n-3}$. Looking at dimensions in Lemma $2.3.1$, two cases bring to mind a basis in the \textit{\textbf{Hoffman's way}}:
\begin{itemize}
\item[$(i)$] For $\mathcal{MT}(\mathcal{O}_{N})$, since $d_{n}= \frac{\varphi(N)}{2}d_{n-1}+ d_{n-2}$, this suggests to look for a basis with $\boldsymbol{1}$ (with $\frac{\varphi(N)}{2}$ choices of $N^{\text{th}}$ roots of unity) and $\boldsymbol{2}$ (1 choice of $N^{\text{th}}$ roots of unity).
\item[$(ii)$]  For  $\mathcal{MT}\left( \mathcal{O}_{p^{r}}\left[  \frac{1}{p} \right] \right)$, where $p \mid N$ and $p$ inert, since $ d_{n}= \left( \frac{\varphi(N)}{2}+1\right)^{n},$ this suggests a basis with only $1$ above, and $( \frac{\varphi(N)}{2}+1)$ choices of $N^{\text{th}}$ roots of unity; in particular if $N=p^{k}$.\\
\end{itemize}
\texttt{Example}: For $N=2$, the recursion relation for dimensions $d_{n}=d_{n-1}+d_{n-2}$ of $\mathcal{H}_{n}^{2}$ suggests, \textit{in the Hoffman's way}, a basis composed of motivic Euler sums with only $1$ and $2$. For instance, the following are candidates conjectured to be a basis, supported by numerical computations:
$$\left\lbrace  \zeta^{\mathfrak{m}} \left( n_{1}, \ldots, n_{p-1}, n_{p} \atop 1, \ldots, 1, -1 \right) , n_{i}\in \lbrace 2, 1\rbrace \right\rbrace, \textsc{ or } \left\lbrace  \zeta^{\mathfrak{m}} \left( 1, \cdots 1, \atop \boldsymbol{s}, -1 \right)\zeta^{\mathfrak{m}} (2)^{\bullet} ,\boldsymbol{s}\in \left\lbrace \lbrace 1\rbrace, \lbrace -1,-1\rbrace\right\rbrace   ^{\ast }\right\rbrace.$$
However, there is not a nice \textit{suitable} filtration\footnote{In the second case, it appears that we could proceed as follows to show the linear independence of these elements, where $p$ equals $1+$ the number of $1$ in the $E_{n}$ element:
Prove that, for $x\in E_{n,p}$ there exists a linear combination $cl(x)\in E_{n,>p}$ such that $x+cl(x)\in\mathcal{F}^{\mathfrak{D}}_{p} \mathcal{H}_{n}$, and then that $\lbrace x+cl(x), x\in E_{n,p} \rbrace$ is precisely a basis for $gr^{\mathfrak{D}}_{p} \mathcal{H}_{n}$, considering, for $2r \leq n-p$:
$$D_{2r+1}: gr^{\mathfrak{D}}_{p} \mathcal{H}_{n}  \rightarrow gr^{\mathfrak{D}}_{p-1} \mathcal{H}_{n-2r-1}.$$} corresponding to the motivic depth which would allow a recursive proof \footnote{A suitable filtration, whose level $0$ would be the power of $\pi$, level $1$ would be linear combinations of $\zeta(odd)\cdot\zeta(2)^{\bullet}$, etc.; as in proofs in $\S 4.5.1$.}. 
\end{itemize}

\chapter{MZV $\star$ and Euler  $\sharp$ sums}

\paragraph{\texttt{Contents}:}
After introducing motivic Euler $\star$, and $\sharp$ sums, with some useful motivic relations (antipodal and hybrid), the third section focuses on some specific Euler $\sharp$ sums, starting by a broad subfamily of \textit{unramified} elements (i.e. which are motivic MZV) and extracting from it a new basis for $\mathcal{H}^{1}$. The fourth section deals with the Hoffman star family, proving it is a basis of $\mathcal{H}^{1}$, up to an analytic conjecture ($\ref{conjcoeff}$). In Appendix $\S 4.7$, some missing coefficients in Lemma $\ref{lemmcoeff}$, although not needed for the proof of the Hoffman $\star$ Theorem $\ref{Hoffstar}$, are discussed. The last section presents a conjectured motivic equality ($\ref{lzg}$) which turns each motivic MZV $\star$ into a motivic Euler $\sharp$ sums of the previous honorary family; in particular, under this conjecture, the two previous bases are identical. The proofs here are partly based on results of Annexe $\S A.1$, which themselves use relations presented in $\S 4.2$.

\section{Star, Sharp versions}

Here are the different variants of motivic Euler sums (MES) used in this chapter, where a $ \pm \star$ resp. $\pm \sharp$ in the notation below $I(\cdots)$ stands for a $\omega_{\pm \star} $ resp. $\omega_{\pm\sharp}$ in the iterated integral:\footnote{Possibly regularized with $(\ref{eq:shufflereg})$.}
\begin{defi}
Using the expression in terms of motivic iterated integrals ($\ref{eq:reprinteg}$), motivic Euler sums are, with $n_{i}\in\mathbb{Z}^{\ast}$, $\epsilon_{i}\mathrel{\mathop:}=sign(n_{i})$:
\begin{equation}\label{eq:mes}
\zeta^{\mathfrak{m}}_{k} \left(n_{1}, \ldots , n_{p} \right) \mathrel{\mathop:}= (-1)^{p}I^{\mathfrak{m}} \left(0; 0^{k}, \epsilon_{1}\cdots \epsilon_{p}, 0^{\mid n_{1}\mid -1} ,\ldots, \epsilon_{i}\cdots \epsilon_{p}, 0^{\mid n_{i}\mid -1} ,\ldots, \epsilon_{p}, 0^{\mid n_{p}\mid-1} ;1 \right).
\end{equation}
$$\text{ With the differentials: } \omega_{\pm\star}\mathrel{\mathop:}= \omega_{\pm 1}- \omega_{0}=\frac{dt}{t(\pm t-1)}, \quad \quad \omega_{\pm\sharp}\mathrel{\mathop:}=2 \omega_{\pm 1}-\omega_{0}=\frac{(t \pm 1)dt}{t(t\mp 1)},$$
\begin{description}
\item[MES ${\star}$] are defined similarly than $(\ref{eq:mes})$ with $\omega_{\pm \star}$ (instead of $\omega_{\pm 1}$), $\omega_{0}$ and a $\omega_{\pm 1}$ at the beginning:
$$\zeta_{k}^{\star,\mathfrak{m}} \left(n_{1}, \ldots , n_{p} \right) \mathrel{\mathop:}= (-1)^{p} I^{\mathfrak{m}} \left(0; 0^{k}, \epsilon_{1}\cdots \epsilon_{p}, 0^{\mid n_{1}\mid-1}, \epsilon_{2}\cdots \epsilon_{p}\star, 0^{\mid n_{2}\mid -1}, \ldots, \epsilon_{p}\star, 0^{\mid n_{p}\mid-1} ;1 \right).$$
\item[MES ${\star\star}$] similarly with only $\omega_{\pm \star}, \omega_{0}$ (including the first):\nomenclature{MES ${\star\star}$, $\zeta^{\star\star,\mathfrak{m}}$}{Motivic Euler sums $\star\star$}
$$\zeta_{k}^{\star\star,\mathfrak{m}} \left(n_{1}, \ldots , n_{p} \right) \mathrel{\mathop:}= (-1)^{p} I^{\mathfrak{m}}  \left(0; 0^{k}, \epsilon_{1}\cdots \epsilon_{p}\star, 0^{\mid n_{1}\mid-1}, \epsilon_{2}\cdots \epsilon_{p}\star, 0^{\mid n_{2}\mid-1}, \ldots, \epsilon_{p}\star, 0^{\mid n_{p}\mid-1} ;1 \right).$$
\item[MES ${\sharp}$] with $\omega_{\pm \sharp},\omega_{0} $ and a $\omega_{\pm 1}$ at the beginning:
$$\zeta_{k}^{\sharp,\mathfrak{m}} \left(n_{1}, \ldots , n_{p} \right) \mathrel{\mathop:}= 2 (-1)^{p} I^{\mathfrak{m}}  \left(0; 0^{k}, \epsilon_{1}\cdots \epsilon_{p}, 0^{\mid n_{1}\mid-1}, \epsilon_{2}\cdots \epsilon_{p}\sharp, 0^{\mid n_{2}\mid -1}, \ldots, \epsilon_{p}\sharp, 0^{\mid n_{p}\mid-1} ;1 \right).$$
\item[MES $\sharp\sharp$] similarly with only $\omega_{\pm \sharp}, \omega_{0}$ (including the first):\nomenclature{MES ${\sharp\sharp}$, $\zeta^{\sharp\sharp,\mathfrak{m}}$}{Motivic Euler sums $\sharp\sharp$}
$$\zeta_{k}^{\sharp\sharp,\mathfrak{m}} \left(n_{1}, \ldots , n_{p} \right) \mathrel{\mathop:}= (-1)^{p} I^{\mathfrak{m}}  \left(0; 0^{k}, \epsilon_{1}\cdots \epsilon_{p}\sharp, 0^{\mid n_{1}\mid-1},  \epsilon_{2}\cdots \epsilon_{p}\sharp, 0^{\mid n_{2}\mid-1}, \ldots, \epsilon_{p}\sharp, 0^{\mid n_{p}\mid-1} ;1 \right).$$
\end{description}
\end{defi}
\textsc{Remarks}:
\begin{itemize}
\item[$\cdot$] The Lie algebra of the fundamental group $\pi_{1}^{dR}(\mathbb{P}^{1}\diagdown \lbrace 0, 1, \infty\rbrace)=\pi_{1}^{dR}(\mathcal{M}_{0,4})$ is generated by $e_{0}, e_{1},e_{\infty}$ with the only condition than $e_{0}+e_{1}+e_{\infty}=0$\footnote{ For the case of motivic Euler sums, it is the Lie algebra generated by $e_{0}, e_{1}, e_{-1}, e_{\infty}$ with the only condition than $e_{0}+e_{1}+e_{-1}+e_{\infty}=0$; similarly for other roots of unity with $e_{\eta}$. Note that $e_{i}$ corresponds to the class of the residue around $i$ in $H_{dR}^{1}(\mathbb{P}^{1} \diagdown \lbrace 0, \mu_{N}, \infty \rbrace)^{\vee}$. }. If we keep $e_{0}$ and $e_{\infty}$ as generators, instead of the usual $e_{0},e_{1}$, it leads towards MMZV  $^{\star\star}$ up to a sign, instead of MMZV since $-\omega_{0}+\omega_{1}- \omega_{\star}=0$. We could also choose $e_{1}$ and $e_{\infty}$ as generators, which leads to another version of MMZV that has not been much studied yet. These versions are equivalent since each one can be expressed as $\mathbb{Q}$  linear combination of another one.
\item[$\cdot$] By linearity and $\shuffle$-regularisation $(\ref{eq:shufflereg})$, all these versions ($\star$, $\star\star$, $\sharp$ or $\sharp\sharp$) are $\mathbb{Q}$-linear combination of motivic Euler sums. Indeed, with $n_{+}$ the number of $+$ among $\circ$:
$$\begin{array}{llll}
\zeta ^{\star,\mathfrak{m}}(n_{1}, \ldots, n_{p})  &=& \sum_{\circ=\mlq + \mrq \text{ or } ,} & \zeta ^{\mathfrak{m}}(n_{1}\circ \cdots \circ n_{p}) \\
\\
\zeta ^{\mathfrak{m}}(n_{1}, \ldots, n_{p})  &=& \sum_{\circ=\mlq + \mrq \text{ or } ,} (-1)^{n_{+}} & \zeta ^{\star,\mathfrak{m}}(n_{1}\circ \cdots \circ n_{p})  \\
\\
 \zeta ^{ \sharp,\mathfrak{m}}(n_{1}, \ldots, n_{p}) &=& \sum_{\circ=\mlq + \mrq \text{ or } ,} 2^{p-n_{+}} & \zeta ^{\mathfrak{m}}(n_{1}\circ \cdots \circ n_{p}) \\
 \\
 \zeta ^{\mathfrak{m}}(n_{1}, \ldots, n_{p}) &=&  \sum_{\circ=\mlq + \mrq \text{ or } ,} (-1)^{n_{+}} 2^{-p} & \zeta ^{ \sharp,\mathfrak{m}}(n_{1}\circ \cdots \circ n_{p})   \\
 \\
  \zeta ^{\star\star,\mathfrak{m}}(n_{1}, \ldots, n_{p}) &=& \sum_{i=0}^{p-1} & \zeta ^{\star,\mathfrak{m}}_{\mid n_{1}\mid+\cdots+\mid n_{i}\mid}(n_{i+1}, \cdots , n_{p}) \\
  & = & \sum_{\circ=\mlq + \mrq \text{ or } ,\atop i=0}^{p-1} & \zeta^{\mathfrak{m}}_{\mid n_{1}\mid+\cdots+\mid n_{i}\mid}(n_{i+1}\circ \cdots \circ n_{p})\\
  \\
 \zeta ^{ \sharp\sharp,\mathfrak{m}}(n_{1}, \ldots, n_{p}) &=& \sum_{ i=0}^{p-1} & \zeta ^{\sharp,\mathfrak{m}}_{\mid n_{1}\mid+\cdots+\mid n_{i}\mid}(n_{i+1}, \cdots , n_{p})  \\
 & = & \sum_{\circ=\mlq + \mrq \text{ or } ,\atop i=0}^{p-1} 2^{p-i-n_{+}} & \zeta^{\mathfrak{m}}_{\mid n_{1}\mid+\cdots+\mid n_{i}\mid}(n_{i+1}\circ \cdots \circ n_{p}) \\
 \\
  \zeta ^{\star,\mathfrak{m}}(n_{1}, \ldots, n_{p})  &=& \zeta ^{\star\star,\mathfrak{m}}(n_{1}, \ldots, n_{p})& -\zeta ^{\star\star,\mathfrak{m}}_{\mid n_{1}\mid}(n_{2}, \ldots, n_{p})
 \\
\\
\zeta ^{\sharp,\mathfrak{m}}(n_{1}, \ldots, n_{p}) &=& \zeta ^{\sharp\sharp}(n_{1}, \ldots, n_{p})
&-\zeta ^{\sharp\sharp,\mathfrak{m}}_{\mid n_{1}\mid}(n_{2}, \ldots, n_{p}) \\ 

\end{array}$$
\texttt{Notation:} Beware, the $\mlq + \mrq$ here is on $n_{i}\in\mathbb{Z}^{\ast}$ is a summation of absolute values while signs are multiplied:
$$n_{1} \mlq + \mrq \cdots   \mlq + \mrq n_{i} \rightarrow sign(n_{1}\cdots n_{i})( \vert n_{1}\vert +\cdots  + \vert n_{i} \vert).$$
\end{itemize}
\texttt{Examples:} Expressing them as $\mathbb{Q}$ linear combinations of motivic Euler sums\footnote{To get rid of the $0$ in front of the MZV, as in the last example, we use the shuffle regularisation $\ref{eq:shufflereg}$.}:
$$\begin{array}{lll}
\zeta^{\star,\mathfrak{m}}(2,\overline{1},3) & = & -I^{\mathfrak{m}}(0;-1,0,-\star, \star,0,0; 1) \\
& = &  \zeta^{\mathfrak{m}}(2,\overline{1},3)+ \zeta^{\mathfrak{m}}(\overline{3},3)+ \zeta^{\mathfrak{m}}(2,\overline{4})+\zeta^{\mathfrak{m}}(\overline{6}) \\
\zeta^{\sharp,\mathfrak{m}}(2,\overline{1},3) &=& - 2 I^{\mathfrak{m}}(0;-1,0,-\sharp, \sharp,0,0; 1) \\
 &=& 8\zeta^{\mathfrak{m}}(2,\overline{1},3)+ 4\zeta^{\mathfrak{m}}(\overline{3},3)+ 4\zeta^{\mathfrak{m}}(2,\overline{4})+2\zeta^{\mathfrak{m}}(\overline{6})\\
 \zeta^{\star\star,\mathfrak{m}}(2,\overline{1},3) &=& -I^{\mathfrak{m}}(0;-\star,0,-\star, \star,0,0; 1) \\
 &=& \zeta^{\star,\mathfrak{m}}(2,\overline{1},3)+ \zeta^{\star,\mathfrak{m}}_{2}(\overline{1},3)+\zeta^{\star,\mathfrak{m}}_{3}(3) \\
  &=&  \zeta^{\star,\mathfrak{m}}(2, \overline{1}, 3)+\zeta^{\star,\mathfrak{m}}(\overline{3},3)+3 \zeta^{\star,\mathfrak{m}}(\overline{2},4)+6\zeta^{\star,\mathfrak{m}}(1, 5)-10\zeta^{\star,\mathfrak{m}}(6)\\
 &=& 11\zeta^{\mathfrak{m}}(\overline{6})+2\zeta^{\mathfrak{m}}(\overline{3}, 3)+\zeta^{\mathfrak{m}}(2, \overline{4})+\zeta^{\mathfrak{m}}(2,\overline{1}, 3)+3\zeta^{\mathfrak{m}}(\overline{2}, 4)+6\zeta^{\mathfrak{m}}(\overline{1}, 5)-10\zeta^{\mathfrak{m}}(6)\\
\end{array}$$

\paragraph{Stuffle.}

One of the most famous relations between cyclotomic MZV, the\textit{ stuffle} relation, coming from the multiplication of series, has been proven to be \textit{motivic} i.e. true for cyclotomic MMZV, which was a priori non obvious. \footnote{The stuffle for these motivic iterated integrals can be deduced from works by Goncharov on mixed Hodge structures, but was also proved in a direct way by G. Racinet, in his thesis, or I. Souderes in $\cite{So}$ via blow-ups. Remark that shuffle relation, coming from the iterated integral representation is clearly \textit{motivic}.} In particular:
\begin{lemm}
$$\zeta^{\mathfrak{m}}\left( a_{1}, \ldots, a_{r} \atop \alpha_{1}, \ldots, \alpha_{r}\right) \zeta^{\mathfrak{m}}\left(  b_{1}, \ldots, b_{s} \atop  \beta_{1}, \ldots, \beta_{s}\right)=\sum_{ \left(  c_{j} \atop \gamma_{j}\right)  = \left(  a_{i} \atop \alpha_{i} \right) ,\left(  b_{i'} \atop \beta_{i'}\right)  \text{ or }\left(  a_{i}+b_{i'} \atop \alpha_{i}\beta_{i'}\right) \atop \text{order } (a_{i}), (b_{i}) \text{ preserved}  }\zeta^{\mathfrak{m}}\left( c_{1}, \ldots, c_{m} \atop \gamma_{1}, \ldots, \gamma_{m} \right) .$$
$$\zeta^{\star,\mathfrak{m}}\left( a_{1}, \ldots, a_{r} \atop \alpha_{1}, \ldots, \alpha_{r}\right) \zeta^{\star, \mathfrak{m}}\left(  b_{1}, \ldots, b_{s} \atop  \beta_{1}, \ldots, \beta_{s}\right)=\sum_{ \left(  c_{j} \atop \gamma_{j}\right)  = \left(  a_{i} \atop \alpha_{i} \right) ,\left(  b_{i} \atop \beta_{i}\right)  \text{ or }\left(  a_{i}+b_{i'} \atop \alpha_{i}\beta_{i'}\right) \atop \text{order } (a_{i}), (b_{i}) \text{ preserved}  }(-1)^{r+s+m}\zeta^{\star, \mathfrak{m}}\left( c_{1}, \ldots, c_{m} \atop \gamma_{1}, \ldots, \gamma_{m} \right) .$$
$$\zeta^{\sharp , \mathfrak{m}}\left( \textbf{a} \atop \boldsymbol{\alpha}  \right) \zeta^{\sharp, \mathfrak{m}}\left( \textbf{ b } \atop  \boldsymbol{\beta} \right)=\sum_{ \left(  c_{j} \atop \gamma_{j}\right)  = \left( a_{i}+\sum_{l=1}^{k} a_{i+l} +b_{i'+l} \atop \alpha_{i}\prod_{l=1}^{k}\alpha_{i+l}\beta_{i'+l}\right) \text{ or }  \left( b_{i'}+\sum_{l=1}^{k} a_{i+l} +b_{i'+l} \atop \beta_{i'}\prod_{l=1}^{k}\alpha_{i+l}\beta_{i'+l}\right) \atop k\geq 0, \text{ order } (a_{i}), (b_{i}) \text{ preserved}}(-1)^{\frac{r+s-m}{2}}\zeta^{\sharp, \mathfrak{m}}\left( c_{1}, \ldots, c_{m} \atop \gamma_{1}, \ldots, \gamma_{m} \right) .$$
\end{lemm}
\noindent\textsc{Remarks:}
\begin{itemize}
\item[$\cdot$] In the depth graded, stuffle corresponds to shuffle the sequences $\left( \boldsymbol{a} \atop \boldsymbol{\alpha} \right) $ and $\left( \boldsymbol{b} \atop \boldsymbol{\beta} \right) $.
\item[$\cdot$] Other identities mixing the two versions could also be stated, such as
$$\zeta^{\star, \mathfrak{m}}\left( a_{1}, \ldots, a_{r} \atop \alpha_{1}, \ldots, \alpha_{r}\right) \zeta^{\mathfrak{m}}\left(  b_{1}, \ldots, b_{s} \atop  \beta_{1}, \ldots, \beta_{s}\right)=\sum_{ \left(  c_{j} \atop \gamma_{j}\right)  = \left(  a_{i} \atop \alpha_{i} \right) ,\left(  b_{i'} \atop \beta_{i'}\right)  \text{ or }\left(  (\sum_{l=1}^{k} a_{i+l})+b_{i'} \atop (\prod_{l=1}^{k}\alpha_{i+l})\beta_{i'}\right) \atop k \geq 1, \text{order } (a_{i}), (b_{i}) \text{ preserved}  }\zeta^{\mathfrak{m}}\left( c_{1}, \ldots, c_{m} \atop \gamma_{1}, \ldots, \gamma_{m} \right) .$$
\end{itemize}

\section{Relations in $\mathcal{L}$}

\subsection{Antipode relation}

In this part, we are interested in some Antipodal relations for motivic Euler sums in the coalgebra $\mathcal{L}$, i.e. modulo products. To explain quickly where they come from, let's go back to two combinatorial Hopf algebra structures.\\
\\
First recall that if $A$ is a graded connected bialgebra, there exists an unique antipode S (leading to a Hopf algebra structure)\footnote{It comes from the usual required relation for the antipode in a Hopf algebra, but because it is graded and connected, we can apply the formula recursively to construct it, in an unique way. }, which is the graded map defined by:
\begin{equation} \label{eq:Antipode} S(x)= -x-\sum S(x_{(1)}) \cdot x_{(2)},
\end{equation}
where $\cdot$ is the product and using Sweedler notations for the coaction:
$$\Delta (x)= 1\otimes x+ x\otimes 1+ \sum x_{(1)}\otimes x_{(2)}= \Delta'(x)+ 1\otimes x+ x\otimes 1 .$$
Hence, in the quotient $A/ A_{>0}\cdot A_{>0} $:
$$S(x) \equiv -x .  $$

\subsubsection{The $\shuffle$ Hopf algebra}

Let $X=\lbrace a_{1},\cdots, a_{n} \rbrace$ an alphabet and $A_{\shuffle}\mathrel{\mathop:}=\mathbb{Q} \langle X^{\times} \rangle$ the $\mathbb{Q}$-vector space generated by words on X, i.e. non commutative polynomials in $a_{i}$. It is easy to see that $A_{\shuffle}$ is a Hopf algebra with the $\shuffle$ shuffle product, the deconcatenation coproduct $\Delta_{D}$ and antipode $S_{\shuffle}$:\nomenclature{$\Delta_{D}$}{the deconcatenation coproduct}
\begin{equation} \label{eq:shufflecoproduct}  \Delta_{D}(a_{i_{1}}\cdots a_{i_{n}})= \sum_{k=0}^{n} a_{i_{1}}\cdots a_{i_{k}} \otimes a_{i_{k+1}} \cdots a_{i_{n}}.
\end{equation}
\begin{equation} \label{eq:shuffleantipode} 
S_{\shuffle} (a_{i_{1}} \cdots a_{i_{n}})= (-1)^{n} a_{i_{n}} \cdots a_{i_{1}}.
\end{equation}
$A_{\shuffle}$ is even a connected graded Hopf algebra, called the \textit{ shuffle Hopf algebra}; the grading coming from the degree of polynomial. By the equivalence of category between $\mathbb{Q}$-Hopf algebra and $\mathbb{Q}$-Affine Group Scheme, it corresponds to:
\begin{equation} \label{eq:gpschshuffle} 
G=\text{Spec} A_{\shuffle} : R \rightarrow \text{Hom}(\mathbb{Q} \langle X \rangle, R)=\lbrace S\in R\langle\langle a_{i} \rangle\rangle \mid \Delta_{\shuffle} S= S\widehat{\otimes} S, \epsilon(S)=1 \rbrace,
\end{equation} 
where $\Delta_{\shuffle}$ is the coproduct dual to the product $\shuffle$:\nomenclature{$\Delta_{\shuffle}$}{the $\shuffle$ coproduct}
$$\Delta_{\shuffle}(a_{i_{1}}\cdots a_{i_{n}})= \left( 1\otimes a_{i_{1}}+ a_{i_{1}}\otimes 1\right) \cdots \left( 1\otimes a_{i_{n}}+ a_{i_{n}}\otimes 1\right) .$$
Let restrict now to $X=\lbrace 0,\mu_{N}\rbrace$; our main interest in this Chapter is $N=2$, but it can be extended to other roots of unity. The shuffle relation for motivic iterated integral relative to $\mu_{N}$:
\begin{equation}\label{eq:shuffleim}
I^{\mathfrak{m}}(0; \cdot ; 1) \text{ is a morphism of Hopf algebra from } A_{\shuffle} \text{ to }  (\mathbb{R},\times): 
\end{equation}
$$I^{\mathfrak{m}}(0; w ; 1) I^{\mathfrak{m}}(0; w' ; 1)= I^{\mathfrak{m}}(0; w\shuffle w' ;1) \text{ with }  w,w' \text{ words in } X.$$

\begin{lemm}[\textbf{Antipode $\shuffle$}]
In the coalgebra $\mathcal{L}$, with $w$ the weight, $\bullet$ standing for MMZV$_{\mu_{N}}$, or $\star\star$ ($N=2$) resp. $\sharp\sharp$-version ($N=2$):
$$\zeta^{\bullet,\mathfrak{l}}_{n-1}\left( n_{1}, \ldots, n_{p} \atop \epsilon_{1}, \ldots, \epsilon_{p} \right) \equiv (-1)^{w+1}\zeta^{\bullet,\mathfrak{l}}_{n_{p}-1}\left( n_{p-1}, \ldots, n_{1},n \atop \epsilon_{p-1}^{-1}, \ldots, \epsilon_{1}^{-1}, \epsilon \right) \text{ where }  \epsilon\mathrel{\mathop:}=\epsilon_{1}\cdot\ldots\cdot\epsilon_{p}.$$
\end{lemm}
\noindent
This formula stated for any $N$ is slightly simpler in the case $N=1,2$ since $n_{i}\in\mathbb{Z}^{\ast}$:
\begin{framed}
\begin{equation}\label{eq:antipodeshuffle2}
\textsc{Antipode } \shuffle \quad :
\begin{array}{l}
 \zeta^{\bullet,\mathfrak{l}}_{n-1}\left( n_{1}, \ldots, n_{p} \right) \equiv(-1)^{w+1}\zeta^{\bullet,\mathfrak{l}}_{\mid n_{p}\mid -1}\left( n_{p-1}, \ldots, n_{1},sign(n_{1}\cdots n_{p}) n \right)\\
 \text{ } \\
 I^{\mathfrak{l}}(0;X;\epsilon)\equiv (-1)^{w} I^{\mathfrak{l}}(\epsilon;\widetilde{X};0) \equiv (-1)^{w+1} I^{\mathfrak{l}}(0; \widetilde{X}; \epsilon)
\end{array}.
\end{equation}
\end{framed}
Here $X$ is any word in $0,\pm 1$ or $0, \pm \star$ or $0, \pm\sharp$, and $\widetilde{X}$ denotes the \textit{reversed} word.

\begin{proof}
For motivic iterated integrals, as said above:
$$ S_{\shuffle} (I^{\mathfrak{m}}(0; a_{1}, \ldots, a_{n}; 1))= (-1)^{n}I^{\mathfrak{m}}(0; a_{n}, \ldots, a_{1}; 1),$$
which, in terms of the MMZV$_{\mu_{N}}$ notation is:
$$S_{\shuffle}\left( \zeta^{\bullet,\mathfrak{l}}_{n-1}\left( n_{1}, \ldots, n_{p} \atop \epsilon_{1}, \ldots, \epsilon_{p} \right) \right)  \equiv (-1)^{w}\zeta^{\bullet,\mathfrak{l}}_{n_{p}-1}\left( n_{p-1}, \ldots, n_{1},n \atop \epsilon_{p-1}^{-1}, \ldots, \epsilon_{1}^{-1}, \epsilon \right) \text{ where }  \epsilon\mathrel{\mathop:}=\epsilon_{1}\cdot\ldots\cdot\epsilon_{p}.$$
Then, if we look at the antipode recursive formula $\eqref{eq:Antipode}$ in the coalgebra $\mathcal{L}$, for $a_{i}\in \lbrace 0, \mu_{N} \rbrace$:
$$ S_{\shuffle} (I^{\mathfrak{l}}(0; a_{1}, \ldots, a_{n}; 1))\equiv - I^{\mathfrak{l}}(0; a_{1}, \ldots, a_{n}; 1).$$
This leads to the lemma above. The $\shuffle$-antipode relation can also be seen at the level of iterated integrals as the path composition modulo products followed by a reverse of path.
\end{proof}

\subsubsection{The $\ast$ Hopf algebra}
Let $Y=\lbrace \cdots, y_{-n}, \ldots, y_{-1}, y_{1},\cdots, y_{n}, \cdots \rbrace$ an infinite alphabet and $A_{\ast}\mathrel{\mathop:}=\mathbb{Q} \langle Y^{\times} \rangle$ the non commutative polynomials in $y_{i}$ with rational coefficients, with $y_{0}=1$ the empty word.
Similarly, it is a graded connected Hopf algebra called the\textit{ stuffle Hopf algebra}, with the stuffle $\ast$ product and the following coproduct:\footnote{For the $\shuffle$ algebra, we had to use the notation in terms of iterated integrals, with $0,\pm 1$, but for the $\ast$ stuffle relation, it is more natural with the Euler sums notation, which corresponds to $y_{n_{i}}, n_{i}\in\mathbb{Z}$.}
\begin{equation} \label{eq:stufflecoproduct} 
\Delta_{D\ast}(y_{n_{1}} \cdots y_{n_{p}})= \sum_{} y_{n_{1}} \cdots y_{n_{i}}\otimes y_{n_{i+1}}, \ldots, y_{n_{p}}, \quad n_{i}\in \mathbb{Z}^{\ast}. 
\end{equation}
\texttt{Nota Bene}: Remark that here we restricted to Euler sums, $N=2$, but it could be extended for other roots of unity, for which stuffle relation has been stated in $\S 4.1$.\\
The completed dual is the Hopf algebra of series $\mathbb{Q}\left\langle \left\langle Y \right\rangle \right\rangle $ with the coproduct:
$$\Delta_{\ast}(y_{n})= \sum_{ k =0 \atop sgn(n)=\epsilon_{1}\epsilon_{2}}^{\mid n \mid} y_{\epsilon_{1} k} \otimes y_{\epsilon_{2}( n-k)}.$$
Now, let introduce the notations:\footnote{Here $\star$ resp. $\sharp$ refers naturally to the Euler $\star$ resp. $\sharp$, sums, as we see in the next lemma. Beware, it is not a $\ast$ homomorphism.}
$$(y_{n_{1}} \cdots y_{n_{p}})^{\star} \mathrel{\mathop:}=  \sum_{1=i_{0}< i_{1} < \cdots < i_{k-1}\leq  i_{k+1}=p \atop k\geq 0} y_{n_{i_{0}}\mlq + \mrq\cdots \mlq + \mrq n_{i_{1}-1}} \cdots y_{n_{i_{j}}\mlq + \mrq\cdots \mlq + \mrq n_{i_{j+1}-1}} \cdots  y_{n_{i_{k}}\mlq + \mrq \cdots \mlq + \mrq n_{i_{k+1}}}.$$
$$(y_{n_{1}} \cdots y_{n_{p}})^{\sharp} \mathrel{\mathop:}=  \sum_{1=i_{0}< i_{1} < \cdots < i_{k-1}\leq  i_{k+1}=p \atop k\geq 0} 2^{k+1} y_{n_{i_{0}}\mlq + \mrq\cdots \mlq + \mrq n_{i_{1}-1}} \cdots y_{n_{i_{j}} \mlq + \mrq \cdots \mlq + \mrq n_{i_{j+1}-1}} \cdots  y_{n_{i_{k}}\mlq + \mrq \cdots \mlq + \mrq n_{i_{k+1}}},$$
where $n_{i}\in\mathbb{Z}^{\ast}$ and the operation $\mlq + \mrq$ indicates that signs are multiplied whereas absolute values are summed. It is straightforward to check that:
\begin{equation} 
\Delta_{D\ast}(w^{\star})=(\Delta_{D\ast}(w))^{\star} , \quad \text{ and } \quad  \Delta_{D\ast}(w^{\sharp})=(\Delta_{D\ast}(w))^{\sharp}.
\end{equation}
As said above, the relation stuffle is motivic:
\begin{center}
$\zeta^{\mathfrak{m}}(\cdot)$ is a morphism of Hopf algebra from $A_{\ast}$ to $(\mathbb{R},\times)$.
\end{center}

\begin{lemm}[\textbf{Antipode $\ast$}]
In the coalgebra $\mathcal{L}$, with $n_{i}\in\mathbb{Z}^{\ast}$
$$\zeta^{\mathfrak{l}}_{n-1}(n_{1}, \ldots, n_{p}) \equiv (-1)^{p+1}\zeta^{\star,\mathfrak{l}}_{n-1}(n_{p}, \ldots, n_{1}).$$
$$\zeta^{\sharp,\mathfrak{l}}_{n-1}(n_{1}, \ldots, n_{p})\equiv (-1)^{p+1}\zeta^{\sharp,\mathfrak{l}}_{n-1}(n_{p}, \ldots, n_{1}).$$
\end{lemm}
\begin{proof}
By recursion, using the formula $\eqref{eq:Antipode}$, and the following 
identity (left to the reader):
$$\sum_{i=0}^{p-1} (-1)^{i}(y_{n_{i}} \cdots y_{n_{1}})^{\star} \ast (y_{n_{i+1}} \cdots y_{n_{p}})= -(-1)^{p} (y_{n_{p}} \cdots y_{n_{1}})^{\star}, $$
we deduce the antipode $S_{\ast}$:
$$S_{\ast} (y_{n_{1}} \cdots y_{n_{p}})= (-1)^{p} (y_{n_{p}} \cdots y_{n_{1}})^{\star} .$$
Similarly:
$$S_{\ast}((y_{n_{1}} \cdots y_{n_{p}})^{\sharp})=-\sum_{i=0}^{n-1} S_{\ast}((y_{n_{1}} \cdots y_{n_{i}})^{\sharp}) \ast (y_{n_{i+1}} \cdots y_{n_{p}})^{\sharp}$$
$$=-\sum_{i=0}^{n-1} (-1)^{i}(y_{n_{i}} \cdots y_{n_{1}})^{\sharp} \ast (y_{n_{i+1}} \cdots y_{n_{p}})^{\sharp}=(-1)^{p}(y_{n_{p}} \cdots y_{n_{1}})^{\sharp}.$$
Then, we deduce the lemma, since $\zeta^{\mathfrak{m}}(\cdot)$ is a morphism of Hopf algebra. Moreover, the formula $\eqref{eq:Antipode}$ in the coalgebra $\mathcal{L}$ gives that:
$$S(\zeta^{\mathfrak{l}}(\textbf{s}))\equiv -\zeta^{\mathfrak{l}}(\textbf{s}).$$
\end{proof}

\subsection{Hybrid relation in $\mathcal{L}$}

In this part, we look at a new relation called \textit{hybrid relation} between motivic Euler sums in the coalgebra $\mathcal{L}$, i.e. modulo products, which comes from the motivic version of the octagon relation (for $N>1$, cf. $\cite{EF}$) \footnote{\begin{figure}[H]
\centering
\includegraphics[]{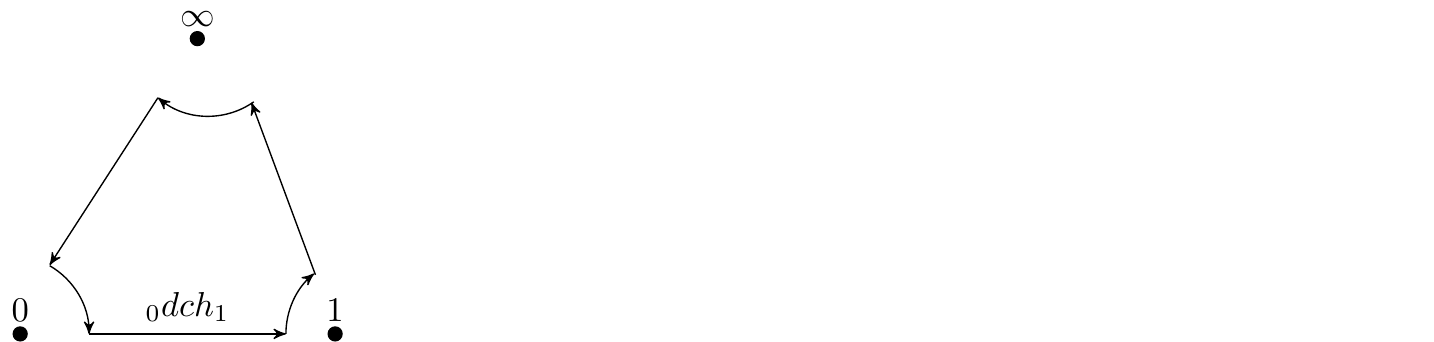}
\caption{For $N=1$, Hexagon relation: $e^{i\pi e_{0}} \Phi(e_{\infty}, e_{0}) e^{i\pi e_{\infty}} \Phi(e_{1},e_{\infty}) e^{i\pi e_{1}}\Phi( e_{0},e_{1})=1.$} \label{fig:hexagon}
\end{figure}}

\begin{figure}[H]
\centering
\includegraphics[]{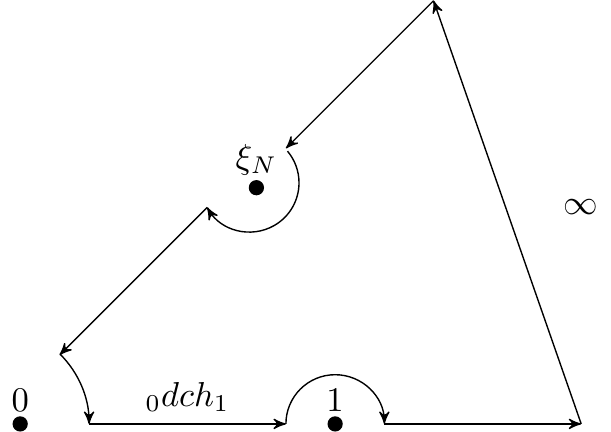}
\caption{Octagon relation, $N>1$:\\
$ \Phi(e_{0}, e_{1}, \ldots, e_{n}) e^{\frac{2 i\pi e_{1}}{N}} \Phi(e_{\infty}, e_{1}, e_{n}, \ldots, e_{2})^{-1} e^{\frac{2 i\pi e_{\infty}}{N}}\Phi(e_{\infty}, e_{n},  \ldots, e_{1}) e^{\frac{2i\pi e_{n}}{N}}\Phi( e_{0},e_{n}, e_{1}, \ldots, e_{n-1})^{-1}e^{\frac{2 i\pi e_{0}}{N}}$\\
$=1$} \label{fig:octagon}
\end{figure}
\noindent
This relation is motivic, and hence valid for the motivic Drinfeld associator $\Phi^{\mathfrak{m}}$ ($\ref{eq:associator}$), replacing $2 i \pi$ by the Lefschetz motivic period $\mathbb{L}^{\mathfrak{m}}$. \\

Let focus on the case $N=2$ and recall that the space of motivic periods of $\mathcal{MT}\left( \mathbb{Z}[\frac{1}{2}]\right)$ decomposes as (cf. $\ref{eq:periodgeomr}$):
\begin{equation}\label{eq:perioddecomp2}
\mathcal{P}_{\mathcal{MT}\left( \mathbb{Z}[\frac{1}{2}]\right)}^{\mathfrak{m}}= \mathcal{H}^{2} \oplus  \mathcal{H}^{2}. \mathbb{L}^{\mathfrak{m}}, \quad \text{ where } \begin{array}{l}
\mathcal{H}^{2} \text{ is } \mathcal{F}_{\infty} \text{ invariant} \\
\mathcal{H}^{2}. \mathbb{L}^{\mathfrak{m}} \text{ is } \mathcal{F}_{\infty} \text{ anti-invariant} 
\end{array}.
\end{equation}
For the motivic Drinfeld associator, seeing the path in the Riemann sphere, it becomes:
\begin{figure}[H]
\centering
\includegraphics[]{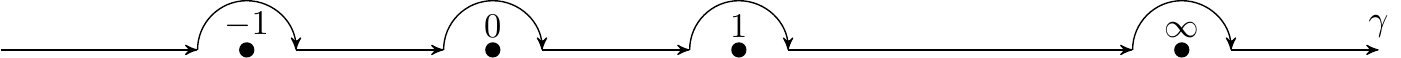}
\caption{Octagon relation, $N=2$ with $e_{0}+e_{1}+e_{-1}+e_{\infty}=0$:\\
$e^{\frac{\mathbb{L}^{\mathfrak{m}} e_{-1}}{2}} \Phi^{\mathfrak{m}}(e_{0}, e_{-1},e_{1})^{-1} e^{ \frac{\mathbb{L}^{\mathfrak{m}}e_{0}}{2}} \Phi^{\mathfrak{m}}(e_{0},e_{1},e_{-1}) e^{\frac{\mathbb{L}^{\mathfrak{m}} e_{1}}{2}}\Phi^{\mathfrak{m}}( e_{\infty},e_{1},e_{-1})^{-1} e^{\frac{\mathbb{L}^{\mathfrak{m}} e_{\infty}}{2}}\Phi^{\mathfrak{m}}( e_{\infty},e_{-1},e_{1})  =1.$} \label{fig:octagon2}
\end{figure}

Let $X= \mathbb{P}^{1}\diagdown \left\lbrace 0, \pm 1, \infty \right\rbrace $. The action of the \textit{real Frobenius} $\boldsymbol{\mathcal{F}_{\infty}}$ on $X(\mathbb{C})$ is induced by complex conjugation. The real Frobenius acts on the Betti realization $\pi^{B}(X (\mathbb{C}))$\footnote{ It is compatible with the groupoid structure of $\pi^{B}$, and the local monodromy. }, and induces an involution on motivic periods, compatible with the Galois action:
$$\mathcal{F}_{\infty}: \mathcal{P}_{\mathcal{MT}(\mathbb{Z}[\frac{1}{2}])}^{\mathfrak{m}} \rightarrow\mathcal{P}_{\mathcal{MT}(\mathbb{Z}[\frac{1}{2}])}^{\mathfrak{m}}.$$
The Lefschetz motivic period $\mathbb{L}^{\mathfrak{m}}$ is anti-invariant by $\mathcal{F}_{\infty}$:
$$\mathcal{F}_{\infty} \mathbb{L}^{\mathfrak{m}}= -\mathbb{L}^{\mathfrak{m}},$$
whereas terms corresponding to real paths in Figure $\ref{fig:octagon2}$, such as Drinfeld associator terms, are obviously invariant by $\mathcal{F}_{\infty}$.\\
\\
The linearized $\mathcal{F}_{\infty}$-anti-invariant part of this octagon relation leads to the following hybrid relation. 

\begin{theo}\label{hybrid}
In the coalgebra $\mathcal{L}^{2}$, with $n_{i}\in \mathbb{Z}^{\ast}$, $w$ the weight:
$$\zeta^{\mathfrak{l}}_{k}\left( n_{0}, n_{1},\ldots, n_{p} \right) + \zeta^{\mathfrak{l}}_{\mid n_{0} \mid +k}\left( n_{1}, \ldots, n_{p} \right) \equiv (-1)^{w+1}\left(  \zeta^{\mathfrak{l}}_{k}\left( n_{p}, \ldots, n_{1}, n_{0} \right) + \zeta^{\mathfrak{l}}_{k+\mid n_{p}\mid}\left( n_{p-1}, \ldots, n_{1},n_{0} \right)\right)$$
Equivalently, in terms of motivic iterated integrals, for $X$ any word in $\lbrace 0, \pm 1 \rbrace$, with $\widetilde{X}$ the reversed word, we obtain both:
$$I^{\mathfrak{l}} (0; 0^{k}, \star, X; 1)\equiv I^{\mathfrak{l}} (0; X, \star, 0^{k}; 1)\equiv (-1)^{w+1} I^{\mathfrak{l}} (0; 0^{k}, \star, \widetilde{X}; 1), $$
$$I^{\mathfrak{l}} (0; 0^{k}, -\star, X; 1)\equiv I^{\mathfrak{l}} (0;- X, -\star, 0^{k}; 1)\equiv (-1)^{w+1} I^{\mathfrak{l}} (0; 0^{k}, -\star, -\widetilde{X}; 1) $$
\end{theo}
The proof is given below, firstly for $k=0$, using octagon relation (Figure $\ref{fig:octagon2}$). The generalization for any $k >0$ is deduced directly from the shuffle regularization $(\ref{eq:shufflereg})$.\\
\\
\textsc{Remarks}:
\begin{itemize}
\item[$\cdot$] This theorem implies notably the famous \textit{depth-drop phenomena} when weight and depth have not the same parity (cf. Corollary $\ref{hybridc}$). 
\item[$\cdot$] Equivalently, this statement is true for $X$ any word in $\lbrace 0, \pm \star \rbrace$. Recall that ($\ref{eq:miistarsharp}$), by linearity:
$$ I^{\mathfrak{m}}(\ldots, \pm \star, \ldots)\mathrel{\mathop:}= I^{\mathfrak{m}}(\ldots, \pm 1, \ldots) - I^{\mathfrak{m}}(\ldots, 0, \ldots).$$
\item[$\cdot$] The point of view adopted by Francis Brown in $\cite{Br3}$, and its use of commutative polynomials (also seen in Ecalle work) can be applied in the coalgebra $\mathcal{L}$ and leads to a new proof of Theorem $\ref{hybrid}$ in the case of MMZV, i.e. $N=1$, sketched in Appendix $A.4$; it uses the stuffle relation and the antipode shuffle. Unfortunately, generalization for motivic Euler sums of this proof is not clear, because of this commutative polynomial setting. 
\end{itemize}
Since Antipode $\ast$ relation expresses $\zeta^{\mathfrak{l}}_{n-1}(n_{1}, \ldots, n_{p})+(-1)^{p} \zeta^{\mathfrak{l}}_{n-1}(n_{p}, \ldots, n_{1})$ in terms of smaller depth (cf. Lemma $4.2.2$), when weight and depth have not the same parity, it turns out that a (motivic) Euler sum can be expressed by smaller depth:\footnote{Erik Panzer recently found a new proof of this depth drop result for MZV at roots of unity, which appear as a special case of some functional equations of polylogarithms in several variables. }
\begin{coro}\label{hybridc}
If $w+p$ odd, a motivic Euler sum in $\mathcal{L}$ is reducible in smaller depth:
$$2\zeta^{\mathfrak{l}}_{n-1}(n_{1}, \ldots, n_{p}) \equiv$$
$$-\zeta^{\mathfrak{l}}_{n+\mid n_{1}\mid -1}(n_{2}, \ldots, n_{p})+(-1)^{p} \zeta^{\mathfrak{l}}_{n+\mid n_{p}\mid -1}(n_{p-1}, \ldots, n_{1})+\sum_{\circ=+ \text{ or } ,\atop \text{at least one } +} (-1)^{p+1} \zeta^{\mathfrak{l}}_{n-1}(n_{p}\circ \cdots \circ n_{1}).$$
\end{coro}

\paragraph{Proof of Theorem $\ref{hybrid}$}
First, the octagon relation (Figure $\ref{fig:octagon2}$) is equivalent to:
\begin{lemm} In  $\mathcal{P}_{\mathcal{MT}\left( \mathbb{Z}[\frac{1}{2}]\right)}^{\mathfrak{m}}\left\langle \left\langle e_{0}, e_{1}, e_{-1}\right\rangle \right\rangle $, with $e_{0} + e_{1}  + e_{-1} +e_{\infty} =0$:
\begin{equation}\label{eq:octagon21}
\Phi^{\mathfrak{m}}(e_{0}, e_{1},e_{-1}) e^{\frac{\mathbb{L}^{\mathfrak{m}} e_{0}}{2}} \Phi^{\mathfrak{m}}(e_{-1}, e_{0},e_{\infty}) e^{\frac{\mathbb{L}^{\mathfrak{m}} e_{-1}}{2}} \Phi^{\mathfrak{m}}(e_{\infty}, e_{-1},e_{1}) e^{\frac{\mathbb{L}^{\mathfrak{m}} e_{\infty}}{2}} \Phi^{\mathfrak{m}}(e_{1}, e_{\infty},e_{0}) e^{\frac{\mathbb{L}^{\mathfrak{m}} e_{1}}{2}}  =1,
\end{equation}
Hence, the linearized octagon relation is:
\begin{multline}\label{eq:octagonlin}
 - e_{0} \Phi^{\mathfrak{l}}(e_{-1}, e_{0},e_{\infty})+  \Phi^{\mathfrak{l}}(e_{-1}, e_{0},e_{\infty})e_{0} +(e_{0}+e_{-1}) \Phi^{\mathfrak{l}}(e_{\infty}, e_{-1},e_{1}) - \Phi^{\mathfrak{l}}(e_{\infty}, e_{-1},e_{1}) (e_{0}+e_{-1})\\
   - e_{1} \Phi^{\mathfrak{l}}(e_{1}, e_{\infty},e_{0}) + \Phi^{\mathfrak{l}}(e_{1}, e_{\infty},e_{0}) e_{1}  \equiv 0.
\end{multline}
\end{lemm}
\begin{proof}
\begin{itemize}
\item[$\cdot$] Let's first remark that:
$$ \Phi^{\mathfrak{m}}(e_{0}, e_{1},e_{-1})= \Phi^{\mathfrak{m}}(e_{1}, e_{0},e_{\infty})^{-1} .$$
Indeed, the coefficient in the series $\Phi^{\mathfrak{m}}(e_{1}, e_{0},e_{\infty})$ of a word $e_{0}^{a_{0}} e_{\eta_{1}} e_{0}^{a_{1}} \cdots e_{\eta_{r}} e_{0}^{a_{r}}$, where $\eta_{i}\in \lbrace\pm 1 \rbrace$ is (cf. $\S 4.6$):
 $$ I^{\mathfrak{m}} \left(0;  (\omega_{1}-\omega_{-1})^{a_{0}} (-\omega_{\mu_{1}})  (\omega_{1}-\omega_{-1})^{a_{1}} \cdots  (-\omega_{\mu_{r}})(\omega_{1}-\omega_{-1})^{a_{r}} ;1 \right) \texttt{ with } \mu_{i}\mathrel{\mathop:}= \left\lbrace \begin{array}{ll}
-\star& \texttt{if } \eta_{i}=1\\
-1 & \texttt{if } \eta_{i}=-1
\end{array}  \right. .$$
Let introduce the following homography $\phi_{\tau\sigma}$ (cf. Annexe $(\ref{homography2})$):
$$\phi_{\tau\sigma}= \phi_{\tau\sigma}^{-1}: t \mapsto \frac{1-t}{1+t} :\left\lbrace  \begin{array}{l}
-\omega_{\star}\mapsto \omega_{\star} \\
-\omega_{1}\mapsto \omega_{-\star}\\
\omega_{-1}-\omega_{1} \mapsto -\omega_{0}\\
\omega_{-1} \mapsto -\omega_{-1}\\
\omega_{-\star}  \mapsto -\omega_{1}
\end{array} \right..$$
If we apply $\phi_{\tau\sigma}$ to the motivic iterated integral above, it gives: $ I^{\mathfrak{m}} \left(1;  \omega_{0}^{a_{0}} \omega_{\eta_{1}} \omega_{0}^{a_{1}} \cdots  \omega_{\eta_{r}}  \omega_{0}^{a_{r}} ;0 \right)$. Hence, summing over words $w$ in $e_{0},e_{1},e_{-1}$:
$$ \Phi^{\mathfrak{m}}(e_{1}, e_{0},e_{\infty})= \sum I^{\mathfrak{m}}(1; w; 0) w$$
Therefore:
$$\Phi^{\mathfrak{m}}(e_{0}, e_{1},e_{-1})\Phi^{\mathfrak{m}}(e_{1}, e_{0},e_{\infty})= \sum_{w, w=uv} I^{\mathfrak{m}}(0; u; 1) I^{\mathfrak{m}}(1; v; 0) w= 1.$$
We used the composition formula for iterated integral to conclude, since for $w$ non empty, $\sum_{w=uv} I^{\mathfrak{m}}(0; u; 1) I^{\mathfrak{m}}(1; v; 0)= I^{\mathfrak{m}}(0; w; 0) =0$.\\
Similarly:
$$\Phi^{\mathfrak{m}}(e_{0}, e_{-1},e_{1})= \Phi^{\mathfrak{m}}(e_{-1}, e_{0},e_{\infty})^{-1} , \quad \text{ and } \quad  \Phi^{\mathfrak{m}}(e_{\infty}, e_{1},e_{-1})= \Phi^{\mathfrak{m}}(e_{1}, e_{\infty},e_{0})^{-1}.$$ 
The identity $\ref{eq:octagon21}$ follows from $\ref{fig:octagon2}$.\\

\item[$\cdot$] Let consider both paths on the Riemann sphere $\gamma$ and $\overline{\gamma}$, its conjugate: \footnote{Path $\gamma$ corresponds to the cycle $\sigma$, $1 \mapsto \infty \mapsto -1 \mapsto 0 \mapsto 1$ (cf. in Annexe $\ref{homography2}$). Beware, in the figure, the position of both path is not completely accurate in order to distinguish them.} \\
\\
\includegraphics[]{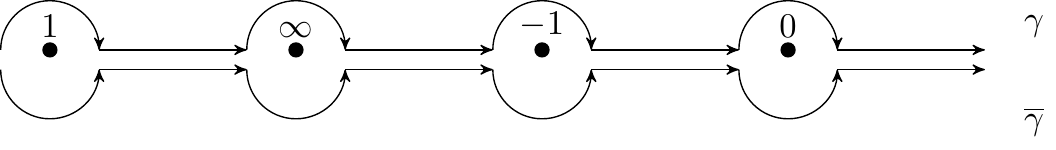}\\
Applying $(id-\mathcal{F}_{\infty})$ to the octagon identity $\ref{eq:octagon21}$ \footnote{The identity $\ref{eq:octagon21}$ corresponds to the path $\gamma$ whereas applying $\mathcal{F}_{\infty}$ to the path $\gamma$ corresponds to the path $\overline{\gamma}$ represented.} leads to:
\begin{small}
\begin{multline}\label{eq:octagon22}
\Phi^{\mathfrak{m}}(e_{0}, e_{1},e_{-1}) e^{\frac{\mathbb{L}^{\mathfrak{m}} e_{0}}{2}} \Phi^{\mathfrak{m}}(e_{-1}, e_{0},e_{\infty}) e^{\frac{\mathbb{L}^{\mathfrak{m}} e_{-1}}{2}} \Phi^{\mathfrak{m}}(e_{\infty}, e_{-1},e_{1}) e^{\frac{\mathbb{L}^{\mathfrak{m}} e_{\infty}}{2}} \Phi^{\mathfrak{m}}(e_{1}, e_{\infty},e_{0}) e^{\frac{\mathbb{L}^{\mathfrak{m}} e_{1}}{2}}\\
-\Phi^{\mathfrak{m}}(e_{0}, e_{1},e_{-1}) e^{-\frac{\mathbb{L}^{\mathfrak{m}} e_{0}}{2}} \Phi^{\mathfrak{m}}(e_{-1}, e_{0},e_{\infty}) e^{-\frac{\mathbb{L}^{\mathfrak{m}} e_{-1}}{2}} \Phi^{\mathfrak{m}}(e_{\infty}, e_{-1},e_{1}) e^{-\frac{\mathbb{L}^{\mathfrak{m}} e_{\infty}}{2}} \Phi^{\mathfrak{m}}(e_{1}, e_{\infty},e_{0}) e^{-\frac{\mathbb{L}^{\mathfrak{m}} e_{1}}{2}}=0.
\end{multline}
\end{small}
By $(\ref{eq:perioddecomp2})$, the left side of $(\ref{eq:octagon22})$, being anti-invariant by $\mathcal{F}_{\infty}$, lies in $ \mathcal{H}^{2}\cdot \mathbb{L}^{\mathfrak{m}} \left\langle \left\langle e_{0}, e_{1}, e_{-1} \right\rangle \right\rangle $. Consequently, we can divide it by $\mathbb{L}^{\mathfrak{m}}$ and consider its projection $\pi^{\mathcal{L}}$ in the coalgebra $\mathcal{L} \left\langle \left\langle e_{0}, e_{1}, e_{-1} \right\rangle \right\rangle $, which gives firstly:
\begin{small}
\begin{multline}\label{eq:octagon23}\hspace*{-0.5cm} 
0=\Phi^{\mathfrak{l}}(e_{0}, e_{1},e_{-1}) \pi^{\mathcal{L}} \left( (\mathbb{L}^{\mathfrak{m}})^{-1} \left[ e^{\frac{\mathbb{L}^{\mathfrak{m}} e_{0}}{2}} e^{\frac{\mathbb{L}^{\mathfrak{m}} e_{-1}}{2}}e^{\frac{\mathbb{L}^{\mathfrak{m}} e_{\infty}}{2}} e^{\frac{\mathbb{L}^{\mathfrak{m}} e_{1}}{2}} - e^{-\frac{\mathbb{L}^{\mathfrak{m}} e_{0}}{2}} e^{-\frac{\mathbb{L}^{\mathfrak{m}} e_{-1}}{2}}e^{-\frac{\mathbb{L}^{\mathfrak{m}} e_{\infty}}{2}} e^{-\frac{\mathbb{L}^{\mathfrak{m}} e_{1}}{2}} \right] \right)  \\
\hspace*{-0.5cm}  +\pi^{\mathcal{L}} \left( (\mathbb{L}^{\mathfrak{m}})^{-1} \left[  e^{\frac{\mathbb{L}^{\mathfrak{m}} e_{0}}{2}}  \Phi^{\mathfrak{l}}(e_{-1}, e_{0},e_{\infty})  e^{\frac{\mathbb{L}^{\mathfrak{m}} e_{-1}}{2}} e^{\frac{\mathbb{L}^{\mathfrak{m}} e_{\infty}}{2}}  e^{\frac{\mathbb{L}^{\mathfrak{m}} e_{1}}{2}}- e^{-\frac{\mathbb{L}^{\mathfrak{m}} e_{0}}{2}}  \Phi^{\mathfrak{l}}(e_{-1}, e_{0},e_{\infty})  e^{-\frac{\mathbb{L}^{\mathfrak{m}} e_{-1}}{2}} e^{-\frac{\mathbb{L}^{\mathfrak{m}} e_{\infty}}{2}}  e^{-\frac{\mathbb{L}^{\mathfrak{m}} e_{1}}{2}} \right]  \right) \\
 \hspace*{-0.5cm} + \pi^{\mathcal{L}} \left( (\mathbb{L}^{\mathfrak{m}})^{-1} \left[  e^{\frac{\mathbb{L}^{\mathfrak{m}} e_{0}}{2}}  e^{\frac{\mathbb{L}^{\mathfrak{m}} e_{-1}}{2}}  \Phi^{\mathfrak{l}}(e_{\infty}, e_{-1},e_{1}) e^{\frac{\mathbb{L}^{\mathfrak{m}} e_{\infty}}{2}}  e^{\frac{\mathbb{L}^{\mathfrak{m}} e_{1}}{2}} -  e^{-\frac{\mathbb{L}^{\mathfrak{m}} e_{0}}{2}}  e^{-\frac{\mathbb{L}^{\mathfrak{m}} e_{-1}}{2}}  \Phi^{\mathfrak{l}}(e_{\infty}, e_{-1},e_{1}) e^{-\frac{\mathbb{L}^{\mathfrak{m}} e_{\infty}}{2}}  e^{-\frac{\mathbb{L}^{\mathfrak{m}} e_{1}}{2}} \right]  \right) \\  
\hspace*{-0.5cm}  +\pi^{\mathcal{L}} \left( (\mathbb{L}^{\mathfrak{m}})^{-1} \left[    e^{\frac{\mathbb{L}^{\mathfrak{m}} e_{0}}{2}}  e^{\frac{\mathbb{L}^{\mathfrak{m}} e_{-1}}{2}} e^{\frac{\mathbb{L}^{\mathfrak{m}} e_{\infty}}{2}} \Phi^{\mathfrak{l}}(e_{1}, e_{\infty},e_{0})  e^{\frac{\mathbb{L}^{\mathfrak{m}} e_{1}}{2}} -  e^{-\frac{\mathbb{L}^{\mathfrak{m}} e_{0}}{2}}  e^{-\frac{\mathbb{L}^{\mathfrak{m}} e_{-1}}{2}} e^{-\frac{\mathbb{L}^{\mathfrak{m}} e_{\infty}}{2}} \Phi^{\mathfrak{l}}(e_{1}, e_{\infty},e_{0})  e^{-\frac{\mathbb{L}^{\mathfrak{m}} e_{1}}{2}}  \right] \right)
\end{multline}
\end{small}
The first line is zero (since $e_{0}+e_{1}+ e_{-1}+e_{\infty}=0$) whereas each other line will contribute by two terms, in order to give $(\ref{eq:octagonlin})$. Indeed, the projection $\pi^{\mathcal{L}}(x)$, when seeing $x$ as a polynomial (with only even powers) in $\mathbb{L}^{\mathfrak{m}}$, only keep the constant term; hence, for each term, only one of the exponentials above $e^{x}$ contributes by its linear term i.e. $x$, while the others contribute simply by $1$. For instance, if we examine carefully the second line of $(\ref{eq:octagon23})$, we get:
$$\begin{array}{ll}
= & e_{0}  \Phi^{\mathfrak{l}}(e_{-1}, e_{0},e_{\infty}) + \Phi^{\mathfrak{l}}(e_{-1}, e_{0},e_{\infty}) (e_{-1}+e_{\infty}+e_{1})\\
&  - (-e_{0})  \Phi^{\mathfrak{l}}(e_{-1}, e_{0},e_{\infty}) -    \Phi^{\mathfrak{l}}(e_{-1}, e_{0},e_{\infty}) ( - e_{-1}- e_{\infty}- e_{1}) \\
= & 2 \left[ e_{0}  \Phi^{\mathfrak{l}}(e_{-1}, e_{0},e_{\infty}) -  \Phi^{\mathfrak{l}}(e_{-1}, e_{0},e_{\infty}) e_{0}\right] 
\end{array}.$$
Similarly, the third line of  $(\ref{eq:octagon23})$ is equal to $(e_{0}+e_{-1}) \Phi^{\mathfrak{l}}(e_{\infty}, e_{-1},e_{1}) - \Phi^{\mathfrak{l}}(e_{\infty}, e_{-1},e_{1}) (e_{0}+e_{-1})$ and the last line is equal to $ -e_{1} \Phi^{\mathfrak{l}}(e_{1}, e_{\infty},e_{0}) + \Phi^{\mathfrak{l}}(e_{1}, e_{\infty},e_{0}) e_{1}$. Therefore, $(\ref{eq:octagon23})$ is equivalent to $(\ref{eq:octagonlin})$, as claimed.
\end{itemize}
\end{proof}

This linearized octagon relation $\ref{eq:octagonlin}$, while looking at the coefficient of a specific word in $\lbrace e_{0},e_{1}, e_{-1}\rbrace$, provides an identity between some $ \zeta^{\star\star,\mathfrak{l}} (\bullet)$ and $\zeta^{\mathfrak{l}} (\bullet)$ in the coalgebra $\mathcal{L}$. The different identities obtained in this way are detailed in the $\S 4.6$. In the following proof of Theorem $\ref{hybrid}$, two of those identities are used.

\begin{proof}[Proof of Theorem $\ref{hybrid}$]
The identity with MMZV$_{\mu_{2}}$ is equivalent to, in terms of motivic iterated integrals:\footnote{Indeed, if $\prod_{i=0}^{p} \epsilon_{i}=1$, it corresponds to the first case, whereas  if $\prod_{i=0}^{p} \epsilon_{i}$, we need the second case.}
$$I^{\mathfrak{l}} (0; 0^{k}, \star, X; 1)\equiv I^{\mathfrak{l}} (0; X, \star, 0^{k}; 1) \text{ and } I^{\mathfrak{l}} (0; 0^{k}, -\star, X; 1)\equiv I^{\mathfrak{l}} (0;- X, -\star, 0^{k}; 1).$$
Furthermore, by shuffle regularization formula ($\ref{eq:shufflereg}$), spreading the first $0$ further inside the iterated integrals, the identity  $I^{\mathfrak{l}} (0;\boldsymbol{0}^{k}, \star, X; 1)\equiv  (-1)^{w+1} I^{\mathfrak{l}} (0;\boldsymbol{0}^{k}, \star, \widetilde{X}; 1)$ boils down to the case $k=0$. \\
The notations are as usual: $\epsilon_{i}=\text{sign} (n_{i})$, $\epsilon_{i}=\eta_{i}\eta_{i+1}$,$\epsilon_{p}= \eta_{p}$, $n_{i}=\epsilon_{i}(a_{i}+1)$.
\begin{itemize}
\item[$(i)$] In $(\ref{eq:octagonlin})$, if we look at the coefficient of a specific word in $\lbrace e_{0},e_{1}, e_{-1}\rbrace$ ending and beginning with $e_{-1}$ (as in $\S 4.6$), only two terms contribute, i.e.:
\begin{equation}\label{eq:octagonlinpart1}
 e_{-1}\Phi^{\mathfrak{l}}(e_{\infty}, e_{-1},e_{1})- \Phi^{\mathfrak{l}}(e_{\infty}, e_{-1},e_{1})e_{-1} 
 \end{equation}
The coefficient of $e_{0}^{a_{0}}e_{\eta_{1}} e_{0}^{a_{1}} \cdots e_{\eta_{p}} e_{0}^{a_{p}}$ in $\Phi^{\mathfrak{m}}(e_{\infty}, e_{-1},e_{1})$ is $(-1)^{n+p}\zeta^{\star\star,\mathfrak{m}}_{n_{0}-1} \left( n_{1}, \cdots, n_{p-1}, -n_{p}\right)$.\footnote{The expressions of those associators are more detailed in the proof of Lemma $\ref{lemmlor}$.} Hence, the coefficient in $(\ref{eq:octagonlinpart1})$ (as in $(\ref{eq:octagonlin})$) of the word $e_{-1} e_{0}^{a_{0}} e_{\eta_{1}} \cdots  e_{\eta_{p}} e_{0}^{a_{p}} e_{-1}$ is:
$$ \zeta^{\star\star, \mathfrak{l}}_{\mid n_{0}\mid -1}(n_{1}, \cdots, - n_{p}, 1) -  \zeta^{\star\star, \mathfrak{l}}(n_{0}, n_{1}, \cdots, n_{p-1}, -n_{p})=0, \quad \text{ with} \prod_{i=0}^{p} \epsilon_{i}=1.$$
In terms of iterated integrals, reversing the first one with Antipode $\shuffle$, it is:
$$ I^{\mathfrak{l}} \left(0;-X , \star ;1 \right)\equiv I^{\mathfrak{l}} \left(0; \star, -X ;1 \right), \text{ with } X\mathrel{\mathop:}=0^{n_{0}-1} \eta_{1} 0^{n_{1}-1} \cdots \eta_{p} 0^{n_{p}-1}.$$
Therefore, since $X$ can be any word in $\lbrace 0, \pm \star \rbrace$, by linearity this is also true for any word X in $\lbrace 0, \pm 1 \rbrace$: $ I^{\mathfrak{l}} \left(0;X, \star ;1 \right)\equiv I^{\mathfrak{l}} \left(0; \star, X ;1 \right)$.
\item[$(ii)$] Now, let look at the coefficient of a specific word in $\lbrace e_{0},e_{1}, e_{-1}\rbrace$ beginning by $e_{1}$, and ending by $e_{-1}$. Only two terms in the left side of $(\ref{eq:octagonlin})$ contribute, i.e.:
\begin{equation}\label{eq:octagonlinpart2}
 -e_{1}\Phi^{\mathfrak{l}}(e_{1}, e_{\infty},e_{0})- \Phi^{\mathfrak{l}}(e_{\infty}, e_{-1},e_{1})e_{-1} 
 \end{equation}
The coefficient in this expression of the word $e_{1} e_{0}^{a_{0}} e_{\eta_{1}} \cdots  e_{\eta_{p}} e_{0}^{a_{p}} e_{-1}$ is:
$$ \zeta^{\star\star, \mathfrak{l}}_{\mid n_{0}\mid -1}(n_{1}, \cdots, n_{p}, -1) -  \zeta^{\star\star, \mathfrak{l}}(n_{0}, n_{1}, \cdots, n_{p})=0, \quad \text{ with} \prod_{i=0}^{p} \epsilon_{i}=-1.$$
In terms of iterated integrals, reversing the first one with Antipode $\shuffle$, it is:
$$ I^{\mathfrak{l}} \left(0; - X , -\star ;1 \right)\equiv I^{\mathfrak{l}} \left(0; -\star, X ;1 \right).$$
Therefore, since $X$ can be any word in $\lbrace 0, \pm \star \rbrace$, by linearity this is also true for any word X in $\lbrace 0, \pm 1 \rbrace$.
\end{itemize}
\end{proof}

\paragraph{For Euler $\boldsymbol{\star\star}$ sums. }

\begin{coro}
In the coalgebra $\mathcal{L}^{2}$, with $n_{i}\in\mathbb{Z}^{\ast}$, $n\geq 1$:
\begin{equation}\label{eq:antipodestaresss}
\zeta^{\star\star,\mathfrak{l}}_{n-1}(n_{1}, \ldots, n_{p})\equiv (-1)^{w+1}\zeta^{\star\star,\mathfrak{l}}_{n-1}(n_{p}, \ldots, n_{1}).
\end{equation}
Motivic Euler $\star\star$ sums of depth $p$ in $\mathcal{L}$ form a dihedral group of order $p+1$:
$$\textsc{(Shift)   } \quad \zeta^{\star\star,\mathfrak{l}}_{\mid n\mid -1}(n_{1}, \ldots, n_{p})\equiv \zeta^{\star\star,\mathfrak{l}}_{\mid n_{1}\mid -1}(n_{2}, \ldots, n_{p},n) \quad \text{ where } sgn(n)\mathrel{\mathop:}= \prod_{i} sgn(n_{i}).$$
\end{coro}
\noindent
Indeed, these two identities lead to a dihedral group structure of order $p+1$:  $(\ref{eq:antipodestaresss})$, respectively $\textsc{Shift}$, correspond to the action of a reflection resp. of a cycle of order $p$ on motivic Euler $\star\star$ sums of depth $p$ in $\mathcal{L}$.
\begin{proof}
Writing $\zeta^{\star\star,\mathfrak{m}}$ as a sum of Euler sums:
\begin{small}
$$\zeta^{\star\star,\mathfrak{m}}_{n-1}(n_{1}, \ldots, n_{p})=\sum_{i=1}^{p} \zeta^{\mathfrak{m}}_{n-1+\mid n_{1}\mid+\cdots+ \mid n_{i-1}\mid}(n_{i} \circ \cdots \circ n_{p})=\sum_{r \atop A_{i}} \left( \zeta^{\mathfrak{m}}_{n-1}(A_{1}, \ldots, A_{r})+ \zeta^{\mathfrak{m}}_{n-1+\mid A_{1}\mid}(A_{2}, \ldots, A_{r})\right),$$
\end{small}
where the last sum is over $(A_{i})_{i}$ such that each $A_{i}$ is a non empty \say{sum} of consecutive $(n_{j})'s$, preserving the order; the absolute value being summed whereas the sign of the $n_{i}$ involved are multiplied; moreover, $\mid A_{1}\mid \geq \mid n_{1}\mid $ resp. $\mid A_{r} \mid \geq \mid n_{p}\mid $.\\
Using Theorem $(\ref{hybrid})$ in the coalgebra $\mathcal{L}$, the previous equality turns into:
$$(-1)^{w+1}\sum_{r \atop A_{i}} \left( \zeta^{\mathfrak{l}}_{n-1}(A_{r}, \ldots, A_{1})+ \zeta^{\mathfrak{l}}_{n-1+\mid A_{r}\mid}(A_{r-1}, \ldots, A_{1})\right) \equiv (-1)^{w+1}\zeta^{\star\star,\mathfrak{m}}_{n-1}(n_{p}, \ldots, n_{1}). $$
The identity $\textsc{Shift}$ is obtained as the composition of Antipode $\shuffle$ $(\ref{eq:antipodeshuffle2})$ and the first identity of the corollary.
\end{proof}

\paragraph{For Euler $\boldsymbol{\sharp\sharp}$ sums.}

\begin{coro}
In the coalgebra $\mathcal{L}$, for $n\in\mathbb{N}$, $n_{i}\in\mathbb{Z}^{\ast}$, $\epsilon_{i}\mathrel{\mathop:}=sign(n_{i})$:\footnote{Here, $\mlq - \mrq$ denotes the operation where absolute values are subtracted whereas sign multiplied.}\\
\begin{tabular}{lll}
\textsc{Reverse} & $\zeta^{\sharp\sharp,\mathfrak{l}}_{n}(n_{1}, \ldots, n_{p})+ (-1)^{w}\zeta^{\sharp\sharp,\mathfrak{l}}_{n}(n_{p}, \ldots, n_{1}) \equiv
\left\{
\begin{array}{l}
0   . \\
\zeta^{\sharp,\mathfrak{l}}_{n}(n_{1}, \ldots, n_{p})
  \end{array}\right.$ & $ \begin{array}{l}
\textrm{       if } w+p \textrm{ even } . \\
 \textrm{    if  } w+p \textrm{ odd } \end{array} .$\\
 &&\\
\textsc{Shift} &  $\zeta^{\sharp\sharp,\mathfrak{l}}_{ n -1}(n_{1}, \ldots, n_{p})\equiv \zeta^{\sharp\sharp,\mathfrak{l}}_{\mid n_{1}\mid-1}(n_{2}, \ldots, n_{p},\epsilon_{1}\cdots \epsilon_{p} \cdot n)$ & for  $w+p$ even.\\
&&\\
\textsc{Cut} & $\zeta^{\sharp\sharp,\mathfrak{l}}_{n}(n_{1},\cdots, n_{p}) \equiv \zeta^{\sharp\sharp,\mathfrak{l}}_{n+\mid n_{p}\mid}(n_{1},\cdots, n_{p-1}),$ & for $w+p$ odd.\\
&&\\
\textsc{Minus} & $\zeta^{\sharp\sharp,\mathfrak{l}}_{n-i}(n_{1},\cdots, n_{p}) \equiv  \zeta^{\sharp\sharp,\mathfrak{l}}_{n}\left( n_{1},\cdots, n_{p-1},  n_{p} \mlq - \mrq i)\right)$, & for $\begin{array}{l}
w+p \text{ odd }\\
i \leq \min(n,\mid n_{p}\mid)
\end{array}$.\\
&&\\
\textsc{Sign} & $\zeta^{\sharp\sharp,\mathfrak{l}}_{n}(n_{1},\cdots,  n_{p-1}, n_{p}) \equiv  \zeta^{\sharp\sharp,\mathfrak{l}}_{n}(n_{1},\cdots,  n_{p-1},-n_{p})$,& for $w+p$ odd.
\end{tabular}
$$\quad\Rightarrow \forall  W \in \lbrace 0, \pm \sharp\rbrace^{\times} \text{with an odd number of } 0, \quad I^{\mathfrak{l}}(-1; W ; 1) \equiv 0 .$$
\end{coro}
\noindent
\textsc{Remark}:
In the coaction of Euler sums, terms with $\overline{1}$ can appear\footnote{More precisely, using the notations of Lemma $\ref{lemmt}$, a $\overline{1}$ can appear in terms of the type $T_{\epsilon, -\epsilon}$ for a cut between $\epsilon$ and $-\epsilon$.}, which are clearly not motivic multiple zeta values. The left side corresponding to such a term in the coaction part $D_{2r+1}(\cdot)$ is $I^{ \mathfrak{l}}(1;X;-1)$, X odd weight with $0, \pm \sharp$. It is worth underlying that, for the $\sharp$ family with $\lbrace \overline{even}, odd\rbrace$, these terms disappear by $\textsc{Sign}$, since by constraint on parity, X will always be of even depth for such a cut. This $\sharp$ family is then more suitable for an unramified criterion, cf. $\S 4.3$.
\begin{proof}
These are consequences of the hybrid relation in Theorem $\ref{hybrid}$.
\begin{itemize}
\item[$\cdot$] \textsc{Reverse:}
Writing $\zeta^{\sharp\sharp,\mathfrak{l}}$ as a sum of Euler sums:
\begin{flushleft}
$\hspace*{-0.5cm}\zeta^{\sharp\sharp,\mathfrak{m}}_{k}(n_{1}, \ldots, n_{p}) +(-1)^{w} \zeta^{\sharp\sharp,\mathfrak{m}}_{k}(n_{p}, \ldots, n_{1})$
\end{flushleft}
\begin{small}
$$\hspace*{-0.5cm}\begin{array}{l}
=\sum_{i=1}^{p} 2^{p-i+1-n_{+}}  \zeta^{\mathfrak{m}}_{k+n_{1}+\cdots+ n_{i-1}}(n_{i} \circ \cdots \circ n_{p})   + (-1)^{w}2^{i-n_{+}} \zeta^{\mathfrak{m}}_{k+n_{p}+\cdots+ n_{i+1}}(n_{i} \circ \cdots \circ n_{1})\\
 \\
=\sum_{r \atop A_{i}}2^{r-1} \left(\right. 2 \zeta^{\mathfrak{m}}_{k}(A_{1}, \ldots, A_{r}) +2 (-1)^{w} \zeta^{\mathfrak{m}}_{k}(A_{r}, \ldots, A_{1}) + \zeta^{\mathfrak{m}}_{k+A_{1}}(A_{2}, \ldots, A_{r}) +(-1)^{w} \zeta^{\mathfrak{m}}_{k+A_{r}}(A_{r-1}, \ldots, A_{1}) \left.  \right) 
\end{array}$$
\end{small}
where the sum is over $(A_{i})$ such that each $A_{i}$ is a non empty \say{sum} of consecutive $(n_{j})'s$, preserving the order; i.e. absolute values of $n_{i}$ are summed whereas signs are multiplied; moreover, $A_{1}$ resp. $A_{r}$ are no less than $n_{1}$ resp. $n_{p}$.\\
By Theorem $\ref{hybrid}$, the previous equality turns into, in $\mathcal{L}$:
$$\sum_{r \atop A_{i}}2^{r-1} \left( \zeta^{\mathfrak{l}}_{k}(A_{1}, \ldots, A_{r})+  (-1)^{w} \zeta^{\mathfrak{l}}_{k}(A_{r}, \ldots, A_{1})\right)$$
$$ \equiv 2^{-1} \left( \zeta^{\sharp,\mathfrak{l}}_{k}(n_{1}, \ldots, n_{p})+  (-1)^{w} \zeta^{\sharp,\mathfrak{l}}_{k}(n_{p}, \ldots, n_{1})\right)\equiv 2^{-1}  \zeta^{\sharp,\mathfrak{l}}_{k}(n_{1}, \ldots, n_{p}) \left( 1+ (-1)^{w+p+1} \right).$$
By the Antipode $\star$ relation applied to $\zeta^{\sharp,\mathfrak{l}}$, it implies the result stated, splitting the cases $w+p$ even and $w+p$ odd.
\item[$\cdot$] \textsc{Shift:} Obtained when combining \textsc{Reverse} and \textsc{Antipode} $\shuffle$, when $w+p$ even.
\item[$\cdot$] \textsc{Cut:} Reverse in the case $w+p$ odd implies:
$$\zeta^{\sharp\sharp,\mathfrak{l}}_{n+\mid n_{1}\mid }(n_{2}, \ldots, n_{p})+ (-1)^{w}\zeta^{\sharp\sharp,\mathfrak{l}}_{n}(n_{p}, \ldots, n_{1})  \equiv 0,$$
Which, reversing the variables, gives the Cut rule.
\item[$\cdot$] \textsc{Minus} follows from \textsc{Cut} since, by \textsc{Cut} both sides are equal to $\zeta^{\sharp\sharp,\mathfrak{l}}_{n-i+ \mid  n_{p}\mid}(n_{1},\cdots, n_{p-1})$.
\item[$\cdot$] In \textsc{Cut}, the sign of $n_{p}$ does not matter, hence, using \textsc{Cut} in both directions, with different signs leads to \textsc{Sign}:
$$\zeta^{\sharp\sharp,\mathfrak{l}}_{n} (n_{1},\ldots,n_{p})\equiv \zeta^{\sharp\sharp,\mathfrak{l}}_{n+ \mid n_p\mid } (n_{1}, \ldots,n_{p-1})\equiv \zeta^{\sharp\sharp,\mathfrak{l}}_{n} ( n_{1},\ldots,-n_{p}).$$
Note that, translating in terms of iterated integrals, it leads to, for $X$ any sequence of $0, \pm \sharp$, with $w+p$ odd:
$$ I^{\mathfrak{l}}(0; X ; 1) \equiv  I^{\mathfrak{l}}(0; -X; 1), $$ 
where $-X$ is obtained from $X$ after exchanging $\sharp$ and $-\sharp$. Moreover, $I^{\mathfrak{l}}(0; -X; 1)\equiv I^{\mathfrak{l}}(0; X; -1) \equiv - I^{\mathfrak{l}}(-1; X; 0)$. Hence, we obtain, using the composition rule of iterated integrals modulo product:
$$I^{\mathfrak{l}}(0; X ; 1) + I^{\mathfrak{l}}(-1; X; 0)\equiv I^{\mathfrak{l}}(-1; X ; 1)  \equiv 0.$$
\end{itemize} 
\end{proof}

\section{Euler $\sharp$ sums}

Let's consider more precisely the following family, appearing in Conjecture $\ref{lzg}$, ith only positive odd and negative even integers for arguments:
$$\zeta^{\sharp,  \mathfrak{m}} \left( \lbrace \overline{\text{even }} ,  \text{odd } \rbrace^{\times} \right) .$$
In the iterated integral, this condition means that we see only the following sequences:
\begin{center}
$\epsilon 0^{2a} \epsilon$, $\quad$ or $\quad\epsilon 0^{2a+1} -\epsilon$, $\quad$  with $\quad\epsilon\in \lbrace \pm\sharp \rbrace$.
\end{center}
\begin{theo}\label{ESsharphonorary}
The motivic Euler sums $\zeta^{\sharp,  \mathfrak{m}} (\lbrace \overline{\text{even }},  \text{odd } \rbrace^{\times} )$ are motivic geometric$^{+}$ periods of $\mathcal{MT}(\mathbb{Z})$.\\
Hence, they are $\mathbb{Q}$ linear combinations of motivic multiple zeta values.
\end{theo}
The proof, in $\S 4.3.2$, relies mainly upon the stability under the coaction of this family.\\
This motivic family is even a generating family of motivic MZV:\nomenclature{$\mathcal{B}^{\sharp}$}{is a family of (unramified) motivic Euler $\sharp$ sums, basis of MMZV}
\begin{theo}\label{ESsharpbasis}
The following family is a basis of $\mathcal{H}^{1}$:
$$\mathcal{B}^{\sharp}\mathrel{\mathop:}= \left\lbrace \zeta^{ \sharp,\mathfrak{m}} (2a_{0}+1,2a_{1}+3,\cdots, 2 a_{p-1}+3, \overline{2a_{p}+2})\text{ , } a_{i}\geq 0\right\rbrace .$$
\end{theo}
First, it is worth noticing that this subfamily is also stable under the coaction.\\
\\
\textsc{Remark}:  It is conjecturally the same family as the Hoffman star family $\zeta^{\star} (\boldsymbol{2}^{a_{0}},3,\cdots, 3, \boldsymbol{2}^{a_{p}})$, by Conjecture $(\ref{lzg})$.\\
\\
For that purpose, we use the increasing \textit{depth filtration} $\mathcal{F}^{\mathfrak{D}}$ on $\mathcal{H}^{2}$ such that (cf. $\S 2.4.3$):
\begin{center}
  $\mathcal{F}_{p}^{\mathfrak{D}} \mathcal{H}^{2}$ is generated by Euler sums of depth smaller than $p$.
 \end{center} 
 Note that it is not a grading, but we define the associated graded as the quotient $gr_{p}^{\mathfrak{D}}\mathrel{\mathop:}=\mathcal{F}_{p}^{\mathfrak{D}} \diagup \mathcal{F}_{p-1}^{\mathfrak{D}}$. The vector space $\mathcal{F}_{p}^{\mathfrak{D}}\mathcal{H}$ is stable under the action of $\mathcal{G}$. The linear independence of this $\sharp$ family is proved below thanks to a recursion on the depth and on the weight, using the injectivity of a map $\partial$ where $\partial$ came out of the depth and weight-graded part of the coaction $\Delta$. 

\subsection{Depth graded Coaction}

In Chapter $2$, we defined the depth graded derivations $D_{r,p}$ (cf. $\ref{Drp}$), and $D^{-1}_{r,p}$ ($\ref{eq:derivnp}$) after the projection on the right side, using depth $1$ results: 
$$gr^{\mathcal{D}}_{1} \mathcal{L}_{2r+1}=\mathbb{Q}\zeta^{\mathfrak{l}}(2r+1).$$
Let look at the following maps, whose injectivity is fundamental to the Theorem $\ref{ESsharpbasis}$:
$$ D^{-1}_{2r+1,p} : gr^{\mathfrak{D}}_{p}\mathcal{H}_{n}\rightarrow gr^{\mathfrak{D}}_{p-1}\mathcal{H}_{n-2r-1} .$$
$$\partial_{<n,p} \mathrel{\mathop:}=\oplus_{2r+1<n} D^{-1}_{2r+1,p} .$$
Their explicit expression is:

\begin{lemm}\footnotemark[2]\footnotetext[2]{To be accurate, the term $i=0$ in the first sum has to be understood as:
$$ \frac{2^{2r+1}}{1-2^{2r}}\binom{2r}{2a_{1}+2} \zeta^{\sharp,\mathfrak{m}} (2\alpha+3,  2 a_{2}+3,\cdots, \overline{2a_{p}+2}) . $$
Meanwhile the terms $i=1$, resp. $i=p$ in the second sum have to be understood as:
$$ \frac{2^{2r+1}}{1-2^{2r}}\binom{2r}{2a_{0}+2} \zeta^{\sharp,\mathfrak{m}} (2\alpha+3,  2 a_{2}+3,\cdots, \overline{2a_{p}+2}) \quad \text{ resp. } \quad \frac{2^{2r+1}}{1-2^{2r}}\binom{2r}{2a_{p-1}+2} \zeta^{\sharp,\mathfrak{m}} (\cdots, 2 a_{p-2}+3, \overline{2\alpha+2}).$$}
\begin{multline} \label{eq:dgrderiv} 
D^{-1}_{2r+1,p} \left(  \zeta^{\sharp,\mathfrak{m}} (2a_{0}+1,2a_{1}+3,\cdots, 2 a_{p-1}+3, \overline{2a_{p}+2}) \right) = \\
\delta_{r=a_{0}} \frac{2^{2r+1}}{1-2^{2r}}\binom{2r}{2r+2} \zeta^{\sharp,\mathfrak{m}} (2 a_{1}+3,\cdots, \overline{2a_{p}+2})\\
+ \sum_{0 \leq i \leq p-2, \quad  \alpha \leq a_{i}\atop r=a_{i+1}+a_{i}+1-\alpha} \frac{2^{2r+1}}{1-2^{2r}}\binom{2r}{2a_{i+1}+2} \zeta^{\sharp,\mathfrak{m}} (\cdots, 2 a_{i-1}+3,2\alpha+3,  2 a_{i+2}+3,\cdots, \overline{2a_{p}+2})\\
+ \sum_{1 \leq i \leq p-1, \quad \alpha \leq a_{i} \atop r=a_{i-1}+a_{i}+1-\alpha} \frac{2^{2r+1}}{1-2^{2r}}\binom{2r}{2a_{i-1}+2} \zeta^{\sharp,\mathfrak{m}} (\cdots, 2 a_{i-2}+3,2\alpha+3,  2 a_{i+1}+3,\cdots, \overline{2a_{p}+2})\\
+ \textsc{(Deconcatenation)} \sum_{\alpha \leq a_{p} \atop r=a_{p-1}+a_{p}+1-\alpha} 2 \binom{2r}{2a_{p}+1}\zeta^{\sharp,\mathfrak{m}} (\cdots, 2 a_{p-1}+3,\overline{2\alpha+2}).
\end{multline}
\end{lemm}
\begin{proof}
Looking at the Annexe $A.1$ expression for $D_{2r+1}$, we obtain for $D_{2r+1,p}$ keeping only the cuts of depth one (removing exactly one non zero element):
\begin{multline} \nonumber
D_{2r+1,p} \zeta^{\sharp,\mathfrak{m}} (2a_{0}+1,2a_{1}+3,\cdots, 2 a_{p-1}+3, \overline{2a_{p}+2})=\\
\sum_{i, \alpha \leq a_{i}\atop r=a_{i+1}+a_{i}+1-\alpha} 2 \zeta^{\mathfrak{l}} _{2a_{i}-2\alpha}(2a_{i+1}+3)  \otimes \zeta^{\sharp,\mathfrak{m}} (\cdots, 2 a_{i-1}+3,2\alpha+3,  2 a_{i+2}+3,\cdots, \overline{2a_{p}+2})\\
+\sum_{i, \alpha \leq a_{i} \atop r=a_{i-1}+a_{i}+1-\alpha} 2 \zeta^{\mathfrak{l}} _{2a_{i}-2\alpha}(2a_{i-1}+3)  \otimes \zeta^{\sharp,\mathfrak{m}} (\cdots, 2 a_{i-2}+3,2\alpha+3,  2 a_{i+1}+3,\cdots, \overline{2a_{p}+2})\\
+\sum_{\alpha \leq a_{p} \atop r=a_{p-1}+a_{p}+1-\alpha} 2 \zeta^{\mathfrak{l}} _{2a_{p-1}-2\alpha+1}(\overline{2a_{p}+2})  \otimes \zeta^{\sharp,\mathfrak{m}} (\cdots, 2 a_{p-1}+3,\overline{2\alpha+2}).
\end{multline}
To lighten the result, some cases at the borders ($i=0$, or $i=p$) have been included in the sum, being fundamentally similar (despite some index problems). These are clarified in the previous footnote\footnotemark[2].\\
In particular, with notations of the Lemma $\ref{lemmt}$, $T_{0,0}$ terms can be neglected as they decrease the depth by at least $2$; same for the $T_{0,\epsilon}$ and  $T_{\epsilon,0}$ for cuts between $\epsilon$ and $\pm \epsilon$. To obtain the lemma, it remains to check the coefficient of $\zeta^{\mathfrak{l}}(\overline{2r+1})$ for each term in the left side thanks to the known identities:
$$  \zeta^{\mathfrak{l}}(2r+1)= \frac{-2^{2r}}{2^{2r}-1} \zeta^{\mathfrak{l}}(\overline{2r+1})\quad  \text{  and  } \quad \zeta^{\mathfrak{l}}_{2r+1-a}(a)=(-1)^{a+1}\binom{2r}{a-1} \zeta^{\mathfrak{l}}(2r+1).$$
\end{proof}

\subsection{Proofs of Theorem $4.3.1$ and $4.3.2$}

\begin{proof}[\textbf{Proof of Theorem $4.3.1$}]
By Corollary $5.1.2$, we can prove it in two steps:
\begin{itemize}
\item[$\cdot$] First, checking that $D_{1}(\cdot)=0$ for this family, which is rather obvious by Lemma $\ref{condd1}$ since there is no sequence of the type $\lbrace 0, \epsilon, -\epsilon \rbrace$ or $\lbrace \epsilon, -\epsilon, 0 \rbrace$ in the iterated integral.
\item[$\cdot$] Secondly, we can use a recursion on weight to prove that $D_{2r+1}(\cdot)$, for $r> 0$, are unramified. Consequently, using recursion, this follows from the following statement:
\begin{center}
The family $\zeta^{\sharp,  \mathfrak{m}} \left( \lbrace \overline{\text{even }},  + \text{odd } \rbrace^{\times} \right) $ is stable under $D_{2r+1}$.
\end{center}
This is proved in Lemma $A.1.3$, using the relations of $\S 4.2$ in order to simplify the \textit{unstable cuts}, i.e. the cuts where a sequence of type $\epsilon, 0^{2a+1}, \epsilon$ or $\epsilon, 0^{2a}, -\epsilon$ appears; indeed, these cuts give rise to a $\text{even}$ or to a $\overline{ \text{ odd}}$ in the $\sharp$ Euler sum.
\end{itemize}
One fundamental observation about this family, used in Lemma $A.1.3$ is: for a subsequence of odd length from the iterated integral, because of these patterns of  $\epsilon, \boldsymbol{0}^{2a}, \epsilon$, or $\epsilon, \boldsymbol{0}^{2a+1}, -\epsilon$, we can put in relation the depth $p$, the weight $w$ and $s$ the number of sign changes among the $\pm\sharp$:
$$w\equiv p-s  \pmod{2}.$$
It means that if we have a cut $\epsilon_{0},\cdots \epsilon_{p+1}$ of odd weight, then:
\begin{center}
\textsc{Either:} Depth $p$ is odd, $s$ even, $\epsilon_{0}=\epsilon_{p+1}$,    \textsc{Or:} Depth $p$ is even, $s$ odd, $\epsilon_{0}=-\epsilon_{p+1}$.
\end{center}
\end{proof}

\begin{proof}[\textbf{Proof of Theorem $4.3.2$}]
By a cardinality argument, it is sufficient to prove the linear independence of the family, which is based on the injectivity of $\partial_{<n,p}$. Let us define: \footnote{Sub-$\mathbb{Q}$ vector space of $\mathcal{H}^{1}$ by previous Theorem.}\nomenclature{$\mathcal{H}^{odd\sharp}$}{$\mathbb{Q}$-vector space generated by $\zeta^{ \sharp,\mathfrak{m}} (2a_{0}+1,2a_{1}+3,\cdots, 2 a_{p-1}+3, \overline{2a_{p}+2})$}
\begin{center}
$\mathcal{H}^{odd\sharp}$:  $\mathbb{Q}$-vector space generated by $\zeta^{ \sharp,\mathfrak{m}} (2a_{0}+1,2a_{1}+3,\cdots, 2 a_{p-1}+3, \overline{2a_{p}+2})$.
\end{center}
The first thing to remark is that $\mathcal{H}^{odd\sharp}$ is stable under these derivations, by the expression obtained in Lemma $A.1.4$.:
$$D_{2r+1} (\mathcal{H}_{n}^{odd\sharp}) \subset  \mathcal{L}_{2r+1} \otimes \mathcal{H}_{n-2r-1}^{odd\sharp},$$
Now, let consider the restriction on $\mathcal{H}^{odd\sharp}$ of $\partial_{<n,p}$ and prove:
$$\partial_{<n,p}: gr^{\mathfrak{D}}_{p} \mathcal{H}_{n}^{odd\sharp} \rightarrow \oplus_{2r+1<n}  gr^{\mathfrak{D}}_{p-1}\mathcal{H}_{n-2r-1}^{odd\sharp} \text{  is bijective. }$$
The formula $\eqref{eq:dgrderiv}$ gives the explicit expression of this map. Let us prove more precisely:
\begin{center}
$M^{\mathfrak{D}}_{n,p}$ the matrix of $\partial_{<n,p}$ on $\left\lbrace \zeta^{ \sharp ,\mathfrak{m}} (2a_{0}+1,2a_{1}+3,\cdots, 2 a_{p-1}+3, \overline{2a_{p}+2})\right\rbrace $ in terms of $\left\lbrace \zeta^{ \sharp ,\mathfrak{m}} (2b_{0}+1,2b_{1}+3,\cdots, 2 b_{p-2}+3, \overline{2b_{p-1}+2})\right\rbrace $ is invertible.
\end{center}
\texttt{Nota Bene}: The matrix $M^{\mathfrak{D}}_{n,p}$ is well (uniquely) defined provided that the $\zeta^{ \sharp ,\mathfrak{m}}$ of the second line are linearly independent. So first, we have to consider the formal matrix associated $\mathbb{M}^{\mathfrak{D}}_{n,p}$ defined explicitly (combinatorially) by the formula for the derivations given, and prove $\mathbb{M}^{\mathfrak{D}}_{n,p}$ is invertible. Afterwards, we could state that $M^{\mathfrak{D}}_{n,p}$ is well defined and invertible too since equal to $\mathbb{M}^{\mathfrak{D}}_{n,p}$.
\begin{proof}
The invertibility comes from the fact that the (strictly) smallest terms $2$-adically in $\eqref{eq:dgrderiv}$ are the deconcatenation ones, which is an injective operation. More precisely, let $\widetilde{M}^{\mathfrak{D}}_{n,p}$ be the matrix $\mathbb{M}_{n,p}$ where we have multiplied each line corresponding to $D_{2r+1}$ by ($2^{-2r}$). Then, order elements on both sides by lexicographical order on  ($a_{p}, \ldots, a_{0}$), resp. ($r,b_{p-1}, \ldots, b_{0}$), such that the diagonal corresponds to $r=a_{p}+1$ and $b_{i}=a_{i}$ for $i<p$. The $2$ -adic valuation of all the terms in $(\ref{eq:dgrderiv})$ (once divided by $2^{2r}$) is at least $1$, except for the deconcatenation terms since:
$$v_{2}\left( 2^{-2r+1} \binom{2r}{2a_{p}+1} \right)  \leq 0 \Longleftrightarrow  v_{2}\left( \binom{2r}{2a_{p}+1} \right)  \leq 2r-1.$$
Then, modulo $2$, only the deconcatenation terms remain, so the matrix $\widetilde{M}^{\mathfrak{D}}_{n,p}$ is triangular with $1$ on the diagonal. This implies that $\det (\widetilde{M}^{\mathfrak{D}}_{n,p})\equiv 1   \pmod{2}$, and in particular is non zero: the matrix $\widetilde{M}^{\mathfrak{D}}_{n,p}$ is invertible, and so does $\mathbb{M}^{\mathfrak{D}}_{n,p}$.
\end{proof}
This allows us to complete the proof since it implies:
\begin{center}
The elements of $\mathcal{B}^{\sharp}$ are linearly independent.
\end{center}
\begin{proof}
First, let prove the linear independence of this family of the same depth and weight, by recursion on $p$. For depth $0$, this is obvious since $\zeta^{\mathfrak{m}}(\overline{2n})$ is a rational multiple of $\pi^{2n}$.\\
Assuming by recursion on the depth that the elements of weight $n$ and depth $p-1$ are linearly independent, since $M^{\mathfrak{D}}_{n,p}$ is invertible, this means both that the  $\zeta^{ \sharp,\mathfrak{m}} (2a_{0}+1,2a_{1}+3,\cdots, 2 a_{p-1}+3, \overline{2a_{p}+2})$ of weight $n$ are linearly independent and that $\partial_{<n,p}$ is bijective, as announced before.\\
The last step is just to realize that the bijectivity of $\partial_{<n,l}$ also implies that elements of different depths are also linearly independent. The proof could be done by contradiction: by applying $\partial_{<n,p}$ on a linear combination where $p$ is the maximal depth appearing, we arrive at an equality between same level elements.
\end{proof}
\end{proof}

\section{Hoffman $\star$}

\begin{theo}\label{Hoffstar}
If the analytic conjecture  ($\ref{conjcoeff}$) holds, then the motivic \textit{Hoffman} $\star$ family $\lbrace \zeta^{\star,\mathfrak{m}} (\lbrace 2,3 \rbrace^{\times})\rbrace$ is a basis of $\mathcal{H}^{1}$, the space of MMZV.
\end{theo}

For that purpose, we define an increasing filtration $\mathcal{F}^{L}_{\bullet}$ on $\mathcal{H}^{2,3}$, called \textbf{level}, such that:
\begin{equation}\label{eq:levelf}
 \mathcal{F}^{L}_{l}\mathcal{H}^{2,3} \text{ is spanned by } \zeta^{\star,\mathfrak{m}} (\boldsymbol{2}^{a_{0}},3,\cdots,3, \boldsymbol{2}^{a_{p}}) \text{, with less than 'l' } 3. 
\end{equation}
It corresponds to the motivic depth for this family, as we see through the proof below and the coaction calculus.\\
\paragraph{Sketch. } The vector space $\mathcal{F}^{L}_{l}\mathcal{H}^{2,3}$ is stable under the action of $\mathcal{G}$ ($\ref{eq:levelfiltstrable}$). The linear independence of the Hoffman $\star$ family is proved below ($ § 4.4.2$) thanks to a recursion on the level and on the weight, using the injectivity of a map $\partial^{L}$ where $\partial^{L}$ came out of the level and weight-graded part of the coaction $\Delta$ (cf. $4.4.2$). The injectivity is proved via $2$-adic properties of some coefficients conjectured in $\ref{conjcoeff}$.\\
Indeed, when computing the level graded coaction (cf. Lemma $4.4.2$) on the Hoffman $\star$ elements, looking at the left side, some elements appear, such as $\zeta^{\star\star,\mathfrak{m}}(\boldsymbol{2}^{a},3,\boldsymbol{2}^{b})$ but also $\zeta^{\star\star,\mathfrak{m}}(\boldsymbol{2}^{a},3,\boldsymbol{2}^{b})$. These are not always of depth $1$ as we could expect,\footnote{As for the Hoffman non $\star$ case done by Francis Brown, using a result of Don Zagier for level $1$.} but at least are abelians: product of motivic simple zeta values, as proved in Lemma $\ref{lemmcoeff}$.\\
To prove the linear independence of Hoffman $\star$ elements, we will then need to know some coefficients appearing in Lemma $\ref{lemmcoeff}$ (or at least the 2-adic valuation) of $\zeta(weight)$ for each of these terms, conjectured in $\ref{conjcoeff}$, which is the only missing part of the proof, and can be solved at the analytic level.\\

\subsection{Level graded coaction}

Let use the following form for a MMZV$^{\star}$, gathering the $2$:
$$\zeta^{\star, \mathfrak{m}} (\boldsymbol{2}^{a_{0}},c_{1},\cdots,c_{p}, \boldsymbol{2}^{a_{p}}), \quad c_{i}\in\mathbb{N}^{\ast}, c_{i}\neq 2.$$
This writing is suitable for the Galois action (and coaction) calculus, since by the antipode relations ($\S 4.2$), many of the cuts from a $2$ to a $2$ get simplified (cf.  Annexe $\S A.1$).\\
For the Hoffman family, with only $2$ and $3$, the expression obtained is:\footnote{Cf. Lemma $A.1.2$; where $\delta_{2r+1}$ means here that the left side has to be of weigh $2r+1$.}\\
\begin{flushleft}
\hspace*{-0.7cm}$D_{2r+1}   \zeta^{\star, \mathfrak{m}} (\boldsymbol{2}^{a_{0}},3,\cdots,3, \boldsymbol{2}^{a_{p}})$
\end{flushleft}
\begin{multline} \label{eq:dr3}
\hspace*{-1.3cm}= \delta_{2r+1}\sum_{i<j} \left[  
\begin{array}{lll}
 + \quad \zeta^{\star\star, \mathfrak{l}}_{1} (\boldsymbol{2}^{a_{i+1}},3,\cdots,3, \boldsymbol{2}^{\leq a_{j}}) & \otimes & \zeta^{\star, \mathfrak{m}} (\cdots,3, \boldsymbol{2}^{1+a_{i}+ \leq a_{j}},3, \cdots)\\
 - \quad \zeta^{\star\star, \mathfrak{l}}_{1} (\boldsymbol{2}^{\leq a_{i}},3,\cdots,3, \boldsymbol{2}^{ a_{j-1}}) & \otimes & \zeta^{\star, \mathfrak{m}} (\cdots,3, \boldsymbol{2}^{1+a_{j}+ \leq a_{i}},3, \cdots)\\
 + \left(  \zeta^{\star\star, \mathfrak{l}}_{2} (\boldsymbol{2}^{a_{i+1}},3,\cdots, \boldsymbol{2}^{a_{j}},3) +  \zeta^{\star\star, \mathfrak{l}}_{1} (\boldsymbol{2}^{<a_{i}},3,\cdots, \boldsymbol{2}^{a_{j}},3) \right) & \otimes&  \zeta^{\star, \mathfrak{m}} (\cdots,3, \boldsymbol{2}^{<a_{i}},3,\boldsymbol{2}^{a_{j+1}},3, \cdots)\\
  -  \left(\zeta^{\star\star, \mathfrak{l}}_{2} (\boldsymbol{2}^{a_{j+1}},3,\cdots,3) + \zeta^{\star\star, \mathfrak{l}}_{1}(\boldsymbol{2}^{<a_{j}},3,\cdots,3) \right)& \otimes &  \zeta^{\star, \mathfrak{m}} (\cdots,3, \boldsymbol{2}^{a_{i-1}},3,\boldsymbol{2}^{< a_{j}},3, \cdots) \\
\end{array}  \right] \\
\quad \quad \begin{array}{lll}
\quad \quad+ \quad\delta_{2r+1} \quad \left( \zeta^{\star, \mathfrak{l}} (\boldsymbol{2}^{a_{0}},3,\cdots,3, \boldsymbol{2}^{\leq a_{i}})- \zeta^{\star\star, \mathfrak{l}} (\boldsymbol{2}^{\leq a_{i}},3,\cdots,3, \boldsymbol{2}^{a_{0}}) \right) & \otimes & \zeta^{\star, \mathfrak{m}} (\boldsymbol{2}^{\leq a_{i}},3, \cdots)\\
 \quad\quad +\quad \delta_{2r+1} \quad\zeta^{\star\star, \mathfrak{l}} (\boldsymbol{2}^{\leq a_{j}},3,\cdots,3, \boldsymbol{2}^{ a_{p}}) & \otimes & \zeta^{\star, \mathfrak{m}} (\cdots,3, \boldsymbol{2}^{\leq a_{j}}).
\end{array}
\end{multline}
In particular, the coaction on the Hoffman $\star$ elements is stable. \\
By the previous expression $(\ref{eq:dr3})$, we see that each cut (of odd length) removes at least one $3$. It means that the level filtration is stable under the action of $\mathcal{G}$ and:
\begin{equation} \label{eq:levelfiltstrable}
D_{2r+1}(\mathcal{F}^{L}_{l}\mathcal{H}^{2,3}) \subset  \mathcal{L}_{2r+1} \otimes \mathcal{F}^{L}_{l-1}\mathcal{H}_{n-2r-1}^{2,3} .
\end{equation}
Then, let consider the level graded derivation:
\begin{equation}
gr^{L}_{l} D_{2r+1}: gr^{L}_{l}\mathcal{H}_{n}^{2,3} \rightarrow  \mathcal{L}_{2r+1} \otimes gr^{L}_{l-1}\mathcal{H}_{n-2r-1}^{2,3}.
\end{equation}
If we restrict ourselves to the cuts in the coaction that remove exactly one $3$ in the right side, the formula $(\ref{eq:dr3})$ leads to:
\begin{flushleft}
\hspace*{-0.5cm}$gr^{L}_{l} D_{2r+1}   \zeta^{\star, \mathfrak{m}} (\boldsymbol{2}^{a_{0}},3,\cdots,3, \boldsymbol{2}^{a_{p}}) =$
\end{flushleft}
\begin{multline}\label{eq:gdr3}
\hspace*{-1.5cm}\begin{array}{lll}
\quad - \delta_{a_{0} < r \leq a_{0}+a_{1}+2} \quad \zeta^{\star\star, \mathfrak{l}}_{2} (\boldsymbol{2}^{a_{0}}, 3, \boldsymbol{2}^{r-a_{0}-2})  &\otimes & \zeta^{\star, \mathfrak{m}} (\boldsymbol{2}^{ a_{0}+a_{1}+1-r},3, \cdots)
\end{array}\\
\hspace*{-1.3cm}\sum_{i<j} \left[ \begin{array}{l}
\delta_{r\leq a_{i}} \quad \zeta^{\star\star, \mathfrak{l}}_{1} (\boldsymbol{2}^{r}) \quad \quad \quad \quad \quad  \otimes  \left(  \zeta^{\star, \mathfrak{m}} (\cdots,3, \boldsymbol{2}^{a_{i-1}+ a_{i}-r+1},3, \cdots)  - \zeta^{\star, \mathfrak{m}} (\cdots,3, \boldsymbol{2}^{a_{i+1}+  a_{i}-r+1},3, \cdots) \right) \\
 + \left( \delta_{r=a_{i}+2} \zeta^{\star\star, \mathfrak{l}}_{2} (\boldsymbol{2}^{a_{i}},3) +  \delta_{r< a_{i}+a_{i-1}+3} \zeta^{\star\star, \mathfrak{l}}_{1} (\boldsymbol{2}^{r-a_{i}-3}, 3, \boldsymbol{2}^{a_{i}},3) \right)  \otimes  \zeta^{\star, \mathfrak{m}} (\cdots,3, \boldsymbol{2}^{a_{i}+a_{i-1}-r+1},3,\boldsymbol{2}^{a_{i+1}},3, \cdots)\\
  - \left( \delta_{r=a_{i}+2} \zeta^{\star\star, \mathfrak{l}}_{2} (\boldsymbol{2}^{a_{i}},3) + \delta_{r< a_{i}+a_{i+1}+3} \zeta^{\star\star, \mathfrak{l}}_{1}(\boldsymbol{2}^{r-a_{i}-3},3, \boldsymbol{2}^{a_{i}}, 3) \right) \otimes  \zeta^{\star, \mathfrak{m}} (\cdots,3, \boldsymbol{2}^{a_{i-1}},3,\boldsymbol{2}^{a_{i}+a_{i+1}-r+1},3, \cdots) 
\end{array} \right] \\
\hspace*{-2cm} \textsc{(D)}  \begin{array}{lll}
 +\delta_{a_{p}+1 \leq r \leq a_{p}+a_{p-1}+1}  \quad  \zeta^{\star\star, \mathfrak{l}} (\boldsymbol{2}^{r- a_{p}-1},3, \boldsymbol{2}^{ a_{p}}) &\otimes & \zeta^{\star, \mathfrak{m}} (\cdots,3, \boldsymbol{2}^{a_{p}+ a_{p-1}-r+1}). 
 \end{array}
 \end{multline}
By the antipode $\shuffle$ relation (cf. $\ref{eq:antipodeshuffle2}$):
$$\zeta^{\star\star, \mathfrak{l}}_{1} (\boldsymbol{2}^{a},3, \boldsymbol{2}^{b},3)= \zeta^{\star\star, \mathfrak{l}}_{2} (\boldsymbol{2}^{b},3, \boldsymbol{2}^{a+1})=\zeta^{\star\star, \mathfrak{l}}(\boldsymbol{2}^{b+1},3, \boldsymbol{2}^{a+1})- \zeta^{\star, \mathfrak{l}}(\boldsymbol{2}^{b+1},3, \boldsymbol{2}^{a+1}).$$
Then, by Lemma $\ref{lemmcoeff}$, all the terms appearing in the left side of $gr^{L}_{l} D_{2r+1}$ are product of simple MZV, which turns into, in the coalgebra $\mathcal{L}$ a rational multiple of $\zeta^{\mathfrak{l}}(2r+1)$:
$$gr^{L}_{l} D_{2r+1} (gr^{L}_{l}\mathcal{H}_{n}^{2,3}) \subset \mathbb{Q}\zeta^{\mathfrak{l}}(2r+1)\otimes gr^{L}_{l-1}\mathcal{H}_{n-2r-1}^{2,3}.$$
\\
Sending $\zeta^{\mathfrak{l}}(2r+1)$ to $1$ with the projection $\pi:\mathbb{Q} \zeta^{\mathfrak{l}}(2r+1)\rightarrow\mathbb{Q}$, we can then consider:\nomenclature{$\partial^{L}_{r,l}$ and $\partial^{L}_{<n,l}$}{defined as composition from derivations} 
\begin{description}
\item[$\boldsymbol{\cdot\quad \partial^{L}_{r,l}}$] $ : gr^{L}_{l}\mathcal{H}_{n}^{2,3}\rightarrow gr^{L}_{l-1}\mathcal{H}_{n-2r-1}^{2,3},  \quad \text{ defined as the composition }$
$$\partial^{L}_{r,l}\mathrel{\mathop:}=gr_{l}^{L}\partial_{2r+1}\mathrel{\mathop:}=m\circ(\pi\otimes id)(gr^{L}_{l} D_{r}): \quad gr^{L}_{l}\mathcal{H}_{n}^{2,3} \rightarrow \mathbb{Q}\otimes_{\mathbb{Q}} gr^{L}_{l-1}\mathcal{H}_{n-2r-1}^{2,3}  \rightarrow gr^{L}_{l-1}\mathcal{H}_{n-2r-1}^{2,3} .$$
\item[$\boldsymbol{\cdot\quad \partial^{L}_{<n,l}}$] $\mathrel{\mathop:}=\oplus_{2r+1<n}\partial^{L}_{r,l} .$
\\
\end{description}
The injectivity of this map is the keystone of the Hoffman$^{\star}$ proof. Its explicit expression is:
\begin{lemm} 
\begin{flushleft}
$\partial^{L}_{r,l} (\zeta^{\star, \mathfrak{m}} (\boldsymbol{2}^{a_{0}},3,\cdots,3, \boldsymbol{2}^{a_{p}}))=$
\end{flushleft}
$$\begin{array}{l}
 \quad  - \delta_{a_{0} < r \leq a_{0}+a_{1}+2} \widetilde{B}^{a_{0}+1,r-a_{0}-2}  \zeta^{\star, \mathfrak{m}} (\boldsymbol{2}^{ a_{0}+a_{1}+1-r},3, \cdots)  \\
   \\
+ \sum_{i<j} \left[ \begin{array}{l}
 \delta_{r\leq a_{i}}C_{r}  \left(  \zeta^{\star, \mathfrak{m}} (\cdots,3, \boldsymbol{2}^{a_{i-1}+ a_{i}-r+1},3, \cdots) - \zeta^{\star, \mathfrak{m}} (\cdots,3, \boldsymbol{2}^{a_{i+1}+  a_{i}-r+1},3, \cdots) \right) \\
   \\
+\delta_{a_{i}+2\leq r \leq a_{i}+a_{i-1}+2} \widetilde{B}^{a_{i}+1,r-a_{i}-2}  \zeta^{\star, \mathfrak{m}} (\cdots,3, \boldsymbol{2}^{a_{i}+a_{i-1}-r+1},3,\boldsymbol{2}^{a_{i+1}},3, \cdots) \\
   \\
-  \delta_{a_{i}+2 \leq r\leq a_{i}+a_{i+1}+2} \widetilde{B}^{a_{i}+1,r-a_{i}-2} \zeta^{\star, \mathfrak{m}} (\cdots,3, \boldsymbol{2}^{a_{i-1}},3,\boldsymbol{2}^{a_{i}+a_{i+1}-r+1},3, \cdots) \\
\end{array} \right]  \\
 \\
  \textsc{(D)} + \delta_{a_{p}+1 \leq r \leq a_{p}+a_{p-1}+1} B^{r-a_{p}-1,a_{p}} \zeta^{\star, \mathfrak{m}} (\cdots,3, \boldsymbol{2}^{a_{p}+ a_{p-1}-r+1}) , \\
      \\
      \quad \quad \quad   \text{ with }  \widetilde{B}^{a,b}\mathrel{\mathop:}=B^{a,b}C_{a+b+1}-A^{a,b}.
\end{array}$$ 

\end{lemm}
\begin{proof}
Using Lemma $\ref{lemmcoeff}$ for the left side of $gr^{L}_{p} D_{2r+1}$, and keeping just the coefficients of $\zeta^{2r+1}$, we obtain easily this formula. In particular:
\begin{flushleft}
$\zeta^{\star\star, \mathfrak{l}}_{2} (\boldsymbol{2}^{a},3, \boldsymbol{2}^{b})=\zeta^{\star\star, \mathfrak{l}}(\boldsymbol{2}^{a+1},3, \boldsymbol{2}^{b})- \zeta^{\star, \mathfrak{l}}(\boldsymbol{2}^{a+1},3, \boldsymbol{2}^{b}) = \widetilde{B}^{a+1,b} \zeta^{\mathfrak{l}}(\overline{2a+2b+5}).$\\
$\zeta^{\star\star, \mathfrak{l}}_{1} (\boldsymbol{2}^{a},3, \boldsymbol{2}^{b},3)= \zeta^{\star\star, \mathfrak{l}}_{2} (\boldsymbol{2}^{b},3, \boldsymbol{2}^{a+1})= \widetilde{B}^{b+1,a+1}\zeta^{\mathfrak{l}}(\overline{2a+2b+7}).$
\end{flushleft}
\end{proof}

\subsection{Proof of Theorem $4.4.1$}

Since the cardinal of the Hoffman $\star$ family in weight $n$ is equal to the dimension of $\mathcal{H}_{n}^{1}$, \footnote{Obviously same recursive relation: $d_{n}=d_{n-2}+d_{n-3}$} it remains to prove that they are linearly independent:
\begin{center}
\texttt{Claim 1}: The Hoffman $\star$ elements are linearly independent.
\end{center}
It fundamentally use the injectivity of the map defined above, $\partial^{L}_{<n,l}$, via a recursion on the level. Indeed, let first prove the following statement:
\begin{equation} \label{eq:bijective} \texttt{Claim 2}: \quad \partial^{L}_{<n,l}: gr^{L}_{l}\mathcal{H}_{n}^{2,3}\rightarrow \oplus_{2r+1<n} gr^{L}_{l-1}\mathcal{H}_{n-2r-1}^{2,3} \text{  is bijective}.
\end{equation}
Using the Conjecture $\ref{conjcoeff}$ (assumed for this theorem), regarding the $2$-adic valuation of these coefficients, with $r=a+b+1$:\footnote{The last inequality comes from the fact that $v_{2} (\binom{2r}{2b+1} )<2r $.}
\begin{equation}\label{eq:valuations}
\hspace*{-0.7cm}\left\lbrace  \begin{array}{ll}
 C_{r}=\frac{2^{2r+1}}{2r+1}  &\Rightarrow  v_{2}(C_{r})=2r+1 .\\
 \widetilde{B}^{a,b}\mathrel{\mathop:}= B^{a,b}C_{r}-A^{a,b}=2^{2r+1}\left(  \frac{1}{2r+1}-\frac{\binom{2r}{2a}}{2^{2r}-1} \right) &\Rightarrow  v_{2}(\widetilde{B}^{a,b}) \geq 2r+1.\\
B^{a,b}C_{r}=C_{r}-2\binom{2r}{2b+1} &\Rightarrow   v_{2}(B^{0,r-1}C_{r})= 2+ v_{2}(r) \leq v_{2}(B^{a,b}C_{r}) < 2r+1 .
\end{array} \right. 
\end{equation}
The deconcatenation terms in $\partial^{L}_{<n,l}$, which correspond to the terms with $B^{a,b}C_{r}$ are then the smallest 2-adically, which is crucial for the injectivity.\\
\\
Now, define a matrix $M_{n,l}$ as the matrix of $\partial^{L}_{<n,l}$ on $\zeta^{\star, \mathfrak{m}} (\boldsymbol{2}^{a_{0}},3,\cdots,3, \boldsymbol{2}^{a_{l}})$ in terms of $\zeta^{\star, \mathfrak{m}} (\boldsymbol{2}^{b_{0}},3,\cdots,3, \boldsymbol{2}^{b_{l-1}})$; even if up to now, we do not know that these families are linearly independent. We order the elements on both sides by lexicographical order on ($a_{l}, \ldots, a_{0}$), resp. ($r,b_{l-1}, \ldots, b_{0}$), such that the diagonal corresponds to $r=a_{l}$ and $b_{i}=a_{i}$ for $i<l$ and claim:
\begin{center}
\texttt{Claim 3}: The matrix $M_{n,l}$ of $\partial^{L}_{<n,l}$ on the Hoffman $\star$ elements is invertible 
\end{center}
\begin{proof}[\texttt{Proof of Claim 3}]
Indeed, let $\widetilde{M}_{n,l}$ be the matrix $M_{n,l}$ where we have multiplied each line corresponding to $D_{2r+1}$ by ($2^{-v_{2}(r)-2}$). Then modulo $2$, because of the previous computations on the $2$-adic valuations of the coefficients, only the deconcatenations terms remain. Hence, with the previous order, the matrix is, modulo $2$, triangular with $1$ on the diagonal; the diagonal being the case where $B^{0,r-1}C_{r}$ appears. This implies that $\det (\widetilde{M}_{n,l})\equiv 1   \pmod{2}$, and in particular is non zero. Consequently, the matrix $\widetilde{M}_{n,l}$ is invertible and so does $M_{n,l}$.
\end{proof}
Obviously, $\texttt{Claim 3} \Rightarrow \texttt{Claim 2} $, but it will also enables us to complete the proof:

\begin{proof}[\texttt{Proof of Claim 1}] Let first prove it for the Hoffman $\star$ elements of a same level and weight, by recursion on level. Level $0$ is obvious: $\zeta^{\star,\mathfrak{m}}(2)^{n}$ is a rational multiple of $(\pi^{\mathfrak{m}})^{2n}$. Assuming by recursion on the level that the Hoffman $\star$ elements of weight $\leq n$ and level $l-1$ are linearly independent, since $M_{n,l}$ is invertible, this means both that the Hoffman $\star$ elements of weight $n$ and level $l$ are linearly independent.\\
The last step is to realize that the bijectivity of $\partial^{L}_{<n,l}$ also implies that Hoffman $\star$ elements of different levels are linearly independent. Indeed, proof can be done by contradiction: applying  $\partial^{L}_{<n,l}$ to a linear combination of Hoffman $\star$ elements, $l$ being the maximal number of $3$, we arrive at an equality between same level elements, and at a contradiction.
\end{proof}

\subsection{Analytic conjecture}

Here are the equalities needed for Theorem $4.4.1$, known up to some rational coefficients:
\begin{lemm} \label{lemmcoeff}
With  $w$, $d$ resp.  $ht$ denoting the weight, the depth, resp. the height:
\begin{itemize}
\item[$(o)$] $\begin{array}{llll}
\zeta^{\mathfrak{m}}(\overline{r}) & = & (2^{1-r}-1) &\zeta^{\mathfrak{m}}(r).\\
\zeta^{\mathfrak{m}}(2n) & = & \frac{\mid B_{n}\mid 2^{3n-1}3^{n}}{(2n)!} &\zeta^{\mathfrak{m}}(2)^{n}.
\end{array}$
\item[$(i)$] $\zeta^{\star,\mathfrak{m}}(\boldsymbol{2}^{n})= -2 \zeta^{\mathfrak{m}}(\overline{2n}) =\frac{(2^{2n}-2)6^{n}}{(2n)!}\vert B_{2n}\vert\zeta^{\mathfrak{m}}(2)^{n}.$
\item[$(ii)$] $\zeta^{\star,\mathfrak{m}}_{1}(\boldsymbol{2}^{n})= -2 \sum_{r=1}^{n} \zeta^{\mathfrak{m}}(2r+1)\zeta^{\star,\mathfrak{m}}(\boldsymbol{2}^{n-r}).$
\item[$(iii)$] 
\begin{align}
\zeta^{\star\star,\mathfrak{m}}(\boldsymbol{2}^{n}) & = \sum_{d \leq n} \sum_{w(\textbf{m})=2n \atop ht(\textbf{m})=d(\textbf{m})=d} 2^{2n-2d}\zeta^{\mathfrak{m}}(\textbf{m})  \\
& =\sum_{2n=\sum s_{k}(2i_{k}+1)+2S \atop i_{k}\neq i_{j}}  \left( \prod_{k=1}^{p} \frac{C_{i_{k}}^{s_{k}}} {s_{k}!} \zeta^{\mathfrak{m}}(\overline{2i_{k}+1})^{s_{k}} \right) D_{S} \zeta^{\mathfrak{m}}(2)^{S}. \nonumber\\
\zeta^{\star\star,\mathfrak{m}}_{1}(\boldsymbol{2}^{n}) & =-\sum_{d \leq n} \sum_{w(\textbf{m})=2n+1 \atop ht(\textbf{m})=d(\textbf{m})=d} 2^{2n+1-2d}\zeta^{\mathfrak{m}}(\textbf{m}) \\
&=\sum_{2n+1=\sum s_{k}(2i_{k}+1)+2S \atop i_{k}\neq i_{j}}  \left( \prod_{k=1}^{p} \frac{C_{i_{k}}^{s_{k}}} {s_{k}!} \zeta^{\mathfrak{m}}(\overline{2i_{k}+1})^{s_{k}}\right) D_{S} \zeta^{\mathfrak{m}}(2)^{S}\nonumber
\end{align}

\item[$(iv)$] $\zeta^{\star,\mathfrak{m}}(\boldsymbol{2}^{a},3,\boldsymbol{2}^{b})= \sum A^{a,b}_{r} \zeta^{\mathfrak{m}}(\overline{2r+1})\zeta^{\star,\mathfrak{m}}(\boldsymbol{2}^{n-r}).$
\item[$(v)$] \begin{align}
\zeta^{\star\star,\mathfrak{m}}(\boldsymbol{2}^{a},3,\boldsymbol{2}^{b}) &= \sum_{w=\sum s_{k}(2i_{k}+1)+2S \atop i_{k}\neq i_{j}} B^{a,b}_{i_{1},\cdots, i_{p}\atop s_{1}\cdots s_{p}} \left( \prod_{k=1}^{p} \frac{C_{i_{k}}^{s_{k}}} {s_{k}!} \zeta^{\mathfrak{m}}(\overline{2i_{k}+1})^{s_{k}}\right)  D_{S} \zeta^{\mathfrak{m}}(2)^{S}.\\
\zeta^{\star\star,\mathfrak{m}}_{1}(\boldsymbol{2}^{a},3,\boldsymbol{2}^{b}) &=D^{a,b} \zeta^{\mathfrak{m}}(2)^{\frac{w}{2}}+ \sum_{w=\sum s_{k}(2i_{k}+1)+2S \atop i_{k}\neq i_{j}} B^{a,b}_{i_{1},\cdots, i_{p}\atop s_{1}\cdots s_{p}}  \left( \prod_{k=1}^{p} \frac{C_{i_{k}}^{s_{k}}} {s_{k}!} \zeta^{\mathfrak{m}}(\overline{2i_{k}+1})^{s_{k}}\right)  D_{S}\zeta^{\mathfrak{m}}(2)^{S}.
\end{align}
\end{itemize}
Where:
\begin{itemize}

\item[$\cdot$] $C_{r}=\frac{2^{2r+1}}{2r+1}$, $D_{S}$ explicit\footnote{Cf. Proof.} and with the following constraint:
\begin{equation} \label{eq:constrainta}
A^{a,b}_{r}=A_{r}^{a,r-a-1}+C_{r} \left( B^{r-b-1,b}- B^{r-a-1,a} +\delta_{r\leq b}-\delta_{r\leq a} \right). 
\end{equation}
\item[$\cdot$] The recursive formula for $B$-coefficients, where $B^{x,y}\mathrel{\mathop:}=B^{x,y}_{x+y+1 \atop 1}$ and $r<a+b+1$:
  \begin{equation} \label{eq:constraintb} 
\begin{array}{lll }
 B^{a,b}_{r \atop 1} & = & \delta_{r\leq b} -  \delta_{r< a}+ B^{r-b-1,b}+\frac{D^{a-r-1,b}}{a+b-r+1}+\delta_{r=a} \frac{2(2^{2b+1}-1)6^{b+1} \mid B_{2b+2} \mid}{(2b+2)! D_{b+1}}.\\
 B^{a,b}_{i_{1},\cdots, i_{p}\atop s_{1}\cdots s_{p}} &=& \left\{
\begin{array}{l}
 \delta_{i_{1}\leq b } - \delta_{i_{1}< a } + B^{i_{1}-b-1,b} + B^{a-i_{1}-1,b}_{i_{1}, \ldots, i_{p}\atop s_{1}-1, \ldots, s_{p}}  \quad  \text{ for } \sum s_{k} \text{ odd }  \\
 \delta_{i_{1}\leq b } - \delta_{i_{1}\leq a } + B^{i_{1}-b-1,b} +B^{a-i_{1},b}_{i_{1}, \ldots, i_{p}\atop s_{1}-1, \ldots, s_{p}}  \quad \text{ else }.
  \end{array}
  \right. 
\end{array}  
    \end{equation}
\end{itemize}

\end{lemm}
\noindent
Before giving the proof, here is the (analytic) conjecture remaining on some of these coefficients, sufficient to complete the Hoffman $\star$ basis proof (cf. Theorem $\ref{Hoffstar}$):
\begin{conj}\label{conjcoeff}
The equalities $(v)$ are satisfied for real MZV, with:
$$B^{a,b}=1-\frac{2}{C_{a+b+1}}\binom{2a+2b+2}{2b+1}.$$
\end{conj}
\textsc{Remarks:} 
\begin{itemize}
\item[$\cdot$] This conjecture is of an entirely different nature from the techniques developed in this thesis. We can expect that it can proved using analytic methods as the usual techniques of identifying hypergeometric series, as in $\cite{Za}$, or $\cite{Li}$.
\item[$\cdot$] The equality $(iv)$ is already proven in the analytic case by Ohno-Zagier (cf.$\cite{IKOO}$, $\cite{Za}$), with the values of the coefficient $A_{r}^{a,b}$ given below. Nevertheless, as we will see through the proofs below, to make the coefficients for the (stronger) motivic identity $(iv)$ explicit, we need to prove the other identities in $(v)$. 
\item[$\cdot$]  We will use below a result of Ohno and Zagier on sums of MZV of fixed weight, depth and height to conclude for the coefficients for $(iii)$.
\end{itemize}

\begin{theo}
If the analytic conjecture ($\ref{conjcoeff}$) holds, the equalities $(iv)$, $(v)$ are true in the motivic case, with the same values of the coefficients. In particular:
$$A_{r}^{a,b}= 2\left( -\delta_{r=a}+ \binom{2r}{2a} \right)  \frac{2^{2r}}{2^{2r}-1}-2\binom{2r}{2b+1}.$$
\end{theo}
\begin{proof}
Remind that if we know a motivic equality up to one unknown coefficient (of $\zeta(weight)$), the analytic result analogue enables us to conclude on its value by Corollary $\ref{kerdn}$.\\
Let assume now, in a recursion on $n$, that we know $\lbrace B^{a,b}, D^{a,b}, B_{i_{1} \cdots i_{p} \atop s_{1} \cdots s_{p} }^{a,b} \rbrace_{a+b+1<n}$ and consider $(a,b)$ such that $a+b+1=n$. Then, by $(\ref{eq:constraintb})$, we are able to compute the $B_{\textbf{i}\atop \textbf{s}}^{a,b}$ with $(s,i)\neq (1,n)$. Using the analytic $(v)$ equality, and Corollary $\ref{kerdn}$, we deduce the only remaining unknown coefficient $B^{a,b}$ resp. $D^{a,b}$ in $(v)$.\\
Lastly, by recursion on $n$ we deduce the $A_{r}^{a,b}$ coefficients: let assume they are known for $a+b+1<n$, and take $(a,b)$ with $a+b+1=n$. By 
the constraint $(\ref{eq:constrainta})$, since we already know $B$ and $C$ coefficients, we deduce $A_{r}^{a,b}$ for $r<n$. The remaining coefficient, $A_{n}^{a,b}$, is obtained using the analytic $(iv)$ equality and Corollary $\ref{kerdn}$.
\end{proof}

\paragraph{\texttt{Proof of} Lemma $\ref{lemmcoeff}$.}:
\begin{proof}
Computing the coaction on these elements, by a recursive procedure, we are able to prove these identities up to some rational coefficients, with the Corollary $\ref{kerdn}$. When the analytic analogue of the equality is known for MZV, we \textit{may} conclude on the value of the remaining rational coefficient of $\zeta^{\mathfrak{m}}(w)$ by identification (as for $(i),(ii),(iii)$). However, if the family is not stable under the coaction , (as for $(iv)$) knowing the analytic case is not enough.\\
\texttt{Nota Bene:} This proof refers to the expression of $D_{2r+1}$ in Lemma $\ref{lemmt}$: we look at cuts of length $2r+1$ among the sequence of $0, 1, $ or $\star$ (in the iterated integral writing); there are different kind of cuts (according their extremities), and each cut may bring out two terms ($T_{0,0}$ and $T_{0,\star}$ for instance). The simplifications are illustrated by the diagrams, where some arrows (term of a cut) get simplified by rules specified in Annexe $A$.\\
\begin{itemize}
\item[$(i)$] The corresponding iterated integral: 
$$I^{\mathfrak{m}}(0; 1, 0, \star, 0 \cdots, \star, 0; 1).$$
The only possible cuts of odd length are between two $\star$ ($T_{0,\star}$ and $T_{\star,0}$) or $T_{1,0}$ from the first $1$ to a $\star$, or $T_{0,1}$ from a $\star$ to the last $1$. By \textsc{ Shift }(\ref{eq:shift}), these cuts get simplified two by two. Since $D_{2r+1}(\cdot)$, for $2r+1<2n$ are all zero, it belongs to $\mathbb{Q}\zeta^{\mathfrak{m}}(2n)$, by Corollary $\ref{kerdn}$). Using the (known) analytic equality, we can conclude.
\item[$(ii)$] It is quite similar to $(i)$: using $\textsc{ Shift }$ $(\ref{eq:shift})$, it remains only the cut:\\
\includegraphics[]{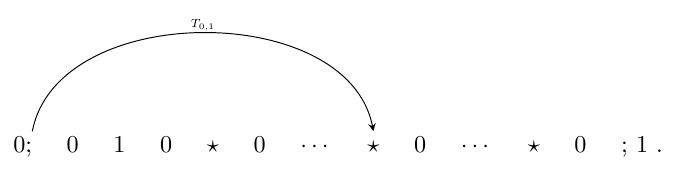}\\
$$\text{i.e.}: \quad D_{2r+1} (\zeta^{\star,\mathfrak{m}}_{1}(\boldsymbol{2}^{n}))= \zeta^{\mathfrak{l},\star}_{1}(\boldsymbol{2}^{r})\otimes \zeta^{\star,\mathfrak{m}}(\boldsymbol{2}^{n-r})=-2 \zeta^{ \mathfrak{l}}(\overline{2r+1})\otimes \zeta^{ \star,\mathfrak{m}}(\boldsymbol{2}^{n-r}).$$
The last equality is deduced from the recursive hypothesis (smaller weight). The analytic equality (coming from the Zagier-Ohno formula, and the $\shuffle$ regulation) enables us to conclude on the value of the remaining coefficient of $\zeta^{\mathfrak{m}}(2n+1)$.
\item[$(iii)$] Expressing these ES$\star\star$ as a linear combination of ES by $\shuffle$ regularisation:
$$\hspace*{-0.5cm}\zeta^{\star\star,\mathfrak{m}}(\boldsymbol{2}^{n})= \sum_{k_{i} \text{ even}} \zeta^{\mathfrak{m}}_{2n-\sum k_{i}}(k_{1},\cdots, k_{p})=\sum_{n_{i}\geq 2} \left(  \sum_{k_{i} \text{ even} \atop k_{i} \leq n_{i}}  \binom{n_{1}-1}{k_{1}-1} \cdots \binom{n_{d}-1}{k_{d}-1}  \right) \zeta^{\mathfrak{m}}(n_{1},\cdots, n_{d}) .$$
Using the multi-binomial formula:
$$2^{\sum m_{i}}=\sum_{l_{i} \leq m_{i}} \binom{m_{1}}{l_{1}}(1-(-1))^{l_{1}} \cdots \binom{m_{d}}{l_{d}}(1-(-1))^{l_{d}}= 2^{d} \sum_{l_{i} \leq m_{i}\atop l_{i} \text{ odd  }} \binom{m_{1}}{l_{1}} \cdots \binom{m_{d}}{l_{d}} .$$
Thus:
$$\zeta^{\star\star,\mathfrak{m}}(\boldsymbol{2}^{n})=\sum_{d \leq n} \sum_{w(\textbf{m})=2n \atop ht(\textbf{m})=d(\textbf{m})=d} 2^{2n-2d}\zeta^{\mathfrak{m}}(\textbf{m}).$$
Similarly for $(4.27)$, since:
$$\zeta^{\star\star,\mathfrak{m}}_{1}(\boldsymbol{2}^{n})=\sum_{k_{i} \text{ even}} \zeta^{\mathfrak{m}}_{2n+1-\sum k_{i}}(k_{1},\cdots, k_{p})=\sum_{d \leq n} \sum_{w(\textbf{m})=2n \atop ht(\textbf{m})=d(\textbf{m})=d} 2^{2n-2d}\zeta^{\mathfrak{m}}(\textbf{m}).$$
Now, using still only $\textsc{ Shift }$ $(\ref{eq:shift})$, it remains the following cuts:\\
\includegraphics[]{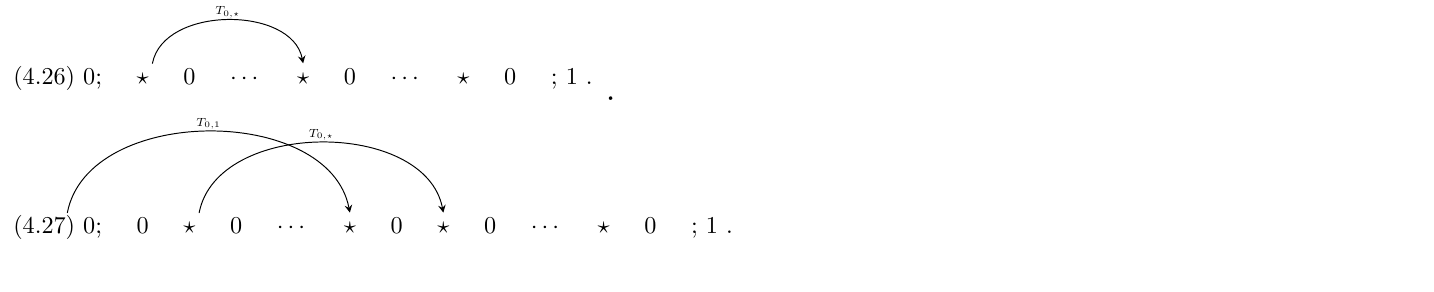}\\
With a recursion on $n$ for both $(4.26)$, $(4.27)$, we deduce:
$$D_{2r+1}(\zeta^{\star\star,\mathfrak{m}}(\boldsymbol{2}^{n}))=\zeta^{\star\star,\mathfrak{m}}_{1}(\boldsymbol{2}^{r})\otimes \zeta^{\star\star,\mathfrak{l}}_{1}(\boldsymbol{2}^{n-r})=C_{r} \zeta^{ \mathfrak{l}}(\overline{2r+1})\otimes  \zeta^{\star\star,\mathfrak{m}}_{1}(\boldsymbol{2}^{n-r-1}).$$
$$D_{2r+1}(\zeta^{\star\star,\mathfrak{m}}_{1}(\boldsymbol{2}^{n}))=\zeta^{\star\star,\mathfrak{l}}_{1}(\boldsymbol{2}^{r})\otimes \zeta^{\star\star,\mathfrak{m}}(\boldsymbol{2}^{n-r})=C_{r} \zeta^{ \mathfrak{l}}(\overline{2r+1})\otimes \zeta^{\star\star,\mathfrak{m}}(\boldsymbol{2}^{n-r}).$$
To find the remaining coefficients, we need the analytic result corresponding, which is a consequence of the sum relation for MZV of fixed weight, depth and height, by Ohno and Zagier ($\cite{OZa}$, Theorem $1$), via the hypergeometric functions.\\
Using $\cite{OZa}$, the generating series of these sums is, with $\alpha,\beta=\frac{x+y \pm \sqrt{(x+y)^{2}-4z}}{2}$:
$$\begin{array}{lll}
\phi_{0}(x,y,z)\mathrel{\mathop:} & = &  \sum_{s\leq d \atop w\geq d+s}  \left( \sum \zeta(\textbf{k}) \right) x^{w-d-s}y^{d-s}z^{s-1} \\
& = & \frac{1}{xy-z} \left( 1- \exp \left( \sum_{m=2}^{\infty} \frac{\zeta(m)}{m}(x^{m}+y^{m}-\alpha^{m}-\beta^{m}) \right)   \right) .
\end{array}$$
From this, let express the generating series of both $\zeta^{\star\star}(\boldsymbol{2}^{n})$ and $\zeta^{\star\star}_{1}(\boldsymbol{2}^{n})$:
$$\phi(x)\mathrel{\mathop:}= \sum_{w} \left( \sum_{ht(\textbf{k})=d(\textbf{k})=d\atop w\geq 2d} 2^{w-2d} \zeta(\textbf{k}) \right) x^{w-2}= \phi_{0}(2x, 0, x^{2}).$$
Using the result of Ohno and Don Zagier:
$$\phi(x)= \frac{1}{x^{2}} \left(\exp \left( \sum_{m=2}^{\infty} \frac{2^{m}-2}{m} \zeta(m) x^{m} \right)   -1\right).$$
Consequently, both $\zeta^{\star\star}(\boldsymbol{2}^{n})$ and $\zeta^{\star\star}_{1}(\boldsymbol{2}^{n})$ can be written explicitly as polynomials in simple zetas. For $\zeta^{\star\star}(\boldsymbol{2}^{n})$, by taking the coefficient of $x^{2n-2}$ in $\phi(x)$:
$$\zeta^{\star\star}(\boldsymbol{2}^{n})= \sum_{\sum m_{i} s_{i}=2n \atop m_{i}\neq m_{j}} \prod_{i=1}^{k} \left(  \frac{1}{s_{i} !}\left( \zeta(m_{i}) \frac{2^{m_{i}}-2}{m_{i}}\right)^{s_{i}} \right) .$$
Gathering the zetas at even arguments, it turns into:
$$\zeta^{\star\star}(\boldsymbol{2}^{n})= \sum_{\sum (2i_{k}+1) s_{k}+2S=2n \atop i_{k}\neq i_{j}} \prod_{i=1}^{p} \left(  \frac{1}{s_{k} !}\left( \zeta(2 i_{k}+1) \frac{2^{2 i_{k}+1}-2}{2i_{k}+1}\right)^{s_{k}} \right) d_{S} \zeta(2)^{S}, $$
\begin{equation}\label{eq:coeffds}
 \text{ where }  d_{S}\mathrel{\mathop:}=3^{S}\cdot 2^{3S}\sum_{\sum m_{i} s_{i}=S \atop m_{i}\neq m_{j}} \prod_{i=1}^{k} \left( \frac{1}{s_{i}!} \left( \frac{\mid B_{2m_{i}}\mid (2^{2m_{i}-1}-1) } {2m_{i} (2m_{i})!}\right)^{s_{i}}  \right).
\end{equation}
It remains to turn $\zeta(odd)$ into $\zeta(\overline{odd})$ by $(o)$ to fit the expression of the Lemma:
$$\zeta^{\star\star}(\boldsymbol{2}^{n})= \sum_{\sum (2i_{k}+1) s_{k}+2S=2n \atop i_{k}\neq i_{j}} \prod_{i=1}^{p} \left(  \frac{1}{s_{k} !}\left( c_{i_{k}}\zeta(\overline{2 i_{k}+1}) \right)^{s_{k}} \right) d_{S} \zeta(2)^{S}, \text{ where } c_{r}=\frac{2^{2r+1}}{2r+1}.$$
It is completely similar for $\zeta^{\star\star}_{1}(\boldsymbol{2}^{n})$: by taking the coefficient of $x^{2n-3}$ in $\phi(x)$, we obtained the analytic analogue of $(4.25)$, with the same coefficients $d_{S}$ and $c_{r}$.\\
Now, using these analytic results for $(4.26)$, $(4.27)$, by recursion on the weight, we can identify the coefficient $D_{S}$ and $C_{r}$ with resp. $d_{S}$ and $c_{r}$, since there is one unknown coefficient at each step of the recursion.
\item[$(iv)$] After some simplifications by Antipodes rules ($\S A.1$), only the following cuts remain:\\
\includegraphics[]{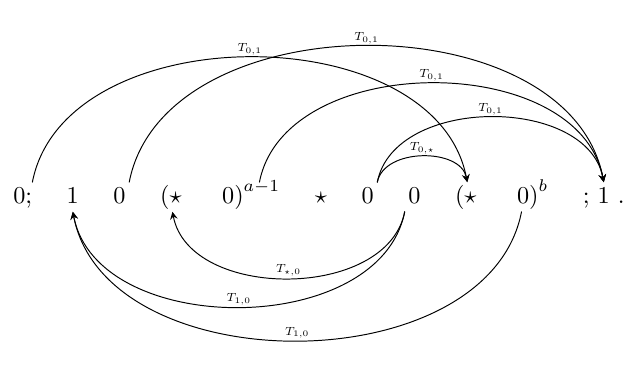}\\
This leads to the formula:\\
$$D_{2r+1} (\zeta^{\star,\mathfrak{m}}(\boldsymbol{2}^{a},3,\boldsymbol{2}^{b}))=  \left(\zeta^{\star,\mathfrak{m}}(\boldsymbol{2}^{a},3,\boldsymbol{2}^{r-a-1})+\right.$$
$$ \left.\left(  \delta_{r \leq b}-\delta_{r \leq a}\right)  \zeta^{\star\star,\mathfrak{m}}_{1}(\boldsymbol{2}^{r})  + \zeta^{\star\star , \mathfrak{m}}(\boldsymbol{2}^{r-b-1},3,\boldsymbol{2}^{b}) -\zeta^{\star\star ,\mathfrak{m}}(\boldsymbol{2}^{r-a-1},3,\boldsymbol{2}^{a})\right) \otimes \zeta^{\star,\mathfrak{m}}(\boldsymbol{2}^{n-r}).$$
In particular, the Hoffman $\star$ family is not stable under the coaction, so we need first to prove $(v)$, and then:
$$\hspace*{-0.7cm}D_{2r+1} (\zeta^{\star ,\mathfrak{m}}(\boldsymbol{2}^{a},3,\boldsymbol{2}^{b}))= \left( A_{r}^{a,r-a-1}+C_{r} \left( B^{r-b-1,b}- B^{r-a-1,a} +\delta_{r\leq b}-\delta_{r\leq a} \right)\right)  \zeta^{ \mathfrak{l}}(\overline{2r+1})\otimes \zeta^{\star ,\mathfrak{m}}(\boldsymbol{2}^{n-r}). $$
It leads to the constraint $(\ref{eq:constrainta})$ above for coefficients $A$. To make these coefficients explicit, apart from the known analytic Ohno Zagier formula, we need the analytic analogue of $(v)$ identities, as stated in Conjecture $\ref{conjcoeff}$.
\item[$(v)$] By Annexe rules, the following cuts get simplified (by colors, above with below):\footnote{The vertical arrows indicates a cut from the $\star$ to a $\star$ of the same group.}\\
\includegraphics[]{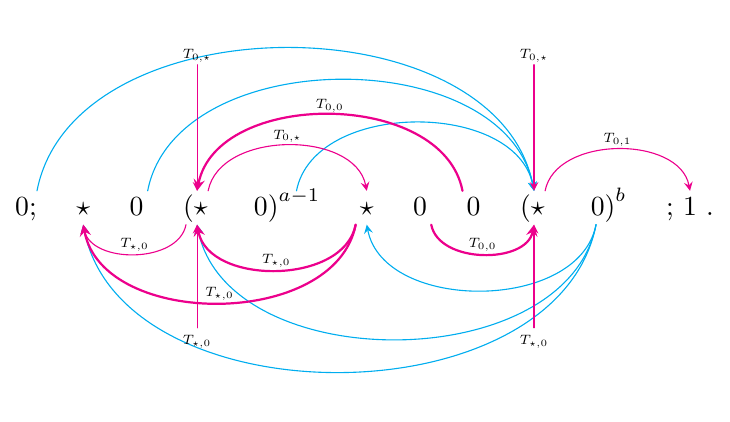}\\
Indeed, cyan arrows get simplified by \textsc{Antipode} $\shuffle$, $T_{0,0}$ resp. $T_{0, \star}$ above with $T_{0,0}$ resp. $T_{\star,0}$ below; magenta ones by $\textsc{ Shift }$ $(\ref{eq:shift})$, term above with the term below shifted by two on the left. It remains the following cuts for $(4.28)$:\\
\includegraphics[]{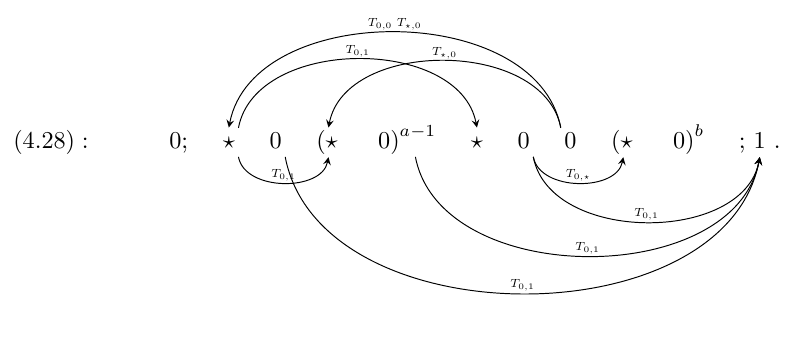}\\
In a very similar way, the simplifications lead to the following remaining terms:\\
\includegraphics[]{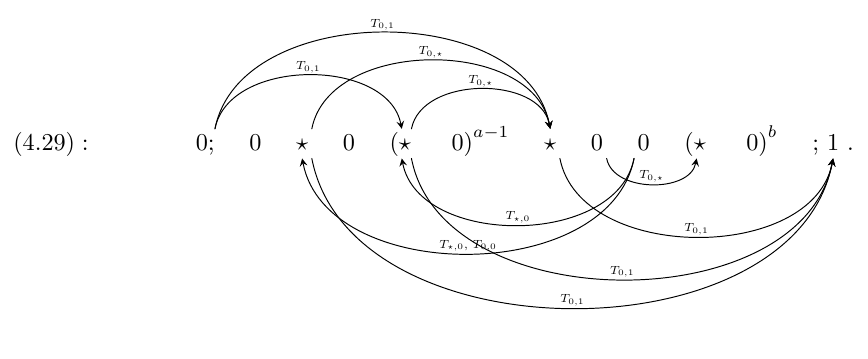}\\
Then, the derivations reduce to: 
$$\hspace*{-0.7cm}D_{2r+1} (\zeta^{\star\star ,\mathfrak{m}}(\boldsymbol{2}^{a},3,\boldsymbol{2}^{b}))= \left( \left(  \delta_{r\leq b}-\delta_{r \leq a}\right)  \zeta^{\star\star ,\mathfrak{l}}_{1}(\boldsymbol{2}^{r}) +\delta_{r> b}\zeta^{\star\star, \mathfrak{m-l}}(\boldsymbol{2}^{r-b-1},3,\boldsymbol{2}^{b})\right) \otimes \zeta^{\star\star ,\mathfrak{m}}(\boldsymbol{2}^{n-r}) +$$
$$\hspace*{+1cm} +\delta_{r\leq a-1} \zeta^{\star\star ,\mathfrak{l}}_{1}(\boldsymbol{2}^{r}) \otimes \zeta^{\star\star ,\mathfrak{m}}_{1}(\boldsymbol{2}^{a-r-1},3,\boldsymbol{2}^{b})+ \delta_{r=a} \zeta^{\star\star ,\mathfrak{l}}_{1}(\boldsymbol{2}^{a})\otimes \zeta^{\star\star ,\mathfrak{m}}_{2}(\boldsymbol{2}^{b}) .$$
$$\hspace*{-1.4cm}D_{2r+1} (\zeta^{\star\star ,\mathfrak{m}}_{1}(\boldsymbol{2}^{a},3,\boldsymbol{2}^{b}))= \left( \left(  \delta_{r\leq b}-\delta_{r \leq a}\right)  \zeta^{\star\star ,\mathfrak{l}}_{1}(\boldsymbol{2}^{r})  + \zeta^{\star\star ,\mathfrak{l}}(\boldsymbol{2}^{r-b-1},3,\boldsymbol{2}^{b})\right)\otimes \zeta^{\star\star,\mathfrak{m}}_{1}(\boldsymbol{2}^{n-r}) +$$
$$ + \zeta^{\star\star ,\mathfrak{l}}_{1}(\boldsymbol{2}^{r})\otimes \zeta^{\star\star,\mathfrak{m}}(\boldsymbol{2}^{a-r},3,\boldsymbol{2}^{b}) .$$
\hspace*{-0.5cm}With a recursion on $w$ for both:
$$\hspace*{-1.4cm}\begin{array}{ll}
D_{2r+1} (\zeta^{\star\star ,\mathfrak{m}}(\boldsymbol{2}^{a},3,\boldsymbol{2}^{b})) & = C_{r} \zeta^{ \mathfrak{l}}(\overline{2r+1})\otimes\\
& \left( \left( \delta_{r\leq b}-\delta_{r < a}  + B_{r}^{r-b-1,b}\right) \zeta^{\star\star ,\mathfrak{m}}(\boldsymbol{2}^{n-r}) + \zeta^{\star\star ,\mathfrak{m}}_{1}(\boldsymbol{2^{a-r-1},}3,\boldsymbol{2}^{b})+ \delta_{r=a}\zeta^{\star ,\mathfrak{m}}(\boldsymbol{2}^{b+1}) \right) .\\
& \\
D_{2r+1} (\zeta^{\star\star ,\mathfrak{m}}_{1}(\boldsymbol{2}^{a},3,\boldsymbol{2}^{b}))& = C_{r} \zeta^{ \mathfrak{l}}(\overline{2r+1})\otimes\left( \left(  \delta_{r\leq b}-\delta_{r \leq a}   + B_{r}^{r-b-1,b}\right) \zeta^{ \star\star ,\mathfrak{m}}_{1}(\boldsymbol{2}^{n-r}) + \zeta^{\star\star ,\mathfrak{m}}(\boldsymbol{2}^{a-r},3,\boldsymbol{2}^{b}) \right).
\end{array}$$
This leads to the recursive formula $(\ref{eq:constraintb})$ for $B$.
\end{itemize}
\end{proof}


\section{Motivic generalized Linebarger Zhao Conjecture}

We conjecture the following motivic identities, which express each motivic MZV $\star$ as a motivic Euler $\sharp$ sum:
\begin{conj}\label{lzg}  
For $a_{i},c_{i} \in \mathbb{N}^{\ast}$, $c_{i}\neq 2$,
$$\zeta^{\star, \mathfrak{m}} \left( \boldsymbol{2}^{a_{0}},c_{1},\cdots,c_{p}, \boldsymbol{2}^{a_{p}}\right)  =(-1)^{1+\delta_{c_{1}}}\zeta^{\sharp,  \mathfrak{m}} \left(B_{0},\boldsymbol{1}^{c_{1}-3 },\cdots,\boldsymbol{1}^{ c_{i}-3 },B_{i}, \ldots, B_{p}\right), $$
where $\left\lbrace \begin{array}{l}
B_{0}\mathrel{\mathop:}=  \pm (2a_{0}+1-\delta_{c_{1}})\\
B_{i}\mathrel{\mathop:}= \pm(2a_{i}+3-\delta_{c_{i}}-\delta_{c_{i+1}})\\
B_{p}\mathrel{\mathop:}=\pm ( 2 a_{p}+2-\delta_{c_{p}})
\end{array}\right.$, with $\pm\mathrel{\mathop:}=\left\lbrace \begin{array}{l}
- \text{ if } \mid B_{i}\mid \text{ even} \\
+ \text{ if } \mid B_{i}\mid \text{ odd}
\end{array} \right.$, $\begin{array}{l}
\delta_{c}\mathrel{\mathop:}=\delta_{c=1},\\
\text{the Kronecker symbol}.
\end{array}$
and $\boldsymbol{1}^{n}:=\boldsymbol{1}^{min(0,n)}$ is a sequence of $n$ 1 if $n\in\mathbb{N}$, an empty sequence else.
\end{conj}
\textsc{Remarks}:
\begin{itemize}
\item[$\cdot$] Motivic Euler $\sharp$ sums appearing on the right side have already been proven to be unramified in $\S 4.3$, i.e. MMZV. 
\item[$\cdot$] This conjecture implies that the motivic Hoffman $\star$ family is a basis, since it corresponds here to the motivic Euler $\sharp$ sum family proved to be a basis in Theorem $\ref{ESsharpbasis}$: cf. ($\ref{eq:LZhoffman}$).
\item[$\cdot$] The number of sequences of consecutive $1$ in $\zeta^{\star}$, $n_{1}$ is linked with the number of even in $\zeta^{\sharp}$, $n_{e}$, here by the following formula:\\
$$n_{e}=1+2n_{1}-2\delta_{c_{p}} -\delta_{c_{1}}.$$
In particular, when there is no $1$ in the MMZV $\star$, there is only one even (at the end) in the Euler sum $\sharp$. There are always at least one even in the Euler sums.
\end{itemize}

Special cases of this conjecture, which are already proven for real Euler sums (references indicated in the braket), but remain conjectures in the motivic case:
\begin{description}
\item[Two-One] [Ohno Zudilin, $\cite{OZ}$.]
\begin{equation}\label{eq:OZ21}
\zeta^{\star, \mathfrak{m}} (\boldsymbol{2}^{a_{0}},1,\cdots,1, \boldsymbol{2}^{a_{p}})= - \zeta^{\sharp,  \mathfrak{m}} \left( \overline{2a_{0}}, 2a_{1}+1, \ldots, 2a_{p-1}+1, 2 a_{p}+1\right) . 
\end{equation}
\item [Three-One] [Broadhurst et alii, $\cite{BBB}$.] \footnote{The Three-One formula was conjectured for real Euler sums by Zagier, proved by Broadhurst et alii in $\cite{BBB}$.}
\begin{equation}\label{eq:Z31} \zeta^{\star, \mathfrak{m}} (\boldsymbol{2}^{a_{0}},1,\boldsymbol{2}^{a_{1}},3 \cdots,1, \boldsymbol{2}^{a_{p-1}}, 3, \boldsymbol{2}^{a_{p}}) = -\zeta^{\sharp,  \mathfrak{m}} \left( \overline{2a_{0}}, \overline{2a_{1}+2}, \ldots, \overline{2a_{p-1}+2}, \overline{2 a_{p}+2} \right) .
\end{equation}
\item[Linebarger-Zhao$\star$] [Linebarger Zhao, $\cite{LZ}$]  With $c_{i}\geq 3$:
\begin{equation}\label{eq:LZ}
\zeta^{\star, \mathfrak{m}} \left( \boldsymbol{2}^{a_{0}},c_{1},\cdots,c_{p}, \boldsymbol{2}^{a_{p}}\right)  =
-\zeta^{\sharp,  \mathfrak{m}} \left( 2a_{0}+1,\boldsymbol{1}^{ c_{1}-3  },\cdots,\boldsymbol{1}^{ c_{i}-3  },2a_{i}+3, \ldots, \overline{ 2 a_{p}+2} \right) 
\end{equation}
In particular, when all $c_{i}=3$:
\begin{equation}\label{eq:LZhoffman}
\zeta^{\star, \mathfrak{m}} \left( \boldsymbol{2}^{a_{0}},3,\cdots,3, \boldsymbol{2}^{a_{p}}\right)  = - \zeta^{\sharp,  \mathfrak{m}} \left( 2a_{0}+1, 2a_{1}+3, \ldots, 2a_{p-1}+3, \overline{2 a_{p}+2}\right) . 
\end{equation}
\end{description}
\texttt{Examples}: Particular identities implied by the previous conjecture, sometimes known for MZV and which could then be proven for motivic Euler sums directly with the coaction:
\begin{itemize}
\item[$\cdot$] $ \zeta^{\star, \mathfrak{m}}(1, \left\lbrace 2 \right\rbrace^{n} )=2 \zeta^{ \mathfrak{m}}(2n+1).$
\item[$\cdot$] $ \zeta^{\star, \mathfrak{m}}(1, \left\lbrace 2 \right\rbrace^{a}, 1, \left\lbrace 2 \right\rbrace^{b} )= \zeta^{ \sharp\mathfrak{m} }(2a+1,2b+1)= 4 \zeta^{ \mathfrak{m} }(2a+1,2b+1)+ 2 \zeta^{ \mathfrak{m} }(2a+2b+2). $
\item[$\cdot$] $ \zeta^{ \mathfrak{m}} (n)= - \zeta^{\sharp, \mathfrak{m}} (\lbrace 1\rbrace^{n-2}, -2)= -\sum_{ w(\boldsymbol{k})=n \atop \boldsymbol{k} \text{admissible}} \boldsymbol{2}^{p} \zeta^{\mathfrak{m}}(k_{1}, \ldots, k_{p-1}, -k_{p}).$
\item[$\cdot$] $ \zeta^{ \star, \mathfrak{m}} (\lbrace 2 \rbrace ^{n})= \sum_{\boldsymbol{k} \in \lbrace \text{ even }\rbrace^{\times} \atop w(\boldsymbol{k})= 2n} \zeta^{\mathfrak{m}} (\boldsymbol{k})=- 2 \zeta^{\mathfrak{m}} (-2n) .$
\end{itemize}
We paved the way for the proof of Conjecture $\ref{lzg}$, bringing it back to an identity in $\mathcal{L}$:
\begin{theo}
Let assume:
\begin{itemize}
\item[$(i)$] The analytic version of $\ref{lzg}$ is true.
\item[$(ii)$] In the coalgebra $\mathcal{L}$, i.e. modulo products, for odd weights:
\begin{equation}\label{eq:conjid}
 \zeta^{\sharp,  \mathfrak{l}} _{B_{0}-1}(\boldsymbol{1}^{ \gamma_{1}},\cdots, \boldsymbol{1}^{\gamma_{p}  },B_{p})\equiv
  \zeta^{\star\star, \mathfrak{l}}_{2} (\boldsymbol{2}^{a_{0}-1},c_{1},\cdots,\boldsymbol{2}^{a_{p}})-\zeta^{\star\star, \mathfrak{l}}_{1} (\boldsymbol{2}^{a_{0}}, c_{1}-1, \ldots, \boldsymbol{2}^{a_{p}}) ,
\end{equation}
\begin{flushright}
with $c_{1}\geq 3$, $a_{0}>0$,  $\gamma_{i}=c_{i}-3 + 2\delta_{c_{i}}$   and   $\left\lbrace \begin{array}{l}
B_{0}= 2a_{0}+1-\delta_{c_{1}}\\
B_{i}=2a_{i}+3-\delta_{c_{i}}-\delta_{c_{i+1}}\\
B_{p}=2a_{p}+3-\delta_{c_{p}}
\end{array} \right. $.
\end{flushright}
\end{itemize}
Then:
\begin{enumerate}[I.]
\item Conjecture $\ref{lzg}$ is true, for motivic Euler sums.
\item In the coalgebra $\mathcal{L}$, for odd weights, with $c_{1}\geq 3$ and the previous notations:
\begin{equation}\label{eq:toolid}
\zeta^{\sharp,  \mathfrak{l}} (\boldsymbol{1}^{ \gamma_{1}},\cdots, \boldsymbol{1}^{ \gamma_{p} },B_{p})\equiv  - \zeta^{\star, \mathfrak{l}}_{1} (c_{1}-1, \boldsymbol{2}^{a_{1}},c_{2},\cdots,c_{p}, \boldsymbol{2}^{a_{p}}).
\end{equation}
\end{enumerate}
\end{theo}
\texttt{ADDENDUM:} The hypothesis $(i)$ is proved: J. Zhao deduced it from its Theorem 1.4 in $\cite{Zh3}$.\\
\\
\textsc{Remark:} The $(ii)$ hypothesis should be proven either directly via the various relations in $\mathcal{L}$ proven in $\S 4.2$ (as for $\ref{eq:toolid}$), or using the coaction, which would require the analytic identity corresponding. Beware, $(ii)$ would only be true in $\mathcal{L}^{2}$, not in $\mathcal{H}^{2}$.

\begin{proof}
To prove this equality $1.$ at a motivic level by recursion, we would need to proof that the coaction is equal on both side, and use the conjecture analytic version of the same equality. We prove $I$ and $II$ successively, in a same recursion on the weight:
\begin{enumerate}[I.]
\item Using the formulas of the coactions $D_{r}$ for these families (Lemma $A.1.2$ and $A.1.4$), we can gather terms in both sides according to the right side, which leads to three types:
$$ \begin{array}{llll}
(a) &  \zeta^{\star, \mathfrak{m}} (\cdots,\boldsymbol{2}^{a_{i}}, \alpha, \boldsymbol{2}^{\beta}, c_{j+1}, \cdots)  & \longleftrightarrow &   \zeta^{\sharp ,\mathfrak{m}}(B_{0} \cdots, B_{i}, \textcolor{magenta}{1^{\gamma}, B}, 1^{\gamma_{j+1}}, \ldots, B_{p}) \\
(b) &   \zeta^{\star, \mathfrak{m}} (\cdots,\boldsymbol{2}^{a_{i-1}}, c_{i}, \boldsymbol{2}^{\beta}, c_{j+1}, \cdots)    & \longleftrightarrow &   \zeta^{\sharp ,\mathfrak{m}}(B_{0} \cdots, B_{i-1}, 1^{\gamma_{i}}, \textcolor{green}{B}, 1^{\gamma_{j+1}}, \ldots, B_{p})   \\
(c) &   \zeta^{\star, \mathfrak{m}} (\cdots, c_{i}, \boldsymbol{2}^{\beta}, \alpha, \boldsymbol{2}^{a_{j}}, \cdots)    & \longleftrightarrow & \zeta^{\sharp, \mathfrak{m}}(B_{0} \cdots, 1^{\gamma_{i+1}},\textcolor{cyan}{ B, 1^{\gamma}}, B_{j+1}, \ldots, B_{p}) 
\end{array},$$
with $ \gamma=\alpha-3$ and $B=2\beta +3-\delta_{c_{j+1}}$, or $B=2\beta+3 - \delta_{c_{i}}- \delta_{c_{j+1}}$ for $(b)$.\\
The third case, antisymmetric of the first case, may be omitted below. By recursive hypothesis, these right sides are equal and it remains to compare the left sides associated:
\begin{enumerate}
\item On the one hand, by lemma $A.1.2$, the left side corresponding:
$$ \delta_{3\leq \alpha \leq c_{i+1}-1 \atop 0\leq \beta  a_{j}}  \zeta^{\star, \mathfrak{l}}_{c_{i+1}-\alpha}- (\boldsymbol{2}^{ a_{j}-\beta}, \ldots, \boldsymbol{2}^{a_{i+1}}).$$
On the other hand (Lemma $A.1.4$), the left side is:
$$-\delta_{2 \leq B \leq B_{j} \atop 0\leq\gamma\leq\gamma_{i+1}-1}\zeta^{\sharp,\mathfrak{l}}(B_{j}-B+1, 1^{\gamma_{j}}, \ldots, 1^{\gamma_{i+1}-\gamma-1}).$$
They are both equal, by $\ref{eq:toolid}$, where $c_{i+1}-\alpha+2$ corresponds to $c_{1} $ and is greater than $3$.\\
\item By lemma $A.1.2$, the left side corresponding for $\zeta^{\star}$:
$$\hspace*{-1.2cm}\begin{array}{llll}
-& \delta_{c_{i}>3} \zeta^{\star\star, \mathfrak{l}}_{2} (\boldsymbol{2}^{a_{i}}, \ldots, \boldsymbol{2}^{ a_{j}-\beta-1}) & + & \delta_{c_{j}>3} \zeta^{\star\star, \mathfrak{l}}_{2} (\boldsymbol{2}^{a_{j}}, \ldots, \boldsymbol{2}^{ a_{i}-\beta-1})  \\
 - & \delta_{c_{i}=1} \zeta^{\star\star, \mathfrak{l}} (\boldsymbol{2}^{a_{j}-\beta}, \ldots, \boldsymbol{2}^{ a_{i}}) & +&  \delta_{c_{j+1}=1} \zeta^{\star\star, \mathfrak{l}} (\boldsymbol{2}^{a_{i}-\beta}, \ldots, \boldsymbol{2}^{ a_{j}})\\
   + & \delta_{c_{i+1}=1 \atop \beta> a_{i}} \zeta^{\star\star, \mathfrak{l}}_{1} (\boldsymbol{2}^{a_{i}+a_{j}-\beta}, \ldots, \boldsymbol{2}^{ a_{i+1}})  & - & \delta_{c_{j}=1 \atop \beta> a_{j}} \zeta^{\star\star, \mathfrak{l}}_{1} (\boldsymbol{2}^{a_{i}+a_{j}-\beta}, \ldots, \boldsymbol{2}^{ a_{j-1}})\\
   - & \delta_{a_{j}< \beta \leq a_{i}+ a_{j}+1}  \zeta^{\star\star, \mathfrak{l}}_{c_{j}-2} (\boldsymbol{2}^{a_{j-1}}, \ldots, \boldsymbol{2}^{a_{i}+ a_{j}-\beta+1})  &  + & \delta_{a_{i}< \beta \leq a_{j}+a_{i}+1} \zeta^{\star\star, \mathfrak{l}}_{c_{i+1}-2} (\boldsymbol{2}^{a_{i+1}}, \ldots, \boldsymbol{2}^{ a_{i}+ a_{j} -\beta+1}) .
\end{array}$$
It should correspond to (using still lemma $A.1.4$), with $B_{k}=2a_{k}+3-\delta_{c_{k}}-\delta_{c_{k+1}}$, $\gamma_{k}=c_{k}-3+2\delta_{c_{k}}$ and $B=2\beta+3 - \delta_{c_{i}}- \delta_{c_{j+1}}$:
$$\left( \delta_{B_{i}< B}\zeta^{\sharp\sharp,\mathfrak{l}}_{B_{i}+B_{j}-B}(1^{\gamma_{j}}, \ldots, 1^{\gamma_{i+1}}) - \delta_{B_{j}< B}\zeta^{\sharp\sharp,\mathfrak{l}}_{B_{i}+B_{j}-B}(1^{\gamma_{i+1}}, \ldots, 1^{\gamma_{j}}) \right.  $$
$$\left. + \zeta^{\sharp\sharp,\mathfrak{l}}_{B_{i}-B}(1^{\gamma_{i+1}}, \ldots, B_{j}) - \zeta^{\sharp\sharp,\mathfrak{l}}_{B_{j}-B}(1^{\gamma_{j}}, \ldots, B_{i})\right) .$$
The first line has even depth, while the second line has odd depth, as noticed in Lemma $A.1.4$.
Let distinguish three cases, and assume $a_{i}<a_{j}$:\footnote{The case $a_{j}<a_{i}$ is anti-symmetric, hence analogue.}
\begin{itemize}

\item[$(i)$]  When $\beta< a_{i}<a_{j}$, we should have:
\begin{equation}\label{eq:ci}
 \zeta^{\sharp\sharp,\mathfrak{l}}_{B_{i}-B}(1^{\gamma_{i+1}}, \ldots, B_{j}) - \zeta^{\sharp\sharp,\mathfrak{l}}_{B_{j}-B}(1^{\gamma_{j}}, \ldots, B_{i})   \text{ equal to:} 
\end{equation} 
$$\begin{array}{llll}
-  \delta_{c_{i} >3} & \zeta^{\star\star, \mathfrak{l}}_{2} (\boldsymbol{2}^{a_{j}-\beta-1}, \ldots, \boldsymbol{2}^{ a_{i}})  & - \delta_{c_{i}=1}  &  \zeta^{\star\star, \mathfrak{l}} (\boldsymbol{2}^{a_{j}-\beta}, \ldots, \boldsymbol{2}^{ a_{i}}) \\
+\delta_{c_{j+1}>3} & \zeta^{\star\star, \mathfrak{l}}_{2} (\boldsymbol{2}^{a_{i}-\beta-1}, \ldots, \boldsymbol{2}^{ a_{j}})  & + \delta_{c_{j+1}=1}  & \zeta^{\star\star, \mathfrak{l}} (\boldsymbol{2}^{a_{i}-\beta}, \ldots, \boldsymbol{2}^{ a_{j}}) 
\end{array}$$
\begin{itemize}
\item[$\cdot$] Let first look at the case where $c_{i}>3, c_{j+1}>3$. Renumbering the indices, using $\textsc{Shift}$ for odd depth for the second line, it is equivalent to, with $\alpha=\beta +1$, $B_{p}=2a_{p}+3, B_{0}=2a_{0}+3$:
$$\begin{array}{llll}
&\zeta^{\star\star, \mathfrak{l}}_{2} (\boldsymbol{2}^{a_{0}-\alpha},c_{1},\cdots,c_{p},\boldsymbol{2}^{a_{p}}) & -& \zeta^{\star\star, \mathfrak{l}}_{2} (\boldsymbol{2}^{a_{0}},c_{1},\cdots,c_{p},\boldsymbol{2}^{a_{p}-\alpha}) \\
\equiv & \zeta^{\sharp\sharp,  \mathfrak{l}} _{B_{0}-B}(1^{\gamma_{1}},\cdots, 1^{\gamma_{p}},B_{p}) & - & \zeta^{\sharp\sharp,  \mathfrak{l}} _{B_{p}-B}(B_{0}, 1^{\gamma_{1}},\cdots, 1^{\gamma_{p}})\\
\equiv & \zeta^{\sharp\sharp,  \mathfrak{l}} _{B_{p}-1}(B_{0}-B+1,1^{\gamma_{1}},\cdots, 1^{\gamma_{p}}) & -& \zeta^{\sharp\sharp,\mathfrak{l}} _{B_{p}-B}(B_{0}, 1^{ \gamma_{1}},\cdots, 1^{\gamma_{p}})\\
\equiv & \zeta^{\sharp,  \mathfrak{l}} _{B_{p}-1}(B_{0}-B+1,1^{\gamma_{1}},\cdots, 1^{\gamma_{p}}) & -& \zeta^{\sharp,\mathfrak{l}} _{B_{p}-B}(B_{0}, 1^{ \gamma_{1}},\cdots, 1^{\gamma_{p}}).
\end{array}$$
This boils down to $(\ref{eq:conjid})$ applied to each $\zeta^{\star\star}_{2}$, since by \textsc{Shift} $(\ref{eq:shift})$ the two terms of the type $\zeta^{\star\star}_{1}$ get simplified.\\
\item[$\cdot$] Let now look at the case where $c_{i}=1, c_{j+1}>3$ \footnote{The case $c_{j+1}=1, c_{i}>3$ being analogue, by symmetry.}; hence $B_{i}=2a_{i}+2-\delta_{c_{i+1}}$, $B=2\beta+2$. In a first hand, we have to consider:
$$ \zeta^{\star\star, \mathfrak{l}}_{2} (\boldsymbol{2}^{a_{i}-\beta-1},c_{i+1},\cdots,c_{j},\boldsymbol{2}^{a_{j}}) - \zeta^{\star\star, \mathfrak{l}} (\boldsymbol{2}^{a_{j}-\beta},c_{j},\cdots,c_{i+1},\boldsymbol{2}^{a_{i}}).$$
By renumbering indices in $\ref{eq:ci}$, the correspondence boils down here to the following $\boldsymbol{\diamond} = \boldsymbol{\Join}$, where $B_{0}=2a_{0}+3-\delta_{c_{1}}$, $B_{i}=2a_{i}+3-\delta_{c_{i}}-\delta_{c_{i+1}}$, $B=2\beta +2$:
$$ (\boldsymbol{\diamond}) \quad  \zeta^{\star\star, \mathfrak{l}}_{2} (\boldsymbol{2}^{a_{0}-\beta},c_{1},\cdots,c_{p},\boldsymbol{2}^{a_{p}}) - \zeta^{\star\star, \mathfrak{l}} (\boldsymbol{2}^{a_{0}+1},c_{1},\cdots,c_{p},\boldsymbol{2}^{a_{p}-\beta})$$
$$(\boldsymbol{\Join}) \quad \zeta^{\sharp\sharp,\mathfrak{l}}_{B_{0}-B+1}(1^{\gamma_{1}}, \ldots,1^{\gamma_{p}}, B_{p}) - \zeta^{\sharp\sharp,\mathfrak{l}}_{B_{p}-B}(1^{\gamma_{p}}, \ldots,1^{\gamma_{1}}, B_{0}+1).$$
Turning in $(\boldsymbol{\diamond})$ the second term into a $\zeta^{\star, \mathfrak{l}}(2, \cdots)+ \zeta^{\star\star, \mathfrak{l}}_{2} (\cdots)$, and applying the identity $(\ref{eq:conjid})$ for both terms $\zeta^{\star\star, \mathfrak{l}}_{2}(\cdots)$ leads to:
$$\hspace*{-0.7cm}(\boldsymbol{\diamond}) \left\lbrace \begin{array}{lll}
 + \zeta^{\star\star, \mathfrak{l}}_{1} (\boldsymbol{2}^{a_{0}-\beta+1},c_{1}-1,\cdots,c_{p},\boldsymbol{2}^{a_{p}}) & - \zeta^{\star\star, \mathfrak{l}}_{1} (\boldsymbol{2}^{a_{0}+1},c_{1}-1,\cdots,c_{p},\boldsymbol{2}^{a_{p}-\beta})& \quad (\boldsymbol{\diamond_{1}})  \\
+ \zeta^{\sharp, \mathfrak{l}}_{B_{0}-B+1} (\boldsymbol{1}^{\gamma_{1}},\cdots,\boldsymbol{1}^{\gamma_{p}},B_{p}) &- \zeta^{\sharp, \mathfrak{l}}_{B_{0}-1} (\boldsymbol{1}^{\gamma_{1}},\cdots,\boldsymbol{1}^{\gamma_{p}},B_{p}-B+2)& \quad(\boldsymbol{\diamond_{2}})  \\
- \zeta^{\star, \mathfrak{l}} (\boldsymbol{2}^{a_{0}+1},c_{1},\cdots,c_{p},\boldsymbol{2}^{a_{p}-\beta}) &  & \quad (\boldsymbol{\diamond_{3}}) \\
\end{array} \right. $$
The first line, $(\boldsymbol{\diamond}_{1}) $ by $\textsc{Shift}$ is zero. We apply $\textsc{Antipode}$ $\ast$ on the terms of the second line, then turn each into a difference $\zeta^{\sharp\sharp}_{n}(m, \cdots)- \zeta^{\sharp\sharp}_{n+m}(\cdots)$; the terms of the type $\zeta^{\sharp\sharp}_{n+m}(\cdots)$, are identical and get simplified:

$$(\boldsymbol{\diamond_{2}}) \quad \begin{array}{lll}
\equiv & \zeta^{\sharp\sharp, \mathfrak{l}}_{B_{0}-B+1} (B_{p},\boldsymbol{1}^{\gamma_{p}},\cdots,\boldsymbol{1}^{\gamma_{1}}) & - \zeta^{\sharp\sharp, \mathfrak{l}}_{B_{0}-B+1+ B_{p}} (\boldsymbol{1}^{\gamma_{p}},\cdots,\boldsymbol{1}^{\gamma_{1}}) \\
& -\zeta^{\sharp\sharp, \mathfrak{l}}_{B_{0}-1} (B_{p}-B+2,\boldsymbol{1}^{\gamma_{p}},\cdots,\boldsymbol{1}^{\gamma_{1}}) & + \zeta^{\sharp\sharp, \mathfrak{l}}_{B_{0}-B+1+ B_{p}} (\boldsymbol{1}^{\gamma_{p}},\cdots,\boldsymbol{1}^{\gamma_{1}}) \\
\equiv & \zeta^{\sharp\sharp, \mathfrak{l}}_{B_{0}-B+1} (B_{p},\boldsymbol{1}^{\gamma_{p}},\cdots,\boldsymbol{1}^{\gamma_{1}}) & - \zeta^{\sharp\sharp, \mathfrak{l}}_{B_{0}-1} (B_{p}-B+2,\boldsymbol{1}^{\gamma_{p}},\cdots,\boldsymbol{1}^{\gamma_{1}}).
\end{array}$$
Furthermore, applying the recursion hypothesis (I.), i.e. conjecture $\ref{lzg}$ on $(\boldsymbol{\diamond}_{3})$, and turn it into a difference of $\zeta^{\sharp\sharp}$:
$$(\boldsymbol{\diamond_{3}})\quad \begin{array}{ll}
& - \zeta^{\star, \mathfrak{l}} (\boldsymbol{2}^{a_{0}+1},c_{1},\cdots,c_{p},\boldsymbol{2}^{a_{p}-\beta})\\
\equiv & - \zeta^{\sharp, \mathfrak{l}} (B_{p}-B+1,\boldsymbol{1}^{\gamma_{p}},\cdots,\boldsymbol{1}^{\gamma_{1}},B_{0})\\
\equiv & - \zeta^{\sharp\sharp, \mathfrak{l}} (B_{p}-B+1,\boldsymbol{1}^{\gamma_{p}},\cdots,\boldsymbol{1}^{\gamma_{1}},B_{0}) + \zeta^{\sharp\sharp, \mathfrak{l}}_{B_{p}-B+1} (\boldsymbol{1}^{\gamma_{p}},\cdots,\boldsymbol{1}^{\gamma_{1}},B_{0})
\end{array}$$
When adding $(\boldsymbol{\diamond_{2}})$ and $(\boldsymbol{\diamond_{3}})$ to get $(\boldsymbol{\diamond})$, the two last terms (odd depth) being simplified by $\textsc{Shift}$, it remains:
$$(\boldsymbol{\diamond}) \quad  \zeta^{\sharp\sharp, \mathfrak{l}}_{B_{0}-B+1} (B_{p},\boldsymbol{1}^{\gamma_{p}},\cdots,\boldsymbol{1}^{\gamma_{1}}) - \zeta^{\sharp\sharp, \mathfrak{l}} (B_{p}-B+1,\boldsymbol{1}^{\gamma_{p}},\cdots,\boldsymbol{1}^{\gamma_{1}},B_{0}). $$
This, applying \textsc{Antipode} $\ast$ to the first term, $\textsc{Cut}$ and $\textsc{Shift}$ to the second, corresponds to $(\boldsymbol{\Join})$.\\
\end{itemize}
\item[$(ii)$] When $\beta >a_{j}>a_{i}$, we should have:
$$\begin{array}{lll}
& - \zeta^{\star\star, \mathfrak{l}}_{c_{j}-2} (\boldsymbol{2}^{a_{j-1}}, \ldots, \boldsymbol{2}^{a_{i}+ a_{j}-\beta+1}) & + \zeta^{\star\star, \mathfrak{l}}_{c_{i+1}-2} (\boldsymbol{2}^{a_{i+1}}, \ldots, \boldsymbol{2}^{ a_{i}+ a_{j} -\beta+1})\\
\equiv & + \zeta^{\sharp\sharp,\mathfrak{l}}_{B_{i}+B_{j}-B}(1^{\gamma_{j}}, \ldots, 1^{\gamma_{i+1}}) & -\zeta^{\sharp\sharp,\mathfrak{l}}_{B_{i}+B_{j}-B}(1^{\gamma_{i+1}}, \ldots, 1^{\gamma_{j}}).
\end{array}$$
Using \textsc{Shift} $(\ref{eq:shift})$ for the first line, and renumbering the indices, it is equivalent to, with $c_{1},c_{p} \geq 3$ and $a_{0}>0$:
\begin{equation} \label{eq:corresp3}
\zeta^{\star\star, \mathfrak{l}}_{1} (\boldsymbol{2}^{a_{0}},c_{1}-1,\cdots,c_{p})- \zeta^{\star\star, \mathfrak{l}}_{1} (\boldsymbol{2}^{a_{0}},c_{1},\cdots,c_{p}-1)
\end{equation}
$$ \equiv \zeta^{\sharp\sharp,  \mathfrak{l}} _{B_{0}+2}(1^{\gamma_{1}},\cdots, 1^{\gamma_{p}})-\zeta^{\sharp\sharp,  \mathfrak{l}} _{B_{0}+2}(1^{\gamma_{p}},\cdots, 1^{\gamma_{1}}) \equiv \zeta^{\sharp,  \mathfrak{l}} _{B_{0}+2}(1^{\gamma_{1}}, \ldots, 1^{\gamma_{p}}).$$
The last equality comes from Corollary $4.2.7$, since depth is even. By $(\ref{eq:corresp})$ applied on each term of the first line
$$\zeta^{\star\star, \mathfrak{l}}_{1} (\boldsymbol{2}^{a_{0}},c_{1}-1,\cdots,c_{p})- \zeta^{\star\star, \mathfrak{l}}_{1} (\boldsymbol{2}^{a_{0}},c_{1},\cdots,c_{p}-1)$$
$$\hspace*{-1.5cm} \equiv \zeta^{\star\star, \mathfrak{l}}_{2} (\boldsymbol{2}^{a_{0}-1},c_{1},\cdots,c_{p})+ \zeta^{\sharp  \mathfrak{l}} _{2a_{0}}(3,1^{\gamma_{p}},\cdots, 1^{\gamma_{1}}) - \zeta^{\star\star, \mathfrak{l}}_{2} (c_{p},\cdots,c_{1},\boldsymbol{2}^{a_{0}-1}) - \zeta^{\sharp\sharp,  \mathfrak{l}}_{2}(2a_{0}+1,1^{\gamma_{p}},\cdots, 1^{\gamma_{1}}).$$
By Antipode $\shuffle$, the $\zeta^{\star\star}$ get simplified, and by the definition of $\zeta^{\sharp\sharp}$, the previous equality is equal to: 
$$\equiv- \zeta^{\sharp  \mathfrak{l}} _{2a_{0}+3}(1^{\gamma_{p}},\cdots, 1^{\gamma_{1}}) + \zeta^{\sharp\sharp  \mathfrak{l}} _{2a_{0}}(3,1^{\gamma_{p}},\cdots, 1^{\gamma_{1}}) + \zeta^{\sharp\sharp  \mathfrak{l}} _{2a_{0}+4}(1^{\gamma_{p}-1},\cdots, 1^{\gamma_{1}})  $$
$$ - \zeta^{\sharp\sharp  \mathfrak{l}} _{2}(2a_{0}+1, 1^{\gamma_{p}},\cdots, 1^{\gamma_{1}}) + \zeta^{\sharp\sharp  \mathfrak{l}} _{2a_{0}+3}(1^{\gamma_{r}},\cdots, 1^{\gamma_{1}}).$$
Then, by \textsc{Shift} $(\ref{eq:shift})$, the second and fourth term get simplified while the third and fifth term get simplified by \textsc{Cut} $(\ref{eq:cut})$. It remains:
$$- \zeta^{\sharp , \mathfrak{l}} _{2a_{0}+3}(1^{\gamma_{p}},\cdots, 1^{\gamma_{1}}), \quad \text{ which leads straight to } \ref{eq:corresp3}.$$
\item[$(iii)$] When $a_{i}< \beta <a_{j}$, we should have:
\begin{multline}\nonumber
 - \zeta^{\star\star, \mathfrak{l}}_{2} (\boldsymbol{2}^{a_{i}}, \ldots, \boldsymbol{2}^{a_{j}-\beta-1})  +  \zeta^{\star\star, \mathfrak{l}}_{c_{i+1}-2} (\boldsymbol{2}^{a_{i+1}}, \ldots, \boldsymbol{2}^{a_{i}+ a_{j} -\beta+1})  \\
\equiv \zeta^{\sharp\sharp,\mathfrak{l}}_{B_{i}+B_{j}-B}(1^{\gamma_{j}}, \ldots, 1^{\gamma_{i+1}})-\zeta^{\sharp\sharp,\mathfrak{l}}_{B_{j}-B}(1^{\gamma_{j}}, \ldots, B_{i}). 
\end{multline}
Using resp. \textsc{Antipode} \textsc{Shift} $(\ref{eq:shift})$ for the first line, and re-ordering the indices, it is equivalent to, with $c_{1}\geq 3$, $B_{0}=2a_{0}+1-\delta c_{1}$ here:
\begin{equation} \label{eq:corresp} \zeta^{\star\star, \mathfrak{l}}_{2} (\boldsymbol{2}^{a_{0}-1},c_{1},\cdots,c_{p},\boldsymbol{2}^{a_{p}})- \zeta^{\star\star, \mathfrak{l}}_{1} (\boldsymbol{2}^{a_{0}},c_{1}-1,\cdots,c_{p},\boldsymbol{2}^{a_{p}})  
\end{equation}
$$\equiv \zeta^{\sharp\sharp,  \mathfrak{l}} _{B_{p}-1}(B_{0}, 1^{\gamma_{1}},\cdots, 1^{\gamma_{p}}) - \zeta^{\sharp\sharp,  \mathfrak{l}} _{B_{0}+B_{p}-1}(1^{ \gamma_{p}},\cdots, 1^{\gamma_{1}}) \equiv  \zeta^{\sharp,  \mathfrak{l}} _{B_{0}-1}(1^{\gamma_{1}},\cdots, 1^{\gamma_{p}},B_{p}).$$
This matches with the identity $\ref{eq:conjid}$; the last equality coming from $\textsc{Shift}$ since depth is odd.
\end{itemize}
\item Antisymmetric of the first case.\\
\end{enumerate}

\item Let us denote the sequences $\textbf{X}=\boldsymbol{2}^{a_{1}}, \ldots , \boldsymbol{2}^{a_{p}}$ and $\textbf{Y}= \boldsymbol{1}^{\gamma_{1}-1}, B_{1},\cdots,  \boldsymbol{1}^{\gamma_{p}} $.\\
We want to prove that:
\begin{equation} \label{eq:1234567}
\zeta^{\sharp,\mathfrak{l}} (1,\textbf{Y},B_{p})\equiv -\zeta^{\star,\mathfrak{l}}_{1} (c_{1}-1,\textbf{X})
\end{equation}
Relations used are mostly these stated in $\S 4.2$. Using the definition of $\zeta^{\star\star}$:
\begin{equation}\label{eq:12345}
\begin{array}{ll}
-\zeta^{\star,\mathfrak{l}}_{1} (c_{1}-1,\textbf{X}) & \equiv -\zeta^{\star\star,\mathfrak{l}}_{1} (c_{1}-1,\textbf{X})+ \zeta^{\star\star,\mathfrak{l}}_{c_{1}} (\textbf{X})\\
& \equiv - \zeta^{\star\star,\mathfrak{l}} (1,c_{1}-1,\textbf{X})+ \zeta^{\star,\mathfrak{l}} (1,c_{1}-1,\textbf{X})+ \zeta^{\star\star,\mathfrak{l}}(c_{1},\textbf{X})- \zeta^{\star,\mathfrak{l}} (c_{1},\textbf{X}) \\
& \equiv \zeta^{\star,\mathfrak{l}} (1,c_{1}-1,\textbf{X})- \zeta^{\star,\mathfrak{l}} (c_{1},\textbf{X})- \zeta^{\star,\mathfrak{l}}(c_{1}-1,\textbf{X},1).
\end{array}
\end{equation}
There, the first and third term in the second line, after applying \textsc{Shift}, have given the last $\zeta^{\star}$ in the last line.\\
Using then Conjecture $\ref{lzg}$, in terms of MMZV$^{\sharp}$, then MMZV$^{\sharp\sharp}$, it gives:
\begin{multline}
\zeta^{\sharp,\mathfrak{l}} (2,\textbf{Y},B_{p}-1)+ \zeta^{\sharp,\mathfrak{l}}  (1,1,\textbf{Y},B_{p}-1)+ \zeta^{\sharp,\mathfrak{l}} (1,\textbf{Y},B_{p}-1,1)\\
 \equiv \zeta^{\sharp\sharp,\mathfrak{l}} (2,\textbf{Y},B_{p}-1)- \zeta^{\sharp\sharp,\mathfrak{l}} _{2}(\textbf{Y},B_{p}-1)+ \zeta^{\sharp\sharp,\mathfrak{l}}  (1,1,\textbf{Y},B_{p}-1)\\
-\zeta^{\sharp\sharp,\mathfrak{l}}_{1} (1,\textbf{Y},B_{p}-1)+ \zeta^{\sharp\sharp,\mathfrak{l}} (1,\textbf{Y},B_{p}-1,1)-\zeta^{\sharp\sharp,\mathfrak{l}}_{1} (\textbf{Y},B_{p}-1,1)   
\end{multline}
First term (odd depth)\footnote{Since weight is odd, we know also depth parity of these terms.} is simplified with the last, by $\textsc{Schift}$. Fifth term (even depth) get simplified by \textsc{Cut} with the fourth term. Hence it remains two terms of even depth:
$$\equiv - \zeta^{\sharp\sharp,\mathfrak{l}} _{2}(\textbf{Y},B_{p}-1)+ \zeta^{\sharp\sharp,\mathfrak{l}}  (1,1,\textbf{Y},B_{p}-1) \equiv - \zeta^{\sharp\sharp,\mathfrak{l}} _{1}(\textbf{Y},B_{p})+ \zeta^{\sharp\sharp,\mathfrak{l}}_{B_{p}-1}  (1,1,\textbf{Y}) , $$
where \textsc{Minus} resp. \textsc{Cut} have been applied. This matches with $(\ref{eq:1234567})$ since, by $\textsc{Shift}:$
$$\equiv - \zeta^{\sharp\sharp,\mathfrak{l}} _{1}(\textbf{Y},B_{p})+ \zeta^{\sharp\sharp,\mathfrak{l}} (1,\textbf{Y},B_{p})\equiv \zeta^{\sharp,\mathfrak{l}}(1,\textbf{Y},B_{p}). $$
The case $c_{1}=3$ slightly differs since $(\ref{eq:12345})$ gives, by recursion hypothesis I.($\ref{lzg}$):
$$ -\zeta^{\star,\mathfrak{l}}_{1} (2,\textbf{X})\equiv \zeta^{\sharp,\mathfrak{l}} (B_{1}+1,\textbf{Y}',B_{p}-1)+ \zeta^{\sharp,\mathfrak{l}}  (1,B_{1},\textbf{Y}',B_{p}-1)+ \zeta^{\sharp,\mathfrak{l}} (B_{1},\textbf{Y}',B_{p}-1,1),$$
where $\textbf{Y}'= \boldsymbol{1}^{\gamma_{2}},\cdots,  \boldsymbol{1}^{\gamma_{p}} $, odd depth.
Turning into MES$^{\sharp\sharp}$, and using identities of $\S 4.2$ in the same way than above, leads to the result. Indeed, from:
$$\equiv \zeta^{\sharp\sharp,\mathfrak{l}} (B_{1}+1,\textbf{Y}',B_{p}-1)+ \zeta^{\sharp\sharp,\mathfrak{l}}  (1,B_{1},\textbf{Y}',B_{p}-1)+ \zeta^{\sharp\sharp,\mathfrak{l}} (B_{1},\textbf{Y}',B_{p}-1,1)$$
$$-\zeta^{\sharp\sharp,\mathfrak{l}}_{B_{1}+1} (\textbf{Y}',B_{p}-1)- \zeta^{\sharp\sharp,\mathfrak{l}}_{1}  (B_{1},\textbf{Y}',B_{p}-1)-\zeta^{\sharp\sharp,\mathfrak{l}}_{B_{1}} (\textbf{Y}',B_{p}-1,1)$$
First and last terms get simplified via $\textsc{Shift}$, while third and fifth term get simplified by $\textsc{Cut}$; besides, we apply \textsc{minus}  for second term, and \textsc{minus} for the fourth term, which are both of even depth. This leads to $\ref{eq:toolid}$, using again $\textsc{Shift}$ for the first term:
$$\begin{array}{l}
\equiv\zeta^{\sharp\sharp,\mathfrak{l}}_{B_{p}-1}  (1,B_{1},\textbf{Y}')-\zeta^{\sharp\sharp,\mathfrak{l}}_{B_{1}} (\textbf{Y}',B_{p}) \\
\equiv \zeta^{\sharp\sharp,\mathfrak{l}} (B_{1},\textbf{Y}',B_{p})-\zeta^{\sharp\sharp,\mathfrak{l}}_{B_{1}} (\textbf{Y}',B_{p}) \\
\equiv \zeta^{\sharp,\mathfrak{l}} (B_{1},\textbf{Y}',B_{p}).
\end{array}$$
\end{enumerate}
\end{proof}

\section{Appendix $1$: From the linearized octagon relation}

The identities in the coalgebra $\mathcal{L}$ obtained from the linearized octagon relation $\ref{eq:octagonlin}$:

\begin{lemm}\label{lemmlor}
In the coalgebra $\mathcal{L}$, $n_{i}\in\mathbb{Z}^{\ast}$:\footnote{Here, $\mlq + \mrq$ still denotes the operation where absolute values are summed and signs multiplied.}
\begin{itemize}
\item[$(i)$] $\zeta^{\star\star, \mathfrak{l}}(n_{0},\cdots, n_{p})= (-1)^{w+1} \zeta^{\star\star,\mathfrak{l}}(n_{p},\cdots, n_{0})$.
\item[$(ii)$] $\zeta^{\mathfrak{l}}(n_{0},\cdots, n_{p})+(-1)^{w+p} \zeta^{\star\star\mathfrak{l}}(n_{0},\cdots, n_{p})+(-1)^{p} \zeta^{\star\star\mathfrak{l}}_{n_{p}}(n_{p-1},\cdots,n_{1},n_{0})=0$.
\item[$(iii)$] 
$$\hspace*{-1cm}\zeta^{\mathfrak{l}}_{n_{0}-1}(n_{1},\cdots, n_{p})-  \zeta^{\mathfrak{l}}_{n_{0}}(n_{1},\cdots,n_{p-1}, n_{p}\mlq + \mrq 1 )=(-1)^{w} \left[ \zeta^{\star\star,\mathfrak{l}}_{n_{0}-1}(n_{1},\cdots, n_{p})-  \zeta^{\star\star,\mathfrak{l}}_{n_{0}}(n_{1},\cdots,n_{p-1}, n_{p}\mlq + \mrq 1)\right].$$
\end{itemize}
\end{lemm}
\begin{proof}
The sign of $n_{i}$ is denoted $\epsilon_{i}$ as usual. First, we remark that, with $\eta_{i}=\pm 1$, $n_{i}=\epsilon_{i} (a_{i}+1)$, and $\epsilon_{i}= \eta_{i}\eta_{i+1}$:
$$\hspace*{-0.5cm}\begin{array}{ll}
\Phi^{\mathfrak{m}}(e_{\infty}, e_{-1},e_{1}) & = \sum  I^{\mathfrak{m}} \left(0;  (-\omega_{0})^{a_{0}} (-\omega_{-\eta_{1}\star})  (-\omega_{0})^{a_{1}} \cdots (-\omega_{-\eta_{p}\star})  (-\omega_{0})^{a_{p}} ;1 \right)  e_{0}^{a_{0}}e_{\eta_{1}} e_{0}^{a_{1}} \cdots e_{\eta_{p}} e_{0}^{a_{p}}\\
& \\
& = \sum  (-1)^{n+p}\zeta^{\star\star,\mathfrak{m}}_{n_{0}-1} \left( n_{1}, \cdots, n_{p-1}, -n_{p}\right)  e_{0}^{a_{0}}e_{\eta_{1}} e_{0}^{a_{1}} \cdots e_{\eta_{p}} e_{0}^{a_{p}}. \\
\end{array}$$
Similarly, with $ \mu_{i}\mathrel{\mathop:}= \left\lbrace \begin{array}{ll}
\star & \texttt{if } \eta_{i}=1\\
1 & \texttt{if } \eta_{i}=-1
\end{array}  \right. $, applying the homography $\phi_{\tau\sigma}$ to get the second line:
$$\hspace*{-0.5cm}\begin{array}{ll}
\Phi^{\mathfrak{m}}(e_{-1}, e_{0},e_{\infty}) & = \sum I^{\mathfrak{m}} \left(0;  (\omega_{1}-\omega_{-1})^{a_{0}} \omega_{\mu_{1}}  (\omega_{1}-\omega_{-1})^{a_{1}} \cdots \omega_{\mu_{p}} (\omega_{1}-\omega_{-1})^{a_{p}} ;1 \right)  e_{0}^{a_{0}}e_{\eta_{1}} e_{0}^{a_{1}} \cdots e_{\eta_{p}} e_{0}^{a_{p}}\\
& \\
\Phi^{\mathfrak{l}}(e_{-1}, e_{0},e_{\infty}) & = \sum (-1)^{p} I^{\mathfrak{m}} \left(0;  0^{a_{0}} \omega_{-\eta_{1}}  0^{a_{1}} \cdots \omega_{-\eta_{p}} 0^{a_{p}} ;1 \right)  e_{0}^{a_{0}}e_{\eta_{1}} e_{0}^{a_{1}} \cdots e_{\eta_{p}} e_{0}^{a_{p}}\\
& \\
& = \sum  \zeta^{\mathfrak{m}}_{n_{0}-1} \left( n_{1}, \cdots, n_{p-1}, -n_{p}\right)  e_{0}^{a_{0}}e_{\eta_{1}} e_{0}^{a_{1}} \cdots e_{\eta_{p}} e_{0}^{a_{p}}. \\
\end{array}$$
Lastly, still using  $\phi_{\tau\sigma}$, with here $\mu_{i}\mathrel{\mathop:}= \left\lbrace \begin{array}{ll}
\star & \texttt{if } \eta_{i}=1\\
1 & \texttt{if } \eta_{i}=-1
\end{array}  \right. $:
$$\hspace*{-0.5cm}\begin{array}{ll}
\Phi^{\mathfrak{m}}(e_{1}, e_{\infty},e_{0}) & = \sum  I^{\mathfrak{m}} \left(0;  (\omega_{-1}-\omega_{1})^{a_{0}} \omega_{\mu_{1}}  (\omega_{-1}-\omega_{1})^{a_{1}} \cdots \omega_{\mu_{p}} (\omega_{-1}-\omega_{1})^{a_{p}} ;1 \right)  e_{0}^{a_{0}}e_{\eta_{1}} e_{0}^{a_{1}} \cdots e_{\eta_{p}} e_{0}^{a_{p}}\\
& \\
\Phi^{\mathfrak{l}}(e_{1}, e_{\infty},e_{0}) & = \sum  (-1)^{w+1} I^{\mathfrak{m}} \left(0;  0^{a_{0}} \omega_{\eta_{1}\star}  0^{a_{1}} \cdots \omega_{\eta_{p}\star} 0^{a_{p}} ;1 \right)  e_{0}^{a_{0}}e_{\eta_{1}} e_{0}^{a_{1}} \cdots e_{\eta_{p}} e_{0}^{a_{p}}\\
& \\
& = \sum  (-1)^{n+p+1}\zeta^{\star\star,\mathfrak{m}}_{n_{0}-1} \left( n_{1}, \cdots, n_{p-1}, n_{p}\right)  e_{0}^{a_{0}}e_{\eta_{1}} e_{0}^{a_{1}} \cdots e_{\eta_{p}} e_{0}^{a_{p}}. \\
\end{array}$$

\begin{itemize}
\item[$(i)$] 
This case is the one used in Theorem $\ref{hybrid}$. This identity is equivalent to, in terms of iterated integrals, for $X$ any sequence of $\left\lbrace  0, \pm 1 \right\rbrace $ or of $\left\lbrace  0, \pm \star \right\rbrace $:
$$\left\lbrace  \begin{array}{llll}
I^{\mathfrak{l}}(0;0^{k}, \star, X ;1) & = & I^{\mathfrak{l}}(0; X, \star, 0^{k}; 1)  & \text{ if } \prod_{i=0}^{p} \epsilon_{i}=1 \Leftrightarrow \eta_{0}=1\\
I^{\mathfrak{l}}(0;0^{k}, -\star, X ;1) & = & I^{\mathfrak{l}}(0; -X, -\star, 0^{k}; 1)  & \text{ if } \prod_{i=0}^{p} \epsilon_{i}=-1 \Leftrightarrow \eta_{0}=-1\\
\end{array} \right. $$
The first case is deduced from $\ref{eq:octagonlin}$ when looking at the coefficient of a word beginning and ending by $e_{-1}$ (or beginning and ending by $e_{1}$), whereas the second case is obtained from the coefficient of a word beginning by $e_{-1}$ and ending by $e_{1}$, or beginning by $e_{1}$ and ending by $e_{-1}$.

\item[$(ii)$] Let split into two cases, according to the sign of $\prod \epsilon_{i}$:
\begin{itemize}
\item[$\cdot$]  In $\ref{eq:octagonlin}$, when looking at the coefficient of a word beginning by $e_{1}$ and ending by $e_{0}$, only these three terms contribute:
$$ \Phi^{\mathfrak{l}}(e_{-1}, e_{0},e_{\infty})e_{0}- \Phi^{\mathfrak{l}}(e_{\infty}, e_{-1},e_{1})e_{0}- e_{1} \Phi^{\mathfrak{l}}(e_{1}, e_{\infty},e_{0}) .$$
Moreover, the coefficient of $e_{-1} e_{0}^{a_{0}} e_{\eta_{1}} \cdots  e_{\eta_{p}} e_{0}^{a_{p}+1}$ is, using the expressions above for $\Phi^{\mathfrak{l}}(\cdot)$:
\begin{multline}\nonumber
 (-1)^{p} I^{\mathfrak{l}}(0; -1, -X; 1)+ (-1)^{w+1} I^{\mathfrak{l}}(0; -\star, -X_{\star}; 1)+ (-1)^{w}I^{\mathfrak{l}}(0; X_{\star}, 0; 1)=0, \\
\text{where }\begin{array}{l}
X:= \omega_{0}^{a_{0}} \omega_{\eta_{1}} \cdots  \omega_{\eta_{p}} \omega_{0}^{a_{p}}\\
X_{\star}:= \omega_{0}^{a_{0}} \omega_{\eta_{1}\star} \cdots  \omega_{\eta_{p}\star} \omega_{0}^{a_{p}}
\end{array}.
\end{multline}
In terms of motivic Euler sums, it is, with $\prod \epsilon_{i}=1$:
$$ \zeta^{\mathfrak{l}} (n_{0},\cdots, -n_{p}) +(-1)^{w+p} \zeta^{\star\star\mathfrak{l}}(n_{0},\cdots, -n_{p})+(-1)^{w+p} \zeta^{\star\star\mathfrak{l}}_{n_{0}-1}(n_{1},\cdots,n_{p-1}, n_{p}\mlq + \mrq 1)=0.$$
Changing $n_{p}$ into $-n_{p}$, and applying Antipode $\shuffle$ to the last term, it gives, with now $\prod \epsilon_{i}=-1$: 
$$ \zeta^{\mathfrak{l}} (n_{0},\cdots, n_{p}) +(-1)^{w+p} \zeta^{\star\star\mathfrak{l}}(n_{0},\cdots, n_{p})+(-1)^{p} \zeta^{\star\star\mathfrak{l}}_{n_{p}}(n_{p-1},\cdots,n_{1},n_{0})=0.$$ 
\item[$\cdot$] Similarly, for the coefficient of a word beginning by $e_{-1}$ and ending by $e_{0}$, only these three terms contribute:
$$ \Phi^{\mathfrak{l}}(e_{-1}, e_{0},e_{\infty})e_{0}- \Phi^{\mathfrak{l}}(e_{\infty}, e_{-1},e_{1})e_{0}+ e_{-1} \Phi^{\mathfrak{l}}(e_{\infty}, e_{-1},e_{1}) .$$
Similarly than above, it leads to the identity, with $\prod \epsilon_{i}=-1$:
$$ \zeta^{\mathfrak{l}} (n_{0},\cdots, -n_{p}) +(-1)^{w+p} \zeta^{\star\star\mathfrak{l}}(n_{0},\cdots, -n_{p})+(-1)^{w+p} \zeta^{\star\star\mathfrak{l}}_{n_{0}-1}(n_{1},\cdots,n_{p-1},-(n_{p}\mlq + \mrq 1))=0.$$
Changing $n_{p}$ into $-n_{p}$, and applying Antipode $\shuffle$ to the last term, it gives, with now $\prod \epsilon_{i}=1$: 
$$ \zeta^{\mathfrak{l}} (n_{0},\cdots, n_{p}) +(-1)^{w+p+1} \zeta^{\star\star\mathfrak{l}}(n_{0},\cdots, n_{p})+(-1)^{p} \zeta^{\star\star\mathfrak{l}}_{n_{p}}(n_{p-1},\cdots,n_{1},n_{0})=0.$$ 
\end{itemize}

\item[$(iii)$] When looking at the coefficient of a word beginning by $e_{0}$ and ending by $e_{0}$ in $\ref{eq:octagonlin}$, only these three terms contribute:
$$ -e_{0} \Phi^{\mathfrak{l}}(e_{-1}, e_{0},e_{\infty})+ \Phi^{\mathfrak{l}}(e_{-1}, e_{0},e_{\infty})e_{0}+ e_{0} \Phi^{\mathfrak{l}}(e_{\infty}, e_{-1},e_{1})- \Phi^{\mathfrak{l}}(e_{\infty}, e_{-1},e_{1})e_{0}.$$
If we identify the coefficient of the word $ e_{0}^{a_{0}+1} e_{-\eta_{1}} \cdots  e_{-\eta_{p}} e_{0}^{a_{p}+1}$, it leads straight to the identity $(iii)$.
\end{itemize}
\textsc{Remark}: Looking at the coefficient of words beginning by $e_{0}$ and ending by $e_{1}$ or $e_{-1}$ in $\ref{eq:octagonlin}$ would lead to the same identity than the second case.
\end{proof}

\section{Appendix $2$: Missing coefficients}

In Lemma $\ref{lemmcoeff}$, the coefficients $D_{a,b}$ appearing (in $(v)$) are the only one which are not conjectured. Albeit these values are not required for the proof of Theorem $4.4.1$, we provide here a table of values in small weights. Let examine the coefficient corresponding to $\zeta^{\star}(\boldsymbol{2}^{n})$ instead of $\zeta^{\star}(2)^{n}$, which is (by $(i)$ in Lemma $\ref{lemmcoeff}$), with $n=a+b+1$:
\begin{equation}
\widetilde{D}^{a,b}\mathrel{\mathop:}= \frac{(2n)!}{6^{n}\mid B_{2n}\mid (2^{2n}-2)} D^{a,b}  \quad \text{  and   }  \quad \widetilde{D}_{n}\mathrel{\mathop:}= \frac{(2n)!}{6^{n}\mid B_{2n}\mid (2^{2n}-2)}D_{n} .
\end{equation}
We have an expression $(\ref{eq:coeffds})$ for $D_{n}$, albeit not very elegant, which would give:
\begin{equation} \label{eq:coeffdstilde}
\widetilde{D}_{n}= \frac{2^{2n} (2n)!}{(2^{2n}-2)\mid B_{2n}\mid }\sum_{\sum m_{i} s_{i}=n \atop m_{i}\neq m_{j}} \prod_{i=1}^{k} \left( \frac{1}{s_{i}!} \left( \frac{\mid B_{2m_{i}}\mid (2^{2m_{i}-1}-1) } {2m_{i} (2m_{i})!}\right)^{s_{i}}  \right).
\end{equation}
Here is a table of values for $\widetilde{D}_{n}$ and $\widetilde{D}^{k,n-k-1}$ in small weights:\\
\\
  \begin{tabular}{| l || c | c |c | c |}
   \hline
     $\cdot \quad \quad \diagdown  n$ & $2$ &  $3$ & $4$ &  $5$   \\ \hline
     $\widetilde{D_{n}}$ & $\frac{19}{2^{3}-1}$ & $\frac{275}{2^{5}-1}$& $\frac{11813}{3(2^{7}-1)}$ & $\frac{783}{7}$\\
     & & & & \\ \hline
    $\widetilde{D}_{k,n-1-k}$ &$\frac{-12}{7}$ & $\frac{-84}{31}, \frac{160}{31}$& $\frac{1064}{127}, \frac{-1680}{127}, \frac{-9584}{381}$ & $\frac{189624}{2555}$,$\frac{-137104}{2555}$,$\frac{-49488}{511}$,$\frac{-17664}{511}$ \\ 
     & & & &\\
    \hline
     \hline
     $\cdot \quad \quad \diagdown  n$ & $6$ &  $7$ & $8$ &  $9$  \\ \hline
     $\widetilde{D_{n}}  \quad \quad $ & $\frac{581444793}{691(2^{11}-1)}$& $\frac{263101079}{21(2^{13}-1)}$& $\frac{6807311830555}{3617(2^{15}-1)}$& $\frac{124889801445461}{43867(2^{17}-1)}$\\ 
      & & & & \\
     \hline
  \end{tabular} \\
  \\
  \\
The denominators of  $\widetilde{D_{n}},\widetilde{D}_{k,n-1-k}$ can be written as $(2^{2n-1}-1)$ times the numerator of the Bernoulli number $B_{2n}$. No formula has been found yet for their numerators, that should involve binomial coefficients. These coefficients are related since, by shuffle:
$$\begin{array}{lll}
 & \zeta^{\star\star, \mathfrak{m}}_{2} (\boldsymbol{2}^{n})+ \sum_{k=0}^{n-1}\zeta^{\star\star, \mathfrak{m}}_{1} (\boldsymbol{2}^{k},3,\boldsymbol{2}^{n-k-1}) & =0\\
 &  \zeta^{\star\star, \mathfrak{m}}(\boldsymbol{2}^{n+1})-\zeta^{\star, \mathfrak{m}}(\boldsymbol{2}^{n+1}) \sum_{k=0}^{n-1}\zeta^{\star\star, \mathfrak{m}}_{1} (\boldsymbol{2}^{k},3,\boldsymbol{2}^{n-k-1}) & =0.
\end{array}$$
Identifying the coefficients of $\zeta^{\star}(\boldsymbol{2}^{n})$ in formulas $(iii),(v)$ in Lemma $\ref{lemmcoeff}$ leads to:
\begin{equation}\label{eq:coeffdrel}
1-\widetilde{D_{n}}= \sum_{k=0}^{n-1} \widetilde{D}_{k,n-1-k}.
\end{equation}

\chapter{Galois Descents}

\paragraph{\texttt{Contents}: }
The first section gives the general picture (for any $N$), sketching the Galois descent ideas. The second section focuses on the cases $N=2,3,4,\mlq 6\mrq, 8$, defining the filtrations by the motivic level associated to each descent, and displays both results and proofs. Some examples in small depth for are given in the Annexe $\S A.2$.\\
\\
\texttt{Notations}: For a fixed $N$, let $k_{N}\mathrel{\mathop:}=\mathbb{Q}(\xi_{N})$, where $\xi_{N}\in\mu_{N}$ is a primitive $N^{\text{th}}$ root of unity, and $\mathcal{O}_{N}$ is the ring of integers of $k_{N}$. The subscript or exponent $N$ will be omitted when it is not ambiguous. For the general case, the decomposition of $N$ is denoted $N=\prod q_{i}= \prod p_{i}^{\alpha}$.\\

\section{Overview}

\paragraph{Change of field. }
As said in Chapter $3$, for each $N, N'$ with $N' | N$, the Galois action on $\mathcal{H}_{N}$ and $\mathcal{H}_{N'}$ is determined by the coaction $\Delta$. More precisely, let consider the following descent\footnote{More generally, there are Galois descents $(\mathcal{d})=(k_{N}/k_{N'}, M/M')$ from $\mathcal{H}^{\mathcal{MT}\left( \mathcal{O}_{k_{N}} \left[ \frac{1}{M}\right] \right)  }$, to $\mathcal{H}^{\mathcal{MT}\left( \mathcal{O}_{k_{N'}} \left[ \frac{1}{M'}\right]\right)  }$, with $N'\mid N$, $M'\mid M$, with a set of derivations $\mathscr{D}^{\mathcal{d}} \subset \mathscr{D}^{N}$ associated.}, \texttt{assuming $\phi_{N'}$ is an isomorphism of graded Hopf comodules}: \footnote{Conjecturally as soon as $N'\neq p^{r}$, $p\geq 5$. Proven for  $N'=1,2,3,4,\mlq 6\mrq,8$.}
$$\xymatrixcolsep{5pc}\xymatrix{
\mathcal{H}^{N} \ar@{^{(}->}[r]
^{\phi_{N}}_{n.c} & \mathcal{H}^{\mathcal{MT}_{\Gamma_{N}}}  \\
\mathcal{H}^{N'}\ar[u]_{\mathcal{G}^{N/N'}} \ar@{^{(}->}[r]
_{n.c}^{\phi_{N'}\atop \sim} &\mathcal{H}^{\mathcal{MT}_{\Gamma_{N'}}} \ar[u]^{\mathcal{G}^{\mathcal{MT}}_{N/N'}} }$$
Let choose a basis for $gr_{1}\mathcal{L}_{r}^{\mathcal{MT}_{N'}}$, and extend it into a basis of $gr_{1}\mathcal{L}_{r}^{\mathcal{MT}_{N}}$:
$$ \left\lbrace \zeta^{\mathfrak{m}}(r; \eta'_{i,r}) \right\rbrace_{i} \subset \left\lbrace \zeta^{\mathfrak{m}}(r; \eta'_{i,r}) \right\rbrace \cup \left\lbrace \zeta^{\mathfrak{m}}(r; \eta_{i}) \right\rbrace_{1\leq i \leq c_{r}}, $$
$$\text{where } \quad c_{r} =\left\{
\begin{array}{ll} 
a_{N}-a_{N'}= \frac{\varphi(N)-\varphi(N')}{2}+p(N)-p(N') & \text{ if } r=1\\
b_{N}-b_{N'}= \frac{\varphi(N)-\varphi(N')}{2}   & \text{ if } r>1\\
\end{array} \right. .$$
Then, once this basis fixed, let split the set of derivations $\mathscr{D}^{N}$ into two parts (cf. $\S 2.4.4$), one corresponding to $\mathcal{H}^{N'}$:\nomenclature{$\mathscr{D}^{\backslash \mathcal{d}}$ and $\mathscr{D}^{\mathcal{d}} $}{sets of derivations associated to a descent $\mathcal{d}$}
\begin{equation}
\label{eq:derivdescent}
\mathscr{D}^{N} =  \mathscr{D}^{\backslash \mathcal{d}} \uplus \mathscr{D}^{\mathcal{d}}  \quad  \text{ where }\quad  \left\lbrace \begin{array}{l}
\mathscr{D}^{\backslash \mathcal{d}} =\mathscr{D}^{N'}\mathrel{\mathop:}=  \cup_{r}  \left\lbrace  D_{r}^{\eta'_{i,r}} \right\rbrace_{1\leq i \leq c_{r}} \\
\mathscr{D}^{\mathcal{d}}\mathrel{\mathop:}= \cup_{r} \left\lbrace D^{\eta_{i,r}}_{r} \right\rbrace_{1\leq i \leq c_{r}} 
\end{array} \right.  .
\end{equation}

\texttt{Examples:}
\begin{itemize}
\item[$\cdot$] For the descent from $\mathcal{MT}_{3}$ to $\mathcal{MT}_{1}$: $\mathscr{D}^{(k_{3}/\mathbb{Q}, 3/1)}=\left\lbrace  D^{\xi_{3}}_{1}, D^{\xi_{3}}_{2r}, r>0 \right\rbrace $.
\item[$\cdot$] For the descent from $\mathcal{MT}_{8}$ to $\mathcal{MT}_{4}$: $\mathscr{D}^{(k_{8}/k_{4}, 2/2)}=\left\lbrace  D^{\xi_{8}}_{r}-D^{-\xi_{8}}_{r}, r>0 \right\rbrace $.
\item[$\cdot$] For the descent from $\mathcal{MT}_{9}$ to $\mathcal{MT}_{3}$:  $\mathscr{D}^{(k_{9}/k_{3}, 3/3)}=\left\lbrace  D^{\xi_{9}}_{r}-D^{-\xi^{4}_{9}}_{r}, D^{\xi_{9}}_{r}-D^{-\xi^{7}_{9}}_{r} r>0 \right\rbrace $.\footnote{By the relations in depth $1$, since:
$$\zeta^{\mathfrak{a}} \left( r\atop \xi^{3}_{9}\right)= 3^{r-1} \left( \zeta^{\mathfrak{a}} \left( r\atop \xi^{1}_{9}\right) + \zeta^{\mathfrak{a}} \left( r\atop \xi^{4}_{9}\right)+ \zeta^{\mathfrak{a}} \left( r\atop \xi^{7}_{9}\right)  \right)  \quad \text{etc.}$$}
\end{itemize}

\begin{theo}
Let $N'\mid N$ such that $\mathcal{H}^{N'}\cong \mathcal{H}^{\mathcal{MT}_{\Gamma_{N'}}}$.\\
Let $\mathfrak{Z}\in gr^{\mathfrak{D}}_{p}\mathcal{H}_{n}^{N}$, depth graded MMZV relative to $\mu_{N}$.\\
Then $\mathfrak{Z}\in gr^{\mathfrak{D}}_{p}\mathcal{H}^{N'}$, i.e. $\mathfrak{Z}$ is a depth graded MMZV relative to $\mu_{N'}$ modulo smaller depth if and only if:
$$ \left( \forall r<n, \forall D_{r,p}\in\mathscr{D}_{r}^{\mathcal{d}},\quad   D_{r,p}(\mathfrak{Z})=0\right)  \quad \textrm{  and  }  \quad  \left( \forall r<n, \forall D_{r,p} \in\mathscr{D}^{\diagdown\mathcal{d}}, \quad D_{r,p}(\mathfrak{Z})\in gr^{\mathfrak{D}}_{p-1}\mathcal{H}^{N'}\right) .$$ 
\end{theo}
\begin{proof}
In the $(f_{i})$ side, the analogue of this theorem is pretty obvious, and the result can be transported via $\phi$, and back since $\phi_{N'}$ isomorphism by assumption.
\end{proof}
This is a very useful recursive criterion (derivation strictly decreasing weight and depth) to determine if a (motivic) multiple zeta value at $\mu_{N}$ is in fact a (motivic) multiple zeta value at $\mu_{N'}$, modulo smaller depth terms; applying it recursively, it could also take care of smaller depth terms. This criterion applies for motivic MZV$_{\mu_{N}}$, and by period morphism is deduced for MZV$_{\mu_{N}}$.\\

\paragraph{Change of Ramification.} If the descent has just a ramified part, the criterion can be stated in a non depth graded version. Indeed, there, since only weight $1$ matters, to define the derivation space $\mathcal{D}^{\mathcal{d}}$ as above ($\ref{eq:derivdescent}$), we need to choose a basis for $\mathcal{O}_{N}^{\ast}\otimes \mathbb{Q}$, which we complete with $\left\lbrace \xi^{\frac{N}{q_{i}}}_{N}\right\rbrace_{i\in I}$ into a basis for $\Gamma_{N}$. Then, with $N=\prod p_{i}^{\alpha_{i}}=\prod q_{i}$:
\begin{theo}\label{ramificationchange}
Let $\mathfrak{Z}\in \mathcal{H}_{n}^{N}\subset \mathcal{H}^{\mathcal{MT}_{\Gamma_{N}}}$, MMZV relative to $\mu_{N}$.\\
Then $\mathfrak{Z}\in \mathcal{H}^{\mathcal{MT}(\mathcal{O}_{N})}$ unramified if and only if:
$$ \left( \forall i\in I, D^{\xi^{\frac{N}{q_{i}}}}_{1}(\mathfrak{Z})=0\right)  \quad \textrm{  and  }  \quad  \left( \forall r<n, \forall D_{r}\in\mathscr{D}^{\diagdown\mathcal{d}}, \quad D_{r}(\mathfrak{Z})\in \mathcal{H}^{\mathcal{MT}(\mathcal{O}_{N})}\right) .$$ 
\end{theo}
\texttt{Nota Bene}: Intermediate descents and change of ramification, keeping part of some of the weight $1$ elements $\left\lbrace \xi^{\frac{N}{q_{i}}}_{N}\right\rbrace$ could also be stated.\\
\\
\texttt{Examples}:
\begin{description}
\item[$\boldsymbol{N=2}$:]  As claimed in the introduction, the descent between $\mathcal{H}^{2}$ and $\mathcal{H}^{1}$ is precisely measured by $D_{1}$:\footnote{$\mathscr{D}^{(\mathbb{Q}/\mathbb{Q}, 2/1)}=\left\lbrace  D^{-1}_{1} \right\rbrace $ with the above notations; and $D^{-1}_{1}$ is here simply denoted $D_{1}$ .}
\begin{coro}\label{criterehonoraire}
Let $\mathfrak{Z}\in\mathcal{H}^{2}=\mathcal{H}^{\mathcal{MT}_{2}}$, a motivic Euler sum.\\
Then $\mathfrak{Z}\in\mathcal{H}^{1}=\mathcal{H}^{\mathcal{MT}_{1}}$, i.e. $\mathfrak{Z}$ is a motivic multiple zeta value if and only if:
$$D_{1}(\mathfrak{Z})=0 \quad \textrm{  and  } \quad D_{2r+1}(\mathfrak{Z})\in\mathcal{H}^{1}.$$ 
\end{coro}
\item[$\boldsymbol{N=3,4,6}$:]
\begin{coro}\label{ramif346}
Let $N\in \lbrace 3,4,6\rbrace$ and $\mathfrak{Z}\in\mathcal{H}^{\mathcal{MT}(\mathcal{O}_{N} \left[ \frac{1}{N}\right] )}$, a motivic MZV$_{\mu_{N}}$.\\
Then $\mathfrak{Z}$ is unramified, $\mathfrak{Z}\in\mathcal{H}^{\mathcal{MT} (\mathcal{O}_{N})}$ if and only if:
$$D_{1}(\mathfrak{Z})=0 \textrm{  and  } \quad D_{r}(\mathfrak{Z})\in\mathcal{H}^{\mathcal{MT} (\mathcal{O}_{N})}.$$ 
\end{coro}
\item[$\boldsymbol{N=p^{r}}$:] A basis for $\mathcal{O}^{N}\otimes \mathbb{Q}$ is formed by: $\left\lbrace  \frac{1-\xi^{k}}{1-\xi}  \right\rbrace_{k\wedge p=1 \atop 0<k\leq\frac{N}{2}} $,  which corresponds to 
$$\text{ a basis of }  \mathcal{A}_{1}^{\mathcal{MT}(\mathcal{O}_{N})} \quad : \left\lbrace  \zeta^{\mathfrak{m}}\left( 1 \atop \xi^{k} \right)- \zeta^{\mathfrak{m}}\left( 1 \atop \xi \right) \right\rbrace  _{ k\wedge p=1 \atop 0<k\leq\frac{N}{2}} .$$
It can be completed in a basis of $\mathcal{A}_{1}^{N}$ with $\zeta^{\mathfrak{m}}\left( 1 \atop \xi^{1} \right)$. \footnote{With the previous theorem notations, $\mathcal{D}^{\mathcal{d}}=\lbrace D^{\xi}_{1}\rbrace$ whereas $\mathcal{D}^{\diagdown \mathcal{d}}= \lbrace D^{\xi^{k}}_{1}-D^{\xi}_{1} \rbrace_{k\wedge p=1 \atop 1<k\leq\frac{N}{2} } \cup_{r>1} \lbrace D^{\xi^{k}}_{r}\rbrace_{k\wedge p=1 \atop 0<k\leq\frac{N}{2}}$; where $D^{\xi}_{1}$ has to be understood as the projection of the left side over $\zeta^{\mathfrak{a}}\left( 1 \atop \xi \right)$ in respect to the basis above of $\mathcal{H}_{1}^{\mathcal{MT}(\mathcal{O}_{N})}$ more $\zeta^{\mathfrak{a}}\left( 1 \atop \xi \right)$. This leads to a criterion equivalent to $(\ref{ramifpr})$.} However, if we consider the basis of $\mathcal{A}_{1}^{N}$ formed by primitive roots of unity up to conjugates, the criterion for the descent could also be stated as follows: 
\begin{coro}\label{ramifpr}
Let $N=p^{r}$ and $\mathfrak{Z}\in\mathcal{H}^{\mathcal{MT}_{\Gamma_{N}}}=\mathcal{H}^{\mathcal{MT}(\mathcal{O}_{N} \left[ \frac{1}{p}\right] )}$, relative to $\mu_{N}$\footnote{For instance a MMZV relative to $\mu_{N}$. Beware, for $p>5$, there could be other periods.}.\\
Then $\mathfrak{Z}$ is unramified, $\mathfrak{Z}\in\mathcal{H}^{\mathcal{MT} (\mathcal{O}_{N})}$ if and only if:
$$\sum_{k\wedge p=1 \atop 0<k\leq\frac{N}{2}} D^{\xi^{k}_{N}}_{1}(\mathfrak{Z})=0 \quad \textrm{  and  } \quad \forall \left\lbrace  \begin{array}{l}
 r>1 \\
1<k\leq\frac{N}{2}\\
k\wedge p=1 \\
\end{array} \right.  , \quad D^{\xi^{k}_{N}}_{r}(\mathfrak{Z})\in\mathcal{H}^{\mathcal{MT} (\mathcal{O}_{N})}.$$ 
\end{coro}

\end{description}

\section{Descents for $\boldsymbol{N=2,3,4,\mlq 6\mrq,8}$.}

\subsection{Depth $\boldsymbol{1}$}

Let start with depth $1$ results, deduced from Lemma $2.4.1$ (from $\cite{De}$), fundamental to initiate the recursion later.
\begin{lemm} The basis for $gr^{\mathfrak{D}}_{1} \mathcal{A}$ is:
$$\left\{ \zeta^{\mathfrak{a}}\left(r;  \xi \right) \text{ such that } \left\{
\begin{array}{ll}  
r>1 \text{ odd  }  & \text{ if }N=1 \\
r \text{ odd  } & \text{ if }N=2 \\
r>0 & \text{ if } N=3,4 \\
r>1 & \text{  if } N=6 \\
\end{array} \right.   \right\rbrace  $$
For $N=8$, the basis for $gr^{\mathfrak{D}}_{1} \mathcal{A}_{r}$ is two dimensional, for all $r>0$:
$$\left\{ \zeta^{\mathfrak{a}}\left(r;  \xi \right), \zeta^{\mathfrak{a}}\left(r;  -\xi \right)\right\rbrace.$$
\end{lemm} 

Let make these relations explicit in depth $1$ for $N=2,3,4,\mlq 6\mrq,8$, since we would use some $p$-adic properties of the basis elements in our proof:

\begin{description}
\item[\textsc{For $N=2$:}]
The distribution relation in depth 1 is:
$$\zeta^{\mathfrak{a}}\left( {2 r + 1 \atop  1}\right) = (2^{-2r}-1)\zeta^{\mathfrak{a}}\left( {2r+1 \atop -1}\right) .$$
\item[\textsc{For $N=3$:}]
$$  \zeta^{\mathfrak{l}} \left( {2r+1 \atop  1} \right)\left(1-3^{2r}\right)= 2\cdot 3^{2r}\zeta^{\mathfrak{l}}\left({2r+1 \atop \xi}\right) \quad \quad \zeta^{\mathfrak{l}}\left({2r \atop  1}\right)=0 \quad \quad  \zeta^{\mathfrak{l}}\left({r \atop  \xi}\right) =\left(-1\right)^{r-1}  \zeta^{\mathfrak{l}}\left({r \atop  \xi^{-1}}\right). $$
\item[\textsc{For $N=4$:}]
$$\begin{array}{lllllll}
\zeta^{\mathfrak{l}}\left({ r \atop 1} \right) (1-2^{r-1}) & = & 2^{r-1}\cdot \zeta^{\mathfrak{l}}\left( {r\atop -1} \right) \text{  for } r\neq 1 & \quad & \zeta^{\mathfrak{l}}\left({1\atop   1}\right) & = & \zeta^{\mathfrak{l}}\left( {2r\atop -1} \right)=0 \\
\zeta^{\mathfrak{l}}\left({2r+1\atop -1}\right) & = & 2^{2r+1}  \zeta^{\mathfrak{l}}\left( {2r+1\atop \xi} \right) & \quad & \zeta^{\mathfrak{l}}\left( {r \atop \xi} \right) & = & \left(-1\right)^{r-1}  \zeta^{\mathfrak{l}}\left({r\atop  \xi^{-1}}\right).
\end{array}$$
\item[\textsc{For $N=6$:}]
$$\begin{array}{lllllll}
\zeta^{\mathfrak{l}}\left({r\atop  1}\right)\left(1-2^{r-1}\right) & = & 2^{r-1}\zeta^{\mathfrak{l}}\left({ r \atop -1} \right) \text{  for } r\neq 1 & \quad & \zeta^{\mathfrak{l}}\left( {1 \atop 1} \right) & = & \zeta^{\mathfrak{l}}\left({2r\atop  -1}\right)=0\\
\zeta^{\mathfrak{l}}\left(  {2r+1 \atop -1} \right)  & = & \frac{2\cdot 3^{2r}}{1-3^{2r}}  \zeta^{\mathfrak{l}}\left( {2r+1 \atop  \xi} \right)& \quad & \zeta^{\mathfrak{l}}\left( {r \atop \xi^{2}} \right)  & = & \frac{2^{r-1}}{1-(-2)^{r-1}}  \zeta^{\mathfrak{l}}\left( {r \atop  \xi} \right).\\
\zeta^{\mathfrak{l}}\left({r\atop  \xi} \right) &=&\left(-1\right)^{r-1}  \zeta^{\mathfrak{l}}\left( {r \atop \xi^{-1}} \right) &\quad & \zeta^{\mathfrak{l}}\left({r \atop  -\xi} \right)  & = & \left(-1\right)^{r-1}  \zeta^{\mathfrak{l}}\left( {r \atop -\xi^{-1}} \right).
\end{array}$$
\item[\textsc{For $N=8$:}]
$$\begin{array}{lllllll}
  \zeta^{\mathfrak{l}}\left({ r \atop  1} \right)& =& \frac{ 2^{r-1}}{\left(1-2^{r-1}\right)}\zeta^{\mathfrak{l}}\left({r\atop  -1}\right) \text{  for } r\neq 1  &  \quad &  \zeta^{\mathfrak{l}}\left( {1 \atop  1} \right) &=&\zeta^{\mathfrak{l}}\left({2r\atop  -1}\right)=0  \\
 \zeta^{\mathfrak{l}}\left({ r \atop  -i }\right) &=& 2^{r-1}  \left( \zeta^{\mathfrak{l}}\left({r\atop  \xi}\right) + \zeta^{\mathfrak{l}}\left({r\atop  -\xi}\right) \right) &  \quad & \zeta^{\mathfrak{l}}\left( {2r+1 \atop  -1} \right) &=& 2^{2r+1}  \zeta^{\mathfrak{l}}\left({2r+1\atop  i}\right)   \\
\zeta^{\mathfrak{l}}\left({ r\atop  \pm \xi} \right) &=&\left(-1\right)^{r-1}  \zeta^{\mathfrak{l}}\left( {r \atop \pm \xi^{-1} }\right)  &  \quad & \zeta^{\mathfrak{l}}\left( {r \atop  i} \right) &=&\left(-1\right)^{r-1}  \zeta^{\mathfrak{l}}\left( {r \atop  -i}\right) \\
\end{array}$$

\end{description}

\subsection{Motivic Level filtration}

Let fix a descent $(\mathcal{d})=(k_{N}/k_{N'}, M/M')$ from $\mathcal{H}^{\mathcal{MT}\left( \mathcal{O}_{k_{N}} \left[ \frac{1}{M}\right] \right)  }$, to $\mathcal{H}^{\mathcal{MT}\left( \mathcal{O}_{k_{N'}} \left[ \frac{1}{M'}\right]\right)  }$, with $N'\mid N$, $M'\mid M$, among these considered in this section, represented in Figures $\ref{fig:d248}, \ref{fig:d36}$.\\
Let us define a motivic level increasing filtration $\mathcal{F}^{\mathcal{d}}$ associated, from the set of derivations associated to this descent, $\mathscr{D}^{\mathcal{d}} \subset \mathscr{D}^{N}$, defined in $(\ref{eq:derivdescent})$.

\begin{defi}
The filtration by the \textbf{motivic level} associated to a descent $(\mathcal{d})$ is defined recursively on $\mathcal{H}^{N}$ by:
\begin{itemize}
\item[$\cdot$] $\mathcal{F}^{\mathcal{d}} _{-1} \mathcal{H}^{N}=0$.
\item[$\cdot$] $\mathcal{F}^{\mathcal{d}} _{i} \mathcal{H}^{N}$ is the largest submodule of $\mathcal{H}^{N}$ such that $\mathcal{F}^{\mathcal{d}}_{i}\mathcal{H}^{N}/\mathcal{F}^{\mathcal{d}} _{i-1}\mathcal{H}^{N}$ is killed by $\mathscr{D}^{\mathcal{d}}$, $\quad$ i.e. is in the kernel of $\oplus_{D\in \mathscr{D}^{\mathcal{d}}} D$.
\end{itemize}
\end{defi}
It's a graded Hopf algebra's filtration:
$$\mathcal{F} _{i}\mathcal{H}. \mathcal{F}_{j}\mathcal{H} \subset \mathcal{F}_{i+j}\mathcal{H} \text{  ,  } \quad \Delta (\mathcal{F}_{n}\mathcal{H})\subset \sum_{i+j=n} \mathcal{F}_{i}\mathcal{A} \otimes \mathcal{F}_{j}\mathcal{H}.$$
The associated graded is denoted: $gr^{\mathcal{d}}  _{i}$ and the quotients, coalgebras compatible with $\Delta$: 
\begin{equation}
\label{eq:quotienth} \mathcal{H}^{\geq 0}  \mathrel{\mathop:}= \mathcal{H} \text{  ,  } \boldsymbol{\mathcal{H}^{\geq i}}\mathrel{\mathop:}= \mathcal{H}/ \mathcal{F}_{i-1}\mathcal{H} \text{ with the projections :}\quad \quad \forall j\geq i \text{  , } \pi_{i,j}:  \mathcal{H}^{\geq i} \rightarrow  \mathcal{H}^{\geq j}.
\end{equation}
Note that, via the isomorphism $\phi$, the motivic filtration on $\mathcal{H}^{\mathcal{MT}_{N}}$ corresponds to\footnote{In particular, remark that $\dim \mathcal{F}^{\mathcal{d}} _{i} \mathcal{H}_{n}^{\mathcal{MT}_{N}}$ are known.}:
\begin{equation}
\label{eq:isomfiltration}\mathcal{F}^{\mathcal{d}} _{i} \mathcal{H}^{\mathcal{MT}_{N}} \longleftrightarrow \left\langle  x\in H^{N} \mid Deg^{\mathcal{d}} (x) \leq i \right\rangle _{\mathbb{Q}} ,
\end{equation}
where $Deg^{\mathcal{d}}$ is the degree in $\left\lbrace  \lbrace f^{j}_{r} \rbrace_{b_{N'}<j\leq b_{N} \atop r>1} , \lbrace f^{j}_{1} \rbrace_{a_{N'}<j\leq a_{N}} \right\rbrace $, which are the images of the complementary part of $ gr_{1}\mathcal{L}^{\mathcal{MT}_{N'}}$ in the basis of $gr_{1}\mathcal{L}^{\mathcal{MT}_{N}}$.\\
\\
\texttt{Example}: For the descent between $\mathcal{H}^{\mathcal{MT}_{2}}$ and $\mathcal{H}^{\mathcal{MT}_{1}}$, since $gr_{1}\mathcal{L}^{\mathcal{MT}_{2}}= \left\langle \zeta^{\mathfrak{m}}(-1), \left\lbrace  \zeta^{\mathfrak{m}}(2r+1)\right\rbrace  _{r>0}\right\rangle$:
$$\mathcal{F} _{i} \mathcal{H}^{\mathcal{MT}_{2}} \quad \xrightarrow[\sim]{\quad \phi}\quad  \left\langle  x\in \mathbb{Q}\langle f_{1}, f_{3}, \cdots \rangle\otimes \mathbb{Q}[f_{2}]  \mid Deg_{f_{1}} (x) \leq i \right\rangle _{\mathbb{Q}} \text{ , where } Deg_{f_{1}}= \text{ degree in } f_{1}.$$
\\
By definition of these filtrations:
\begin{equation}D_{r,p}^{\eta} \left( \mathcal{F}_{i}\mathcal{H}_{n} \right) \subset \left\lbrace  \begin{array}{ll} 
\mathcal{F}_{i-1}\mathcal{H}_{n-r} & \text{ if }D_{r,p}^{\eta}\in\mathscr{D}^{\mathcal{d}}_{r} \\
\mathcal{F}_{i}\mathcal{H}_{n-r} & \text{ if } D_{r,p}^{\eta}\in\mathscr{D}^{\backslash\mathcal{d}}_{r}
\end{array} \right. . 
\end{equation}
Similarly, looking at $\partial_{n,p}$ (cf. $\ref{eq:pderivnp}$):
\begin{equation} \partial_{n,p}(\mathcal{F}_{i-1}\mathcal{H}_{n}) \subset \oplus_{r<n} \left( gr_{p-1}^{\mathfrak{D}}  \mathcal{F}_{i-2}\mathcal{H}_{n-r}\right)^{\text{ card } \mathscr{D}^{\mathcal{d}}_{r}} \oplus_{r<n} \left( gr_{p-1}^{\mathfrak{D}} \mathcal{F}_{i-1}\mathcal{H}_{n-r}\right) ^{\text{ card } \mathscr{D}^{\backslash\mathcal{d}}_{r}}.
\end{equation}
This allows us to pass to quotients, and define $D^{\eta,i,\mathcal{d}}_{n,p}$ and $\partial^{i,\mathcal{d}}_{n,p}$:\nomenclature{$D^{\eta,i,\mathcal{d}}_{n,p}$ and $\partial^{i,\mathcal{d}}_{n,p}$}{quotient maps}
\begin{equation}
\label{eq:derivinp}
\boldsymbol{D^{\eta,i,\mathcal{d}}_{n,p}}: gr_{p}^{\mathfrak{D}} \mathcal{H}_{n}^{\geq i} \rightarrow \left\lbrace  \begin{array}{ll} 
 gr_{p-1}^{\mathfrak{D}} \mathcal{H}_{n-r}^{\geq i-1} & \text{ if }D_{r,p}^{\eta}\in\mathscr{D}^{\mathcal{d}}_{r} \\
gr_{p-1}^{\mathfrak{D}} \mathcal{H}_{n-r}^{\geq i} & \text{ if } D_{r,p}^{\eta}\in\mathscr{D}^{\backslash\mathcal{d}}_{r}
\end{array} \right. 
\end{equation}
\begin{framed}
\begin{equation}
\label{eq:pderivinp}
\boldsymbol{\partial^{i,\mathcal{d}}_{n,p}}: gr_{p}^{\mathfrak{D}} \mathcal{H}_{n}^{\geq i} \rightarrow  \oplus_{r<n} \left( gr_{p-1}^{\mathfrak{D}} \mathcal{H}_{n-r}^{\geq i-1}\right)^{\text{ card } \mathscr{D}^{\mathcal{d}}_{r}}  \oplus_{r<n} \left( gr_{p-1}^{\mathfrak{D}} \mathcal{H}_{n-r}^{\geq i}\right)^{\text{ card } \mathscr{D}^{\backslash\mathcal{d}}_{r}} . 
\end{equation}
\end{framed}
The bijectivity of this map is essential to the results stated below. 

\subsection{General Results}
In the following results, the filtration considered $\mathcal{F}_{i}$ is the filtration by the motivic level associated to the (fixed) descent $\mathcal{d}$ while the index $i$, in $\mathcal{B}_{n, p, i}$ refers to the level notion for elements in $\mathcal{B}$ associated to the descent $\mathcal{d}$.\footnote{Precisely defined, for each descent in $\S 5.2.5 $.}\\
We first obtain the following result on the depth graded quotients, for all $i\geq 0$, with:
$$\mathbb{Z}_{1[P]} \mathrel{\mathop:}= \frac{\mathbb{Z}}{1+ P\mathbb{Z}}=\left\{ \frac{a}{1+b P}, a,b\in\mathbb{Z} \right\} \text{ with }  \begin{array}{ll}
P=2 & \text{ for } N=2,4,8 \\
P=3 & \text{ for } N=3,6
\end{array} .$$

\begin{lemm}
\begin{itemize}
\item[$\cdot$] $$\mathcal{B}_{n, p, \geq i} \text{  is a linearly free family of  } gr_{p}^{\mathfrak{D}} \mathcal{H}_{n}^{\geq i} \text{ and defines a } \mathbb{Z}_{1[P]} \text{ structure :}$$
Each element $\mathfrak{Z}= \zeta^{\mathfrak{m}}\left( z_{1}, \ldots , z_{p} \atop \epsilon_{1}, \ldots, \epsilon_{p}\right)\in \mathcal{B}_{n,p} $ decomposes in a $\mathbb{Z}_{1[P]}$-linear combination of $\mathcal{B}_{n, p, \geq i}$ elements, denoted $cl_{n,p,\geq i}(\mathfrak{Z})$ in $gr_{p}^{\mathfrak{D}} \mathcal{H}_{n}^{\geq i}$, which defines, in an unique way:
$$cl_{n,p,\geq i}: \langle\mathcal{B}_{n, p, \leq i-1}\rangle_{\mathbb{Q}} \rightarrow \langle\mathcal{B}_{n, p, \geq i}\rangle_{\mathbb{Q}}.$$
\item[$\cdot$]
The following map $\partial^{i,\mathcal{d}}_{n,p}$ is bijective:
$$\partial^{i,\mathcal{d}}_{n,p}: gr_{p}^{\mathfrak{D}} \langle \mathcal{B}_{n, \geq i} \rangle_{\mathbb{Q}} \rightarrow \oplus_{r<n} \left( gr_{p-1}^{\mathfrak{D}} \langle \mathcal{B}_{n-1, \geq i-1} \rangle_{\mathbb{Q}} \right) ^{\oplus \text{ card } \mathcal{D}^{\mathcal{d}}_{r}} \oplus_{r<n} \left( gr_{p-1}^{\mathfrak{D}} \langle \mathcal{B}_{n-2r-1, \geq i} \rangle_{\mathbb{Q}} \right) ^{\oplus \text{ card } \mathcal{D}^{\backslash\mathcal{d}}_{r}}.$$
\end{itemize}
\end{lemm}
\nomenclature{$cl_{n,p,\geq i}$, or $cl_{n,\leq p,\geq i}$}{maps whose existence is proved in $\S 5.2$}Before giving the proof, in the next section, let present its consequences such as bases for the quotient, the filtration and the graded spaces for each descent considered:
\begin{theo}
\begin{itemize}
\item[$(i)$] $\mathcal{B}_{n,\leq p, \geq i}$ is a basis of $\mathcal{F}_{p}^{\mathfrak{D}} \mathcal{H}_{n}^{\geq i}=\mathcal{F}_{p}^{\mathfrak{D}} \mathcal{H}_{n}^{\geq i, \mathcal{MT}}$.
\item[$(ii)$] 
\begin{itemize}
\item[$\cdot$] $\mathcal{B}_{n, p, \geq i}$ is a basis of $gr_{p}^{\mathfrak{D}} \mathcal{H}_{n}^{\geq i}=gr_{p}^{\mathfrak{D}} \mathcal{H}_{n}^{\geq i, \mathcal{MT}}$ on which it defines a $\mathbb{Z}_{1[P]}$-structure:\\
Each element $\mathfrak{Z}= \zeta^{\mathfrak{m}}\left( z_{1}, \ldots , z_{p} \atop \epsilon_{1}, \ldots, \epsilon_{p}\right)$ decomposes in a $\mathbb{Z}_{1[P]}$-linear combination of $\mathcal{B}_{n, p, \geq i}$ elements, denoted $cl_{n,p,\geq i}(\mathfrak{Z})$ in $gr_{p}^{\mathfrak{D}} \mathcal{H}_{n}^{\geq i}$, which defines in an unique way: 
$$cl_{n,p,\geq i}: \langle\mathcal{B}_{n, p, \leq i-1}\rangle_{\mathbb{Q}} \rightarrow \langle\mathcal{B}_{n, p, \geq i}\rangle_{\mathbb{Q}} \text{ such that } \mathfrak{Z}+cl_{n,p,\geq i}(\mathfrak{Z})\in \mathcal{F}_{i-1}\mathcal{H}_{n}+ \mathcal{F}^{\mathfrak{D}}_{p-1}\mathcal{H}_{n}.$$
\item[$\cdot$] The following map is bijective:
$$\partial^{i, \mathcal{d}}_{n,p}: gr_{p}^{\mathfrak{D}} \mathcal{H}_{n}^{\geq i} \rightarrow  \oplus_{r<n} \left( gr_{p-1}^{\mathfrak{D}} \mathcal{H}_{n-1}^{\geq i-1}\right) ^{\oplus \text{ card } \mathcal{D}^{\mathcal{d}}_{r}} \oplus_{r<n} \left( gr_{p-1}^{\mathfrak{D}} \mathcal{H}_{n-r}^{\geq i}\right) ^{\oplus \text{ card } \mathcal{D}^{\backslash\mathcal{d}}_{r}}. $$
\item[$\cdot$] $\mathcal{B}_{n,\cdot, \geq i} $ is a basis of $\mathcal{H}^{\geq i}_{n} =\mathcal{H}^{\geq i, \mathcal{MT}}_{n}$.
\end{itemize}
	\item[$(iii)$] We have the two split exact sequences in bijection:
$$ 0\longrightarrow \mathcal{F}_{i}\mathcal{H}_{n} \longrightarrow \mathcal{H}_{n} \stackrel{\pi_{0,i+1}} {\rightarrow}\mathcal{H}_{n}^{\geq i+1} \longrightarrow 0$$
$$ 0 \rightarrow \langle \mathcal{B}_{n, \cdot, \leq i} \rangle_{\mathbb{Q}} \rightarrow \langle\mathcal{B}_{n} \rangle_{\mathbb{Q}} \rightarrow \langle \mathcal{B}_{n, \cdot, \geq i+1} \rangle_{\mathbb{Q}} \rightarrow 0 .$$
The following map, defined in an unique way:
	$$cl_{n,\leq p,\geq i}: \langle\mathcal{B}_{n, p, \leq i-1}\rangle_{\mathbb{Q}} \rightarrow \langle\mathcal{B}_{n, \leq p, \geq i}\rangle_{\mathbb{Q}} \text{ such that } \mathfrak{Z}+cl_{n,\leq p,\geq i}(\mathfrak{Z})\in \mathcal{F}_{i-1}\mathcal{H}_{n}.$$
\item[$(iv)$] A basis for the filtration spaces $\mathcal{F}_{i} \mathcal{H}^{\mathcal{MT}}_{n}=\mathcal{F}_{i} \mathcal{H}_{n}$:
$$\cup_{p} \left\{ \mathfrak{Z}+ cl_{n, \leq p, \geq i+1}(\mathfrak{Z}), \mathfrak{Z}\in \mathcal{B}_{n, p, \leq i} \right\}.$$
\item[$(v)$] A basis for the graded space $gr_{i} \mathcal{H}^{\mathcal{MT}}_{n}=gr_{i} \mathcal{H}_{n}$:
$$\cup_{p} \left\{ \mathfrak{Z}+ cl_{n, \leq p, \geq i+1}(\mathfrak{Z}), \mathfrak{Z}\in \mathcal{B}_{n, p, i} \right\}.$$
\end{itemize}
\end{theo}
The proof is given in $\S 5.2.4$, and the notion of level resp. motivic level, some consequences and specifications for $N=2,3,4,\mlq 6\mrq,8$ individually are provided in  $\S 5.2.5$. Some examples in small depth are displayed in Appendice $A.2$.\\
\\
\\
\texttt{{\large Consequences, level $i=0$:}} 
\begin{itemize}
\item[$\cdot$] The level $0$ of the basis elements $\mathcal{B}^{N}$ forms a basis of $\mathcal{H}^{N} = \mathcal{H}^{\mathcal{MT}_{N}}$, for $N=2,3,4,\mlq 6 \mrq, 8$. This gives a new proof (dual) of Deligne's result (in $\cite{De}$).\\
The level $0$ of this filtration is hence isomorphic to the following algebras:\footnote{The equalities of the kind $\mathcal{H}^{\mathcal{MT}_{N}}= \mathcal{H}^{N}$ are consequences of the previous theorem for $N=2,3,4,\mlq 6 \mrq,8$, and by F. Brown for $N=1$ (cf. $\cite{Br2}$). Moreover, we have inclusions of the kind $\mathcal{H}^{\mathcal{MT}_{N'}} \subseteq \mathcal{F}_{0}^{k_{N}/k_{N'},M/M'}\mathcal{H}^{\mathcal{MT}_{N}}$ and we deduce the equality from dimensions at fixed weight.}
$$ \mathcal{F}_{0}^{k_{N}/k_{N'},M/M'}\mathcal{H}^{\mathcal{MT}_{N}}=\mathcal{F}_{0}^{k_{N}/k_{N'},M/M'}\mathcal{H}^{N}=\mathcal{H}^{\mathcal{MT}_{N',M'}}="\mathcal{H}^{N',M'}" .$$
Hence the inclusions in the following diagram are here isomorphisms: 
$$\xymatrix{
\mathcal{F}_{0}^{k_{N}/k_{N'},M/M'}\mathcal{H}^{\mathcal{MT}_{N}}  &  \mathcal{H}^{\mathcal{MT}_{N'}}  \ar@{^{(}->}[l]\\
\mathcal{F}_{0}^{k_{N}/k_{N'},M/M'}\mathcal{H}^{N} \ar@{^{(}->}[u]  &  \mathcal{H}^{N'}  \ar@{^{(}->}[l] \ar@{^{(}->}[u]}.$$
\item[$\cdot$] It gives, considering such a descent $(k_{N}/k_{N'},M/M')$, a basis for $\mathcal{F}^{0}\mathcal{H}^{N}= \mathcal{H}^{\mathcal{MT}_{N',M'}}$ in terms of the basis of $\mathcal{H}^{N}$. For instance, it leads to a new basis for motivic multiple zeta values in terms of motivic Euler sums, or motivic MZV$_{\mu_{3}}$.\\
Some other $0$-level such as  $\mathcal{F}_{0}^{k_{N}/k_{N},P/1}$, $N=3,4$ which should reflect the descent from $\mathcal{MT}(\mathcal{O}_{N}\left[ \frac{1}{P}\right] )$ to $\mathcal{MT}(\mathcal{O}_{N})$ are not known to be associated to a fundamental group, but the previous result enables us to reach them. We obtain a basis for:
\begin{itemize}
\item[$\bullet$] $\boldsymbol{\mathcal{H}^{\mathcal{MT}(\mathbb{Z}\left[\frac{1}{3}\right])}}$ in terms of the basis of $\mathcal{H}^{\mathcal{MT}(\mathcal{O}_{3}[\frac{1}{3}])}$.
\item[$\bullet$] $\boldsymbol{\mathcal{H}^{\mathcal{MT}(\mathcal{O}_{3})}}$ in terms of the basis of $\mathcal{H}^{\mathcal{MT}(\mathcal{O}_{3}[\frac{1}{3}])}$.
\item[$\bullet$] $\boldsymbol{\mathcal{H}^{\mathcal{MT}(\mathcal{O}_{4})}}$ in terms of the basis of $\mathcal{H}^{\mathcal{MT}(\mathcal{O}_{4}[\frac{1}{4}])}$.
\\
\end{itemize}
\end{itemize}

\subsection{Proofs}

As proved below, Theorem $5.2.4$ boils down to the Lemma $5.2.3$. Remind the map $\partial^{i,\mathcal{d}}_{n,p}$:
$$\partial^{i,\mathcal{d}}_{n,p}: gr_{p}^{\mathfrak{D}} \mathcal{H}_{n}^{\geq i} \rightarrow  \oplus_{r<n} \left( gr_{p-1}^{\mathfrak{D}} \mathcal{H}_{n-r}^{\geq i-1}\right)^{\text{ card } \mathscr{D}^{\mathcal{d}}_{r}}  \oplus_{r<n} \left( gr_{p-1}^{\mathfrak{D}} \mathcal{H}_{n-r}^{\geq i}\right)^{\text{ card } \mathscr{D}^{\backslash\mathcal{d}}_{r}}.$$
We will look at its image on $\mathcal{B}_{n,p,\geq i} $  and prove both the injectivity of $\partial^{i,\mathcal{d}}_{n,p}$ as considered in Lemma $5.2.3$, and the linear independence of these elements $\mathcal{B}_{n,p,\geq i}$.

\paragraph{\Large { Proof of Lemma $\boldsymbol{5.2.3}$ for $\boldsymbol{N=2}$:}}

The formula $(\ref{Drp})$ for $D^{-1}_{2r+1,p}$ on $\mathcal{B}$ elements:\footnote{Using identity: $\zeta^{\mathfrak{a}}(\overline{2 r + 1 })= (2^{-2r}-1)\zeta^{\mathfrak{a}}(2r+1 )$. Projection on $\zeta^{\mathfrak{l}}(\overline{2r+1})$ for the left side.}
\begin{multline} \label{Deriv2} D^{-1}_{2r+1,p} \left(\zeta^{\mathfrak{m}} (2x_{1}+1, \ldots , \overline{2x_{p}+1}) \right) =  \\
 \frac{2^{2r}}{1-2^{2r}}\delta_{r =x_{1}}  \cdot \zeta^{\mathfrak{m}} (2 x_{2}+1, \ldots, \overline{2x_{p}+1})  \\
\frac{2^{2r}}{1-2^{2r}} \left\lbrace \begin{array}{l}
\sum_{i=1}^{p-2} \delta_{x_{i+1}\leq r < x_{i}+ x_{i+1} } \binom{2r}{2x_{i+1}}   \\
-\sum_{i=1}^{p-1}  \delta_{x_{i}\leq r < x_{i}+ x_{i+1}} \binom{2r}{2x_{i}} 
\end{array} \right.  \cdot \zeta^{\mathfrak{m}} \left( \cdots ,2x_{i-1}+1, 2 (x_{i}+x_{i+1}-r) +1, 2 x_{i+2}+1, \cdots \right)   \\
\textrm{\textsc{(d) }} +\delta_{x_{p} \leq r \leq x_{p}+ x_{p-1}} \binom{2r}{2x_{p}}  \cdot\zeta^{\mathfrak{m}} \left( \cdots ,2x_{p-2}+1, \overline{2 (x_{p-1}+x_{p}-r) +1}\right)   
\end{multline}
Terms of type \textsc{(d)} play a particular role since they correspond to deconcatenation for the coaction, and will be the terms of minimal $2$-adic valuation.\\
$D^{-1}_{1,p}$ acts as a deconcatenation on this family:
\begin{equation} \label{Deriv21} D^{-1}_{1,p} \left(\zeta^{\mathfrak{m}} (2x_{1}+1, \ldots , \overline{2x_{p}+1}) \right) = \left\{
\begin{array}{ll}
  0 & \text{ if } x_{p}\neq 0 \\
  \zeta^{\mathfrak{m}} (2x_{1}+1, \ldots , \overline{2x_{p-1}+1}) & \text{ if } x_{p}=0 .\\
\end{array}
\right. \end{equation}
For $N=2$, $\partial^{i}_{n,p}$ ($\ref{eq:pderivinp}$) is simply:
\begin{equation}\label{eq:pderivinp2}
\partial^{i}_{n,p}: gr_{p}^{\mathfrak{D}} \mathcal{H}_{n}^{\geq i} \rightarrow  gr_{p-1}^{\mathfrak{D}} \mathcal{H}_{n-1}^{\geq i-1} \oplus_{1<2r+1\leq n-p+1} gr_{p-1}^{\mathfrak{D}} \mathcal{H}_{n-2r-1}^{\geq i}. 
\end{equation}
Let prove all statements of Lemma $5.2.3$, recursively on the weight, and then recursively on depth and on the level, from $i=0$.
\begin{proof}
By recursion hypothesis, weight being strictly smaller, we assume that:	
	$$\mathcal{B}_{n-1,p-1,\geq i-1} \oplus_{1<2r+1\leq n-p+1} \mathcal{B}_{n-2r-1,p-1,\geq i} \text{ is a  basis of  } $$
	$$gr_{p-1}^{\mathfrak{D}} \mathcal{H}_{n-1}^{\geq i-1,\mathcal{B}} \oplus_{1<2r+1\leq n-p+1} gr_{p-1}^{\mathfrak{D}} \mathcal{H}_{n-2r-1}^{\geq i,\mathcal{B}}. $$
	\begin{center}
\textsc{Claim:}	The matrix $M^{i}_{n,p}$ of $\left(\partial^{i,\mathcal{d}}_{n,p} (z) \right)_{z\in \mathcal{B}_{n, p, \geq i}}$ on these spaces is invertible.
	\end{center}
\texttt{Nota Bene:} Here $D^{-1}_{1}(z)$, resp. $D^{-1}_{2r+1,p}(z)$ are expressed in terms of $\mathcal{B}_{n-1,p-1,\geq i-1} $ resp. $\mathcal{B}_{n-2r-1,p-1,\geq i}$.\\
It will prove both the bijectivity of $\partial^{i,\mathcal{d}}_{n,p}$ as considered in the lemma and the linear independence of $\mathcal{B}_{n, p, \geq i}$. Let divide $M^{i}_{n,p}$ into four blocks, with the first column corresponding to elements of $\mathcal{B}_{n, p, \geq i}$ ending by $1$:
\begin{center}
  \begin{tabular}{ l || c | c ||}
     & $x_{p}=0$ &  $x_{p}>0$ \\ \hline
     $D_{1,p}$ & M$1$ & M$2$ \\
    $D_{>1,p}$ & M$3$ & M$4$ \\
    \hline
  \end{tabular}
\end{center}
According to ($\ref{Deriv21}$), $D^{-1}_{1,p}$ is zero on the elements not ending by 1, and acts as a deconcatenation on the others. Therefore, M$3=0$, so $M^{i}_{n,p}$ is lower triangular by blocks, and the left-upper-block M$1$ is diagonal invertible. It remains to prove the invertibility of the right-lower-block $\widetilde{M}\mathrel{\mathop:}=M4$, corresponding to $D^{-1}_{>1,p}$ and to the elements of $\mathcal{B}_{n, p, \geq i}$ not ending by 1.\\
\\
Notice that in the formula $(\ref{Deriv2})$ of $D_{2r+1,p}$, applied to an element of $\mathcal{B}_{n, p, \geq i}$, most of terms appearing have a number of $1$ greater than $i$ but there are also terms in $\mathcal{B}_{n-2r-1,p-1,i-1}$, with exactly $(i-1)$ \say{$1$} for type $\textsc{a,b,c}$ only. We will make disappear the latter modulo $2$, since they are $2$-adically greater. \\
More precisely, using recursion hypothesis (in weight strictly smaller), we can replace them in $gr_{p-1} \mathcal{H}^{\geq i}_{n-2r-1}$ by a $\mathbb{Z}_{\text{odd}}$-linear combination of elements in $\mathcal{B}_{n-2r-1, p-1, \geq i}$, which does not lower the $2$-adic valuation. It is worth noticing that the type \textsc{d} elements considered are now always in $\mathcal{B}_{n-2r-1,p-1,\geq i}$, since we removed the case $x_{p}= 0$.\\

Once done, we can construct the matrix $\widetilde{M}$ and examine its entries.\\
Order elements of $\mathcal{B}$ on both sides by lexicographic order of its \say{reversed} elements:
\begin{center}
$(x_{p},x_{p-1},\cdots, x_{1})$ for the colums, $(r,y_{p-1},\cdots, y_{1})$ for the rows.
\end{center}
Remark that, with such an order, the diagonal corresponds to the deconcatenation terms: $r=x_{p}$ and $x_{i}=y_{i}$.\\
Referring to $(\ref{Deriv2})$, and by the previous remark, we see that $\widetilde{M}$ has all its entries of 2-adic valuation positive or equal to zero, since the coefficients in $(\ref{Deriv2})$ are in $2^{2r}\mathbb{Z}_{\text{odd}}$ (for types \textsc{a,b,c}) or of the form $\mathbb{Z}_{\text{odd}}$ for types \textsc{d,d'}. If we look only at the terms with $2$-adic valuation zero, (which comes to consider $\widetilde{M}$ modulo $2$), it only remains in $(\ref{Deriv2})$ the terms of type \textsc{(d,d')}, that is:
\begin{multline}\nonumber
 D_{2r+1,p} (\zeta^{\mathfrak{m}}(2x_{1}+1, \ldots, \overline{2x_{p}+1}))  \equiv \delta_{ r = x_{p}+ x_{p-1}} \binom{2r}{2x_{p}}  \zeta^{\mathfrak{m}} (2x_{1}+1, \ldots ,2x_{p-2}+1, \overline{1})  \\
 + \delta_{x_{p} \leq r < x_{p}+ x_{p-1}} \binom{2r}{2x_{p}} \zeta^{\mathfrak{m}} (2x_{1}+1, \ldots ,2x_{p-2}+1, \overline{2 (x_{p-1}+x_{p}-r) +1})  \pmod{ 2}.
 \end{multline}
Therefore, modulo 2, with the order previously defined, it remains only an upper triangular matrix ($\delta_{x_{p}\leq r}$), with 1 on the diagonal ($\delta_{x_{p}= r}$, deconcatenation terms). Thus, $\det\widetilde{M}$ has a 2-adic valuation equal to zero, and in particular can not be zero, that's why $\widetilde{M}$ is invertible.\\

The $\mathbb{Z}_{odd}$ structure is easily deduced from the fact that the determinant of $\widetilde{M}$ is odd, and the observation that if we consider $D_{2r+1,p} (\zeta^{\mathfrak{m}} (z_{1}, \ldots, z_{p}))$, all the coefficients are integers.
\end{proof}

\paragraph{{\Large Proof of Lemma $\boldsymbol{5.2.3}$ for other $\boldsymbol{N}$.}}

These cases can be handled in a rather similar way than the case $N=2$, except that the number of generators is different and that several descents are possible, hence there will be several notions of level and filtrations by the motivic level, one for each descent. Let fix a descent $\mathcal{d}$ and underline the differences in the proof:

\begin{proof}
In the same way, we prove by recursion on weight, depth and level, that the following map is bijective:
$$\partial^{i,\mathcal{d}}_{n,p}: gr_{p}^{\mathfrak{D}} \langle \mathcal{B}_{n, \geq i} \rangle_{\mathbb{Q}} \rightarrow \oplus_{r<n} \left( gr_{p-1}^{\mathfrak{D}} \langle \mathcal{B}_{n-1, \geq i-1} \rangle_{\mathbb{Q}} \right) ^{\oplus \text{ card } \mathscr{D}^{\mathcal{d}}_{r}} \oplus_{r<n} \left( gr_{p-1}^{\mathfrak{D}} \langle \mathcal{B}_{n-2r-1, \geq i} \rangle_{\mathbb{Q}} \right) ^{\oplus \text{ card } \mathscr{D}^{\backslash\mathcal{d}}_{r}}.$$
\begin{center}
I.e the matrix $M^{i}_{n,p}$ of $\left(\partial^{i}_{n,p} (z) \right)_{z\in \mathcal{B}_{n, p, \geq i}}$ on $\oplus_{r<n} \mathcal{B}_{n-r,p-1,\geq i-1}^{\text{ card } \mathscr{D}^{\mathcal{d}}_{r}} \oplus_{r<n} \mathcal{B}_{n-r,p-1,\geq i}^{\text{ card } \mathscr{D}^{\backslash\mathcal{d}}_{r}}$ \footnote{Elements in arrival space are linearly independent by recursion hypothesis.} is invertible.
\end{center}
 As before, by recursive hypothesis, we replace elements of level $\leq i$ appearing in $D^{i}_{r,p}$, $r\geq 1$ by $\mathbb{Z}_{1[P]}$-linear combinations of elements of level $\geq i$ in the quotient $gr_{p-1}^{\mathfrak{D}} \mathcal{H}_{n-r}^{\geq i}$, which does not decrease the $P$-adic valuation.\\
Now looking at the expression for $D_{r,p}$ in Lemma $2.4.3$, we see that on the elements considered, \footnote{i.e. of the form $\zeta^{\mathfrak{m}} \left({x_{1}, \ldots , x_{p} \atop \epsilon_{1}, \ldots ,\epsilon_{p-1}, \epsilon_{p}\xi_{N} }\right)$, with $\epsilon_{i}\in \pm 1$ for $N=8$, $\epsilon_{i}=1$ else.} the left side is:
\begin{center}
Either $\zeta^{\mathfrak(l)}\left(  r\atop 1 \right) $ for type $\textsc{a,b,c} \qquad    $      Or $\zeta^{\mathfrak(l)}\left(  r\atop \xi \right) $ for Deconcatenation terms.
\end{center}
Using results in depth $1$ of Deligne and Goncharov (cf. $\S 2.4.3$), the deconcatenation terms are $P$-adically smaller. \\
\texttt{For instance}, for $N=\mlq 6 \mrq$, $r$ odd:
$$\zeta^{\mathfrak{l}}\left( r;  1\right) =\frac{2\cdot 6^{r-1}}{(1-2^{r-1})(1-3^{r-1})} \zeta^{\mathfrak{l}}(r;  \xi) , \quad \text{ and } \quad v_{3} \left( \frac{2\cdot 6^{r-1}}{(1-2^{r-1})(1-3^{r-1})}\right)  >0 .$$
\texttt{Nota Bene:} For $N=8$, $D_{r}$ has two independent components, $D_{r}^{\xi}$ and $D_{r}^{-\xi}$. We have to distinguish them, but the statement remains similar since the terms appearing in the left side are either $\zeta^{\mathfrak(l)}\left( r\atop \pm 1 \right)$, or deconcatenation terms, $\zeta^{\mathfrak(l)}\left(  r\atop \pm \xi \right)$, $2$-adically smaller by $\S 4.1$.\\
Thanks to congruences modulo $P$, only the deconcatenation terms remain:\\
$$D_{r,p} \left(\zeta^{\mathfrak{m}} \left({x_{1}, \ldots , x_{p} \atop \epsilon_{1}, \ldots ,\epsilon_{p-1},\epsilon_{p} \xi }\right)\right) = $$
$$  \delta_{ x_{p} \leq r \leq x_{p}+ x_{p-1}-1} (-1)^{r-x_{p}} \binom{r-1}{x_{p}-1} \zeta ^{\mathfrak{l}} \left( r\atop  \epsilon_{p}\xi \right) \otimes \zeta^{\mathfrak{m}} \left({ x_{1}, \ldots, x_{p-2}, x_{p-1}+x_{p}-r\atop  \epsilon_{1},  \cdots, \epsilon_{p-2}, \epsilon_{p-1}\epsilon_{p}\xi} \right) \pmod{P}.$$
As in the previous case, the matrix being modulo $P$ triangular with $1$ on the diagonal, has a determinant congruent at $1$ modulo $P$, and then, in particular, is invertible.
\\
\end{proof}


\paragraph{{\Large \texttt{EXAMPLE for} $\boldsymbol{N=2}$}:} Let us illustrate the previous proof by an example, for weight $n=9$, depth $p=3$, level $i=0$, with the previous notations.\\
Instead of $\mathcal{B}_{9, 3, \geq 0}$, we will restrict to the subfamily (corresponding to $\mathcal{A}$):
$$\mathcal{B}_{9, 3, \geq 0}^{0}\mathrel{\mathop:}= \left\{ \zeta^{\mathfrak{m}}(2a+1,2b+1,\overline{2c+1}) \text{ of weight } 9 \right\} \subset$$
$$ \mathcal{B}_{9, 3, \geq 0}\mathrel{\mathop:}= \left\{ \zeta^{\mathfrak{m}}(2a+1,2b+1,\overline{2c+1})\zeta^{\mathfrak{m}}(2)^{s}\text{ of weight } 9 \right\}$$
Note that $\zeta^{\mathfrak{m}}(2)$ being trivial under the coaction, the matrix $M_{9,3}$ is diagonal by blocks following the different values of $s$ and we can prove the invertibility of each block separately; here we restrict to the block $s=0$. The matrix $\widetilde{M}$ considered represents the coefficients of:
$$\zeta^{\mathfrak{m}}(\overline{2r+1})\otimes \zeta^{\mathfrak{m}}(2x+1,\overline{2y+1})\quad \text{ in }\quad  D_{2r+1,3}(\zeta^{\mathfrak{m}}(2a+1,2b+1,\overline{2c+1})).$$
The chosen order for the columns, resp. for the rows \footnote{I.e. for $\zeta^{\mathfrak{m}}(2a+1,2b+1,2c+1)$ resp. for $(D_{2r+1,3}, \zeta^{\mathfrak{m}}(2x+1,\overline{2y+1}))$.} is the lexicographic order applied to $(c,b,a)$ resp. to $(r,y,x)$. Modulo $2$, it only remains the terms of type \textsc{d,d'}, that is:
$$  D_{2r+1,3} (\zeta^{\mathfrak{m}}(2a+1, 2b+1, \overline{2c+1}))  \equiv \delta_{c \leq r \leq  b+c} \binom{2r}{2c} \zeta^{\mathfrak{m}} (2a+1, \overline{2 (b+c-r) +1}) \text{  }  \pmod{ 2}.$$
With the previous order, $\widetilde{M}_{9,3}$ is then, modulo $2$:\footnote{Notice that the first four rows are exact: no need of congruences modulo $2$ for $D_{1}$ because it acts as a deconcatenation on the base.}\\
\\
\begin{tabular}{c|c|c|c|c|c|c|c|c|c|c}
   $D_{r}, \zeta\backslash$ $\zeta$& $7,1,\overline{1}$ & $5,3,\overline{1}$ & $3,5,\overline{1}$& $1,7,\overline{1}$& $5,1,\overline{3}$&$3,3,\overline{3}$&$1,5,\overline{3}$&$3,1,\overline{5}$&$1,3,\overline{5}$ & $1,1,\overline{7}$ \\
  \hline
  $D_{1},\zeta^{\mathfrak{m}}(7,\overline{1})$ & $1$ & $0$ &$0$ &$0$ &$0$ &$0$ &$0$ &$0$ &$0$ &$0$ \\
   $D_{1},\zeta^{\mathfrak{m}}(5,\overline{3})$ & $0$ & $1$ &$0$ &$0$ &$0$ &$0$ &$0$ &$0$ &$0$ &$0$ \\
  $D_{1},\zeta^{\mathfrak{m}}(3,\overline{5})$ & $0$ & $0$ &$1$ &$0$ &$0$ &$0$ &$0$ &$0$ &$0$ &$0$ \\
  $D_{1},\zeta^{\mathfrak{m}}(1,\overline{7})$ & $0$ & $0$ &$0$ &$1$ &$0$ &$0$ &$0$ &$0$ &$0$ &$0$ \\
  $D_{3},\zeta^{\mathfrak{m}}(5,\overline{7})$ & $0$ & $0$ &$0$ &$0$ &$1$ &$0$ &$0$ &$0$ &$0$ &$0$ \\
  $D_{3},\zeta^{\mathfrak{m}}(3,\overline{3})$ & $0$ & $0$ &$0$ &$0$ &$0$ &$1$ &$0$ &$0$ &$0$ &$0$ \\
  $D_{3},\zeta^{\mathfrak{m}}(1,\overline{5})$ & $0$ & $0$ &$0$ &$0$ &$0$ &$0$ &$1$ &$0$ &$0$ &$0$ \\
  $D_{5},\zeta^{\mathfrak{m}}(3,\overline{1})$ & $0$ & $0$ &$0$ &$0$ &$0$ &$\binom{4}{2}$ &$0$ &$1$ &$0$ &$0$ \\
  $D_{5},\zeta^{\mathfrak{m}}(1,\overline{3})$ & $0$ & $0$ &$0$ &$0$ &$0$ &$0$ &$\binom{4}{2}$ &$0$ &$1$ &$0$ \\
  $D_{7},\zeta^{\mathfrak{m}}(1,\overline{1})$ & $0$ & $0$ &$0$ &$0$ &$0$ &$0$ &$\binom{6}{2}$ &$0$ &$\binom{6}{4}$ &$1$ \\
  \\
\end{tabular}.
As announced, $\widetilde{M}$ modulo $2$ is triangular with $1$ on the diagonal, thus obviously invertible.

\paragraph{ {\Large Proof of the Theorem  $\boldsymbol{5.2.4}$}.}
\begin{proof}
This Theorem comes down to the Lemma $5.2.3$ proving the freeness of $\mathcal{B}_{n, p, \geq i}$ in $gr_{p}^{\mathfrak{D}} \mathcal{H}_{n}^{\geq i}$ defining a $\mathbb{Z}_{odd}$-structure:
\begin{itemize}
 \item[$(i)$] By this Lemma, $\mathcal{B}_{n, p, \geq i}$ is linearly free in the depth graded, and $\partial^{i,\mathcal{d}}_{n,p}$, which decreases strictly the depth, is bijective on $\mathcal{B}_{n, p, \geq i}$. The family $\mathcal{B}_{n, \leq p, \geq i}$, all depth mixed is then linearly independent on $\mathcal{F}_{p}^{\mathfrak{D}} \mathcal{H}_{n}^{\geq i}\subset \mathcal{F}_{p}^{\mathfrak{D}} \mathcal{H}_{n}^{\geq i, \mathcal{MT}}$: easily proved by application of $\partial^{i,\mathcal{d}}_{n,p}$.\\
 By a dimension argument, since $\dim \mathcal{F}_{p}^{\mathfrak{D}} \mathcal{H}_{n}^{\geq i, \mathcal{MT}}= \text{ card } \mathcal{B}_{n, \leq p, \geq i}$, we deduce the generating property.
	\item[$(ii)$] By the lemma, this family is linearly independent, and by $(i)$ applied to depth $p-1$, 
	$$gr_{p}^{\mathfrak{D}} \mathcal{H}_{n}^{\geq i}\subset gr_{p}^{\mathfrak{D}} \mathcal{H}_{n}^{\geq i, \mathcal{MT}}.$$
	Then, by a dimension argument, since $\dim gr_{p}^{\mathfrak{D}} \mathcal{H}_{n}^{\geq i, \mathcal{MT}} = \text{ card } \mathcal{B}_{n, p, \geq i}$ we conclude on the generating property. The $\mathbb{Z}_{odd}$ structure has been proven in the previous lemma.\\
	By the bijectivity of $\partial_{n,p}^{i,\mathcal{d}}$ (still previous lemma), which decreases the depth, and using the freeness of the elements of a same depth in the depth graded, there is no linear relation between elements of  $\mathcal{B}_{n,\cdot, \geq i}$ of different depths in $\mathcal{H}_{n}^{\geq i} \subset \mathcal{H}^{\geq i \mathcal{MT}}_{n}$. The family considered is then linearly independent in $\mathcal{H}_{n}^{\geq i}$. Since $\text{card } \mathcal{B}_{n,\cdot, \geq i} =\dim \mathcal{H}^{\geq i, \mathcal{MT}}_{n}$, we conclude on the equality of the previous inclusions.
	\item[$(iii)$] The second exact sequence is obviously split since $ \mathcal{B}_{n, \cdot,\geq i+1}$ is a subset of $\mathcal{B}_{n}$. We already know that $\mathcal{B}_{n}$ is a basis of $\mathcal{H}_{n}$ and $\mathcal{B}_{n, \cdot, \geq i+1}$ is a basis of $\mathcal{H}_{n}^{\geq i+1}$. Therefore, it gives a map $\mathcal{H}_{n} \leftarrow\mathcal{H}_{n}^{\geq i+1}$ and split the first exact sequence. \\
	The construction of $cl_{n,\leq p, \geq i}(x)$, obtained from $cl_{n,p, \geq i}(x)$ applied repeatedly, is the following: 
\begin{center}
	$x\in\mathcal{B}_{n, \cdot, \leq i-1} $ is sent on $\bar{x}\in \mathcal{H}_{n}^{\geq i} \cong \langle\mathcal{B}_{n, \leq p, \geq i}\rangle_{\mathbb{Q}} $ by the projection $\pi_{0,i}$ and so $x -\bar{x} \in \mathcal{F}_{i-1}\mathcal{H}$.
\end{center}
Notice that the problem of making $cl(x)$ explicit boils down to the problem of describing the map $\pi_{0,i}$ in the bases $\mathcal{B}$. 
	\item[$(iv)$] By the previous statements, these elements are linearly independent in $\mathcal{F}_{i} \mathcal{H}^{MT}_{n}$. Moreover, their cardinal is equal to the dimension of $\mathcal{F}_{i} \mathcal{H}^{MT}_{n}$. It gives the basis announced, composed of elements $x\in \mathcal{B}_{n, \cdot, \leq i}$, each corrected by an element denoted $cl(x)$ of $ \langle\mathcal{B}_{n, \cdot, \geq i+1}\rangle_{\mathbb{Q}}$.
	\item[$(v)$] By the previous statements, these elements are linearly independent in $gr_{i} \mathcal{H}_{n}$, and by a dimension argument, we can conclude.
\end{itemize}
\end{proof}

\subsection{Specified Results}

\subsubsection{\textsc{The case } $N=2$.}
Here, since there is only one Galois descent from $\mathcal{H}^{2}$ to $\mathcal{H}^{1}$, the previous exponents for level filtrations can be omitted, as the exponent $2$ for $\mathcal{H}$ the space of motivic Euler sums. Set $\mathbb{Z}_{\text{odd}}= \left\{ \frac{a}{b} \text{ , } a\in\mathbb{Z}, b\in 2 \mathbb{Z}+1  \right\}$, rationals having a $2$-adic valuation positive or infinite. Let us define particular families of motivic Euler sums, a notion of level and of motivic level. 
\begin{defi}
\begin{itemize}
	\item[$\cdot$] $\mathcal{B}^{2}\mathrel{\mathop:}=\left\{\zeta^{\mathfrak{m}}(2x_{1}+1, \ldots, 2 x_{p-1}+1,\overline{2 x_{p}+1}) \zeta(2)^{\mathfrak{m},k}, x_{i} \geq 0, k \in \mathbb{N} \right\}.$\\
Here, the level is defined as the number of $x_{i}$ equal to zero.
	\item[$\cdot$] The filtration by the motivic ($\mathbb{Q}/\mathbb{Q},2/1$)-level, 
$$\mathcal{F}_{i}\mathcal{H}\mathrel{\mathop:}=\left\{ \mathfrak{Z} \in \mathcal{H}, \textrm{ such that } D^{-1}_{1}\mathfrak{Z} \in \mathcal{F}_{i-1} \mathcal{H} \text{ , } \forall r>0, D^{1}_{2r+1}\mathfrak{Z} \in \mathcal{F}_{i}\mathcal{H} \right\}.$$
\begin{center}
I.e. $\mathcal{F}_{i}$ is the largest submodule such that $\mathcal{F}_{i} / \mathcal{F}_{i-1}$ is killed by $D_{1}$.
\end{center}
\end{itemize}
\end{defi}
This level filtration commutes with the increasing depth filtration.\\
\\
\textsc{Remarks}: The increasing or decreasing filtration defined from the number of 1 appearing in the motivic multiple zeta values is not preserved by the coproduct, since the number of 1 can either decrease or increase (by at the most 1) and is therefore not \textit{motivic}.\\
\\
Let list some consequences of the results in $\S 5.2.3$, which generalize in particular a result similar to P. Deligne's one (cf. $\cite{De}$):
\begin{coro} The map $\mathcal{G}^{\mathcal{MT}} \rightarrow \mathcal{G}^{\mathcal{MT}'}$ is an isomorphism.\\
The elements of $\mathcal{B}_{n}$, $\zeta^{\mathfrak{m}}(2x_{1}+1, \ldots, \overline{2 x_{p}+1}) \zeta(2)^{k}$ of weight $n$, form a basis of motivic Euler sums of weight $n$, $\mathcal{H}^{2}_{n}=\mathcal{H}^{\mathcal{MT}_{2}}_{n}$, and define a $\mathbb{Z}_{odd}$-structure on the motivic Euler sums.
\end{coro}
\noindent
The period map, $\text{per}: \mathcal{H} \rightarrow \mathbb{C}$, induces the following result for the Euler sums:
\begin{center}
Each Euler sum is a $\mathbb{Z}_{odd}$-linear combination of Euler sums \\
$\zeta(2x_{1}+1, \ldots, \overline{2 x_{p}+1}) \zeta(2)^{k}, k\geq 0, x_{i} \geq 0$ of the same weight.
\end{center}
\newpage
\noindent
Here is the result on the $0^{\text{th}}$ level of the Galois descent from $\mathcal{H}^{1}$ to $\mathcal{H}^{2}$:
\begin{coro}
$$\mathcal{F}_{0}\mathcal{H}^{\mathcal{MT}_{2}}=\mathcal{F}_{0}\mathcal{H}^{2}=\mathcal{H}^{\mathcal{MT}_{1}}=\mathcal{H}^{1} .$$
A basis of motivic multiple zeta values in weight $n$, is formed by terms of $\mathcal{B}_{n}$ with $0$-level each corrected by linear combinations of elements of $\mathcal{B}_{n}$ of level $1$: 
\begin{multline}\nonumber
\mathcal{B}_{n}^{1}\mathrel{\mathop:}=\left\{   \zeta^{\mathfrak{m}}(2x_{1}+1, \ldots, \overline{2x_{p}+1})\zeta^{\mathfrak{m}}(2)^{s} + \sum_{y_{i} \geq 0 \atop \text{at least one } y_{i} =0} \alpha_{\textbf{x} , \textbf{y}} \zeta^{\mathfrak{m}}(2y_{1}+1, \ldots, \overline{2y_{p}+1})\zeta^{\mathfrak{m}}(2)^{s} + \right.\\
\left. \sum_{\text{lower depth } q<p, z_{i}\geq 0 \atop \text{ at least one } z_{i} =0} \beta_{\textbf{x}, \textbf{z}} \zeta^{\mathfrak{m}}(2 z_{1}+1, \ldots, \overline{2 z_{q}+1})\zeta^{\mathfrak{m}}(2)^{s}, x_{i}>0 , \alpha_{\textbf{x} , \textbf{y}} , \beta_{\textbf{x} , \textbf{z}} \in\mathbb{Q},\right\}_{\sum x_{i}= \sum y_{i}=\sum z_{i}= \frac{n-p}{2} -s}.
\end{multline}
\end{coro}

\paragraph{Honorary.}
About the first condition in $\ref{criterehonoraire}$ to be honorary:
\begin{lemm}\label{condd1}
Let $\zeta^{\mathfrak{m}}(n_{1},\cdots,n_{p}) \in\mathcal{H}^{2}$, a motivic Euler sum, with $n_{i}\in\mathbb{Z}^{\ast}$, $ n_{p}\neq 1$. Then:
$$\forall i \text{ ,  } n_{i}\neq -1 \Rightarrow D_{1}(\zeta^{\mathfrak{m}}(n_{1},\cdots,n_{p}))=0 $$
\end{lemm}
\begin{proof}
Looking at all iterated integrals of length $1$ in $\mathcal{L}$, $I^{\mathfrak{l}}(a;b;c)$, $a,b,c\in \lbrace 0,\pm 1\rbrace$: the only non zero ones are these with a consecutive $\lbrace 1,-1\rbrace$ or $\lbrace -1,1\rbrace$ sequence in the iterated integral, with the condition that extremities are different, that is:
$$I(0;1;-1), I(0;-1;1), I(1;-1;0), I(-1;+1;0), I(-1;\pm 1;1), I(1;\pm 1;-1).$$
Moreover, they are all equal to $\pm \log^{\mathfrak{a}} (2)$ in the Hopf algebra $\mathcal{A}$. Consequently, if there is no $-1$ in the Euler sums notation, it implies that $D_{1}$ would be zero.
\end{proof}

\paragraph{Comparison with Hoffman's basis. } Let compare:
\begin{itemize}
\item[$(i)$] The Hoffman basis of $\mathcal{H}^{1}$ formed by motivic MZV with only $2$ and $3$ ($\cite{Br2}$)
$$\mathcal{B}^{H}\mathrel{\mathop:}= \left\{\zeta^{\mathfrak{m}} (x_{1}, \ldots, x_{k}), \text{  where } x_{i}\in\left\{2,3\right\} \right\}.$$
\item[$(ii)$] 
$\mathcal{B}^{1}$, the base of $\mathcal{H}^{1}$ previously obtained (Corollary $5.2.7$).
\end{itemize}

Beware, the index $p$ for $\mathcal{B}^{H}$ indicates the number of 3 among the $x_{i}$, whereas for $\mathcal{B}^{1}$, it still indicates the depth; in both case, it can be seen as the \textit{motivic depth} (cf. $\S 2.4.3$):

\begin{coro}
$\mathcal{B}^{1}_{n,p}$ is a basis of $gr_{p}^{\mathfrak{D}} \langle\mathcal{B}^{H}_{n,p}\rangle_{\mathbb{Q}}$ and defines a $\mathbb{Z}_{\text{odd}}$-structure.\\
I.e. each element of the Hoffman basis of weight $n$ and with $p$ three, $p>0$, decomposes into a $\mathbb{Z}_{\text{odd}}$-linear combination of $\mathcal{B}^{1}_{n,p}$ elements plus terms of depth strictly less than $p$.
\end{coro}
\begin{proof}
Deduced from the previous results, with the $\mathcal{Z}_{odd}$ structure of the basis for Euler sums.
\end{proof}

\subsubsection{\textsc{The cases } $N=3,4$.}

For $N=3,4$ there are a generator in each degree $\geq 1$ and two Galois descents. \\
\begin{defi}
\begin{itemize}
	\item[$\cdot$] \textbf{Family:}  $\mathcal{B}\mathrel{\mathop:}=\left\{\zeta^{\mathfrak{m}}\left({x_{1}, \ldots,x_{p}\atop  1, \ldots , 1, \xi }\right) (2i \pi)^{s,\mathfrak{m}}, x_{i} \geq 1, s \geq 0 \right\}$. 
	\item[$\cdot$] \textbf{Level:} 
$$\begin{array}{lll}
\text{ The $(k_{N}/k_{N},P/1)$-level } & \text{ is defined as } & \text{ the number of $x_{i}$ equal to 1 }\\
\text{ The $(k_{N}/\mathbb{Q},P/P)$-level } & \text{  } & \text{ the number of  $x_{i}$ even }\\
\text{ The $(k_{N}/\mathbb{Q},P/1)$-level } & \text{  } & \text{ the number of even $x_{i}$ or equal to $1$ }
\end{array}	$$
	\item[$\cdot$]  \textbf{Filtrations by the motivic level:} 
$\mathcal{F}^{\mathcal{d}} _{-1} \mathcal{H}^{N}=0$ and $\mathcal{F}^{\mathcal{d}} _{i} \mathcal{H}^{N}$ is the largest submodule of $\mathcal{H}^{N}$ such that $\mathcal{F}^{\mathcal{d}}_{i}\mathcal{H}^{N}/\mathcal{F}^{\mathcal{d}} _{i-1}\mathcal{H}^{N}$ is killed by $\mathscr{D}^{\mathcal{d}}$, where

$$\mathscr{D}^{\mathcal{d}}	= \begin{array}{ll}
\lbrace D^{\xi}_{1} \rbrace &  \text{ for } \mathcal{d}=(k_{N}/k_{N},P/1)\\
\lbrace(D^{\xi}_{2r})_{r>0} \rbrace & \text{ for } \mathcal{d}=(k_{N}/\mathbb{Q},P/P)\\
\lbrace D^{\xi}_{1},(D^{\xi}_{2r})_{r>0} \rbrace &  \text{ for } \mathcal{d}=(k_{N}/\mathbb{Q},P/1) \\
\end{array}.
$$
\end{itemize}
\end{defi}
\textsc{Remarks}: 
\begin{itemize}
\item[$\cdot$] As before, the increasing, or decreasing, filtration that we could define by the number of 1 (resp. number of even) appearing in the motivic multiple zeta values is not preserved by the coproduct, since the number of 1 can either diminish or increase (at most 1), so is not motivic. 
\item[$\cdot$] An effective way of seeing those motivic level filtrations, giving a recursive criterion:
$$\hspace*{-0.5cm}\mathcal{F}_{i}^{k_{N}/\mathbb{Q},P/P }\mathcal{H}= \left\{ \mathfrak{Z} \in \mathcal{H}, \textrm{ s. t. } \forall r > 0 \text{  , }  D^{\xi}_{2r}(\mathfrak{Z}) \in \mathcal{F}_{i-1}^{k_{N}/\mathbb{Q},P/P}\mathcal{H} \text{  , } \forall r \geq 0 \text{  , }  D^{\xi}_{2r+1}(\mathfrak{Z}) \in \mathcal{F}_{i}^{ k_{N}/\mathbb{Q},P/P}\mathcal{H} \right\}.$$
\end{itemize}
\noindent
We deduce from the result in $\S 5.2.3$ a result of P. Deligne ($i=0$, cf. $\cite{De}$):
\begin{coro}
The elements of $\mathcal{B}^{N}_{n,p, \geq i}$ form a basis of $gr_{p}^{\mathfrak{D}} \mathcal{H}_{n}/ \mathcal{F}_{i-1} \mathcal{H}_{n}$.\\
In particular the map $\mathcal{G}^{\mathcal{MT}_{N}} \rightarrow \mathcal{G}^{\mathcal{MT}_{N}'}$ is an isomorphism.
The elements of $\mathcal{B}_{n}^{N}$, form a basis of motivic multiple zeta value relative to $\mu_{N}$,  $\mathcal{H}_{n}^{N}$.
\end{coro}
The level $0$ of the filtrations considered for $N' \vert N\in \left\lbrace 3,4 \right\rbrace $ gives the Galois descents:
\begin{coro}
 A basis of $\mathcal{H}_{n}^{N'} $ is formed by elements of $\mathcal{B}_{n}^{N}$ of level $0$ each corrected by linear combination of elements  $\mathcal{B}_{n}^{N}$ of level $ \geq 1$. In particular, with $\xi$ primitive:
\begin{itemize}
	\item[$\cdot$] \textbf{Galois descent} from $N'=1$ to $N=3,4$: A basis of motivic multiple zeta values: 
$$\hspace*{-0.5cm}\mathcal{B}^{1 ; N}  \mathrel{\mathop:}=  \left\{  \zeta^{\mathfrak{m}}\left({2x_{1}+1, \ldots, 2x_{p}+1\atop  1, \ldots, 1, \xi} \right)  \zeta^{\mathfrak{m}}(2)^{s}   + \sum_{y_{i} \geq 0 \atop \text{ at least one $y_{i}$ even or } = 1} \alpha_{\textbf{x},\textbf{y}} \zeta^{\mathfrak{m}} \left({y_{1}, \ldots, y_{p}\atop 1, \ldots, 1, \xi } \right)\zeta^{\mathfrak{m}}(2)^{s}  \right.$$
$$ \left. + \sum_{\text{ lower depth } q<p,  \atop \text{ at least one even or } = 1} \beta_{\textbf{x},\textbf{z}}  \zeta^{\mathfrak{m}}\left({z_{1}, \ldots, z_{q}\atop 1, \ldots, 1, \xi } \right)\zeta^{\mathfrak{m}}(2)^{s} \text{ ,  } x_{i}>0 , \alpha_{\textbf{x},\textbf{y}} , \beta_{\textbf{x},\textbf{z}} \in\mathbb{Q} \right\}. $$
	\item[$\cdot$]  \textbf{Galois descent} from $N'=2$ to $N=4$: A basis of motivic Euler sums:
$$\hspace*{-0.5cm}\mathcal{B}^{2; 4}\mathrel{\mathop:}= \left\{  \zeta^{\mathfrak{m}} \left({2x_{1}+1, \ldots, 2x_{p}+1\atop  1, \ldots, 1, \xi_{4}} \right)\zeta^{\mathfrak{m}}(2)^{s} + \sum_{y_{i}>0 \atop \text{at least one even}} \alpha_{\textbf{x},\textbf{y}} \zeta^{\mathfrak{m}}\left({y_{1}, \ldots, y_{p}\atop 1, \ldots, 1, \xi_{4} } \right)\zeta^{\mathfrak{m}}(2)^{s} \right.$$
$$ \left.   +\sum_{\text{lower depth } q<p \atop z_{i}>0, \text{at least one even}} \beta_{\textbf{x},\textbf{z}} \zeta^{\mathfrak{m}}\left( z_{1}, \ldots, z_{q} \atop 1, \ldots, 1, \xi_{4} \right) \zeta^{\mathfrak{m}}(2)^{s}   \text{  ,  }  x_{i}\geq 0 , \alpha_{\textbf{x},\textbf{y}}, \beta_{\textbf{x},\textbf{z}}\in\mathbb{Q}  \right\} .$$
\item[$\cdot$] Similarly, replacing $\xi_{4}$ by $\xi_{3}$ in $\mathcal{B}^{2; 4}$, this gives a basis of:
$$\mathcal{F}^{k_{3}/\mathbb{Q},3/3}_{0} \mathcal{H}_{n}^{3}=\boldsymbol{\mathcal{H}_{n}^{\mathcal{MT}(\mathbb{Z}[\frac{1}{3}])}}.$$
\item[$\cdot$] A basis of $\mathcal{F}^{k_{N}/k_{N},P/1}_{0} \mathcal{H}_{n}^{N}=\boldsymbol{\mathcal{H}_{n}^{\mathcal{MT}(\mathcal{O}_{N})}}$, with $N= 3,4$:
$$\hspace*{-0.5cm}\mathcal{B}^{N \text{ unram}}\mathrel{\mathop:}= \left\{  \zeta^{\mathfrak{m}} \left({x_{1}, \ldots, x_{p} \atop  1, \ldots, 1, \xi} \right)\zeta^{\mathfrak{m}}(2)^{s} + \sum_{y_{i}>0 \atop \text{at least one } 1} \alpha_{\textbf{x},\textbf{y}} \zeta^{\mathfrak{m}}\left({y_{1}, \ldots, y_{p}\atop 1, \ldots, 1, \xi} \right)\zeta^{\mathfrak{m}}(2)^{s} \right.$$
$$ \left.   +\sum_{\text{lower depth } q<p \atop z_{i}>0, \text{at least one } 1} \beta_{\textbf{x},\textbf{z}} \zeta^{\mathfrak{m}}\left( z_{1}, \ldots, z_{q} \atop 1, \ldots, 1, \xi \right) \zeta^{\mathfrak{m}}(2)^{s}   \text{  ,  }  x_{i} > 0 , \alpha_{\textbf{x},\textbf{y}}, \beta_{\textbf{x},\textbf{z}}\in\mathbb{Q}  \right\} .$$
\end{itemize}
\end{coro}
\noindent
\texttt{Nota Bene:} Notice that for the last two level $0$ spaces, $\mathcal{H}_{n}^{\mathcal{MT}(\mathcal{O}_{N})}$, $N=3,4$ and $\mathcal{H}_{n}^{\mathcal{MT}(\mathbb{Z}[\frac{1}{3}])}$, we still do not have another way to reach them, since those categories of mixed Tate motives are not simply generated by a motivic fundamental group.

\subsubsection{\textsc{The case } $N=8$.}

For $N=8$ there are two generators in each degree $\geq 1$ and three possible Galois descents: with $\mathcal{H}^{4}$, $\mathcal{H}^{2}$ or $\mathcal{H}^{1}$.\\

\begin{defi}
\begin{itemize}
	\item[$\cdot$] \textbf{Family:} $\mathcal{B}\mathrel{\mathop:}=\left\{\zeta^{\mathfrak{m}}\left( {x_{1}, \ldots,x_{p}\atop  \epsilon_{1}, \ldots , \epsilon_{p-1},\epsilon_{p} \xi }\right)(2i \pi)^{s,\mathfrak{m}}, x_{i} \geq 1, \epsilon_{i}\in \left\{\pm 1\right\} s \geq 0 \right\}$. 
		\item[$\cdot$] \textbf{Level}, denoted $i$: 
$$\begin{array}{lll}
 \text{ The $(k_{8}/k_{4},2/2)$-level } &  \text{ is the number of } &  \text{  $\epsilon_{j}$ equal to $-1$ } \\
  \text{ The $(k_{8}/\mathbb{Q},2/2)$-level } &  \text{  } &  \text{ $\epsilon_{j}$ equal to $-1$ $+$ even $x_{j}$ } \\
   \text{ The $(k_{8}/\mathbb{Q},2/1)$-level } &  \text{  } &  \text{ $\epsilon_{j}$ equal to $-1$, $+$ even $x_{j}$ $+$ $x_{j}$ equal to $1$. } 
\end{array}$$

\item[$\cdot$]  \textbf{Filtrations by the motivic level:} 
$\mathcal{F}^{\mathcal{d}} _{-1} \mathcal{H}^{8}=0$ and $\mathcal{F}^{\mathcal{d}} _{i} \mathcal{H}^{8}$ is the largest submodule of $\mathcal{H}^{8}$ such that $\mathcal{F}^{\mathcal{d}}_{i}\mathcal{H}^{8}/\mathcal{F}^{\mathcal{d}} _{i-1}\mathcal{H}^{8}$ is killed by $\mathscr{D}^{\mathcal{d}}$, where
$$\mathscr{D}^{\mathcal{d}}	= \begin{array}{ll}
\left\lbrace (D^{\xi}_{r}- D^{-\xi}_{r})_{r>0} \right\rbrace  &  \text{ for } \mathcal{d}=(k_{8}/k_{4},2/2)\\
\left\lbrace  (D^{\xi}_{2r+1}- D^{-\xi}_{2r+1})_{r\geq 0}, (D^{\xi}_{2r})_{r>0},( D^{-\xi}_{2r})_{r>0}  \right\rbrace  & \text{ for } \mathcal{d}=(k_{8}/\mathbb{Q},2/2)\\
\left\lbrace  (D^{\xi}_{2r+1}- D^{-\xi}_{2r+1})_{r> 0}, D^{\xi}_{1},  D^{-\xi}_{1}, (D^{\xi}_{2r})_{r>0},( D^{-\xi}_{2r})_{r>0}  \right\rbrace  &  \text{ for } \mathcal{d}=(k_{8}/\mathbb{Q},2/1) \\
\end{array}.
$$			
\end{itemize}
\end{defi}

\begin{coro}
 A basis of $\mathcal{H}_{n}^{N'} $ is formed by elements of $\mathcal{B}_{n}^{N}$ of level $0$ each corrected by linear combination of elements  $\mathcal{B}_{n}^{N}$ of level $ \geq 1$. In particular, with $\xi$ primitive:
\begin{description}
\item[$\boldsymbol{8 \rightarrow 1} $:] A basis of MMZV:
$$\hspace*{-0.5cm}\mathcal{B}^{1;8}\mathrel{\mathop:}= \left\{ \zeta^{\mathfrak{m}}\left( 2x_{1}+1, \ldots, 2x_{p}+1 \atop  1, \ldots, 1, \xi \right)\zeta^{\mathfrak{m}}(2)^{s} + \sum_{y_{i} \text{at least one even or } =1 \atop { or  one } \epsilon_{i}=-1 } \alpha_{\textbf{x},\textbf{y}} \zeta^{\mathfrak{m}}\left( y_{1}, \ldots, y_{p} \atop \epsilon_{1}, \ldots, \epsilon_{p-1}, \epsilon_{p}\xi  \right)\zeta^{\mathfrak{m}}(2)^{s} \right.$$
$$\left. + \sum_{q<p \text{ lower depth, level } \geq 1} \beta_{\textbf{x},\textbf{z}} \zeta^{\mathfrak{m}}\left(z_{1}, \ldots, z_{q} \atop  \widetilde{\epsilon}_{1}, \ldots, \widetilde{\epsilon}_{q}\xi \right)\zeta^{\mathfrak{m}}(2)^{s}  \text{  , }x_{i}>0 , \alpha_{\textbf{x},\textbf{y}}, \beta_{\textbf{x},\textbf{z}}\in\mathbb{Q}\right\}.$$
\item[$\boldsymbol{8 \rightarrow 2 } $:] A basis of motivic Euler sums:
$$\hspace*{-0.5cm}\mathcal{B}^{2;8} \mathrel{\mathop:}=  \left\{  \zeta^{\mathfrak{m}} \left(2x_{1}+1, \ldots, 2x_{p}+1 \atop  1, \ldots, 1, \xi\right)\zeta^{\mathfrak{m}}(2)^{s}  +\sum_{y_{i} \text{ at least one even} \atop \text{or one }\epsilon_{i}=-1} \alpha_{\textbf{ x},\textbf{y}} \zeta^{\mathfrak{m}}\left( y_{1}, \ldots, y_{p} \atop \epsilon_{1}, \ldots, \epsilon_{p-1}, \epsilon_{p}\xi  \right)\zeta^{\mathfrak{m}}(2)^{s}  \right.$$
$$\left. + \sum_{\text{lower depth} q<p \atop \text{with level} \geq 1 } \beta_{\textbf{x},\textbf{z}} \zeta^{\mathfrak{m}}\left(z_{1}, \ldots, z_{q} \atop \widetilde{\epsilon}_{1}, \ldots, \widetilde{\epsilon}_{q}\xi\right)\zeta^{\mathfrak{m}}(2)^{s} \text{  ,  }x_{i}\geq 0,\alpha_{\textbf{x},\textbf{y}}, \beta_{\textbf{x},\textbf{y}} \in\mathbb{Q} \right\}.$$
	\item[$\boldsymbol{ 8 \rightarrow 4 } $:] A basis of MMZV relative to $\mu_{4}$:
$$\hspace*{-0.5cm}\mathcal{B}^{4;8}\mathrel{\mathop:}= \left\{ \zeta^{\mathfrak{m}}\left( x_{1}, \ldots, x_{p} \atop  1, \ldots, 1, \xi \right)(2i\pi)^{s}  + \sum_{\text{ at least one }\epsilon_{i}=-1} \alpha_{\textbf{x}, \textbf{y}} \zeta^{\mathfrak{m}}\left(y_{1}, \ldots, y_{p} \atop \epsilon_{1}, \ldots, \epsilon_{p-1},\epsilon_{p} \xi \right)(2 i \pi)^{s} \right.$$
$$ \left. + \sum_{\text{lower depth, level } \geq 1} \beta_{\textbf{x},\textbf{z}} \zeta^{\mathfrak{m}}\left( z_{1}, \ldots,z_{q} \atop  \widetilde{\epsilon}_{1}, \ldots, \widetilde{\epsilon}_{q}\xi  \right)(2i \pi)^{s} \alpha_{\textbf{x},\textbf{y}}, \beta_{\textbf{x},\textbf{z}}\in\mathbb{Q} \right\}.$$
\end{description}
\end{coro}

\subsubsection{\textsc{The case } $N=\mlq 6 \mrq$.}

For the unramified category $\mathcal{MT}(\mathcal{O}_{6})$, there is one generator in each degree $>1$ and one Galois descent with $\mathcal{H}^{1}$.\\
\\
First, let us point out this sufficient condition for a MMZV$_{\mu_{6}}$ to be unramified:
\begin{lemm}
$$\text{Let  } \quad  \zeta^{\mathfrak{m}} \left( n_{1},\cdots,n_{p} \atop \epsilon_{1}, \ldots, \epsilon_{p} \right)  \in\mathcal{H}^{\mathcal{MT}(\mathcal{O}_{6} \left[ \frac{1}{6}\right] )} \text{ a motivic MZV} _{\mu_{6}}, \quad \text{ such that : \footnote{In the iterated integral notation, the associated roots of unity are $\eta_{i}\mathrel{\mathop:}= (\epsilon_{i}\cdots \epsilon_{p})^{-1}$.}}$$ 
$$\begin{array}{ll}
 & \text{ Each } \eta_{i} \in \lbrace 1, \xi_{6} \rbrace    \\
\textsc{ or }&  \text{ Each } \eta_{i} \in \lbrace 1, \xi^{-1}_{6} \rbrace
\end{array}  \quad \quad \text{ Then, } \quad \zeta^{\mathfrak{m}} \left( n_{1},\cdots,n_{p} \atop \epsilon_{1}, \ldots, \epsilon_{p} \right)  \in \mathcal{H}^{\mathcal{MT}(\mathcal{O}_{6})}$$
\end{lemm}
\begin{proof}
Immediate, by Corollary, $\ref{ramif346}$, and with the expression of the derivations $(\ref{drz})$ since these families are stable under the coaction. 
\end{proof}

\begin{defi}
\begin{itemize}
	\item[$\cdot$]  \textbf{Family}: $\mathcal{B}\mathrel{\mathop:}=\left\{\zeta^{\mathfrak{m}}\left( {x_{1}, \ldots,x_{p}\atop  1, \ldots , 1,\xi) } \right)(2i \pi)^{s,\mathfrak{m}}, x_{i} > 1, s \geq 0 \right\}$. 
			\item[$\cdot$] \textbf{Level:} The $(k_{6}/\mathbb{Q},1/1)$-level, denoted $i$, is defined as the number of even $x_{j}$.
				\item[$\cdot$] \textbf{Filtration by the motivic }   $(k_{6}/\mathbb{Q},1/1)$-\textbf{level}:
\begin{center}
	$\mathcal{F}^{(k_{6}/\mathbb{Q},1/1)} _{-1} \mathcal{H}^{6}=0$ and $\mathcal{F}^{(k_{6}/\mathbb{Q},1/1)} _{i} \mathcal{H}^{6}$ is the largest submodule of $\mathcal{H}^{6}$ such that $\mathcal{F}^{(k_{6}/\mathbb{Q},1/1)}_{i}\mathcal{H}^{6}/\mathcal{F}^{(k_{6}/\mathbb{Q},1/1)} _{i-1}\mathcal{H}^{6}$ is killed by $\mathscr{D}^{(k_{6}/\mathbb{Q},1/1)}=\left\lbrace D^{\xi}_{2r} , r>0 \right\rbrace $.
	\end{center}	
\end{itemize}
\end{defi}

\begin{coro} Galois descent from $N'=1$ to $N=\mlq 6 \mrq$ unramified. A basis of MMZV:
$$\mathcal{B}^{1;6} \mathrel{\mathop:}= \left\{  \zeta^{\mathfrak{m}}\left( 2x_{1}+1, \ldots, 2x_{p}+1 \atop  1, \ldots, 1, \xi \right)\zeta^{\mathfrak{m}}(2)^{s} + \sum_{y_{i} \text{ at least one even}} \alpha_{\textbf{x},\textbf{y}}\zeta^{\mathfrak{m}}\left( y_{1}, \ldots, y_{p} \atop 1, \ldots, 1, \xi \right)\zeta^{\mathfrak{m}}(2)^{s} \right.$$ 
$$\left.   +\sum_{\text{lower depth, level } \geq 1}\beta_{\textbf{x},\textbf{z}} \zeta^{\mathfrak{m}} \left( z_{1}, \ldots, z_{q} \atop 1, \ldots, 1, \xi  \right)\zeta^{\mathfrak{m}}(2)^{s} \text{  , } \alpha_{\textbf{x},\textbf{y}}, \beta_{\textbf{x},\textbf{z}}\in \mathbb{Q},  x_{i}>0 \right\}.$$
\end{coro}

\chapter{Miscellaneous uses of the coaction}

\section{Maximal depth terms, $\boldsymbol{gr^{\mathfrak{D}}_{\max}\mathcal{H}_{n}}$}

The coaction enables us to compute, by a recursive procedure, the coefficients of the terms of \textit{maximal depth}, i.e. the projection on the graded $\boldsymbol{gr^{\mathfrak{D}}_{\max}\mathcal{H}_{n}}$. In particular, let look at:
\begin{itemize}
\item[$\cdot$] For $N=1$, when weight is a multiple of $3$ ($w=3d$), such as depth $p>d$:
$$gr^{\mathfrak{D}}_{p}\mathcal{H}_{3d} =\mathbb{Q} \zeta^{\mathfrak{m}}(3)^{d}.$$
\item[$\cdot$] Another simple case is for $N=2,3,4$, when weight equals depth, which is referred to as the \textit{diagonal comodule}:
$$gr^{\mathfrak{D}}_{p}\mathcal{H}_{p} =\mathbb{Q} \zeta^{\mathfrak{m}}\left( 1 \atop \xi_{N}\right) ^{p}.$$
\end{itemize}
The space $gr^{\mathfrak{D}}_{\max}\mathcal{H}^{N}_{n}$ is usually more than $1$ dimensional, but the methods presented below could generalize.

\subsection{MMZV, weight $\boldsymbol{3d}$.}

\paragraph{Preliminaries: Linearized Ihara action.}
The linearisation of the map $\circ:  U\mathfrak{g} \otimes U\mathfrak{g} \rightarrow  U\mathfrak{g}$ induced by Ihara action (cf. $\S 2.4$) can be defined recursively on words by, with $\eta\in\mu_{N}$:\nomenclature{$\underline{\circ}$}{linearized Ihara action}
\begin{equation}\label{eq:circlinear}
\begin{array}{lll}
 \underline{\circ}: \quad   U\mathfrak{g} \otimes U\mathfrak{g}  \rightarrow  U\mathfrak{g}:  & a \underline{\circ} e_{0}^{n} & = e_{0}^{n} a \\
 & a \underline{\circ} e_{0}^{n}e_{\eta} w & = e_{0}^{n} ([\eta].a)  e_{\eta}w + e_{0}^{n} e_{\eta } ([\eta].a)^{\ast} w + e_{0}^{n} e_{\eta} (a\underline{\circ} w), \\
\end{array}
\end{equation}
where ${\ast}$ stands for the following involution: 
$$(a_{1} \cdots a_{n})^{\ast}\mathrel{\mathop:}=(-1)^{n}a_{n} \cdots a_{1}.$$
For this paragraph, from now, let $N=1$ and let use the \textit{commutative polynomial setting}, introducing the isomorphism:\nomenclature{ $\rho$}{isomorphism used to pass to a commutative polynomial setting}
\begin{align}\label{eq:rho}
 \rho: U \mathfrak{g} & \longrightarrow  \mathbb{Q} \langle Y\rangle\mathrel{\mathop:}=\mathbb{Q} \langle y_{0}, y_{1}, \ldots, y_{n},\cdots \rangle  \\
 e_{0}^{n_{0}}e_{1} e_{0}^{n_{1}} \cdots e_{1} e_{0}^{n_{p}} & \longmapsto  y_{0}^{n_{0}} y_{1}^{n_{1}} \cdots y_{p}^{n_{p}} \nonumber
 \end{align}
Remind that if $\Phi\in U \mathfrak{g}$ satisfies the linearized $\shuffle$ relation, it means that $\Phi$ is primitive for $\Delta_{\shuffle}$, and equivalently that $\phi_{u \shuffle v}=0$, with $\phi_{w}$ the coefficient of $w$ in $\Phi$. In particular, this is verified for $\Phi$ in the motivic Lie algebra $\mathfrak{g}^{\mathfrak{m}}$.\\
This property implies for $f=\rho(\Phi)$ a translation invariance (cf. $6.2$ in $\cite{Br3}$)
\begin{equation} \label{eq:translationinv} 
f(y_{0},y_{1},\cdots, y_{p})= f(0,y_{1}-y_{0}, \ldots, y_{p}-y_{0}).
\end{equation} 
Let consider the map:
\begin{align} \label{eq:fbar} 
\mathbb{Q} \langle Y\rangle & \longrightarrow\mathbb{Q} \langle X\rangle = \mathbb{Q} \langle x_{1}, \ldots, x_{n},\cdots\rangle  , & \\
\quad \quad f & \longmapsto \overline{f} & \text{  where } \overline{f}(x_{1},\cdots, x_{p})\mathrel{\mathop:}=f(0,x_{1},\cdots, x_{p}).\nonumber
\end{align}
If $f$ is translation invariant, $f(y_{0}, y_{1}, \ldots, y_{p})=\overline{f}(y_{1}-y_{0},\cdots, y_{p}-y_{0})$.\\
The image of $\mathfrak{g}^{\mathfrak{m}}$ under $\rho$ is contained in the subspace of polynomial in $y_{i}$ invariant by translation. Hence we can consider alternatively in the following $\phi\in\mathfrak{g}^{\mathfrak{m}}$, $f=\rho(\phi)$ or $\overline{f}$.\\
\\
Since the linearized action $\underline{\circ}$ respects the $\mathcal{D}$-grading, it defines, via the isomorphism $\rho: gr^{r}_{\mathfrak{D}} U \mathfrak{g} \rightarrow \mathbb{Q}[y_{0}, \ldots, y_{r}]$, graded version of $(\ref{eq:rho})$, a map:
$$\underline{\circ}: \mathbb{Q}[y_{0}, \ldots, y_{r}]\otimes \mathbb{Q}[y_{0}, \ldots, y_{s}] \rightarrow \mathbb{Q}[y_{0}, \ldots, y_{r+s}] \text{ , which in the polynomial representation is:}$$
\begin{multline}\label{eq:circpolynom}
f\underline{\circ} g (y_{0}, \ldots, y_{r+s})=\sum_{i=0}^{s} f(y_{i}, \ldots, y_{i+r})g(y_{0}, \ldots, y_{i}, y_{i+r+1}, \ldots,  y_{r+s}) \\
+ (-1)^{\deg f+r}\sum_{i=1}^{s} f(y_{i+r}, \ldots, y_{i})g(y_{0}, \ldots, y_{i-1}, y_{i+r}, \ldots,  y_{r+s}).
\end{multline}
Or via the isomorphism $\overline{\rho}: gr^{r}_{\mathfrak{D}} U \mathfrak{g} \rightarrow \mathbb{Q}[x_{1}, \ldots, x_{r}]$, graded version of $(\ref{eq:fbar})\circ (\ref{eq:rho}) $:
\begin{multline}\label{eq:circpolynomx}
f\underline{\circ} g (x_{1}, \ldots, x_{r+s})=\sum_{i=0}^{s} f(x_{i+1}-x_{i}, \ldots, x_{i+r}-x_{i})g(y_{1}, \ldots, x_{i}, x_{i+r+1}, \ldots,  x_{r+s}) \\
+ (-1)^{\deg f+r}\sum_{i=1}^{s} f(x_{i+r-1}-x_{i+r}, \ldots, x_{i}-x_{i+r})g(x_{1}, \ldots, x_{i-1}, x_{i+r}, \ldots,  x_{r+s}).
\end{multline}

\paragraph{Coefficient of $\boldsymbol{\zeta(3)^{d}}$.}
If the weight $w$ is divisible by $3$, for motivic multiple zeta values, it boils down to compute the coefficient of $\zeta^{\mathfrak{m}}(3)^{\frac{w}{3}}$ and a recursive procedure is given in Lemma $6.1.1$.\\
\\
Since $gr_{d}^{\mathfrak{D}} \mathcal{H}_{3d}^{1} $ is one dimensional, generated by $\zeta^{\mathfrak{m}}(3)^{d}$, we can consider the projection:
\begin{equation}
\vartheta : gr_{d}^{\mathfrak{D}} \mathcal{H}_{3d}^{1} \rightarrow \mathbb{Q}.
\end{equation}
Giving a motivic multiple zeta value $\zeta^{\mathfrak{m}}(n_{1}, \ldots, n_{d})$, of depth $d$ and weight $w=3d$, there exists a rational $\alpha_{\underline{\textbf{n}}}= \vartheta(\zeta(n_{1}, \ldots, n_{d}))$ such that:
\begin{framed}
\begin{equation}
\zeta^{\mathfrak{m}}(n_{1}, \ldots, n_{d}) = \frac{\alpha_{\underline{\textbf{n}}}} {d!} \zeta^{\mathfrak{m}}(3)^{d} + \text{ terms of depth strictly smaller than } d. \footnote{The terms of depth strictly smaller than $d$ can be expressible in terms of the Deligne basis for instance.}
\end{equation} 
\end{framed}
\noindent
In the depth graded in depth 1, $\partial \mathfrak{g}^{\mathfrak{m}}_{1}$, the generators are:
$$\overline{\sigma}_{2i+1}= (-1)^{i} (\text{ad} e_{0})^{2i} (e_{1}) .$$
We are looking at, in the depth graded:
\begin{equation} \label{eq:expcirc3}
\exp_{\circ}(\overline{\sigma_{3}})\mathrel{\mathop:}=\sum_{n=0}^{n}\frac{1}{n!} \overline{\sigma_{3}} \circ \cdots \circ \overline{\sigma_{3}}= \sum_{n=0}^{n}\frac{1}{n!}  (\text{ad}(e_{0})^{2} (e_{1}))^{\underline{\circ} n}.
\end{equation} 
In the commutative polynomial representation, via $\overline{\rho}$, since $ \overline{\rho}(\overline{\sigma}_{2n+1})= x_{1}^{2n}$, it becomes:
$$\sum_{n=0}^{n}\frac{1}{n!} x_{1}^{2} \underline{\circ} (x_{1}^{2}  \underline{\circ}( \cdots (x_{1}^{2}  \underline{\circ}x_{1}^{2}  ) \cdots )).$$

\begin{lemm}
The coefficient of $\zeta^{\mathfrak{m}}(3)^{p}$ in $\zeta^{\mathfrak{m}}(n_{1}, \ldots, n_{p})$ of weight $3p$ is given recursively:
\begin{multline} \label{coeffzeta3}
\alpha_{n_{1}, \ldots, n_{p}}= \delta_{n_{p}=3} \alpha_{n_{1}, \ldots, n_{p-1}}\\
+\sum_{k=1 \atop n_{k}=1 }^{p} \left(  \delta_{n_{k-1}\geq 3} \alpha_{n_{1}, \ldots, n_{k-1}-2,n_{k+1}, \ldots, n_{p}}  -\delta_{n_{k+1}\geq 3} \alpha_{n_{1}, \ldots, n_{k-1},n_{k+1}-2, \ldots, n_{p}} \right) \\
+2 \sum_{k=1 \atop n_{k}=2 }^{p} \left(-\delta_{n_{k-1}\geq 3} \alpha_{n_{1}, \ldots, n_{k-1}-2,n_{k+1}, \ldots, n_{p}} + \delta_{n_{k+1}\geq 3} \alpha_{n_{1}, \ldots, n_{k-1},n_{k+1}-2, \ldots, n_{p}}  \right) .
\end{multline}
\end{lemm}
\noindent
\textsc{Remarks}:
\begin{itemize}
\item[$\cdot$] This is proved for motivic multiple zeta values, and by the period map, it also applies to multiple zeta values.
\item[$\cdot$] This lemma (as the next one, more precise) could be generalized for unramified motivic Euler sums.
\item[$\cdot$] All the coefficients $\alpha$ are all integers.
\end{itemize}
\begin{proof}
Recursively, let consider:
\begin{equation}
P_{n+1} (x_{1},\cdots, x_{n+1})\mathrel{\mathop:}=x_{1}^{2} \underline{\circ}P_{n} (x_{1},\cdots, x_{n}).
\end{equation} 
By the definition of the linearized Ihara action $(\ref{eq:circpolynom})$:
\begin{multline} \nonumber
P_{n+1} (x_{1},\cdots, x_{n+1}) =\sum_{i=0}^{n} (x_{i+1}-x_{i})^{2} P_{n} (x_{1},\cdots, x_{i}, x_{i+2}, \ldots, x_{n+1}) \\
-  \sum_{i=1}^{n} (x_{i+1}-x_{i})^{2} P_{n} (x_{1},\cdots, x_{i-1}, x_{i+1}, \ldots, x_{n+1})\\
=  (x_{n+1}-x_{n})^{2} P_{n}(x_{1}, \ldots, x_{n})+ \sum_{i=0}^{n-1} (x_{i}-x_{i+2})(x_{i}+x_{i+2}-2x_{i+1}) P_{n} (x_{1},\cdots, x_{i}, x_{i+2}, \ldots, x_{n+1}).
\end{multline}
Turning now towards the coefficients $c^{\textbf{i}}$ defined by:
$$ P_{p} (x_{1},\cdots, x_{p})= \sum c^{\textbf{i}} x_{1}^{i_{1}}\cdots x_{p}^{i_{p}}, \quad \text{ we deduce: } $$
\begin{multline} \nonumber
c^{i_{1}, \ldots, i_{p}}= -\delta_{i_{1}=0 \atop i_{2} \geq 2} c^{i_{2}-2,i_{3}, \ldots, i_{p}} + \delta_{i_{p}=2 } c^{i_{1},\cdots, i_{p-1}}  + \delta_{i_{n}=0 \atop i_{p-1} \geq 2} c^{i_{1}, \ldots, i_{p-2},i_{p-1}-2}  - 2 \delta_{i_{p}=1 \atop i_{p-1} \geq 2} c^{i_{1}, \ldots, i_{p-2}, i_{p-1}-1}  \\
+ \sum_{k=2 \atop i_{k}=0 }^{p-1} \left( \delta_{i_{k-1}\geq 2} c_{i_{1}, \ldots, i_{k-1}-2,i_{k+1}, \ldots, i_{p}} -\delta_{i_{k+1}\geq 2} c_{i_{1}, \ldots, i_{k-1},i_{k+1}-2, \ldots, i_{p}} \right)\\
+ 2 \sum_{k=2 \atop i_{k}=1 }^{p-1} \left( -\delta_{i_{k-1}\geq 1} c_{i_{1}, \ldots, i_{k-1}-2,i_{k+1}, \ldots, i_{p}} + \delta_{i_{k+1}\geq 1} c_{i_{1}, \ldots, i_{k-1},i_{k+1}-2, \ldots, i_{p}}  \right) ,
\end{multline}
which gives the recursive formula of the lemma.
\end{proof}

\paragraph{Generalization. }

Another proof of the previous lemma is possible using the dual point of view with the depth-graded derivations $D_{3,p}$, looking at cuts of length $3$ and depth $1$.\footnote{The coefficient $\alpha$ indeed emerges when computing $(D_{3,p})^{\circ p}$.}\\
A motivic multiple zeta value of weight $3d$ and of depth $p>d$ could also be expressed as:
\begin{equation}\label{eq:zeta3d}
\zeta^{\mathfrak{m}}(n_{1}, \ldots, n_{p}) = \frac{\alpha_{\underline{\textbf{n}}}} {d!} \zeta^{\mathfrak{m}}(3)^{d} + \text{ terms of depth strictly smaller than } d.
\end{equation} 
However, to compute this coefficient $\alpha_{\underline{\textbf{n}}}$, we could not work as before in the depth graded; i.e. this time, we have to consider all the possible cuts of length $3$. Then, the coefficient emerges when computing $\boldsymbol{(D_{3})^{\circ d}}$.
\begin{lemm}
The coefficient of $\zeta^{\mathfrak{m}}(3)^{d}$ in $\zeta^{\mathfrak{m}}(n_{1}, \ldots, n_{p})$ of weight $3d$ such that $p>d$,  is given recursively:
\begin{multline}  \label{coeffzeta3g}
\alpha_{n_{1}, \ldots, n_{p}}= \delta_{n_{p}=3} \alpha_{n_{1}, \ldots, n_{p-1}}\\
+\sum_{k=1 \atop n_{k}=1 }^{p} \left(  \delta_{n_{k-1}\geq 3  \atop k\neq 1} \alpha_{n_{1}, \ldots, n_{k-1}-2,n_{k+1}, \ldots, n_{p}}  -\delta_{n_{k+1}\geq 3  \atop k\neq p} \alpha_{n_{1}, \ldots, n_{k-1},n_{k+1}-2, \ldots, n_{p}} \right) \\
+2 \sum_{k=1 \atop n_{k}=2 }^{p} \left(-\delta_{n_{k-1}\geq 3} \alpha_{n_{1}, \ldots, n_{k-1}-2,n_{k+1}, \ldots, n_{p}} + \delta_{n_{k+1}\geq 3 \atop k\neq p} \alpha_{n_{1}, \ldots, n_{k-1},n_{k+1}-2, \ldots, n_{p}}  \right) \\
+ \sum_{k=1 \atop n_{k}=1, n_{k+1}=1 }^{p-1} \left(-\delta_{n_{k-1}\geq 3  \atop k\neq 1} \alpha_{n_{1}, \ldots, n_{k-1}-1,n_{k+2}, \ldots, n_{p}} + \delta_{n_{k+2}\geq 3} \alpha_{n_{1}, \ldots, n_{k-1},n_{k+2}-1, \ldots, n_{p}}  \right)\\
+ \sum_{k=1 \atop n_{k}=1, n_{k+1}=2 }^{p-1} \left( \delta_{n_{k-1}\geq 3 \atop \text{ or } k=1 } \alpha_{n_{1}, \ldots, n_{k-1},n_{k+2}, \ldots, n_{p}} +2 \delta_{n_{k+2}\geq 2 \atop k\neq p-1} \alpha_{n_{1}, \ldots, n_{k-1},n_{k+2}, \ldots, n_{p}}  \right)\\
+ \sum_{k=1 \atop n_{k}=2, n_{k+1}=1 }^{p-1} \left(-2\delta_{n_{k-1}\geq 2 \atop \text{ or } k=1} \alpha_{n_{1}, \ldots, n_{k-1},n_{k+2}, \ldots, n_{p}} - \delta_{n_{k+2}\geq 3 \atop k\neq p-1} \alpha_{n_{1}, \ldots, n_{k-1},n_{k+2}, \ldots, n_{p}}  \right).\\
\\
\end{multline}
\end{lemm}

\begin{proof}
Let list first all the possible cuts of length $3$ and depth $1$ in a iterated integral with $\lbrace 0,1 \rbrace$:\\
\includegraphics[]{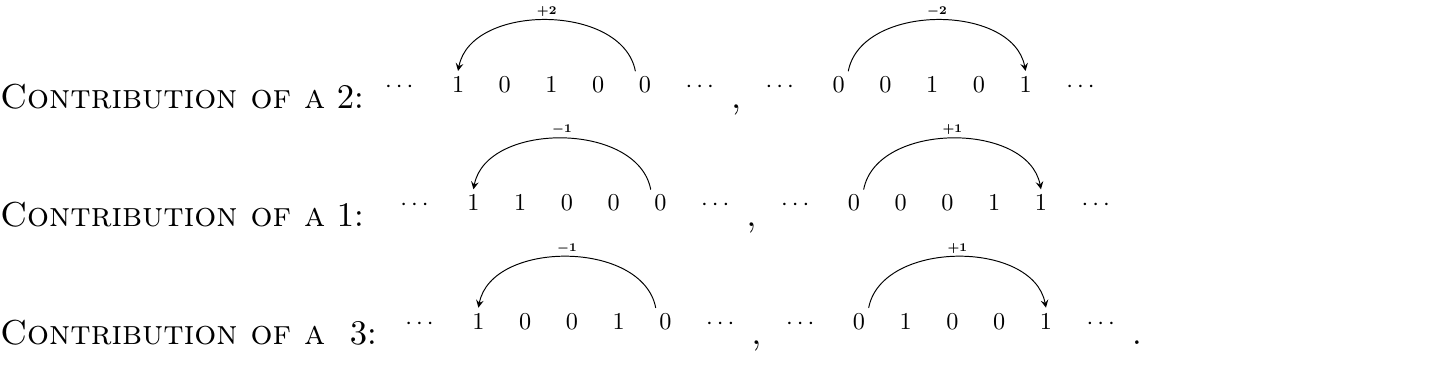}\\
The coefficient above the arrow is the coefficient of $\zeta^{\mathfrak{m}}(3)$ in $I^{\mathfrak{m}}(cut)$, using that:
$$\zeta^{\mathfrak{m}}_{1}(2)=-2\zeta^{\mathfrak{m}}(3), \quad \zeta^{\mathfrak{m}}(1,2)=\zeta^{\mathfrak{m}}(3), \quad \zeta^{\mathfrak{m}}(2,1)=-2\zeta^{\mathfrak{m}}(3), \quad \zeta^{\mathfrak{m}}_{1}(1,1)=\zeta^{\mathfrak{m}}(3).$$
Therefore, when there is a $1$ followed or preceded by something greater than $4$, the contribution is $\pm 1$, while when there is a $2$ followed or preceded by something greater than $3$, the contribution is $\pm 2$ as claimed in the lemma above. The contributions of a $3$ in the third line when followed of preceded by something greater than $2$ get simplified (except if there is a $3$ at the very end); when a $3$ is followed resp. preceded by a $1$ however, we assimilate it to the contribution of a $1$ preceded resp. followed by a $3$; which leads exactly to the penultimate lemma.\\
Additionally to the cuts listed above:\\
\includegraphics[]{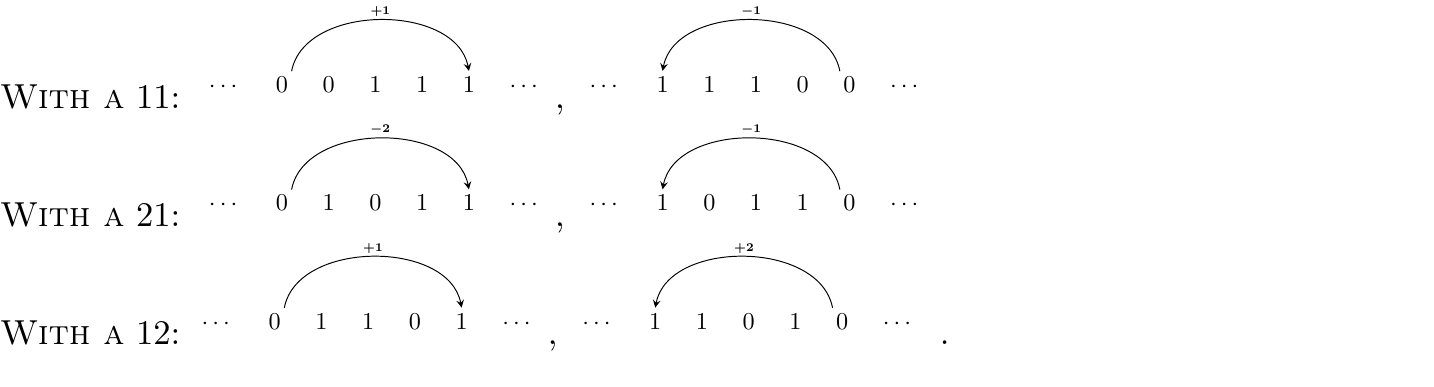}
This analysis leads to the given formula.\\
\end{proof}
In particular, a sequence of the type $\boldsymbol{Y12X}$ resp. $\boldsymbol{X21Y}$ ($X \geq 2, Y \geq 3$) will imply a $(3)$ resp. $(-3)$ times the coefficient of the same sequence without $\boldsymbol{ 12}$, resp. $\boldsymbol{ 21}$.\\
\\
\\
\texttt{{\Large Examples:}} Let list a few families of multiple zeta values for which we have computed explicitly the coefficient $\alpha$ of maximal depth:\\
\\
\begin{tabular}{| c | c | c | }
\hline
Family & Recursion relation & Coefficient $\alpha$\\
\hline
$\zeta^{\mathfrak{m}}(\lbrace 3\rbrace^{p})$ & $\alpha_{\lbrace 3\rbrace^{ p}}=\alpha_{\lbrace 3\rbrace^{ p-1}}$ & $1$\\
$\zeta^{\mathfrak{m}}(\lbrace 1,2\rbrace^{p})$ & $\alpha_{\lbrace 1,2 \rbrace^{ p}}=\alpha_{\lbrace 1,2 \rbrace^{ p-1}}$ & $1$\\
$\zeta^{\mathfrak{m}}(2,4,3\lbrace 3\rbrace^{p})$ & $\alpha_{2,4,\lbrace 3\rbrace^{p}}=\alpha_{2,4,\lbrace 3\rbrace^{p-1}}+2 \alpha_{\lbrace 3\rbrace^{p+1}} $ & $2(p+1)$\\
$\zeta^{\mathfrak{m}}(4,2,\lbrace 3\rbrace^{p})$ & $\alpha_{4,2,\lbrace 3\rbrace^{p}}=3\alpha_{4,2,\lbrace 3\rbrace^{p-1}}-2 \alpha_{\lbrace 3\rbrace^{p+1}} $ & $-3^{p+1}+1$\\
$\zeta^{\mathfrak{m}}(\lbrace 3\rbrace^{p},4,2)$ & $\alpha_{\lbrace 3\rbrace^{p},4,2}=-2\alpha_{\lbrace 3\rbrace^{p+1}} $ & $-2$\\
$\zeta^{\mathfrak{m}}(\lbrace 3\rbrace^{p},2,4)$ & $\alpha_{\lbrace 3\rbrace^{p},2,4}=2\alpha_{\lbrace 3\rbrace^{p-1},2,4}-2 \alpha_{\lbrace 3\rbrace^{p+1}} $ & $(-2)^{p}\frac{4}{3}+\frac{2}{3}$\\
$\zeta^{\mathfrak{m}}(2,\lbrace 3\rbrace^{p},4)$ & $\alpha_{2,\lbrace 3\rbrace^{p},4}=2\alpha_{2,\lbrace 3\rbrace^{p-1},4} $ & $2^{p+1}$\\
$\zeta^{\mathfrak{m}}(4,\lbrace 3\rbrace^{p},2)$ & $\alpha_{4,\lbrace 3\rbrace^{p},2}=-2\alpha_{4,\lbrace 3\rbrace^{p-1},2} $ & $(-2)^{p+1}$\\
$\zeta^{\mathfrak{m}}(1,5,\lbrace 3\rbrace^{p})$ & $\alpha_{1,5,\lbrace 3\rbrace^{p}}=\alpha_{1,5,\lbrace 3\rbrace^{p-1}}-1 $ & $-(p+1)$\\
$\zeta^{\mathfrak{m}}(\lbrace 2\rbrace^{p},\lbrace 4\rbrace^{p})$ & $\alpha_{\lbrace 2\rbrace^{p},\lbrace 4\rbrace^{p}}=  4 \alpha_{\lbrace 2\rbrace^{p-1},\lbrace 4\rbrace^{p-1}}$ & $2^{2p-1}$\\
$\zeta^{\mathfrak{m}}(\lbrace 2\rbrace^{p},\lbrace 3\rbrace^{a} \lbrace 4\rbrace^{p})$ & $\alpha_{\lbrace 2\rbrace^{p},\lbrace 3\rbrace^{a} \lbrace 4\rbrace^{p}}= 2^{a}\alpha_{\lbrace 2\rbrace^{p},\lbrace 4\rbrace^{p}}$ & $2^{a+2p-1}$\\
$\zeta^{\mathfrak{m}}(\lbrace 2\rbrace^{p},p+3)$ & $\alpha _{\lbrace 2\rbrace^{p},p+3}= 2\alpha _{\lbrace 2\rbrace^{p-1},p+1}$  & $2^{p}$\\
$\zeta^{\mathfrak{m}}(2,3,4,\lbrace 3\rbrace^{p})$ & $\alpha_{2,3,4,\lbrace 3\rbrace^{p}}=\alpha_{2,3,4,\lbrace 3\rbrace^{p-1}}+ 2 \alpha_{2,4,\lbrace 3\rbrace^{p}}$ & $2(p+1)(p+2)$\\
$\zeta^{\mathfrak{m}}(2,1,5,4,\lbrace 3\rbrace^{p})$ & $\alpha_{2,1,5,4,\lbrace 3\rbrace^{p}}=\alpha_{2,1,5,4,\lbrace 3\rbrace^{p-1}}- \alpha_{2,3,4,\lbrace 3\rbrace^{p}}$ & $-\frac{2(p+1)(p+2)(p+3)}{3}$\\
$\zeta^{\mathfrak{m}}(\lbrace 2\rbrace^{a}, a+3, \lbrace 3\rbrace^{b})$ & $\alpha_{a; b}=2 \alpha_{a-1;b}+\alpha_{a;b-1 }$ & $2^{a}\binom{a+b}{a} $\\
$\zeta^{\mathfrak{m}}(\lbrace 5,1\rbrace^{n})\text{ with } 3$\footnotemark[1] & $\alpha= \sum_{i=1}^{2p-1} (-1)^{i-1} \alpha_{\lbrace 5,1\rbrace^{p \text{ or } p-1} \text{with } 3} $ \footnotemark[2] & 1\\ \hline
\end{tabular}\\
\footnotetext[1]{Any $\zeta^{\mathfrak{m}}(\lbrace 5,1\rbrace^{p})$ where we have inserted some $3$ in the subsequence.}
\footnotetext[2]{Either a $3$ has been removed, either a $5,1$ resp. $1,5$ has been converted into a $3$ (with a sign coming from if we consider the elements before or after a $1$). If it ends with $3$, the contribution of a $3$ cancel with the contribution of the last $1$.}
\\
\\
\textsc{For instance}, for the coefficient $\alpha_{a;b;c}$ associated to $\zeta^{\mathfrak{m}}(\lbrace 3\rbrace^{a},2, \lbrace 3\rbrace^{b}, 4, \lbrace 3\rbrace^{c} )$, the recursive relation is:
\begin{equation}
\alpha_{a;b;c}=\alpha_{a;b;c-1}+2\alpha_{a;b-1;c}-2\alpha_{a-1;b;c}, \quad \text{ which leads to the formula:}
\end{equation}
\begin{multline}\nonumber 
\alpha_{a;b;c}= (-2)^{a} \sum_{l=0}^{b-1} \sum_{m=0}^{c-1} 2^{l}\frac{(a+l+m-1) !}{(a-1) ! l ! m! } \alpha_{0;b-l;c-m}  + 2^{b} \sum_{k=0}^{a-1} \sum_{m=0}^{c-1} (-2)^{k}  \frac{(b+k+m-1) !}{k !(b-1) ! m! } \alpha_{a-k;0;c-m}\\
 + \sum_{l=0}^{b-1} \sum_{k=0}^{a-1} (-2)^{k} 2^{l} \frac{(k+l+c-1) !}{k! l ! (c-1)! } \alpha_{a-k;b-l;0}.
\end{multline}
Besides, we can also obtain very easily:
$$\hspace*{-0.5cm}\alpha_{a;0;0}= (-2)^{a}\frac{4}{3}+\frac{2}{3} \text{,} \quad \alpha_{0;b;0}=2^{b+1}, \quad  \alpha_{0;0;c}= 2(c+1), \quad \text{ and }  \quad \alpha_{0;b;c}=2^{b+1}\binom{b+c+1}{c} $$
Indeed, using $\sum_{k=0}^{n}\binom{a+k}{a}= \binom{n+a+1}{a+1}$, and $\sum_{k=0}^{n}(n-k) \binom{k}{a}= \binom{n+1}{a+2}$, we deduce:
$$\begin{array}{lll}
\alpha_{0;b;c} &=  2\alpha_{0;b-1;c}+\alpha_{0;b;c-1}& = 2^{b+1} \left(  \sum_{k=0}^{c-1} \binom{b+k-1}{b-1}(c-k+1) + \sum_{k=0}^{b-1}\binom{c+k-1}{c-1} \right)\\
& & =2^{b+1} \left( \binom{b+c+1}{b+1}-\binom{b+c-1}{b-1} + \binom{b+c-1}{b-1}\right)  = 2^{b+1}\binom{b+c+1}{c}.
\end{array}$$
\\
\texttt{Conjectured examples:} \\
\\
\begin{center}
\begin{tabular}{| c | c | }
\hline
Family  & Conjectured coefficient $\alpha$\\
\hline
$\zeta^{\mathfrak{m}}(\lbrace 2, 4 \rbrace^{p })$ & $\alpha_{p}$ such that $ 1-\sqrt{cos(2x)}=\sum \frac{\alpha_{p}(-1)^{p+1} x^{2p}}{(2p)!}$ \\
$\zeta^{\mathfrak{m}}(\lbrace 1, 5 \rbrace^{p})$ &  Euler numbers:  $\frac{1}{cosh(x)}=\sum \frac{\alpha_{p} x^{2p}}{(2p)!}$ \\
$\zeta^{\mathfrak{m}}(\lbrace 1, 5 \rbrace^{p}, 3)$ & $(-1)^{p} \text{Euler ZigZag numbers } E_{2p+1} $ $= 2^{2p+2}(2^{2p+2}-1)\frac{ B_{2p+2}}{2p+2} $ \\ \hline
\end{tabular}
\end{center}

\subsection{$\boldsymbol{N>1}$, The diagonal algebra.}

For $N=2,3,4$, $gr^{\mathfrak{D}}_{d} \mathcal{H}_{d}$ is $1$ dimensional, generated by $\zeta^{\mathfrak{m}}({ 1 \atop \xi})$, where $\xi\in\mu_{N}$ primitive fixed, which allows us to consider the projection:
\begin{equation}
\vartheta^{N} : gr_{d}^{\mathfrak{D}} \mathcal{H}_{d}^{1} \rightarrow \mathbb{Q}.
\end{equation}
Giving a motivic multiple zeta value relative to $\mu_{N}$, of weight $d$, depth $d$, there exists a rational such that:
\begin{framed}
\begin{equation}\label{eq:zeta1d}
\zeta^{\mathfrak{m}}\left(1, \ldots, 1 \atop \epsilon_{1} , \ldots, \epsilon_{d} \right) = \frac{\alpha_{\boldsymbol{\epsilon}}} {n!} \zeta^{\mathfrak{m}}\left( 1 \atop \xi \right) ^{d} + \text{ terms of depth strictly smaller than } d.
\end{equation} 
\end{framed}
The coefficient $\alpha$ being calculated recursively, using depth $1$ results:
\begin{lemm}
$$ \hspace*{-0.5cm}\alpha_{\epsilon_{1}, \ldots, \epsilon_{d}}= \left\lbrace \begin{array}{ll}
1 & \text{if } \boldsymbol{N\in \lbrace 2,3 \rbrace}.\\
 \sum_{k=1 \atop \epsilon_{k}\neq 1 }^{d} \beta_{\epsilon_{k}}\left(  \delta_{\epsilon_{k-1}\epsilon_{k}\neq 1} \alpha_{\epsilon_{1}, \ldots, \epsilon_{k-1}\epsilon_{k},\epsilon_{k+1}, \ldots, \epsilon_{d}}  -\delta_{\epsilon_{k+1}\epsilon_{k}\neq 1 \atop k < d} \alpha_{\epsilon_{1}, \ldots, \epsilon_{k-1},\epsilon_{k+1}\epsilon_{k}, \ldots, \epsilon_{d}} \right) & \text{if } \boldsymbol{N=4}.
\end{array} \right. $$
\begin{flushright}
with $\beta_{\epsilon_{k}}= \left\lbrace \begin{array}{ll}
2 & \text{ if } \epsilon_{k}=-1\\
1 & \text{ else}
\end{array}\right. $.
\end{flushright}
\end{lemm}
\begin{proof}
In regards to redundancy, the proof being in the same spirit than the previous section ($N=1, w=3d$), is left to the reader.\footnote{The cases $N=2,3$ correspond to the case $N=4$ with $\beta$ always equal to $1$.} 
\end{proof}
\textsc{Remarks:} 
\begin{itemize}
\item[$\cdot$] For the following categories, the space $gr^{\mathfrak{D}}_{d} \mathcal{H}_{d}$ is also one dimensional:
$$\mathcal{MT}\left(  \mathcal{O}_{6}\left[ \frac{1}{2}\right] \right) ,\quad \mathcal{MT}\left( \mathcal{O}_{6}\left[ \frac{1}{3}\right] \right) , \quad \mathcal{MT}\left( \mathcal{O}_{5}\right) , \mathcal{MT}\left( \mathcal{O}_{10}\right) , \quad \mathcal{MT}\left( \mathcal{O}_{12}\right) .$$
The recursive method to compute the coefficient of $\zeta^{\mathfrak{m}}\left( 1 \atop \eta\right) ^{d}$ would be similar, except that we do not know a proper basis for these spaces.
\item[$\cdot$] For $N=1$, and $w\equiv 2 \mod 3$ for instance, $gr^{\mathfrak{D}}_{\max}\mathcal{H}_{n}$ is generated by the elements of the Euler $\sharp$ sums basis:\\
$\zeta^{\sharp, \mathfrak{m}} (1, \boldsymbol{s}, \overline{2})$ with $\boldsymbol{s}$ composed of $3$'s and one $5$, $\zeta^{\sharp, \mathfrak{m}} (3, 3 , \ldots, 3, \overline{2})$ and $\zeta^{\sharp, \mathfrak{m}} (1, 3 , \ldots, 3, \overline{4})$.
\end{itemize}

\section{Families of unramified Euler sums.}

The proof relies upon the criterion $\ref{criterehonoraire}$,, which enables us to construct infinite families of unramified Euler sums with parity patterns by iteration on the depth, up to depth $5$.\\
\\
\texttt{Notations:} The occurrences of the symbols $E$ or $O$ can denote arbitrary even or odd integers, whereas every repeated occurrence of symbols $E_{i}$  (respectively $O_{i}$) denotes the same positive even (resp. odd) integer. The bracket $\left\{\cdot, \ldots, \cdot \right\}$ means that every permutation of the listed elements is allowed.
\begin{theo}
The following motivic Euler sums are unramified, i.e. motivic MZV:\footnote{Beware, here, $\overline{O}$ and $\overline{n}$ must be different from $\overline{1}$, whereas $O$ and $n$ may be 1. There is no $\overline{1}$ allowed in these terms if not explicitly written.} \\
\\
\hspace*{-0.5cm} \begin{tabular}{| c | l | l | }
           \hline  
     &  \textsc{Even Weight} & \textsc{Odd Weight}  \\       \hline
\textsc{Depth } 1  &  \text{ All } &  \text{ All }\footnotemark[1] \\    
         \hline    
\textsc{Depth } 2   & $\zeta^{\mathfrak{m}}(\overline{O},\overline{O}),  \zeta^{\mathfrak{m}}(\overline{E},\overline{E})$ & \text{ All }   \\        
         \hline              
\multirow{2}{*}{ \textsc{Depth } 3 } & $\zeta^{\mathfrak{m}}(\left\{E,\overline{O},\overline{O}\right\}), \zeta^{\mathfrak{m}}(O,\overline{E},\overline{O}), \zeta^{\mathfrak{m}}(\overline{O},\overline{E}, O)$  &  $\zeta^{\mathfrak{m}}(\left\{\overline{E},\overline{E},O\right\}), \zeta^{\mathfrak{m}}(\overline{E},\overline{O},E), \zeta^{\mathfrak{m}}(E,\overline{O},\overline{E})$ \\  
 & $ \zeta^{\mathfrak{m}}(\overline{O_{1}}, \overline{E},\overline{O_{1}}), \zeta^{\mathfrak{m}}(O_{1}, \overline{E},O_{1}), \zeta^{\mathfrak{m}}(\overline{E_{1}}, \overline{E},\overline{E_{1}}) .$ & \\
         \hline              
\multirow{2}{*}{ \textsc{Depth } 4 }   &  $\zeta^{\mathfrak{m}}(E,\overline{O},\overline{O},E),\zeta^{\mathfrak{m}}(O,\overline{E},\overline{O},E), $ &   \\   
& $\zeta^{\mathfrak{m}}(O,\overline{E},\overline{E},O), \zeta^{\mathfrak{m}}(E,\overline{O},\overline{E},O)$ & \\
         \hline              
 \textsc{Depth } 5     &   & $\zeta^{\mathfrak{m}}(O_{1}, \overline{E_{1}},O_{1},\overline{E_{1}}, O_{1}).$  \\     
    \hline
  \end{tabular}
Similarly for these linear combinations, in depth $2$ or $3$:
  $$\zeta^{\mathfrak{m}}(n_{1},\overline{n_{2}}) +  \zeta^{\mathfrak{m}}(\overline{n_{2}},n_{1}) , \zeta^{\mathfrak{m}}(n_{1},\overline{n_{2}}) +  \zeta^{\mathfrak{m}}(\overline{n_{1}},n_{2}), \zeta^{\mathfrak{m}}(n_{1},\overline{n_{2}}) -  \zeta^{\mathfrak{m}}(n_{2}, \overline{n_{1}}) .$$
		$$(2^{n_{1}}-1) \zeta^{\mathfrak{m}}(n_{1},\overline{1}) +  (2^{n_{1}-1}-1) \zeta^{\mathfrak{m}}(\overline{1},n_{1}).$$
				$$ \zeta^{\mathfrak{m}}(n_{1},n_{2},\overline{n_{3}}) + (-1)^{n_{1}-1}  \zeta^{\mathfrak{m}}(\overline{n_{3}},n_{2},n_{1}) \text{ with } n_{2}+n_{3} \text{ odd }.$$
\end{theo}
\texttt{Examples}:
These motivic Euler sums are motivic multiple zeta values: 
$$\zeta^{\mathfrak{m}}(\overline{25}, 14,\overline{17}),\zeta^{\mathfrak{m}}(17, \overline{14},17), \zeta^{\mathfrak{m}}(\overline{24}, \overline{14},\overline{24}), \zeta^{\mathfrak{m}}(6, \overline{23}, \overline{17}, 10) , \zeta^{\mathfrak{m}}(13, \overline{24}, 13,\overline{24}, 13).$$
\textsc{Remarks:}
\begin{itemize}
\item[$\cdot$] This result for motivic ES implies the analogue statement for ES.
\item[$\cdot$] Notice that for each honorary MZV above, the reverse sum is honorary too, which was not obvious a priori, since the condition $\textsc{c}1$ below is not symmetric. 
\end{itemize}
\begin{proof} The proof amounts to the straight-forward control that $D_{1}(\cdot)=0$ (immediate) and that all the elements appearing in the right side of $D_{2r+1}$ are unramified, by recursion on depth: here, these elements satisfy the sufficient criteria given below. Let only point out a few things, referring to the expression $(\ref{eq:derhonorary})$:
\begin{description}
\item[\texttt{Terms} $\textsc{c}$:] The symmetry condition $(\textsc{c}4)$, obviously true for these single elements above, get rid of these terms. For the few linear combinations of MES given, the cuts of type (\textsc{c}) get simplified together.
\item[\texttt{Terms} $\textsc{a,b}$:] Checking that the right sides are unramified is straightforward by depth-recursion hypothesis, since only the (previously proven) unramified elements of lower depth emerge. For example, the possible right sides (not canceled by a symmetric cut and up to reversal symmetry) are listed below, for some elements from depth 3. \\
 \hspace*{-1cm} \begin{tabular}{| c | l | l | }
           \hline  
     &  Terms \textsc{a0} & Terms  \textsc{a,b}  \\       \hline
  $(O,\overline{E},\overline{E}) $  &   & $(\overline{E},\overline{E})$ ,$(O,O)$  \\   
  $(\overline{E},O,\overline{E}) $  &  / &   $(\overline{E},\overline{E})$ \\   
 $(E,\overline{O},\overline{E}) $  &  / &  $(\overline{E},\overline{E}), (E,E)$ \\  
         \hline              
    $(E,\overline{O},\overline{O},E) $  &  $(\overline{O},E)$ & $(\overline{E},\overline{O},E), (E,O,E),(E,\overline{O},\overline{E}), (E,O),(O,E)$ \\   
   $(O,\overline{E},\overline{O},E) $ &  $(\overline{E},\overline{O},E),(\overline{O},E)$ & $(\overline{E},\overline{O},E),(O,E,E),(O,\overline{E},\overline{E}), (O,E)$ \\
  $(O,\overline{E},\overline{E},O) $&   $(\overline{E},\overline{E},O) ,(\overline{E},O)$ & $(\overline{E},\overline{E},O), (O,O,O), (O,\overline{E},\overline{E}), (O,E), (E,O)$ \\
    $(\overline{E}, O_{1},\overline{E},O_{1}) $& $(\overline{E},O)$ & $(\overline{E},\overline{E},O) , (\overline{E},O,\overline{E}) ,(\overline{E},\overline{O}), (E,O)$  \\
$(\overline{E_{1}}, \overline{E_{2}},\overline{E_{1}}, \overline{E_{2}}) $& / & $(O,\overline{E},\overline{E}) ,(\overline{E},O,\overline{E}) ,(\overline{E},\overline{E},O) ,(\overline{E},\overline{O}),(\overline{O},\overline{E})$ \\
         \hline    
    $(O_{1}, \overline{E_{1}},O_{1},\overline{E_{1}}, O_{1})$ & $(\overline{E_{1}},O_{1},\overline{E_{1}}, O_{1}),$ & $(\overline{E},O_{1},\overline{E}, O_{1}), (O_{1}, \overline{E},\overline{E}, O_{1}), (O_{1}, \overline{E},O_{1},\overline{E}),$ \\
    & $(O_{1},\overline{E}, O_{1})$ & $(\overline{O},\overline{E},O),(O,\overline{E},\overline{O}), (O,O)$\\
    \hline
  \end{tabular}
\end{description}
It refers to the expression of the derivations $D_{2r+1}$ (from Lemma $\ref{drz}$$\ref{drz}$):
\begin{multline}\label{eq:derhonorary}
 D_{2r+1} \left(\zeta^{\mathfrak{m}} \left(n_{1}, \ldots , n_{p} \right)\right) = \textsc{(a0) }  \delta_{2r+1 = \sum_{k=1}^{i} \mid n_{k} \mid} \zeta^{\mathfrak{l}} (n_{1}, \ldots , n_{i}) \otimes \zeta^{\mathfrak{m}} (n_{i+1},\cdots n_{p}) \\
\textsc{(a,b) }   \sum_{1\leq i < j \leq p \atop 2r+1=\sum_{k=i}^{j} \mid n_{k}\mid  - y } \left\lbrace  \begin{array}{l}
  -\delta_{2\leq y \leq \mid n_{j}\mid } \zeta_{\mid n_{j}\mid -y}^{\mathfrak{l}} (n_{j-1}, \ldots ,n_{i+1}, n_{i}) \\
  +\delta_{2\leq y \leq \mid n_{i}\mid} \zeta_{\mid n_{i}\mid -y}^{\mathfrak{l}} (n_{i+1},  \cdots ,n_{j-1}, n_{j})
 \end{array} \right.  \otimes \zeta^{\mathfrak{m}} (n_{1}, \ldots, n_{i-1},\prod_{k=i}^{j}\epsilon_{k} \cdot y,n_{j+1},\cdots n_{p}). \\
\textsc{(c)  } + \sum_{1\leq i < j \leq p\atop {2r+2=\sum_{k=i}^{j} \mid n_{k}\mid} } \delta_{ \prod_{k=i}^{j} \epsilon_{k} \neq 1}  \left\lbrace  \begin{array}{l} 
+  \zeta_{\mid n_{i}\mid -1}^{\mathfrak{l}} (n_{i+1},  \cdots ,n_{j-1}, n_{j}) \\
- \zeta_{\mid n_{j}\mid -1}^{\mathfrak{l}} (n_{j-1},  \cdots ,n_{i+1}, n_{i})
 \end{array} \right. \otimes \zeta^{\mathfrak{m}} (n_{1}, \ldots, n_{i-1},\overline{1},n_{j+1},\cdots n_{p}).
\end{multline}
\end{proof}

\paragraph{Sufficient condition. }\label{sufficientcondition}
Let $\mathfrak{Z}= \zeta^{\mathfrak{m}}(n_{1}, \ldots, n_{p})$ a motivic Euler sum. These four conditions are \textit{sufficient} for $\mathfrak{Z}$ to be unramified:
\begin{description}
	\item [\textsc{c}1]: No $\overline{1}$ in $\mathfrak{Z}$.
	\item [\textsc{c}2]: For each $(n_{1}, \ldots, n_{i})$ of odd weight, the MES $\zeta^{\mathfrak{m}}(n_{i+1}, \ldots, n_{p})$ is a MMZV.
	\item [\textsc{c}3]: If a cut removes an odd-weight part (such that there is no symmetric cut possible), the remaining MES (right side in terms \textsc{a,b}), is a MMZV.
	\item [\textsc{c}4]: Each sub-sequence $(n_{i}, \ldots, n_{j})$ of even weight such that $\prod_{k=i}^{j} \epsilon_{k} \neq 1$ is symmetric.
\end{description}
\begin{proof}
The condition $\textsc{c}1$ implies that $D_{1}(\mathfrak{Z})=0$; conditions $\textsc{c}2$, resp. $\textsc{c}3$ take care that the right side of terms \textsc{a0}, resp. \textsc{a,b} are unramified, while the condition $\textsc{c}4$ cancels the (disturbing) terms \textsc{c}: indeed, a single ES with an $\overline{1}$ can not be unramified.\\
Note that a MES $\mathfrak{Z}$ of depth $2$, weight $n$ is unramified if and only if $ \left\lbrace \begin{array}{l}
D_{1}(\mathfrak{Z})=0\\
D_{n-1}(\mathfrak{Z})=0
\end{array}\right.$. 
\end{proof}
\noindent
\texttt{Nota Bene:} This criterion is not \textit{necessary}: it does not cover the unramified $\mathbb{Q}$-linear combinations of motivic Euler sums, such as those presented in section $4$, neither some isolated (\textit{symmetric enough}) examples where the unramified terms appearing in $D_{2r+1}$ could cancel between them. However, it embrace general families of single Euler sums which are unramified.\\
\\
Moreover, 
\begin{framed}
\emph{If we \textit{comply with these conditions}, the \textit{only} general families of single MES obtained are the one listed in Theorem $6.2.1$.}
\end{framed}
\begin{proof}[Sketch of the proof]
Notice first that the condition \textsc{c}$4$ implies in particular that there are no consecutive sequences of the type (since it would create type $\textsc{c}$ terms):
$$\textsc{ Seq. not allowed :  }  O\overline{O}, \overline{O}O, \overline{E}E, E\overline{E}.$$
It implies, from depth $3$, that we can't have the sequences (otherwise one of the non allowed depth $2$ sequences above appear in $\textsc{a,b}$ terms):
$$\textsc{ Seq. not allowed : }  \overline{E}\overline{E}\overline{O}, \overline{E}\overline{E}\overline{E}O, \overline{E}\overline{E}O\overline{E}, E\overline{O}\overline{E},EE\overline{O}, \overline{O}EE, \overline{E}OE, \overline{E}\overline{O}\overline{E}, \overline{O}\overline{O}\overline{O}.$$
Going on in this recursive way, carefully, leads to the previous theorem.\\
\texttt{For instance,} let $\mathfrak{Z}$ a MES of even weight, with no $\overline{1}$, and let detail two cases:
\begin{description}
\item[\texttt{Depth} $3$:]  The right side of $D_{2 r+1}$ has odd weight and depth smaller than $2$, hence is always MMZV if there is no $\overline{1}$ by depth $2$ results. It boils down to the condition $\textsc{c}4$: $\mathfrak{Z}$ must be either symmetric (such as $O_{1}E O_{1}$ or $E_{1}EE_{1}$ with possibly one or three overlines) either have exactly two overlines.  Using the analysis above of the allowed sequences in depth $2$ and $3$ for condition $\textsc{c3,4}$ leads to the following:
$$(E,\overline{O},\overline{O}),(\overline{O},\overline{O},E), (O,\overline{E},\overline{O}), (\overline{O},\overline{E}, O), (\overline{O},E,\overline{O}), (\overline{O_{1}}, \overline{E},\overline{O_{1}}), (O_{1}, \overline{E},O_{1}), (\overline{E_{1}}, \overline{E},\overline{E_{1}}) .$$
\item[\texttt{Depth} $4$:] Let $\mathfrak{Z}=\zeta^{\mathfrak{m}}\left( n_{1}, \ldots, n_{4}\right) $, $\epsilon_{i}=sign(n_{i})$. To avoid terms of type $
\textsc{c}$ with a right side of depth $1$: if $\epsilon_{1}\epsilon_{2}\epsilon_{3}\neq 1$, either $n_{1}+n_{2}+n_{3}$ is odd, or $n_{1}=n_{3}$ and $\epsilon_{2}=-1$;  if $\epsilon_{2}\epsilon_{3}\epsilon_{4}\neq 1$, either $n_{2}+n_{3}+n_{4}$ is odd, or $n_{2}=n_{4}$ and $\epsilon_{3}=-1$. The following sequences are then not allowed:
$$ (\overline{E}, O,O,\overline{E}), (\overline{E}, \overline{O},\overline{O},\overline{E}), (\overline{E}, \overline{O},E,O), (\overline{E}, \overline{E},O,O), (O,O,\overline{E}, \overline{E}).$$
\end{description}
\end{proof}

\section{Motivic Identities}

As we have seen above, in particular in Lemma $\ref{lemmcoeff}$, the coaction enables us to prove some identities between MMZV or MES, by recursion on the depth, up to one rational coefficient at each step. This coefficient can be deduced then, if we know the analogue identity for MZV, resp. Euler sums. Nevertheless, a \textit{motivic identity} between MMZV (resp. MES) is stronger than the corresponding relation between real MZV (resp. Euler sums); it may hence require several relations between MZV in order to lift an identity to motivic MZV. An example of such a behaviour occurs with some Hoffman $\star$ elements, ($(iv)$ in Lemma $\ref{lemmcoeff}$).\\
\\
In this section, we list a few examples of identities, picked from the zoo of existing identities, that we are able to lift easily from Euler sums to motivic Euler sums: \textit{Galois trivial} elements (action of the unipotent part of the Galois group being trivial), sums identities, etc. \\
\\
\texttt{Nota Bene}: For other cyclotomic MMZV, we could somehow generalize this idea, but there would be several unknown coefficients at each step, as stated in Theorem $2.4.4$. For $N=3$ or $4$, we have to consider all $D_{r}, 0<r<n$, and there would be one resp. two (if weight even) unknown coefficients at each step ; for $N=\mlq 6 \mrq$, if unramified, considering $D_{r}, r>1$, there would be also one or two unknown coefficients at each step. \\
\\
\texttt{Example:} Here is an identity known for Euler sums, proven at the motivic level by recursion on $n$ via the coaction for motivic Euler sums (and using the analytic identity):
\begin{equation}
\zeta^{\mathfrak{m}}(\lbrace 3 \rbrace^{n})= \zeta^{\mathfrak{m}}(\lbrace 1,2 \rbrace^{n}) = 8^{n} \zeta^{\mathfrak{m}}(\lbrace 1, \overline{2} \rbrace^{n}).
\end{equation}
\begin{proof}
These three families are stable under the coaction:
$$\begin{array}{lllll}
D_{2r+1} (\zeta^{\mathfrak{m}}(\lbrace 3 \rbrace^{n})) & = & \delta_{2r+1=3s} \zeta^{\mathfrak{a}}(\lbrace 3 \rbrace^{s})  & \otimes & \zeta^{\mathfrak{m}}(\lbrace 3 \rbrace^{n-s}) .\\
D_{2r+1} (\zeta^{\mathfrak{m}}(\lbrace 1,2 \rbrace^{n})) & = & \delta_{2r+1=3s} \zeta^{\mathfrak{a}}(\lbrace 1,2 \rbrace^{s}) & \otimes & \zeta^{\mathfrak{m}}(\lbrace 1,2 \rbrace^{n-s}) .\\
D_{2r+1} (\zeta^{\mathfrak{m}}(\lbrace 1,\overline{2} \rbrace^{n})) & = & \delta_{2r+1=3s} \zeta^{\mathfrak{a}}(\lbrace 1,\overline{2} \rbrace^{s}) & \otimes & \zeta^{\mathfrak{m}}(\lbrace 1,\overline{2} \rbrace^{n-s}) .
\end{array}$$
Indeed, in both case, in the diagrams below, cuts $(3)$ and $(4)$ are symmetric and get simplified by reversal, as cuts $(1)$ and $(2)$, except for last cut of type $(1)$ which remains alone:\\
\includegraphics[]{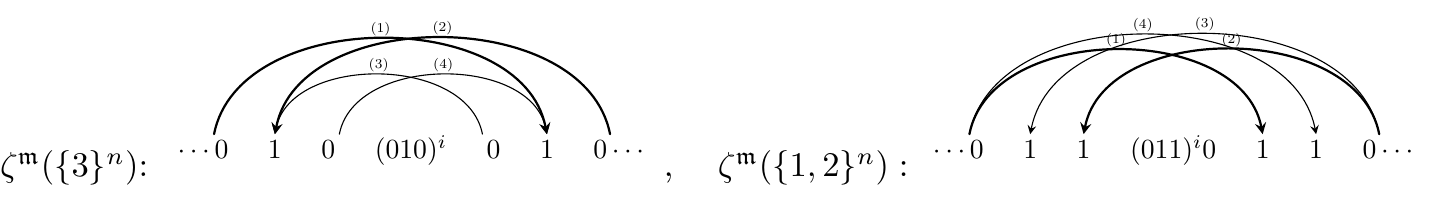}\\
Similarly for $\zeta^{\mathfrak{m}}(\lbrace 1,\overline{2} \rbrace^{n})$: cuts of type $(3)$, $(4)$ resp. $(1), (2)$ get simplified together, except the first one, when $\epsilon=\epsilon'$ in the diagram below. The other possible cuts of odd length would be $(5)$ and $(6)$ below, when $\epsilon=-\epsilon'$, but each is null since antisymmetric.\\
\includegraphics[]{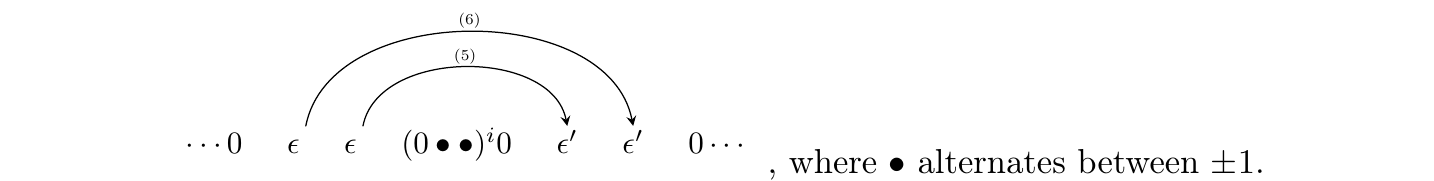}
\end{proof}

\paragraph{Galois trivial.}
The Galois action of the unipotent group $\mathcal{U}$ is trivial on $\mathbb{Q}[\mathbb{L}^{\mathfrak{m}, 2n}]= \mathbb{Q}[\zeta^{\mathfrak{m}}(2)^{n}]$. To prove an element of $\mathcal{H}_{2n}$ is a rational multiple of $\mathbb{L}^{\mathfrak{m},2n}$, it is equivalent to check it is in the kernel of the derivations $D_{2r+1}$, for $1\leq 2r+1<2n$, by Corollary $\ref{kerdn}$. We have to use the (known) analogue identities for MZV to conclude on the value of such a rational.\\
\\
\texttt{Example:} 
\begin{itemize}
\item[$\cdot$] Summing on all the possible ways to insert n $\boldsymbol{2}$'s.
\begin{equation}
\zeta^{\mathfrak{m}}(\left\lbrace 1,3 \right\rbrace^{p} \text{with n } \boldsymbol{2} \text{ inserted } )= \binom{2p+n}{n} \frac{\pi^{4p+2n, \mathfrak{m}}}{(2n+1) (4p+2n+1)!}.
\end{equation}
\item[$\cdot$] More generally, with fixed $(a_{i})$ such that  $\sum a_{i}=n$: \footnote{Both appears also in Charlton's article.$\cite{Cha}$.}
\begin{equation}
\sum_{\sigma\in\mathfrak{S}_{2p}} \zeta^{\mathfrak{m}}(2^{a_{\sigma(0)}} 1 , 2^{a_{\sigma(1)}}, 3, 2^{a_{\sigma(2)}}, \ldots, 1, 2^{a_{\sigma(2p-1)}}, 3, 2^{a_{\sigma(2p)}} )\in\mathbb{Q} \pi^{4p+2n, \mathfrak{m}}.
\end{equation}
\end{itemize}

\begin{proof}
In order to justify why all the derivations $D_{2r+1}$ cancel, the possible cuts of odd length are, with $X= \lbrace 01 \rbrace^{a_{2i+2}+1} \lbrace 10 \rbrace^{a_{2i+3}+1} \cdots \lbrace 01 \rbrace^{a_{2j-2}} \lbrace 10 \rbrace^{a_{2j-1}} $:\\
\includegraphics[]{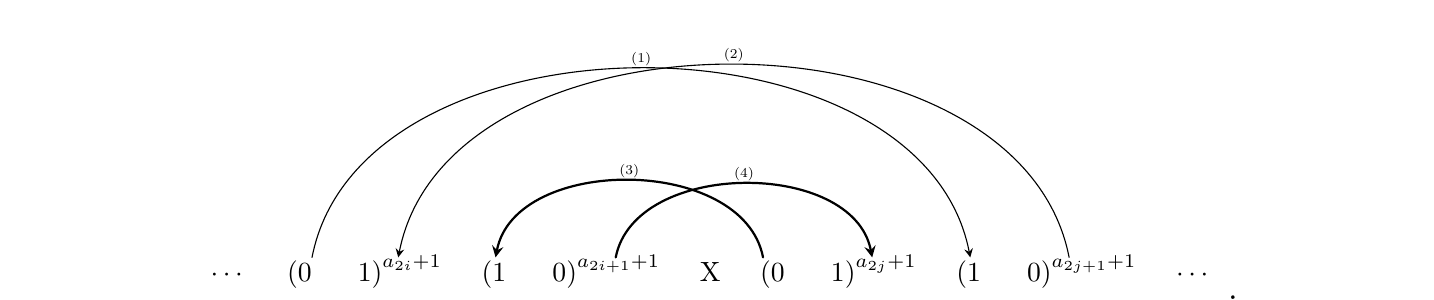}
All the cuts get simplified by \textsc{Antipode} $\shuffle$ $\ref{eq:antipodeshuffle2}$, which proves the result, as follows:
\begin{itemize}
\item[$\cdot$] Cut $(1)$ for $(a_{0}, \ldots, a_{2p})$ with Cut $(2)$ for $(a_{0}, \ldots, a_{2i-1}, a_{2j+1} \cdots, a_{2i}, a_{2j+2 },\cdots,  a_{2p})$.
\item[$\cdot$] Similarly between $(3)$ and $(4)$, which get simplified considering the sequence where $(a_{2i+1}, \ldots, a_{2j})$ is reversed.
\end{itemize}
\end{proof}

\paragraph{Polynomial in simple zetas.} A way to prove that a family of (motivic) MZV are polynomial in simple (motivic) zetas, by recursion on depth:
\begin{lemm}
Let $\mathfrak{Z}\in\mathcal{H}^{1}_{n}$ a motivic multiple zeta value of depth $p$. \\
If the following conditions hold,  $\forall \quad 1<2r+1<n$, $m\mathrel{\mathop:}=\lfloor\frac{n}{2}\rfloor-1$:
\begin{itemize}
\item[$(i)$] $D_{2r+1,p}(\mathfrak{Z})= P^{\mathfrak{Z}}_{r}(\zeta^{\mathfrak{m}}(3),\zeta^{\mathfrak{m}}(5), \ldots, \zeta^{\mathfrak{m}}(2m+1), \zeta^{\mathfrak{m}}(2)),$
$$\text{with } P^{\mathfrak{Z}}_{r}(X_{1},\cdots, X_{m}, Y )= \sum_{2s+\sum (2k+1)\cdot a_{k}=n-2r-1 } \beta^{r}_{a_{1}, \ldots, a_{m}, s} X_{1}^{a_{1}} \cdots X_{m}^{a_{m}} Y^{s}.$$
\item[$(ii)$] For $ a_{k},a_{r}>0 \text{  : } \frac{ \beta^{r}_{a_{1}, \ldots, a_{r}-1,\cdots, a_{m},s}}{a_{r}+1} =\frac{ \beta^{k}_{a_{1}, \ldots, a_{k}-1, \ldots, a_{m},s}}{a_{k}}.$
\end{itemize}
Then, $\mathfrak{Z}$ is a polynomial in depth $1$ MMZV:
$$ \mathfrak{Z}= \alpha \zeta^{\mathfrak{m}}(n)+ \sum_{2s+\sum (2k+1)a_{k}=n} \alpha_{a_{1}, \ldots, a_{m},s} \zeta^{\mathfrak{m}}(3)^{a_{1}} \cdots \zeta^{\mathfrak{m}}(2m+1)^{a_{m}}  \zeta^{\mathfrak{m}}(2)^{s}. \footnote{In particular, $\alpha_{a_{1}, \ldots, a_{m},s} =\frac{\beta^{r}_{a_{1}, \ldots,a_{r}-1, \ldots, a_{m}, s}}{a_{r}}$  for  $a_{r}\neq 0$. }$$
\end{lemm}
\begin{proof}
Immediate with Corollary $\ref{kerdn}$ since:
$$D_{2r+1,p } \left( \zeta^{\mathfrak{m}}(3)^{a_{1}} \cdots \zeta^{\mathfrak{m}}(2m+1)^{a_{m}}  \zeta^{\mathfrak{m}}(2)^{s}\right) = a_{r} \zeta^{\mathfrak{m}}(3)^{a_{1}} \cdots  \zeta^{\mathfrak{m}}(2r+1)^{a_{r}} \cdots \zeta^{\mathfrak{m}}(2m+1)^{a_{m}} \zeta^{\mathfrak{m}}(2)^{s}. $$
\end{proof}
\noindent
\texttt{Example}: Some examples were given in the proof of Lemma $\ref{lemmcoeff}$; the following family is polynomial in zetas \footnote{Proof method: with recursion hypothesis on coefficients, using:
$$D_{2r+1}(\zeta^{\mathfrak{m}} (\left\lbrace 1 \right\rbrace ^{n}, m))= - \sum_{j=\max(0,2r+2-m)}^{\min(n-1,2r-1)} \zeta^{\mathfrak{l}} (\left\lbrace 1 \right\rbrace ^{j}, 2r+1-j) \otimes \zeta^{\mathfrak{m}} (\left\lbrace 1 \right\rbrace ^{n-j-1}, m-2r+j).$$}:
$$\zeta^{\mathfrak{m}} (\left\lbrace 1 \right\rbrace ^{n}, m).$$

\paragraph{Sum formulas.}
Here are listed a few examples of the numerous \textit{sum identities} known for Euler sums\footnote{Usually proved considering the generating function, and expressing it as a hypergeometric function.} which we can lift to motivic Euler sums, via the coaction. For these identities, as we see through the proof, the action of the Galois group is trivial; the families being stable under the derivations, we are able to lift the identity to its motivic version via a simple recursion.

\begin{theo}
Summations, if not precised are done over the admissible multi-indices, with $w(\cdot)$, resp. $d(\cdot)$, resp. $h(\cdot)$ indicating the weight, resp. the depth, resp. the height:
\begin{itemize}
\item[(i)] With fixed even (possibly negative) $\left\lbrace a_{i}\right\rbrace _{1 \leq i \leq p}$ of sum $2n$:\footnote{This would be clearly also true for MMZV$^{\star}$.}
$$\sum_{\sigma\in\mathfrak{S}_{p}} \zeta^{\mathfrak{m}}(a_{\sigma(1)}, \ldots, a_{\sigma(p)}) \in \mathbb{Q} \pi^{2n, \mathfrak{m}}.$$
In particular:\footnote{The precise coefficient is given in $\cite{BBB1}$, $(48)$ and can then be deduced also for the motivic identity.}
$$\zeta^{\mathfrak{m}}(\left\lbrace 2n \right\rbrace^{p} ) , \zeta^{\mathfrak{m}}(\left\lbrace \overline{2n} \right\rbrace^{p} ) \in \mathbb{Q} \pi^{2np, \mathfrak{m}}.$$
More precisely, with Hoffman $\cite{Ho}$ \footnotemark[6] 
\begin{multline}\nonumber
\sum_{\sum n_{i}= 2n} \zeta^{\mathfrak{m}}\left( 2n_{1}, \ldots, 2n_{k}\right) =\\
\frac{1}{2^{2(k-1)}} \binom{2k-1}{k}  \zeta^{\mathfrak{m}}(2n) - \sum_{j=1}^{\lfloor\frac{k-1}{2}\rfloor} \frac{1}{2^{2k-3}(2j+1) B_{2j}} \binom{2k-2j-1}{k} \zeta^{\mathfrak{m}}(2j) \zeta^{\mathfrak{m}}(2n-2j) .
\end{multline}
\item[(ii)]  With Granville $\cite{Gra}$, or Zagier $\cite{Za1}$ \footnotemark[6] 
$$ \sum_{w(\textbf{k})=n, d(\textbf{k})=d } \zeta^{\mathfrak{m} }(\textbf{k})= \zeta^{\mathfrak{m}}(n). $$
\item[(iii)]  With Aoki, Ohno $\cite{AO}$\footnotemark[6] \footnotemark[2]
\begin{align*}
\sum_{w(\textbf{k})=n, d(\textbf{k})=d } \zeta^{\star,\mathfrak{m}}(\textbf{k}) & =  \binom{n-1}{d-1} \zeta^{\mathfrak{m}}(n).\\
 \sum_{w(\textbf{k})=n, h(\textbf{k})=s } \zeta^{\star,\mathfrak{m}}(\textbf{k})& =  2\binom{n-1}{2s-1} (1-2^{1-n}) \zeta^{\mathfrak{m}}(n). 
\end{align*}
\item[(iv)]  With Le, Murakami$\cite{LM}$\footnotemark[6]  $$\sum_{w(\textbf{k})=n, h(\textbf{k})=s } (-1)^{d(\textbf{k})}\zeta^{\mathfrak{m}}(\textbf{k})=\left\lbrace  \begin{array}{ll}
0 & \text{ if } n \text{ odd} . \\ 
\frac{(-1)^{\frac{n}{2}} \pi^{\mathfrak{m},n}}{(n+1)!} \sum_{k=0}^{\frac{n}{2}-s}\binom{n+1}{2k} (2-2^{2k})B_{2k} & \text{ if } n \text{ even} .\\ 
\end{array} \right. $$
\item[(v)] With S. Belcher (?)\footnotemark[6]
$$\hspace*{-0.5cm}\begin{array}{llll}
\sum_{w(\cdot)=2n \atop d(\cdot)=2p } \zeta^{\mathfrak{m}}(odd, odd>1, odd, \ldots, odd, odd>1)& =&  \alpha^{n,p} \zeta^{\mathfrak{m}} (2)^{n}, & \alpha^{n,p}  \in \mathbb{Q}\\
\sum_{w(\cdot)=2n+1 \atop d(\cdot)=2p+1} \zeta^{\mathfrak{m}}(odd, odd>1, odd, \ldots, odd>1, odd)&=& \sum_{i=1}^{n} \beta^{n,p}_{i} \zeta^{\mathfrak{m}}(2i+1) \zeta^{\mathfrak{m}}(2)^{n-i} , & \beta^{n,p}_{i}\in\mathbb{Q}\\
\sum_{w(\cdot)=2n+1 \atop d(\cdot)=2p+1} \zeta^{\mathfrak{m}}(odd>1, odd, \ldots, odd, odd>1)&=& \sum_{i=1}^{n} \gamma^{n,p}_{i} \zeta^{\mathfrak{m}}(2i+1) \zeta^{\mathfrak{m}}(2)^{n-i}, &  \gamma^{n,p}_{i}\in\mathbb{Q} 

\end{array}$$
\end{itemize}
\footnotetext[6]{The person(s) at the origin of the analytic equality for MZV, used in the proof for motivic MZV.}
\end{theo}
\noindent
\textsc{Remark}: The permutation identity $(i)$ would in particular imply that all sum of MZV at even arguments times a symmetric function of these same arguments are rational multiple of power of $\mathbb{L}^{\mathfrak{m}}$. \\
Many specific identities, in small depth have been already found (as Machide in $\cite{Ma}$, resp. Zhao, Guo, Lei in $\cite{GLZ}$, etc.), and can be directly deduced for motivic MZV, such as:
\begin{align*}
\hspace*{-2.5cm}\sum_{k=1}^{n-1} \zeta(2k, 2n-2k) \quad\quad  & \left\lbrace  \begin{array}{lll}
1 & =& \frac{3}{4} \zeta(2n)\\
4^{k}+4^{n-k} &=& (n+\frac{4}{3}+\frac{4^{n}}{6})\zeta(2n) \\
(2k-1)(2n-2k-1) &=& \frac{3}{4} (n-3) \zeta(2n)
\end{array} \right. \\
 \hspace*{-0.3cm}\sum \zeta(2i, 2j, 2n-2i-2j) & \left\lbrace  \begin{array}{lll}
1 & =& \frac{5}{8} \zeta(2n)- \frac{1}{4}  \zeta(2n-2) \zeta(2)\\
ij +jk+ki &=& \frac{5n}{64} \zeta(2n)+(4n-\frac{9}{10})  \zeta(2n-2) \zeta(2) \\
ijk &=& \frac{n}{128} (n-3) \zeta(2n)-\frac{1}{32} \zeta(2n-2) \zeta(2)+\frac{2n-5}{8} \zeta(2n-4) \zeta(4)\\
\end{array} \right.\\
 &  \\
\end{align*}

\begin{proof}
We refer to the formula of the derivations $D_{r}$ in Lemma $\ref{lemmt}$. For many of these equalities, when summing over all the permutations of a certain subset, most of the cuts will get simplified two by two as followed:
\begin{equation}\label{eq:termda}
\zeta^{\mathfrak{m}}\left( k_{1}, \ldots, k_{i}, k_{i+1}, \ldots, k_{j}, k_{j+1}, \cdots k_{d}\right)  \text{ : }   0; \cdots  1 0^{k_{i}-1} \boldsymbol{1 } 0^{k_{i+1}-1} \cdots 0^{k_{j-1}-1} 1 \boldsymbol{0^{k_{j}-1}} 1 0^{k_{j+1}-1}\cdots ; 1 .
\end{equation}
\begin{equation}\label{eq:termdb}
\zeta^{\mathfrak{m}}(k_{1},\cdots, k_{i}, k_{j},  \cdots, k_{i+1}, k_{j+1}, \ldots, k_{d}) \text{ :  } 0; \cdots  1 0^{k_{i}-1} 1 \boldsymbol{0^{k_{j}-1}} \cdots 0^{k_{i+2}-1} 1 0^{k_{i+1}-1} \boldsymbol{1} 0^{k_{j+1}-1}\cdots  ; 1.
\end{equation}
It remains only the first cuts, beginning with the first $0$, such as:
\begin{equation}\label{eq:termd1}
\delta_{2r+1= \sum_{j=1}^{i} k_{j}} \zeta^{\mathfrak{m}}\left( k_{1}, \ldots, k_{i})\otimes \zeta^{\mathfrak{m}}(k_{i+1}, \ldots, k_{d}\right) ,
\end{equation}
and possibly the cuts from a $k_{i}=1$ to $k_{d}$, if the sum is over admissible MMZV: \footnote{There, beware, the MZV at the left side can end by $1$.}
\begin{equation}\label{eq:termdr}
-\delta_{2r+1< \sum_{j=i+1}^{d} k_{j}} \zeta^{\mathfrak{m}}\left( k_{i+1}, \ldots, k_{d-1}, 2r+1- \sum_{j=i+1}^{d-1} k_{j}\right) \otimes \zeta^{\mathfrak{m}}\left( k_{1}, \ldots, k_{i-1}, \sum_{j=i+1}^{d} k_{j} -2r\right)  .
\end{equation}
\begin{itemize}
\item[(i)] From the terms above in $D_{2r+1}$, $(\ref{eq:termda})$, and $(\ref{eq:termdb})$ get simplified together, and there are no terms $(\ref{eq:termd1})$ since the $a_{i}$ are all even. Therefore, it is in the kernel of $\oplus_{2r+1<2n} D_{2r+1}$ with even weight, hence Galois trivial.\\
For instance, for $\zeta^{\mathfrak{m}}(\left\lbrace \overline{2n} \right\rbrace^{p} ) $, with $\epsilon, \epsilon'\in \lbrace\pm 1\rbrace$:\\
\includegraphics[]{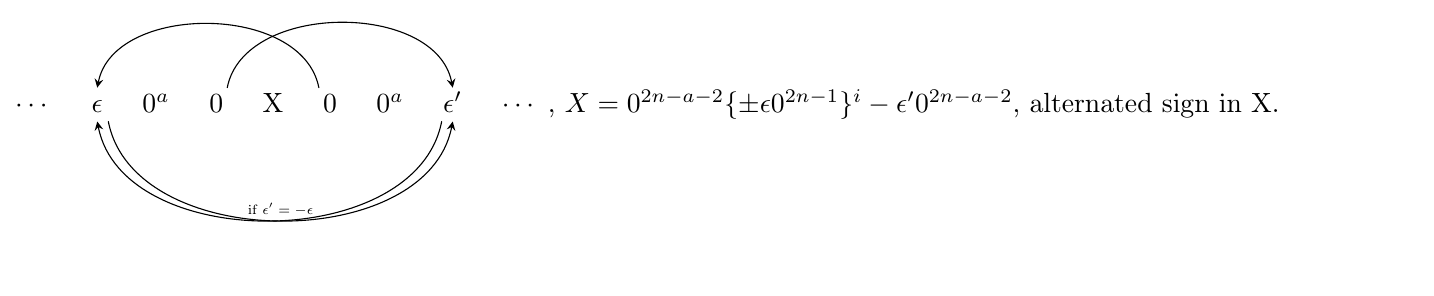}\\
Either, $\epsilon=\epsilon'$ and $X$ is symmetric, and by reversal of path (cf. $\S A.1.1$), cuts above get simplified, or $\epsilon=-\epsilon'$ and $X$ is antisymmetric, and the cuts above still get simplified since $I^{\mathfrak{m}}(\epsilon;0^{a+1} X;0)=-I^{\mathfrak{m}}(0;\widetilde{X} 0^{a+1};\epsilon)=-I^{\mathfrak{m}}(0;X 0^{a+1};-\epsilon)$. 
\item[(ii)] Let us denote this sum $G(n,d)$, and $G_{1}(n,d)$ the corresponding sum where a $1$ at the end is allowed. As explained in the proof's preamble, the remaining cuts being the first ones and the one from a $k_{i}=1$ to the last $k_{d}$:
$$\hspace*{-0.5cm}D_{2r+1}(G (n,d))= \sum_{i=0}^{d-1}  G^{\mathfrak{l}}_{1}(2r+1,i) \otimes G(n-2r-1,d-i) -\sum_{i=0}^{d-1}  G^{\mathfrak{l}}_{1}(2r+1,i) \otimes G(n-2r-1,d-i) =0 .$$

\item[(iii)] This can be proven also computing the coaction, or noticing that it can be deduced from Euler relation above, turning a MZV$^{\star}$ into a sum of MZV of smaller depth, it turns to be:
$$\sum_{i=1}^{d} \sum_{w(\boldsymbol{k})= n, d(\boldsymbol{k})=i}  \binom{n-i-1}{d-i} \zeta^{\mathfrak{m}} (\boldsymbol{k}).$$
For the Aoki-Ohno identity, using the formula for MZV $\star$, and with recursion hypothesis, we could  similarly prove that the coaction is zero on these elements, and conclude with the result for MZV.
\item[(iv)] Let us denote this sum $G_{-}(n,s)$ and $G_{-,(1)}(n,s)$ resp. $G_{-,1}(n,s)$ the analogue sums with possibly a $1$ at the end, resp. with necessarily a $1$ at the end. Looking at the derivations, since we sum over all the permutations of the admissible indices, all the cuts get simplified with its symmetric cut as said above, and it remains only the beginning cut (with the first $0$), and the cut from a $k_{j}=1$ to the last $k_{d}$, which leads to:
\begin{multline}\nonumber
\hspace*{-1cm}D_{2r+1}(G_{-}(n,s))= \sum_{i=0}^{s-1} \left( G^{\mathfrak{l}}_{-,(1)}(2r+1,i) -G^{\mathfrak{l}}_{-}(2r+1,i+1)- G^{\mathfrak{l}}_{-,1}(2r+1,i)\right) \otimes G_{-}(n-2r-1,s-i) \\
= \sum_{i=0}^{s-1} (G^{\mathfrak{l}}_{-}(2r+1,i) -G^{\mathfrak{l}}_{-}(2r+1,i+1))\otimes G_{-}(n-2r-1,s-i).
\end{multline}
Using recursion hypothesis, it cancels, and thus, $G_{alt}(n,s)\in\mathbb{Q} \zeta^{\mathfrak{m}}(n)$. Using the analogue analytic equality, we conclude.

\item[(v)] For odd sequences with alternating constraints ($>1$ or $\geq 1$ for instance), cuts between $k_{i}$ and $k_{j}$ will get simplified with some symmetric terms in the sum, except possibly (when odd length), the first (i.e. from the first $1$ to a first $0$) and the last (i.e. from a last $0$ to the very last $1$) one. More precisely, with $O$ any odd integer, possibly all different:
\begin{itemize}
\item[$\cdot$]
\begin{small}
\begin{multline}\nonumber
\hspace*{-1cm}D_{2r+1} \left( \sum_{w(\cdot)=2n \atop d(\cdot)=2p } \zeta^{\mathfrak{m}}(O, O>1,  \cdots, O, O>1) \right)  \\
 = \sum_{i=0}^{p-1} \left( \sum_{w(\cdot)=2r+1 \atop d(\cdot)=2i+1 } 
\begin{array}{l}
+  \zeta^{\mathfrak{l}}(O, O>1,  \cdots, O>1, O)\\
-\zeta^{\mathfrak{l}}(O, O >1, \ldots, O>1, O)
\end{array}  \right) \otimes  \sum_{w(\cdot)=2n-2r-1 \atop d(\cdot)=2p-2i-1 } \zeta^{\mathfrak{m}}(O, O>1, \ldots, O, O>1)=0.
 \end{multline}
  \end{small}
\item[$\cdot$]
\begin{small}
\begin{multline}\nonumber 
\hspace*{-1cm}D_{2r+1} \left( \sum_{w(\cdot)=2n+1 \atop d(\cdot)=2p+1 } \zeta^{\mathfrak{m}}(O>1, O, \ldots, O, O>1) \right)  \\
= \sum_{i=0}^{p-1} \left( \sum_{w(\cdot)=2r+1 \atop d(\cdot)=2i+1 } \zeta^{\mathfrak{l}}(O>1, O, \ldots, O>1) \right) \otimes  \sum_{w(\cdot)=2n-2r \atop d(\cdot)=2p-2i-1 } \zeta^{\mathfrak{m}}(O, O>1,  \cdots, O, O>1).
 \end{multline}
  \end{small}
By the previous identity, the right side is in $\mathbb{Q}\pi^{2n-2r}$, which proves the result claimed; it gives also the recursion for the coefficients: $\beta^{n,p}_{r}= \sum_{i=0}^{p-1} \beta^{r,i}_{r} \alpha^{n-r,p-i} $.
\item[$\cdot$]
\begin{small}
\begin{multline}\nonumber 
\hspace*{-1cm}D_{2r+1} \left( \sum_{w(\cdot)=2n+1 \atop d(\cdot)=2p+1 } \zeta^{\mathfrak{m}}(O, O>1, \ldots, O>1, O) \right) =\\
  \begin{array}{l}
+\sum_{i=0}^{p-1} \left( \sum_{w(\cdot)=2r+1 \atop d(\cdot)=2i+1 } \zeta^{\mathfrak{m}}(O, O>1, \ldots, O) \right) \otimes  \sum_{w(\cdot)=2n-2r \atop d(\cdot)=2p-2i-1 } \zeta^{\mathfrak{m}}(O, O>1,  \cdots, O>1)\\
+\sum_{i=0}^{p-1} \left( \sum_{w(\cdot)=2r+1,  \atop d(\cdot)=2i+1 } \begin{array}{l}
+ \zeta^{\mathfrak{m}}(O, O>1, \ldots, O)\\
 - \zeta^{\mathfrak{m}}(O, O>1, \ldots,  O)
\end{array}  \right)  \otimes  \sum_{w(\cdot)=2n-2r \atop d(\cdot)=2p-2i-1 } \zeta^{\mathfrak{m}}(O>1, O,  \cdots, O>1, O)
\end{array} 
 \end{multline}
 \end{small}
As above, by recursion hypothesis, the right side of the first sum is in $\mathbb{Q}\pi^{2n-2r}$, which proves the result claimed, the second sum being $0$; the rational coefficients $\gamma$ are given by a recursive relation.
\end{itemize}
\end{itemize}
\end{proof}

\appendix
\chapter{}
\section{Coaction}

The coaction formula given by Goncharov and extended by Brown for motivic iterated integrals applies to the $\star$, $\star\star$, $\sharp$ or $\sharp\sharp$ version by linearity \footnote{Recall the identities $\ref{eq:miistarsharp}$ to turn a $\star$ (resp. $\sharp$) into a $1$ (resp. two times a $1$) minus a $0$.}. Here is the version obtained for MMZV $\star,\star\star$, $\sharp$ or $\sharp\sharp$:\footnote{\textit{For purpose of stability}: if there is a $\pm 1$ at the beginning, as for $\star$  or $\sharp$ versions, the cut with this first $\pm 1$ will be let as a $T_{\pm 1, 0}$ term (and not converted into a $T_{\epsilon, 0}$ less a $T_{0,0}$), in order to still have a $\pm 1$ at the beginning; whereas, if there is no $\pm 1$ at the beginning, as for $\star\star$ or $\sharp\sharp$ version even the first cut (first line) has to be converted into a $T_{0, \epsilon}$ less a $T_{0,0}$, in order to still have a $\epsilon$ at the beginning.}
\begin{lemm}\label{lemmt}
$L$ being a sequence in $\lbrace 0, \pm\star\rbrace$ resp. $\lbrace 0,\pm\sharp\rbrace$, with possibly $1$ at the beginning,  $\epsilon \in \lbrace \pm\star \rbrace$ resp. $\in \lbrace\pm\sharp\rbrace$, and $s_{\epsilon}\mathrel{\mathop:}=sign(\epsilon)$.
$$D_{r} I^{\mathfrak{m}}_{s}\left(0;L;1 \right)=\delta_{ L= A \epsilon B \atop w(A)=r}   I^{\mathfrak{l}}_{k}\left(0;A ;s_{\epsilon} \right)  \otimes  I^{\mathfrak{m}}_{s-k}\left(0;s_{\epsilon}, B ;1 \right) $$
$$+ \sum_{L=A \epsilon B 0 C \atop w(B)=r}  I^{\mathfrak{l}}\left(s_{\epsilon};B ;0 \right) \otimes \left( \underbrace{I^{\mathfrak{m}}_{s}\left(0;A, \epsilon, 0, C ;1\right)}_{T_{\epsilon,0}} + \underbrace{I^{\mathfrak{m}}_{s}(0;A,0,0,C ;1)} _{T_{0,0}} \right)$$
$$+ \sum_{L=A 0 B \epsilon C \atop w(B)=r}  I^{\mathfrak{l}}\left(0;B, s_{\epsilon}\right) \otimes \left( \underbrace{I^{\mathfrak{m}}_{s}\left(0;A,0, \epsilon, C ;1 \right)}_{T_{0, \epsilon}} + \underbrace{I^{\mathfrak{m}}_{s}(0;A,0,0,C ;1)}_{T_{0,0}} \right)$$
$$+ \sum_{L=A \epsilon B \epsilon C\atop w(B)=r}  I^{\mathfrak{l}}\left(0;B, s_{\epsilon}\right) \otimes \left( \underbrace{I^{\mathfrak{m}}_{s}(0;A,\epsilon,0,C ;1)}_{T_{\epsilon,0}} - \underbrace{I^{\mathfrak{m}}_{s}\left(0;A,0, \epsilon, C ;1 \right)}_{T_{0,\epsilon}}  \right)$$

\begin{multline}\nonumber
+ \sum_{L=A  \epsilon B -\epsilon C\atop w(B)=r} \left[  I^{\mathfrak{l}}\left(0;B; -s_{\epsilon} \right) \otimes \underbrace{I^{\mathfrak{m}}_{s}(0;A,\epsilon, 0,C ;1)}_{T_{\epsilon,0}} + I^{\mathfrak{l}}\left(s_{\epsilon};B;0\right)\otimes \underbrace{I^{\mathfrak{m}}_{s}(0; A,0,-\epsilon, C ;1)}_{T_{0,\epsilon}} \right.  \\
\left. I^{\mathfrak{l}}\left(s_{\epsilon};B; -s_{\epsilon} \right) \otimes  \underbrace{I^{\mathfrak{m}}_{s}\left(0;A, \epsilon, - \epsilon, C ;1 \right)}_{T_{\epsilon, -\epsilon}}  \right] .
\end{multline}
\end{lemm}
\newpage
\noindent
\textsc{Remarks}:
\begin{itemize}
\item[$\cdot$]  We will refer to these different terms $T$ for each cut in the whole appendix when using the coaction. 
\item[$\cdot$] The expression of $D_{r}$ for specific MMZV $\star$ and Euler $\sharp$ sums is simplified below thanks to antipodal and hybrid relations, and is fundamentally used in the proofs of Chapter $4$.
\end{itemize}
\begin{proof}
The proof is straightforward from $\eqref{eq:Der}$, using the linearity (with $\ref{eq:miistarsharp}$) in both directions:
\begin{itemize}
\item[$(i)$] First, to turn $\epsilon$ into a difference of $\pm 1$ minus $0$ in order to use $\eqref{eq:Der}$.
\item[$(ii)$] Then, in the right side, a $\pm 1$ appeared inside the iterated integral when looking at the usual coaction formula which is turned into a sum of a term with $\epsilon$ (denoted $T_{\epsilon,0}$ or $T_{0,\epsilon}$) and a term with $0$ (denoted $T_{0,0}$) by linearity of the iterated integrals and in order to end up only with $0,\epsilon$ in the right side.
\end{itemize}
Listing now the different cuts leads to the expression of the lemma, since:
\begin{itemize}
\item[$\cdot$] The first line corresponds to the initial cut (from the $s+1$ first $0$).
\item[$\cdot$] The second line corresponds to a cut either from $\pm\epsilon$ to $0$; the $\pm\epsilon$ being $\pm 1$.
\item[$\cdot$] The third line corresponds to a cut from $0$ to $\pm\epsilon$.
\item[$\cdot$] The fourth line corresponds to cut from $\epsilon$ to $\epsilon$, with two choices: a $\epsilon$ being fixed to $0$, the other one fixed to $1$. Replacing $1$ by $(\epsilon)+(0)$, this leads to a $T_{0,\epsilon}$, a $T_{\epsilon,0}$ and two $T_{0,0}$ terms which get simplified together.
\item[$\cdot$] The last lines correspond to cuts from $\epsilon$ to $-\epsilon$, with three possibilities: one being fixed to $0$, the other one fixed to $\pm 1$, or the first being $1$, the second $-1$. This leads to a $T_{\epsilon,0}$, a $T_{0, -\epsilon}$ and a $T_{\epsilon,-\epsilon}$, since the $T_{0,0}$ terms get simplified.
\end{itemize}
\end{proof}

\subsection{Simplification rules}
This section is devoted on the simplification of the coaction, in the case of motivic Euler sums: we gather terms in $D_{2r+1}$ according to their right side, using relations ($§ 4.2$) between motivic iterated integrals $I^{\mathfrak{l}}\in\mathcal{L}$ to simplify the left side.\\
\\
\texttt{Notations:} We use the notation of the iterated integrals inner sequences and represent a term of a cut in $D_{2r+1}$ (referring to Lemma $4.4.2$) by arrows between two elements of this sequence. The weight of the cut (which is the length of the subsequence in the iterated integral) would always be considered odd here.\footnote{Since we are here only interested in motivic Euler sums, the non zero weight graded parts in the coaction are these corresponding to odd weights: $D_{2r+1}$, $r\geq 0$.} The diagrams show which terms get simplified together: i.e. these which have same right side, but opposite left side, by the relation considered in the coalgebra $\mathcal{L}$. \\
\begin{description}
\item[\textsc{ Composition }:]  The composition rule (cf. $§ 1.6$) in the coalgebra $\mathcal{L}$ boils down to:
\begin{equation} 
I^{\mathfrak{l}}(a; X; b)\equiv -I^{\mathfrak{l}}(b; X; a), \quad  \text{ with $X$ any sequence of  }  0, \pm 1, \pm \star, \pm \sharp.
\end{equation}
It allows us to switch the two extremities of the integral if we multiply by $-1$ the integral: this exchange is considerably used below, without mentioning.
\item[\textsc{ Antipode }  $\shuffle$:] It corresponds to a reversal of path for iterated integrals (cf. $\ref{eq:antipodeshuffle2}$):
\begin{center}
$I^{\mathfrak{l}}(a; X; b)\equiv(-1)^{w}I^{\mathfrak{l}}(b; \widetilde{X}; a)$ for any X sequence of $0, \pm 1, \pm \star, \pm \sharp$. Hence:
\end{center}
\begin{itemize}
\item[$\cdot$] If X symmetric, i.e.  $\widetilde{X}=X$, these two cuts get simplified,  \includegraphics[]{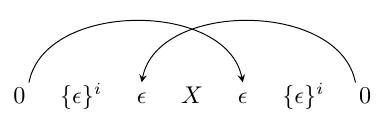}
since  $ I^{\mathfrak{l}}(\epsilon; X \epsilon^{i+1}; 0) \equiv - I^{\mathfrak{l}}(0;  \epsilon^{i+1}\widetilde{X} ; \epsilon) \equiv - I^{\mathfrak{l}}(0;  \epsilon^{i+1} X ; \epsilon)$.
\item[$\cdot$] If X antisymmetric, i.e. $\widetilde{X}=-X$,the cut \includegraphics[]{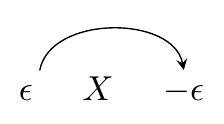} is zero since:\\
$ \begin{array}{ll}
I^{\mathfrak{l}}(\epsilon; X; -\epsilon) & \equiv I^{\mathfrak{l}}(\epsilon; X; 0)+ I^{\mathfrak{l}}(0; X; -\epsilon)\\
&\equiv I^{\mathfrak{l}}(\epsilon; X; 0)- I^{\mathfrak{l}}(-\epsilon; \widetilde{X}; 0)  \\
&\equiv I^{\mathfrak{l}}(\epsilon; X; 0)- I^{\mathfrak{l}}(-\epsilon; -X; 0)\\
&\equiv 0
\end{array}$.
\end{itemize}

\item[\textsc{ Shift }] For MES $\star\star$ and, when weight and depth odd for Euler $\sharp\sharp$ sums:
\begin{equation}\label{eq:shift} \textsc{(Shift) } \zeta^{\bullet}_{n-1} (n_{1},\cdots, n_{p})= \zeta^{\bullet}_{n_{1}-1} (n_{2},\cdots, n_{p},n)
\end{equation}
\includegraphics[]{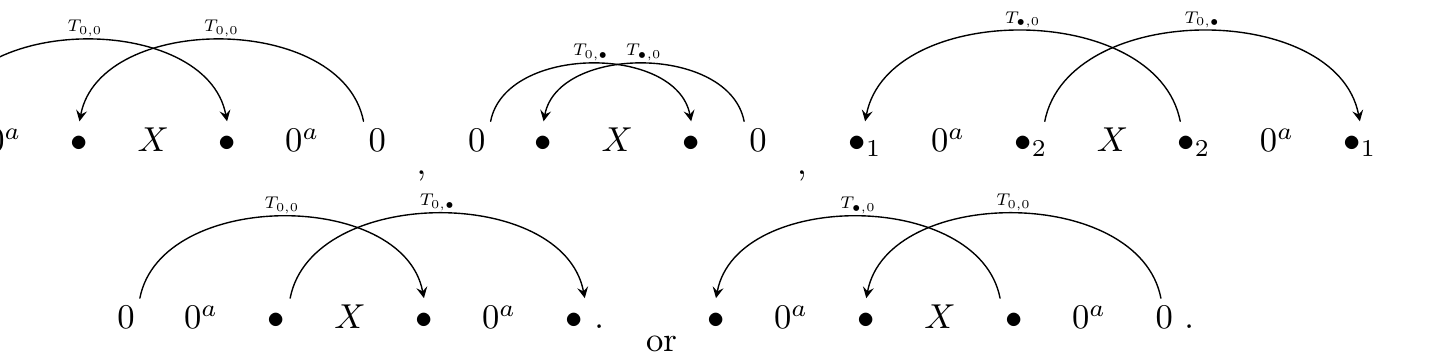}
A dot belongs to $\lbrace \pm\star,\pm\sharp\rbrace$, and two dots with a same index, $\bullet_{i}$ shall be identical.
\item[\textsc{ Cut }] For ES $\sharp\sharp$, with even depth\footnotemark[1], odd weight:\footnotetext[1]{Note that the depth considered here needed to be even is the depth of the bigger cut.}\\
\begin{equation}\label{eq:cut}
\includegraphics[]{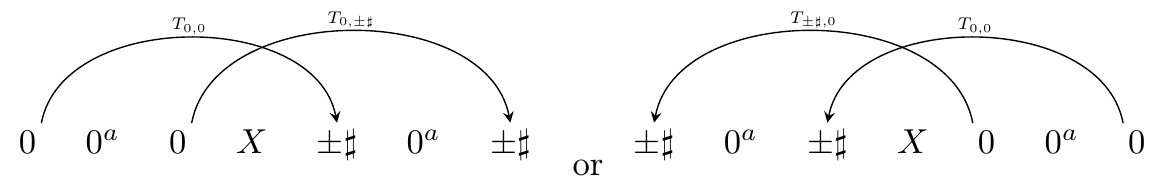}
\end{equation}
\item[\textsc{ Cut Shifted }]: For ES $\sharp\sharp$, with even depth\footnotemark[1], odd weight, composing Cut with Shift: 
\begin{equation}\label{eq:cutshifted} 
\includegraphics[]{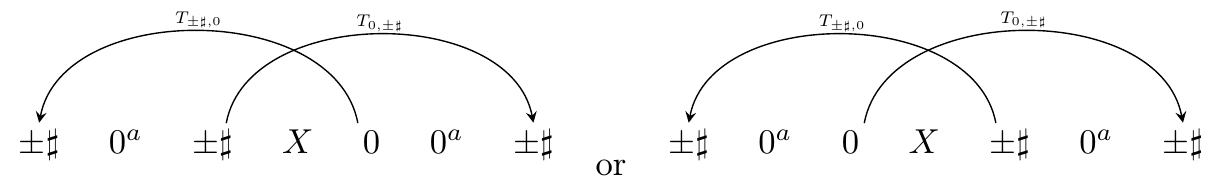}
\end{equation}
\item[\textsc{ Minus }] For ES $\sharp\sharp$, with even depth, odd weight:
\begin{equation}\label{eq:minus}  
\includegraphics[]{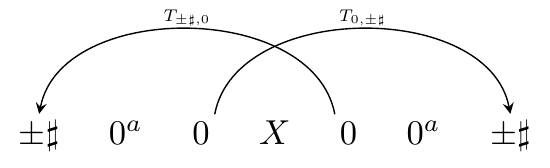}
\end{equation}
\item[\textsc{ Sign }] For ES $\sharp\sharp$ with even depth, odd weight, i.e. $X\in \lbrace 0, \pm \sharp\rbrace^{\times}$:
\begin{equation}\label{eq:sign} 
\includegraphics[]{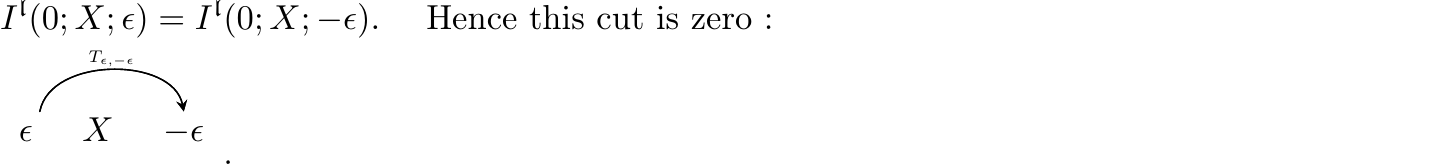}
\end{equation}
\textsc{ Sign } hence also means that the $\pm$ sign at one end of a cut does not matter.
\end{description}

\subsection{MMZV $\star$}

Let express each MMZV$^{\star}$ as:
$$\zeta^{\star, \mathfrak{m}} (2^{a_{0}},c_{1},\cdots,c_{p}, 2^{a_{p}}).$$
As we will see below, this writing is suitable for the coaction expression, since most of the cuts from a $2$ to another $2$ get simplified by the rules above. The iterated integral corresponding:\\
\begin{equation}\label{eq:iistar}
I\left( 0; 1, 0, \left(\star, 0\right)^{a_{0}-1}, \star \cdots 0^{c_{i}-1} \left( \star 0 \right)^{a_{i}} \star , \ldots, 0^{c_{j}-1}  \left(\star 0 \right)^{a_{j}} \star, \ldots, 0^{c_{p}-1}  \left( \star 0 \right)^{a_{p}} ; 1\right) .
\end{equation}
Considering $D_{2r+1}$ after some simplifications:\footnote{Here $\delta_{r}$ underlines that the left side must have a weight equal to $2r+1$.} 
\begin{lemm}
\begin{equation} \label{eq:DerivStar} D_{2r+1} \left(   \zeta^{\star, \mathfrak{m}} (2^{a_{0}},c_{1},\cdots,c_{p}, 2^{a_{p}})\right)  =  
\end{equation}
$$\delta_{r} \sum_{i<j} \left[    \delta_{3\leq \alpha \leq c_{i+1}-1 \atop 0\leq \beta \leq a_{j}} \zeta^{\star, \mathfrak{l}}_{c_{i+1}-\alpha} (2^{a_{j}-\beta}, \ldots, 2^{ a_{i+1}}) \otimes  \zeta^{\star, \mathfrak{m}} (\cdots,2^{a_{i}}, \alpha, 2^{\beta}, c_{j+1}, \cdots) \right.$$
$$\left( -\delta_{c_{i+1}>3} \zeta^{\star\star, \mathfrak{l}}_{2} (2^{a_{j}-\beta-1}, \ldots, 2^{ a_{i+1}})  + \delta_{c_{j+1}>3} \zeta^{\star\star, \mathfrak{l}}_{2} (2^{a_{j}}, \ldots, 2^{ a_{i+1}-\beta -1}) + \right. $$
$$ - \delta_{c_{i+1}=1} \zeta^{\star\star, \mathfrak{l}} (2^{a_{j}-\beta}, \ldots, 2^{ a_{i+1}})  + \delta_{c_{j+1}=1} \zeta^{\star\star, \mathfrak{l}} (2^{a_{i+1}-\beta}, \ldots, 2^{ a_{j}}) $$
$$+ \delta_{c_{i+2}=1 \atop \beta>a_{i+1}} \zeta^{\star\star, \mathfrak{l}}_{1} (2^{a_{j}+a_{i+1}-\beta}, \ldots, 2^{ a_{i+2}})  - \delta_{c_{j}=1 \atop \beta>a_{j}} \zeta^{\star\star, \mathfrak{l}}_{1} (2^{a_{i+1}+a_{j}-\beta}, \ldots, 2^{ a_{j-1}}) .$$
$$ \left. +\delta_{\beta > a_{i+1}}\zeta^{\star\star, \mathfrak{l}}_{c_{i+2}-2} (2^{a_{i+1}+a_{j}-\beta+1}, \ldots, 2^{ a_{i+2}})  - \delta_{\beta > a_{j}} \zeta^{\star\star, \mathfrak{l}}_{c_{j}-2} (2^{a_{i+1}+a_{j}-\beta+1}, \ldots, 2^{ a_{j-1}}) \right) $$
$$\otimes  \zeta^{\star, \mathfrak{m}} (\cdots,2^{a_{i}}, c_{i+1}, 2^{\beta}, c_{j+1}, \cdots)$$
$$\left. - \delta_{3\leq \alpha \leq c_{j}-1 \atop 0\leq \beta \leq a_{i}} \zeta^{\star, \mathfrak{l}}_{c_{j}-\alpha} (2^{a_{i}-\beta}, \ldots, 2^{ a_{j-1}}) \otimes  \zeta^{\star, \mathfrak{m}} (\cdots, c_{i}, 2^{\beta}, \alpha, 2^{a_{j}}, \cdots)\right] . $$
\end{lemm} 

\begin{proof}
We look at cuts of odd interior length between two elements of the sequence inside $\ref{eq:iistar}$. By \textsc{Shift}, the following cuts get simplified:\\
\includegraphics[]{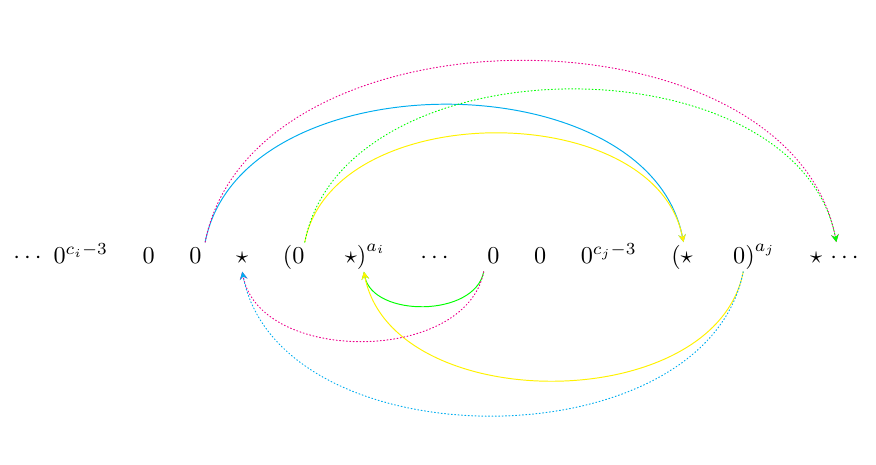}\
More precisely, $T_{0,0}$ resp. $T_{0,\star}$ above get simplified with $T_{0,0}$ resp. $T_{\star,0}$ below (shifted by one at the right), by colors, two by two. The dotted arrows mean that in the particular case where $c_{i}=1$ resp. $c_{j}=1$, only $T_{0,0}$ get simplified.\\
The following arrows get simplified by $\textsc{Shift}$ $(\ref{eq:shift})$, still above with below and by colors:\\
\includegraphics[]{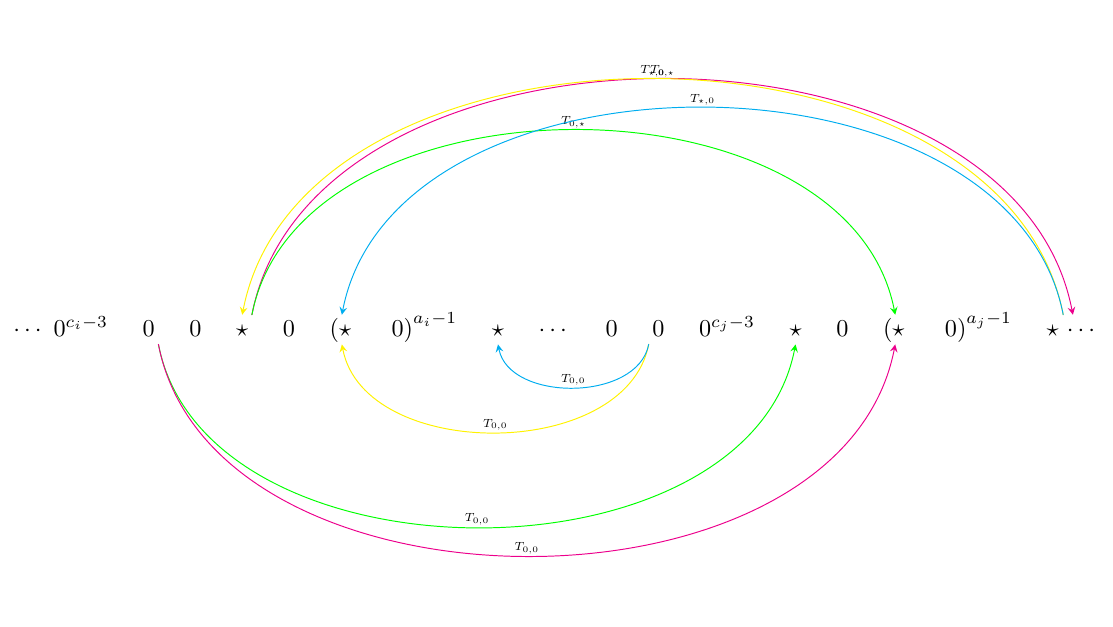}\\
\includegraphics[]{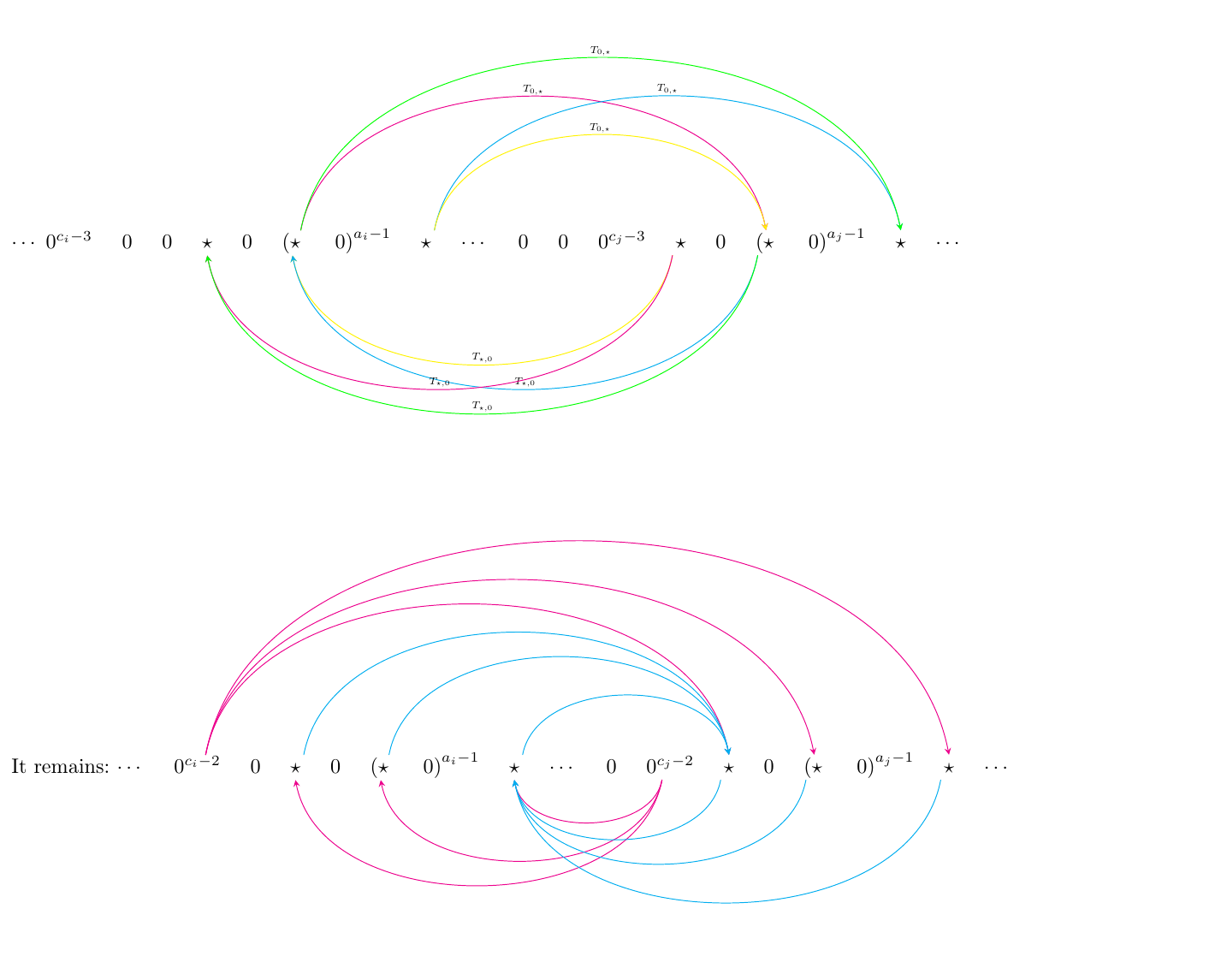}\\
 Cyan arrows above resp. below are $T_{0,\star}$ resp. $T_{0,\star}$ terms; magenta ones above resp. below stand for $T_{0,0}$ and $T_{0,\star}$ resp. for $T_{0,0}$ and $T_{\star,0}$ terms. \\
If $c_{i}=1$ (the case $c_{j}=1$, antisymmetric, is omitted), it remains also the following cuts ($T_{\star,0}$ for black ones or $T_{0,\star}$ for cyan ones):\\
\includegraphics[]{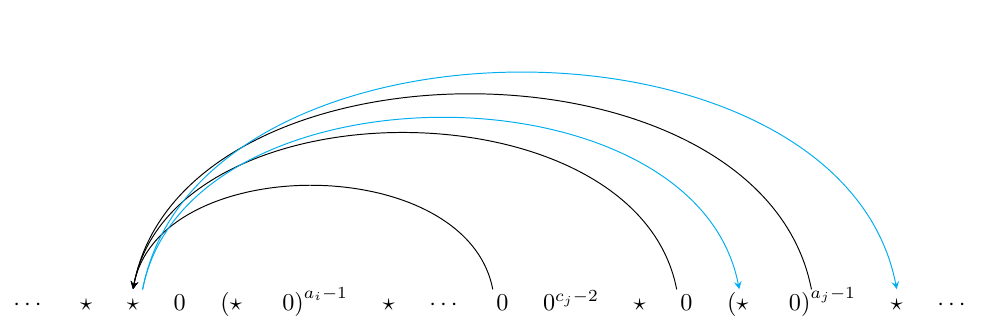}\\
 
Gathering the remaining cuts in this diagram, according the right side:\\
\begin{enumerate}
\item For $\zeta^{\star, \mathfrak{m}} (\cdots,2^{a_{i}}, \boldsymbol{\alpha, 2^{\beta}}, c_{j+1}, \cdots)$:
$$ \left( \textcolor{magenta}{\delta_{3\leq \alpha < c_{i+1} \atop 0\leq \beta < a_{j}}} \zeta^{\star\star, \mathfrak{l}}_{c_{i+1}-\alpha} (2^{a_{i+1}}, \ldots, 2^{ a_{j}-\beta}) - \left( \textcolor{magenta}{\delta_{4\leq \alpha < c_{i+1} \atop 0\leq \beta < a_{j}}} + \textcolor{cyan}{\delta_{\alpha=3 \atop 0\leq \beta < a_{j}}} \right) \zeta^{\star\star, \mathfrak{l}}_{c_{i+1}-\alpha+2} (2^{a_{i+1}}, \ldots, 2^{ a_{j}-\beta-1}) \right).$$
Using \textsc{Shift} $(\ref{eq:shift})$ for the first term and then using the definition of $\zeta^{\star}$ it turns into:
$$ \delta_{3\leq \alpha < c_{i+1} \atop 0\leq \beta <a_{j}} \left(  \zeta^{\star\star, \mathfrak{l}}_{1} (c_{i+1}-\alpha+1, 2^{a_{i+1}}, \ldots, 2^{ a_{j}-\beta-1}) -\zeta^{\star\star, \mathfrak{l}}_{c_{i+1}-\alpha+2} (2^{a_{i+1}}, \ldots, 2^{ a_{j}-\beta-1}) \right) $$
$$=  \delta_{3\leq \alpha < c_{i+1} \atop 0\leq \beta < a_{j}}  \zeta^{\star, \mathfrak{l}}_{1} (c_{i+1}-\alpha+1, 2^{a_{i+1}}, \ldots, 2^{ a_{j}-\beta-1}). $$
Applying antipodes $A_{\shuffle} \circ A_{\ast} \circ A_{\shuffle}$:
$$ =\delta_{3\leq \alpha < c_{i+1} \atop 0\leq \beta < a_{j}}  \zeta^{\star, \mathfrak{l}}_{c_{i+1}-\alpha} (2^{ a_{j}-\beta}, \ldots, 2^{a_{i+1}}),$$
which gives the first line in $\eqref{eq:DerivStar}$.

\item For $\zeta^{\star, \mathfrak{m}} (\cdots,2^{a_{i}}, \boldsymbol{\alpha}, 2^{a_{j}}, c_{j+1}, \cdots)$, the corresponding left sides are:
$$ -\left( \textcolor{magenta}{\delta_{c_{j}+2\leq \alpha \leq c_{i+1}}} + \textcolor{cyan}{\delta_{\alpha=c_{j}+1\atop c_{i+1}>c_{j}} } \right)  \zeta^{\star\star, \mathfrak{l}}_{c_{i+1}+ c_{j}-\alpha} (2^{a_{i+1}}, \ldots, 2^{a_{j-1}}) $$
$$ +\left( \textcolor{magenta}{ \delta_{c_{i+1}+2 \leq \alpha \leq c_{j}}} + \textcolor{cyan}{ \delta_{\alpha = c_{i+1}+1 \atop c_{j}> c_{i+1}}} \right)  \zeta^{\star\star, \mathfrak{l}}_{c_{i+1}+c_{j}-\alpha} (2^{a_{j-1}}, \ldots, 2^{a_{i+1}})  $$
$$+\delta_{3\leq \alpha < c_{i+1}}  \zeta^{\star\star, \mathfrak{l}}_{c_{i+1}-\alpha} (2^{a_{i+1}}, \ldots, c_{j}) -\delta_{3\leq \alpha <c_{j}}\zeta^{\star\star, \mathfrak{l}}_{c_{j}-\alpha} (2^{a_{j-1}}, \ldots,c_{i+1})$$
Using Antipode $\ast$ and turning some $\epsilon$ into $'1+0'$:
$$=+\delta_{3\leq \alpha < c_{i+1}}  \zeta^{\star\star, \mathfrak{l}}_{c_{i+1}-\alpha} (2^{a_{i+1}}, \ldots, c_{j}) -\delta_{3\leq \alpha < c_{j}}\zeta^{\star\star, \mathfrak{l}}_{c_{j}-\alpha} (2^{a_{j-1}}, \ldots,c_{i+1})$$
$$+  (-1)^{{c_{j}<c_{i+1}}} \delta_{\min(c_{j},c_{i+1}) < \alpha \leq \max(c_{j},c_{i+1})}  \zeta^{\star\star, \mathfrak{l}}_{c_{i+1}+ c_{j}-\alpha} (2^{a_{i+1}}, \ldots, 2^{a_{j-1}}) $$
$$= \delta_{3\leq \alpha < c_{i+1}} \zeta^{\star, \mathfrak{l}}_{c_{i+1}-\alpha} (c_{j}, \ldots, 2^{ a_{i+1}}) - \delta_{3\leq \alpha < c_{j} } \zeta^{\star, \mathfrak{l}}_{c_{j}-\alpha} (c_{i+1}, \ldots, 2^{ a_{j-1}})$$
This gives exactly the same expression than the first and fourth case for $\beta=a_{i}$ or $a_{j}$, and are integrated to them in $\eqref{eq:DerivStar}$.

\item For $\zeta^{\star, \mathfrak{m}} (\cdots, c_{i+1},\boldsymbol{2^{\beta}}, c_{j+1}, \cdots)$: \footnote{It includes the case $\alpha=2$.}
$$-\delta_{c_{i+1}>3 \atop 0 \leq \beta < a_{j}}\zeta^{\star\star, \mathfrak{l}}_{2} (2^{a_{j}-\beta-1}, \ldots, 2^{ a_{i+1}})  + \delta_{c_{j+1}>3 \atop 0 \leq \beta < a_{i+1}}  \zeta^{\star\star, \mathfrak{l}}_{2} (2^{a_{j}}, \ldots, 2^{ a_{i+1}-\beta -1}) $$
$$ +\delta_{\beta > a_{i+1} \atop c_{i+1}>3 }\zeta^{\star\star, \mathfrak{l}}_{c_{i+2}-2} (2^{a_{i+1}+a_{j}-\beta+1}, \ldots, 2^{ a_{i+2}})  - \delta_{\beta > a_{j}} \zeta^{\star\star, \mathfrak{l}}_{c_{j}-2} (2^{a_{i+1}+a_{j}-\beta+1}, \ldots, 2^{ a_{j-1}}) $$
$$- \delta_{c_{i+1}=1 \atop 1\leq \beta < a_{j}} \zeta^{\star\star, \mathfrak{l}} (2^{a_{j}-\beta}, \ldots, 2^{ a_{i+1}})  + \delta_{c_{j+1}=1 \atop 1 \leq \beta <a_{i+1}} \zeta^{\star\star, \mathfrak{l}} (2^{a_{i+1}-\beta}, \ldots, 2^{ a_{j}}) $$
$$+ \delta_{c_{i+2}=1 \atop \beta>a_{i+1}} \zeta^{\star\star, \mathfrak{l}}_{1} (2^{a_{j}+a_{i+1}-\beta}, \ldots, 2^{ a_{i+2}})  - \delta_{c_{j}=1 \atop \beta>a_{j}} \zeta^{\star\star, \mathfrak{l}}_{1} (2^{a_{i+1}+a_{j}-\beta}, \ldots, 2^{ a_{j-1}}) .$$
\item For $ \zeta^{\star, \mathfrak{m}} (\cdots, c_{i}, \boldsymbol{2^{\beta}, \alpha}, 2^{a_{j}}, \cdots)$, antisymmetric to 1:
$$ \left( - \textcolor{magenta}{\delta_{2\leq \alpha \leq c_{j}-1 \atop 0\leq \beta \leq a_{i}}} \zeta^{\star\star, \mathfrak{l}}_{c_{j}-\alpha} (2^{a_{j-1}}, \ldots, 2^{ a_{i}-\beta}) + \left(\textcolor{magenta}{\delta_{4\leq \alpha \leq c_{j}+1 \atop 0\leq \beta \leq a_{i}-1}} + \textcolor{cyan}{\delta_{\alpha=3 \atop 0\leq \beta \leq a_{i}-1}} \right) \zeta^{\star\star, \mathfrak{l}}_{c_{j}-\alpha+2} (2^{a_{j-1}}, \ldots, 2^{ a_{i}-\beta-1}) \right) $$
\end{enumerate}
This leads to the lemma, with the second case incorporated in the first and last line.
\end{proof}

\subsection{Euler $\sharp$ sums with $\boldsymbol{\overline{even}}, \boldsymbol{odd}$}

Let us consider the following family:
$$\zeta^{\sharp, \mathfrak{m}}\left( \lbrace\boldsymbol{\overline{even}}, \boldsymbol{odd}\rbrace^{\times} \right) , \text{i.e.  negative even and positive odd integers}$$
which, in terms of iterated integrals corresponds to, with $\epsilon\in \lbrace\pm \sharp\rbrace$:
\begin{equation}\label{eq:iisharp}
I^{ \mathfrak{m}} \left( 0; \left\lbrace \begin{array}{l}
1,  \boldsymbol{0}^{odd} ,-\sharp \\
1,  \boldsymbol{0}^{even} , \sharp
\end{array}\right\rbrace  ,  \cdots, \quad \left\lbrace 
\begin{array}{l}
\epsilon, \boldsymbol{0}^{odd}, -\epsilon \\
\epsilon, \boldsymbol{0}^{even}, \epsilon 
\end{array}\right\rbrace \quad, \cdots; 1 \right) .
\end{equation}

\begin{lemm}
The family $\zeta^{\sharp \mathfrak{m}}\left( \lbrace\overline{even}, odd\rbrace^{\times} \right) $ is stable under the coaction.
\end{lemm}
\begin{proof}
Looking at the possible kinds of cuts, and gathering them according the right side:
\includegraphics[]{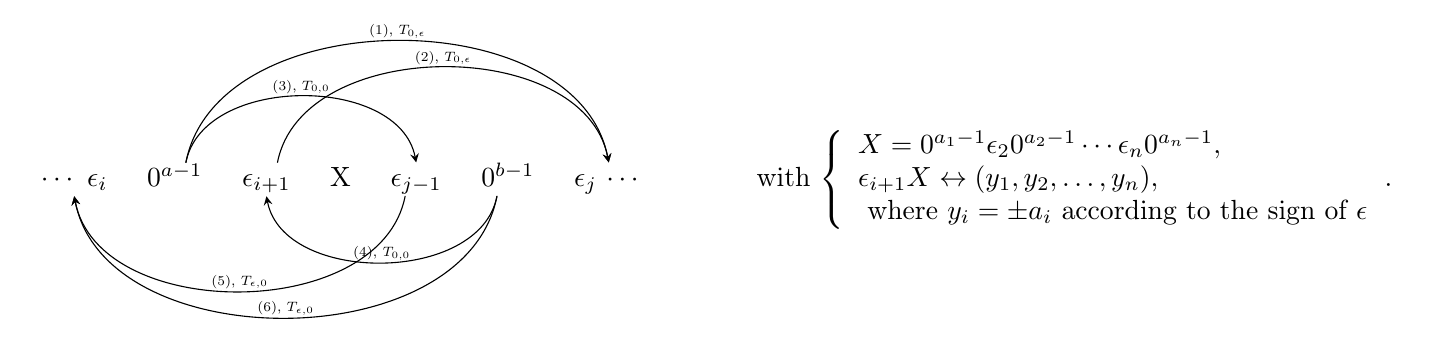}\\
These cuts have the same form for the right side in the coaction:
$$I^{\mathfrak{m}}(0; \cdots, \boldsymbol{\epsilon_{i} 0^{\alpha}  \epsilon_{j}}, \cdots ;1).$$
Notice there would be no term $T_{\epsilon, -\epsilon}$ in a cut from $\epsilon$ to $-\epsilon$ because of \textsc{Sign} $(\ref{eq:sign})$ identity, therefore you have there all the possible cuts pictured.\\
A priori, cuts can create in the right side a sequence $\left\lbrace  \epsilon, \boldsymbol{0}^{even}, -\epsilon \right\rbrace $ or $\left\lbrace \epsilon, \boldsymbol{0}^{odd}, \epsilon\right\rbrace $ inside the iterated integral; these cuts are the \textit{unstable} ones, since they are out of the considered family. However, by coupling these cuts two by two, and using the rules listed at the beginning of the Annexe, the unstable cuts would all get simplified.\\
Indeed, let examine each of the terms $(1-6)$\footnote{There is no remaining cuts between $\epsilon$ and $\epsilon$. Notice also that the left sides of the remaining terms have an even depth.}:\\
\\
\begin{tabular}{c | c | l | l}
Term & Left side & Unstable if & Simplified with \\
\hline
\multirow{4}{*}{$(1)$}   & $\zeta^{ \sharp\sharp,\mathfrak{l}}_{a-1-\alpha} (a_{1}, \ldots, a_{n},b)$, & $n$ even & the previous cut: \\
 &  with $\alpha<a$.  & & either $(6)$ by \textsc{Minus}   \\
 & & & or $(5)$ by \textsc{Cut}  \\
 & & &  or $(3)$ by \textsc{Cut} \footnotemark[1]\\ \hline
\multirow{4}{*}{$(2)$} & $\zeta^{ \sharp\sharp,\mathfrak{l}}_{b-1} (a_{n}, \ldots, a_{1})$ & $\epsilon_{i+1}=\epsilon_{j}$ & the previous cut: \\
 & with $\alpha=a$. & &   either with $(5)$ by \textsc{Shift} \\
  & & &  or with $(6)$ by \textsc{Cut Shifted}  \\
   & & & or with $(3)$ by \textsc{Shift}.\\ \hline
\multirow{4}{*}{$(3)$}  & $-\zeta^{ \sharp\sharp,\mathfrak{l}}_{a+b-\alpha-1} (a_{1}, \ldots, a_{n})$ &  $n$ odd & the following cut:\\ 
 & with $\alpha>b$. & &  either $(1)$ by \textsc{Cut} \\
  & & &  or $(2)$ by \textsc{Shift}  \\
   & & & or $(4)$ by \textsc{Shift}. \\ \hline
   \multirow{4}{*}{$(4)$} & $\zeta^{ \sharp\sharp,\mathfrak{l}}_{a+b-\alpha-1}(a_{n}, \ldots, a_{1})$ &  $n$ odd & the previous cut:\\
 & with $\alpha>a$ .& & either $(6)$ by \textsc{Cut Shifted}  \\
  & & &  or with $(5)$ by \textsc{Shift}  \\
   & & & or with $(3)$ by \textsc{Shift}. \\ \hline
   \multirow{4}{*}{$(5)$}& $-\zeta^{ \sharp\sharp,\mathfrak{l}}_{a-1} (a_{1}, \ldots, a_{n})$ & $\epsilon_{i}=\epsilon_{j-1}$ & the following cut:\\
 & with $\alpha=b$. & &  either  with $(1)$ by \textsc{Cut}, \\
  & & &   or with $(2)$ by \textsc{Shift},  \\
   & & & or with $(4)$ by \textsc{Shift}. \\ \hline
   \multirow{4}{*}{$(6)$} & $-\zeta^{ \sharp\sharp,\mathfrak{l}}_{b-1-\alpha} (a_{n}, \cdots a_{1}, a)$ & $n$ even & the following cut:\\
 & & &  either $(1)$ by \textsc{Minus} \\
  & & &   or with $(2)$ by \textsc{Cut Shifted}  \\
   & & & or with $(4)$ by \textsc{Cut Shifted} \\ \hline
\end{tabular} \\
\footnotetext[1]{It depends on the sign of $b+1-\alpha$ here for instance.}
\end{proof}
\noindent
\paragraph{Derivations.}
Let use the writing of the Conjecture $\ref{conjcoeff}$:
\begin{equation}\label{eq:essharpgather}
 \zeta^{\sharp,\mathfrak{m}}(B_{0}, 1^{\gamma_{1}}, \ldots, 1^{\gamma_{p}}, B_{p}) \text{ with } B_{i}<0 \text{ if and only if } B_{i} \text{ even }.
\end{equation}
\texttt{Nota Bene:} Beware, for instance $B_{i}$ may be equal to $1$, which implies that $\gamma_{i}= \gamma_{i+1}=0$. Indeed, we look at the indices corresponding to a sequence $(2^{a_{0}}, c_{1}, \ldots, c_{p}, 2^{a_{p}})$ as in the Conjecture $\ref{conjcoeff}$:
$$\begin{array}{l}
B_{i}= 2a_{i}+3 - \delta_{c_{i}}- \delta_{c_{i+1}}\\
B_{0}= 2a_{0}+1 - \delta_{c_{1}}\\
B_{p}= 2a_{p}+2 - \delta_{c_{p}}\\
\end{array}, \gamma_{i}\mathrel{\mathop:}= c_{i}-3 +2  \delta_{c_{i}}, \quad  \text{ where }   \left\lbrace  \begin{array}{l} a_{i} \geq 0 \\ c_{i}>0,c_{i}\neq 2 \\
\delta_{c}\mathrel{\mathop:}= \left\lbrace \begin{array}{ll}
1 & \text{ if } c=1\\
0 & \text{ else }.
\end{array}\right.
\end{array}.  \right.  $$
 
\begin{lemm}
\begin{equation}
D_{2r+1}\left( \zeta^{\sharp,\mathfrak{m}}(B_{0}, 1^{\gamma_{1}}, \ldots, 1^{\gamma_{p}}, B_{p})\right) =\footnotemark[2] 
\end{equation}
$$\delta_{r} \left[  -\delta_{{2 \leq B \leq B_{j}+1 \atop 0\leq\gamma\leq\gamma_{i+1}-1 }}\zeta^{\sharp,\mathfrak{l}}(B_{j}-B+1, 1^{\gamma_{j}}, \ldots, 1^{\gamma_{i+1}-\gamma-1})\otimes\zeta^{\sharp,\mathfrak{m}}(B_{0} \cdots, B_{i}, \textcolor{magenta}{1^{\gamma}, B}, 1^{\gamma_{j+1}}, \ldots, B_{p}) \right. $$

$$\left[  
\begin{array}{l}
+ \delta_{B_{i+1}< B}\zeta^{\sharp\sharp,\mathfrak{l}}_{B_{i+1}+B_{j}-B}(1^{\gamma_{j}}, \ldots, 1^{\gamma_{i+2}})\\
- \delta_{B_{j}< B}\zeta^{\sharp\sharp,\mathfrak{l}}_{B_{i+1}+B_{j}-B}(1^{\gamma_{i+2}}, \ldots, 1^{\gamma_{j}})\\
 + \zeta^{\sharp\sharp,\mathfrak{l}}_{B_{i+1}-B}(1^{\gamma_{i+2}}, \ldots, B_{j}) - \zeta^{\sharp\sharp,\mathfrak{l}}_{B_{j}-B}(1^{\gamma_{j}}, \ldots, B_{i+1})
\end{array} \right] \otimes\zeta^{\sharp,\mathfrak{m}}(B_{0} \cdots, B_{i}, 1^{\gamma_{i+1}}, \textcolor{green}{B}, 1^{\gamma_{j+1}}, \ldots, B_{p}) $$

$$\left.  \delta_{{1 \leq B \leq B_{i+1}+1 \atop 0\leq\gamma\leq\gamma_{j+1}-1}}\zeta^{\sharp,\mathfrak{l}}(B_{i+1}-B+1, 1^{\gamma_{i+2}}, \ldots, 1^{\gamma_{j+1}-\gamma-1})\otimes\zeta^{\sharp,\mathfrak{m}}(B_{0} \cdots, 1^{\gamma_{i+1}},\textcolor{cyan}{ B, 1^{\gamma}}, B_{j+1}, \ldots, B_{p}) \right],$$
where $B$ is positive if odd, negative if even.
\end{lemm}
\begin{proof}
\footnotetext[2]{Here $\delta_{r}$ indicates that left side has to be of weight $2r+1$.}
\texttt{Nota Bene}: For the left side, we only look at odd weight $w$, and the parity of the depth $d$ is fundamental since the relations stated above depend on the parity of $w-d$. For instance, for such a sequence $(\boldsymbol{1}^{\gamma_{i}},B_{i}, \ldots,B_{j-1},\boldsymbol{1}^{\gamma_{j}})$ (with the previous notations), $weight- depth$ has the same parity than $\delta_{c_{i}}+\delta_{c_{j}}$. \\
\\
The following cuts get simplified, with \textsc{Shift}, since depth is odd ($B_{i}$ odd if $c_{i},c_{i+1}\neq 1$):\\
\includegraphics[]{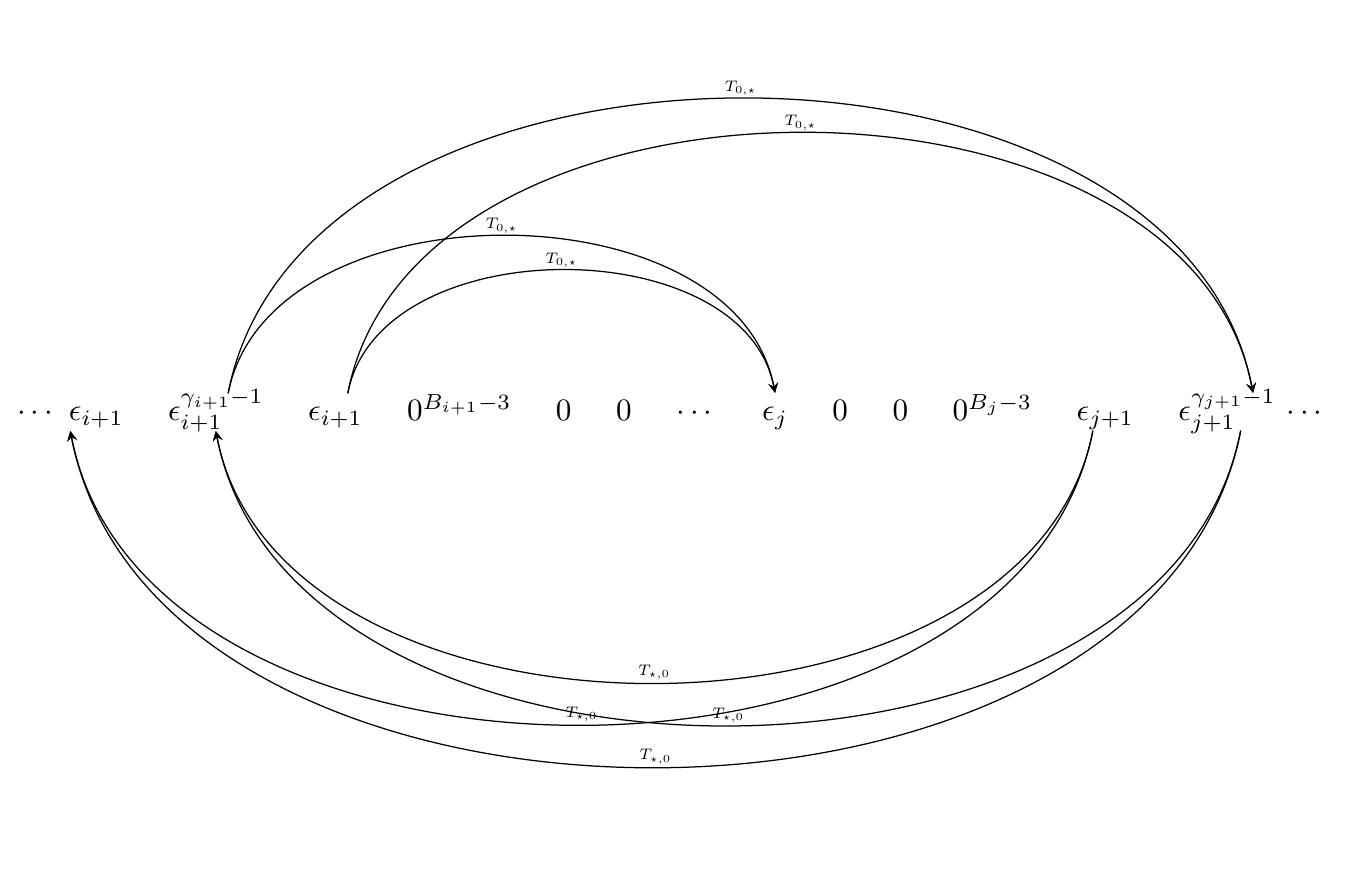}\\ 
It remains, where all the unstable cuts are simplified by the Lemma $A.1.4$, cuts that we can gather into four groups, according to the right side of the coaction:
\begin{itemize}
\item[$(i)$]  $\zeta^{\sharp,\mathfrak{m}}(B_{0} \cdots, B_{i}, \textcolor{magenta}{1^{\gamma}, B}, 1^{\gamma_{j+1}}, \ldots, B_{p}) $.
\item[$(ii)$]  $\zeta^{\sharp,\mathfrak{m}}(B_{0} \cdots, B_{i}, \textcolor{yellow}{1^{\gamma}}, B_{j}, \ldots, B_{p}) $.
\item[$(iii)$]  $\zeta^{\sharp,\mathfrak{m}}(B_{0} \cdots, B_{i}, 1^{\gamma_{i+1}}, \textcolor{green}{B}, 1^{\gamma_{j+1}}, \ldots, B_{p}) $.
\item[$(iv)$]  $\zeta^{\sharp,\mathfrak{m}}(B_{0} \cdots, 1^{\gamma_{i+1}},\textcolor{cyan}{ B, 1^{\gamma}}, B_{j+1}, \ldots, B_{p}) $.
\end{itemize}
It remains, where $(iv)$ terms, antisymmetric of $(i)$ ones, are omitted to lighten the diagrams: \\
\includegraphics[]{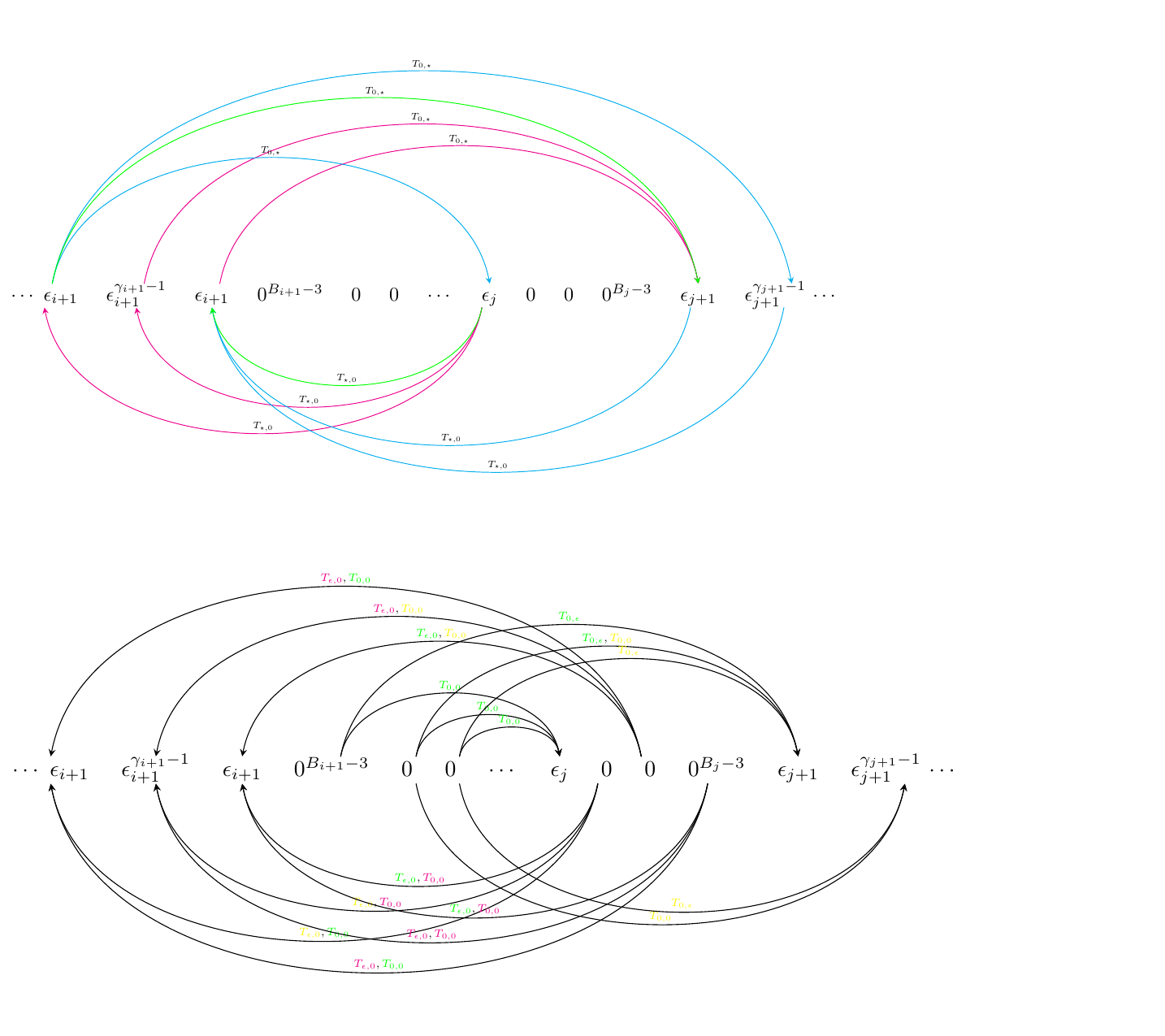}\\
Now, let list these remaining terms, gathered according to their right side as above:
\begin{itemize}
\item[$(i)$] Looking at the magenta terms, with $2 \leq B \leq B_{j}-1$ or $B=B_{j}+1$ and $0\leq\gamma\leq\gamma_{i+1}-1$:
$$\zeta^{\sharp\sharp,\mathfrak{l}}_{B_{j}-B+1}(1^{\gamma_{j}}, \ldots, 1^{\gamma_{i+1}-\gamma-1})-\zeta^{\sharp\sharp,\mathfrak{l}}_{B_{j}-B}(1^{\gamma_{j}}, \ldots, 1^{\gamma_{i+1}-\gamma}) =-\zeta^{\sharp,\mathfrak{l}}(B_{j}-B+1, 1^{\gamma_{j}}, \ldots, 1^{\gamma_{i+1}-\gamma-1})$$
With $even$ depth for the first term and $odd$ for the second since otherwise the cuts would be unstable and simplified by $\textsc{Cut}$; here also $c_{i+1}\neq 1$. 
\item[$(ii)$]  These match exactly with the left side of $(i)$ for $B=B_{j}$ and $(iv)$ terms for $B=B_{i}$.
\item[$(iii)$]  The following cuts:
$$\delta_{B_{i+1}\geq B}\zeta^{\sharp\sharp,\mathfrak{l}}_{B_{i+1}-B}(1^{\gamma_{i+2}}, \ldots, B_{j}) -\delta_{B_{j}\geq B} \zeta^{\sharp\sharp,\mathfrak{l}}_{B_{j}-B}(1^{\gamma_{j}}, \ldots, B_{i+1})+$$ 
$$\delta_{B_{i+1}< B}\zeta^{\sharp\sharp,\mathfrak{l}}_{B_{i+1}+B_{j}-B}(1^{\gamma_{j}}, \ldots, 1^{\gamma_{i+2}}) - \delta_{B_{j}< B}\zeta^{\sharp\sharp,\mathfrak{l}}_{B_{i+1}+B_{j}-B}(1^{\gamma_{i+2}}, \ldots, 1^{\gamma_{j}}) .$$
The parity of $weight-depth$ for the first line is equal to the parity of $\delta_{c_{i+1}}+ \delta_{c_{j+1}}+B$. Notice that if this is even, the first line has odd depth whereas the second line has even depth, and by $\textsc{Cut}$ and $\textsc{Antipode} \ast$, all terms got simplified. Hence, we can restrict to $B$ written as $2\beta+3- \delta_{c_{i+1}}- \delta_{c_{j+1}}$, the first line being of even depth, the second line of odd depth.

\item[$(iv)$]  Antisymmetric of $(i)$.
\end{itemize}
\end{proof}

\section{Galois descent in small depths, $N=2,3,4,\mlq 6 \mrq,8$}

\subsection{$\boldsymbol{N=2}$: Depth $\boldsymbol{2,3}$}
Here we have to consider only one Galois descent, from $\mathcal{H}^{2}$ to $\mathcal{H}^{1}$. \\
In depth $1$ all the $\zeta^{\mathfrak{m}}(\overline{s})$, $s>1$ are MMZV. Let us detail the case of depth 2 and 3 as an application of the results of Chapter $5$. In depth 2, coefficients are explicit:
\begin{lemm} 
The depth $2$ part of the basis of the motivic multiple zeta values is:
$$\left\{ \zeta^{\mathfrak{m}}(2a+1, \overline{2b+1})- \binom {2(a+b)}{2b}  \zeta^{\mathfrak{m}}(1,\overline{2(a+b)+1}), a,b> 0 \right\}.$$
\end{lemm}
\begin{proof}
Indeed, we have if $a,b>0$, $D_{1}(\zeta^{\mathfrak{m}}(2a+1, \overline{2b+1}))=0$ and for $r>0$:
\begin{multline}\nonumber
D_{2r+1,2}(\zeta^{\mathfrak{m}}(2a+1, \overline{2b+1}))= \zeta^{\mathfrak{l}}(2r+1)\otimes \zeta^{\mathfrak{m}}(\overline{2(a+b-r)+1})\\
\left(-\delta_{a \leq r < a+b} \binom{2r}{2a} + \delta_{r=a} + \delta_{b\leq r < a+b} \binom{2r}{2b}(2^{-2r}-1)+ \delta_{r=a+b}(2^{-2r}-2)\binom{2(a+b)}{2b}\right).
\end{multline}
There is only the case $r=a+b$ where a term ($\zeta^{\mathfrak{m}}(\overline{1}))$ which does not belong to $\mathcal{F}_{0}\mathcal{H}$ appears:
$$D_{2r+1,2}(\zeta^{\mathfrak{m}}(2a+1, \overline{2b+1}))\equiv \delta_{r=a+b}(2^{-2r}-2)\binom{2(a+b)}{2b}  \zeta^{\mathfrak{l}}(2r+1)\otimes \zeta^{\mathfrak{m}}(\overline{1}) \text{ in the quotient } \mathcal{H}^{\geq 1}.$$
Referring to the previous results, we can correct $\zeta^{\mathfrak{m}}(2a+1, \overline{2b+1})$ with terms of the same weight, same depth, and with at least one $1$ (not at the end), which here corresponds only to $\zeta^{\mathfrak{m}}(1,\overline{2(a+b)+1})$.\\
Furthermore, the last equality being true in the quotient $\mathcal{H}^{\geq 1}$:
\begin{align*}
D_{2r+1,2}(\zeta^{\mathfrak{m}}(1, \overline{2(a+b)+1})) & = \zeta^{\mathfrak{l}}(2r+1)\otimes (-\delta_{r < a+b}+ \delta_{r=a+b}(2^{-2r}-2)) \zeta^{\mathfrak{m}}(\overline{2(a+b-r)+1})\\
& \equiv \delta_{r=a+b}(2^{-2r}-2) \zeta^{\mathfrak{l}}(2r+1)\otimes \zeta^{\mathfrak{m}}(\overline{1}) .
\end{align*}
According to these calculations of infinitesimal coactions:
$$\zeta^{\mathfrak{m}}(2a+1, \overline{2b+1})- \binom {2(a+b)}{2b}  \zeta^{\mathfrak{m}}(1,\overline{2(a+b)+1}) \text{ belongs to } \mathcal{F}_{0}\mathcal{H} \text{ , i.e. is a MMZV.}$$
\end{proof}
\noindent
\texttt{Examples:} Here are some motivic multiple zeta values:
$$\zeta^{\mathfrak{m}}(3, \overline{3})-6 \zeta^{\mathfrak{m}}(1,\overline{5}) \text{ , } \zeta^{\mathfrak{m}}(3, \overline{5})-15 \zeta^{\mathfrak{m}}(1,\overline{7}) \text{ , }\zeta^{\mathfrak{m}}(5, \overline{3})-15 \zeta^{\mathfrak{m}}(1,\overline{7}) \text{ ,  } \zeta^{\mathfrak{m}}(5, \overline{7})-210 \zeta^{\mathfrak{m}}(1,\overline{11}) .$$
\\
\textsc{Remarks:}
\begin{itemize}
	\item[$\cdot$] The corresponding Euler sums $\left\{ \zeta(2a+1, \overline{2b+1})- \binom {2(a+b)}{2b}  \zeta(1,\overline{2(a+b)+1}), a,b> 0 \right\}$ are a generating family of MZV in depth $2$.
	\item[$\cdot$] Similarly, we can prove that the following elements are (resp. motivic) MZV, if no $\overline{1}$:
	$$\begin{array}{l|ll}
	\text{ } \zeta(\overline{A}, \overline{B}) & \zeta(A,\overline{B}) +\zeta(\overline{A},B) & \text{ if } A,B \text{ odd } \\
	\left. \begin{array}{l}
	 \zeta(A, \overline{B})\\
	 \zeta(\overline{A}, B)
	\end{array} \right\rbrace \text{ if }  A+B \text{ odd } & \zeta(A, \overline{B}) + (-1)^{A} \binom{A+B-2}{A-1} \zeta(1,\overline{A+B-1}) & \text{ if } A+B \text{ even} \\
\begin{array}{l}
\zeta(\overline{1},\overline{1}) -\frac{1}{2}\zeta(\overline{1})^{2} \\
 \zeta(1,\overline{1}) -\frac{1}{2}\zeta(\overline{1})^{2}
\end{array}  & 	\zeta(\overline{A}, B)- (-1)^{A} \binom{A+B-2}{A-1} \zeta(1,\overline{A+B-1}) & \text{ if } \left\lbrace \begin{array}{l}
 A+B \text{ even } \\
 A,B\neq 1  
\end{array}\right. 
	\end{array}$$
\end{itemize}

\begin{lemm} 
The depth $2$ part of the basis of $\mathcal{F}_{1}\mathcal{H}$ is:
$$\left\{ \zeta^{\mathfrak{m}}(2a+1, \overline{2b+1}) , (a,b)\neq(0, 0) \right\}.$$
\end{lemm}
\begin{proof}
No need of correction ($\mathcal{B}_{n,2,\geq 2}$ is empty for $n\neq 2$), these elements belong to $\mathcal{F}_{1}\mathcal{H}$.
\end{proof}

\begin{lemm} The depth $3$ part of the basis of motivic multiple zeta values is:
\begin{multline}
 \left\{\zeta^{\mathfrak{m}}(2a+1,2b+1,\overline{2c+1})-\sum_{k=1}^{a+b+c}\alpha_{k}^{a,b,c}\zeta^{\mathfrak{m}}(1,2(a+b+c-k)+1, \overline{2k+1}) \right. \\ 
\left. -\binom {2(b+c)}{2c}\zeta^{\mathfrak{m}}(2a+1,1,\overline{2(b+c)+1}),a,b,c>0 \right\}.
 \end{multline}
 where $\alpha_{k}^{a,b,c} \in\mathbb{Z}_{\text{odd}}$ are solutions of $M_{3}X=A^{a,b,c}$. With 
$A^{a,b,c}$ such that $r^{\text{th}}-$coefficient is:
$$\delta_{b \leq r < a+b} \binom{2(n-r)}{2c}\binom{2r}{2b} - \delta_{a < r< a+b} \binom{2(n-r)}{2c}\binom{2r}{2a}-\delta_{b\leq r < b+c}\binom{2(n-r)}{2a}\binom{2r}{2b} $$
$$- \delta_{r\leq a}\binom{2(n-r)}{2(b+c)}\binom{2(b+c)}{2c} + \delta_{r<b+c}\binom{2(n-r)}{2a}\binom{2(b+c)}{2c} + \delta_{c\leq r < b+c}\binom{2r}{2c}\binom{2(n-r)}{2a}(2^{-2r}-1).$$ 
$M_{3}$ the matrix whose $(r,k)^{\text{th}}$ coefficient is:
$$\delta_{r=a+b+c}(2^{-2r}-2)\binom{2n}{2k}+ \delta_{k \leq r < n} \binom{2r}{2k}(2^{-2r}-1) - \delta_{r<n-k} \binom{2(n-r)}{2k} - \delta_{n-k \leq r<n} \binom{2r}{2(n-k)}. $$
\end{lemm}
\begin{proof}
Let $\zeta^{\mathfrak{m}}(2a+1,2b+1,\overline{2c+1})$, $a,b,c >0$ fixed, and substract elements of the same weight, of depth 3 until it belongs to $gr_{3} \mathcal{F}_{0}\mathcal{H}$.\\
Let calculate infinitesimal coproducts referring to the formula ($\ref{Deriv2}$) in the quotient $\mathcal{H}^{\geq 1}$ and use previous results for depth 2, with $n=a+b+c$:
\begin{small}
$$D_{2r+1,3}(\zeta^{\mathfrak{m}}(2a+1,2b+1,\overline{2c+1}))\equiv \zeta^{\mathfrak{l}}(2r+1)\otimes  \left[ \delta_{r=b+c} \binom{2(b+c)}{2c}(2^{-2r}-2) \zeta^{\mathfrak{m}}(2a+1, \overline{1}) \right.$$ 
$$\left. + \zeta^{\mathfrak{m}}(1, \overline{2(n-r)+1}) \left( \delta_{a = r} \binom{2(n-r)}{2c}+  \delta_{b \leq r < a+b} \binom{2r}{2b}\binom{2(n-r)}{2c} - \delta_{a\leq r<a+b} \binom{2r}{2a}\binom{2(n-r)}{2c} \right. \right.$$
$$\left. \left. - \delta_{b\leq r<b+c} \binom{2r}{2b}\binom{2(n-r)}{2a}  +\delta_{c \leq r <b+c} \binom{2r}{2c} \binom{2(n-r)}{2a}(2^{-2r}-1)\right) \right].$$
\end{small}
At first, let substract $\binom {2(b+c)}{2c}\zeta(2a+1,1,\overline{2(b+c)+1})$ such that the $D^{-1}_{1,2}  D^{1}_{2r+1,3} $ are equal to zero, which comes to eliminate the term $\zeta^{\mathfrak{m}}(2a+1, \overline{1})$ appearing (case $r=b+c$).\\
So, we are left to substract a linear combination 
$$\sum_{k=1}^{a+b+c} \alpha_{k}^{a,b,c}  \zeta^{\mathfrak{m}}(1,2(a+b+c-k)+1, \overline{2k+1})$$
 such that the coefficients $\alpha_{k}^{a,b,c}$ are solutions of the system $M_{3}X=A^{a,b,c}$ where $A^{a,b,c}= (A^{a,b,c}_{r})_{r}$ satisfying in $\mathcal{H}^{\geq 1}$:
\begin{small}
\begin{multline}\nonumber
D_{2r+1,3} \left( \zeta^{\mathfrak{m}}(2a+1,2b+1,\overline{2c+1})- \binom {2(b+c)}{2c}\zeta(2a+1,1,\overline{2(b+c)+1})\right) \equiv\\
A^{a,b,c}_{r}\zeta^{\mathfrak{l}}(2r+1)\otimes \zeta^{\mathfrak{m}}(1, \overline{2(n-r)+1}),
\end{multline}
\end{small}
and $M_{3}= (m_{r,k})_{r,k}$ matrix such that: 
$$D_{2r+1,3}( \zeta^{\mathfrak{m}}(1,2(a+b+c-k)+1, \overline{2k+1}))= m_{r, k}   \zeta^{\mathfrak{l}}(2r+1)\otimes \zeta^{\mathfrak{m}}(1, \overline{2(n-r)+1}).$$
This system has solutions since, according to Chapter $5$ results, the matrix $M_{3}$ is invertible.\footnote{Indeed, modulo 2, $M_{3}$ is an upper triangular matrix with $1$ on diagonal.}\\
Then, the following linear combination will be in $\mathcal{F}_{0}\mathcal{H}$:
\begin{small}
$$\zeta^{\mathfrak{m}}(2a+1,2b+1,\overline{2c+1})-\sum_{k=1}^{a+b+c} \alpha_{k}^{a,b,c}  \zeta^{\mathfrak{m}}(1,2(a+b+c-k)+1, \overline{2k+1})-\binom {2(b+c)}{2c}\zeta(2a+1,1,\overline{2(b+c)+1}).$$ 
\end{small}
The coefficients $\alpha_{k}^{a,b,c}$ belong to $\mathbb{Z}_{\text{odd}}$ since coefficients are integers, and $\det(M_{3})$ is odd. Referring to the calculus of infinitesimal coactions, $A^{a,b,c}$ and $M_{3}$ are as claimed in lemma.
\end{proof}
\noindent
\texttt{Examples:}
\begin{itemize}
\item[$\cdot$] By applying this lemma, with $a=b=c=1$ we obtain the following MMZV:
$$\zeta^{\mathfrak{m}}(3,3,\overline{3})+ \frac{774}{191} \zeta^{\mathfrak{m}}(1,5, \overline{3})  - \frac{804}{191} \zeta^{\mathfrak{m}}(1,3, \overline{5})  + \frac{450}{191}\zeta^{\mathfrak{m}}(1,1, \overline{7})  -6 \zeta^{\mathfrak{m}}(3,1,\overline{5}).$$
Indeed, in this case, with the previous notations:
$$M_{3}=\begin{pmatrix}
\frac{27}{4}&-1&-1 \\
-\frac{53}{8}&-\frac{111}{16}&-1\\
-\frac{1905}{64}&-\frac{1905}{64}&-\frac{127}{64}
\end{pmatrix} \text{ , } \quad A^{1,1,1}=\begin{pmatrix}
\frac{51}{2} \\
0\\
0
\end{pmatrix} .$$
\item[$\cdot$] Similarly, we obtain the following motivic multiple zeta value:
$$\hspace*{-1.2cm}\zeta^{\mathfrak{m}}(3,3,\overline{5})+ \frac{850920}{203117}\zeta^{\mathfrak{m}}(1,7, \overline{3}) +\frac{838338}{203117}\zeta^{\mathfrak{m}}(1,5, \overline{5}) -\frac{3673590}{203117}\zeta^{\mathfrak{m}}(1,3, \overline{7})+ \frac{20351100}{203117} \zeta^{\mathfrak{m}}(1,1, \overline{9}) -15\zeta^{\mathfrak{m}}(3,1,\overline{7}).$$
$$\hspace*{-2cm}\text{ There: } \quad \quad M_{3}=\begin{pmatrix}
-\frac{63}{4}& 15 & -1& -1\\
-\frac{93}{8}&-\frac{31}{16}&-6&-1 \\
-\frac{1009}{64}&-\frac{1905}{64}&-\frac{1023}{64}&-1\\
-\frac{3577}{64}&-\frac{17885}{128}&-\frac{3577}{64}&-\frac{511}{256}
\end{pmatrix} \text{ , } \quad \quad A^{1,1,2}=\begin{pmatrix}
210\\
\frac{387}{8} \\
0\\
0
\end{pmatrix} .$$
\end{itemize}

\begin{lemm} 
The depth $3$ part of the basis of $\mathcal{F}_{1}\mathcal{H}$ is:
$$\left\{\zeta^{\mathfrak{m}}(2a+1,2b+1,\overline{2c+1})-\delta_{a=0 \atop\text{ or } c=0} (-1)^{\delta_{c=0}} \binom{2(a+b+c)}{2b} \zeta^{\mathfrak{m}}(1,1, \overline{2(a+b+c)+1}) \right.$$
$$ \left. - \delta_{c=0} \binom{2(a+b)}{2b} \zeta(1,2(a+b)+1,\overline{1}), \text{ at most one of } a,b,c \text{ equals zero }\right\}.$$
\end{lemm}
\begin{proof} Let $\zeta^{\mathfrak{m}}(2a+1,2b+1,\overline{2c+1})$ with at most one $1$.\\
\begin{center}
 \emph{Our goal is to annihilate $D^{-1}_{1,3}$ and $\lbrace D^{-1}_{1,3} \circ D^{1}_{2r+1}\rbrace_{r>0}$, in the quotient $\mathcal{H}^{\geq 1}$.}
 \end{center}
Let first cancel $D^{-1}_{1,3}$: if $c\neq 0$, it is already zero; otherwise, for $c=0$, in $\mathcal{H}^{\geq 1}$, according to the results in depth $2$ for $\mathcal{F}_{0}$, we can substract $\binom{2(a+b)}{2a} \zeta(1,2(a+b)+1,\overline{1})$ since:
\begin{small}
$$D_{1,3} (\zeta^{\mathfrak{m}}(2a+1,2b+1,\overline{1}))\equiv \binom{2(a+b)}{2a}\zeta^{\mathfrak{m}}(1,\overline{2(a+b)+1})\equiv \binom{2(a+b)}{2a} D_{1,3} (\zeta^{\mathfrak{m}}(1,2(a+b)+1,\overline{1})).$$
\end{small}
Furthermore, with $\equiv$ standing for an equality in $\mathcal{H}^{\geq 1}$:
\begin{small}
\begin{align*}
D^{-1}_{1,2} D^{1}_{2r+1,3} (\zeta^{\mathfrak{m}}(2a+1,2b+1,\overline{2c+1})) & = \delta_{r=b+c} \binom{2r}{2c}(2^{-2r}-2)\zeta^{\mathfrak{m}}(\overline{2a+1}) \\
 & \equiv \delta_{r=b+c \atop a=0} \binom{2(b+c)}{2c}(2^{-2(b+c)}-2)\zeta^{\mathfrak{m}}(\overline{1}). \\
 D^{-1}_{1,2} D^{1}_{2r+1,3} (\zeta^{\mathfrak{m}}(1,1,\overline{2(a+b+c)+1})) & = \delta_{r=a+b+c} (2^{-2(a+b+c)}-2)\zeta^{\mathfrak{m}}(\overline{1}).\\
 D^{-1}_{1,2} D^{1}_{2r+1,3} (\zeta^{\mathfrak{m}}(1,2(a+b+c)+1, \overline{1})) & = \delta_{r=a+b+c} (2^{-2(a+b+c)}-2)\zeta^{\mathfrak{m}}(\overline{1}).
\end{align*}
\end{small}
Therefore, to cancel $D^{-1}_{1,2}\circ D^{1}_{2r+1,3}$:
\begin{itemize}
\item[$\cdot$] If $a=0$ we substract $\binom{2(b+c)}{2c}\zeta^{\mathfrak{m}}(1,1,\overline{2(b+c)+1}) $.
\item[$\cdot$] If $c=0$, we add $\binom{2(b+c)}{2c}\zeta^{\mathfrak{m}}(1,1,\overline{2(a+b)+1})$.
\end{itemize}
\end{proof}

\paragraph{\textsc{Depth }$\boldsymbol{4}$.} The simplest example in depth $4$ of MMZV obtained by this way, with $\alpha_{i}\in\mathbb{Q}$:
$$-\zeta^{\mathfrak{m}}(3, 3, 3, \overline{3})-\frac{3678667587000}{4605143289541}\zeta^{\mathfrak{m}}(1, 1, 1, \overline{9})+\frac{9187768536750}{4605143289541}\zeta^{\mathfrak{m}}(1, 1, 3, \overline{7})+\frac{41712466500}{4605143289541}\zeta^{\mathfrak{m}}(1, 1, 5, \overline{5})$$
$$-\frac{9160668717750}{4605143289541} \zeta^{\mathfrak{m}}(1, 1, 7, \overline{3})+\frac{11861255103300}{4605143289541}\zeta^{\mathfrak{m}}(1, 3, 1, \overline{7})+\frac{202283196216}{4605143289541}\zeta^{\mathfrak{m}}(1, 3, 3, \overline{5})$$
$$-\frac{993033536436}{4605143289541}\zeta^{\mathfrak{m}}(1, 3, 5, \overline{3})+\frac{8928106562124}{4605143289541}\zeta^{\mathfrak{m}}(1, 5, 1, \overline{5})-\frac{1488017760354}{4605143289541}\zeta^{\mathfrak{m}}(1, 5, 3, \overline{3})$$
$$-\frac{450}{191}\zeta^{\mathfrak{m}}(3, 1, 1, \overline{7})+\frac{804}{191}\zeta^{\mathfrak{m}}(3, 1, 3, \overline{5})-\frac{774}{191}\zeta^{\mathfrak{m}}(3, 1, 5, \overline{3})+6\zeta^{\mathfrak{m}}(3, 3, 1, \overline{5})$$ 
$$+ \alpha_{1} \zeta^{\mathfrak{m}}(1,-11)+ \alpha_{2} \zeta^{\mathfrak{m}}(1,-9)\zeta^{\mathfrak{m}}(2)+ \alpha_{3} \zeta^{\mathfrak{m}}(1,-7)\zeta^{\mathfrak{m}}(2)^{2}+ \alpha_{4} \zeta^{\mathfrak{m}}(1,-5)\zeta^{\mathfrak{m}}(2)^{3}+ \alpha_{5} \zeta^{\mathfrak{m}}(1,-3)\zeta^{\mathfrak{m}}(2)^{4}.$$
$$\quad $$

\subsection{ $\boldsymbol{N=3,4}$: Depth $\boldsymbol{2}$}

Let us detail the case of depth 2 as an application of the results in Chapter $5$ and start by defining some coefficients appearing in the next examples:
\begin{defi}
Set $\alpha^{a,b}_{k}\in\mathbb{Z}$ such that $M(\alpha^{a,b}_{k})_{b+1 \leq k \leq \frac{n}{2}-1 }= A^{a,b}$ with $n=2(a+b+1)$:
$$\hspace*{-0.5cm}M\mathrel{\mathop:}= \left( \binom{2r-1}{2k-1} \right)_{b+1 \leq r,k \leq \frac{n}{2}-1}; A^{a,b}\mathrel{\mathop:}=\left(-\binom{2r-1}{2b}\right)_{b+1 \leq r \leq \frac{n}{2}-1}; \beta^{a,b}\mathrel{\mathop:}= \binom{n-2}{2b} + \sum_{k=b+1}^{a+b} \alpha_{k} \binom{n-2}{2k-1}.$$
\end{defi}
\noindent
\texttt{Nota Bene}: The matrix $M$ having integers as entries and determinant equal to $1$, and $A$ having integer components, the coefficients $\alpha^{a,b}_{k}$ are obviously integers; the matrix $M$ and its inverse are lower triangular with $1$ on the diagonal. Furthermore:\footnote{There $c_{i}\in\mathbb{N} $ does not depend neither on $b$ nor on $a$.}:
\begin{multline}\nonumber
\alpha^{a,b}_{b+i}= (-1)^{i} \binom{2b+2i-1}{2i-1} c_{i},\\
\alpha^{a,b}_{b+1}=-(2b+1), \quad \alpha^{a,b}_{b+2}=2\binom{2b+3}{3}, \quad \alpha^{a,b}_{b+3}=-16\binom{2b+5}{5}, \quad  \alpha^{a,b}_{b+4}=272\binom{2b+7}{7}.\\
\end{multline}

\begin{lemm} 
The depth $2$ part of the basis of MMZV, for even weight $n=2(a+b+1)$, is:
\begin{small}
$$\left\{ \zeta^{\mathfrak{m}}\left( 2a+1, 2b+1 \atop 1, \xi \right)- \beta^{a,b} \zeta^{\mathfrak{m}}\left(1,n-1 \atop 1, \xi \right) - \sum_{k=b+1}^{\frac{n}{2}-1} \alpha^{a,b}_{k} \zeta^{\mathfrak{m}}\left( n-2k, 2k \atop 1, \xi \right), a,b> 0 \right\}.$$
\end{small}
\end{lemm}
\begin{proof}\footnote{We omit the exponent $\xi$ indicating the projection on the second factor of the derivations $D_{r}$, to lighten the notations.} Let $Z=\zeta^{\mathfrak{m}}(2a+1, \overline{2b+1})$ fixed, with $a,b>0$.\\
First we substract a linear combination of $\zeta^{\mathfrak{m}}\left(n-2k, 2k \atop 1, \xi \right)$ in order to cancel $\lbrace D_{2r}\rbrace$. It is possible since in depth 2, because $\zeta^{\mathfrak{l}}\left( 2r \atop  1\right) =0$:
\begin{small}
$$ D_{2r} (\zeta^{\mathfrak{m}}(x_{1}, \overline{x_{2}}))= \delta_{x_{2} \leq 2r \leq x_{1}+x_{2}-1} (-1)^{x_{2}} \binom{2r-1}{x_{2}-1} \zeta^{\mathfrak{l}}\left( 2 r \atop  \xi\right)\otimes \zeta^{\mathfrak{m}}\left( x_{1}+x_{2}-r\atop  \xi \right).$$
\end{small}
Hence it is sufficient to choose $\alpha_{k}$ such that $M\alpha^{a,b}= A^{a,b}$ as in Definition $A.2.5$.\\
Now, it remains to satisfy $D_{1}\circ D_{2r+1}(\cdot)=0$ (for $r=n-1$ only) in order to have an element of $\mathcal{F}^{k_{N}/\mathbb{Q},P/1}_{0}\mathcal{H}_{n}$. In that purpose, let substract $\beta^{a,b} \zeta^{\mathfrak{m}}(1,n-1 ;  1, \xi)$ with $\beta^{a,b}$ as in Definition $A.2.5$) according to the calculation of $D_{1}\circ D_{2r+1}(\cdot)$, left to the reader.\\
\end{proof}
\noindent
\texttt{Examples}: 
The following are motivic multiple zeta values:
\begin{itemize}
\item[$\cdot$] $\zeta^{\mathfrak{m}}\left(5,3 \atop 1, \xi \right)  -75 \zeta^{\mathfrak{m}}\left( 1,7 \atop 1, \xi \right) + 3 \zeta^{\mathfrak{m}}\left(4, 4 \atop 1, \xi \right) - 20 \zeta^{\mathfrak{m}}\left( 2, 6 \atop 1, \xi \right).$
\item[$\cdot$] $\zeta^{\mathfrak{m}}\left(3,5 \atop 1, \xi \right) +15 \zeta^{\mathfrak{m}}\left( 1, 7 \atop 1, \xi \right) + 5 \zeta^{\mathfrak{m}}\left(6, 2 \atop 1, \xi \right) .$
\item[$\cdot$] $\zeta^{\mathfrak{m}}\left(5,5 \atop 1, \xi \right) -350 \zeta^{\mathfrak{m}}\left( 1, 9 \atop 1, \xi \right) + 5 \zeta^{\mathfrak{m}}\left( 4, 6 \atop 1, \xi \right) -70 \zeta^{\mathfrak{m}}\left( 2, 8 \atop 1, \xi \right).$
\item[$\cdot$] $\zeta^{\mathfrak{m}}\left( 7, 5 \atop 1, \xi \right) +12810 \zeta^{\mathfrak{m}}\left( 1, 11 \atop 1, \xi \right) + 5 \zeta^{\mathfrak{m}}\left( 6, 6 \atop 1, \xi \right) -70 \zeta^{\mathfrak{m}}\left( 4,8 \atop 1, \xi \right)+ 2016 \zeta^{\mathfrak{m}}\left( 2, 10 \atop 1, \xi \right).$

\item[$\cdot$]$\zeta^{\mathfrak{m}}\left(9, 5 \atop 1, \xi \right) -685575 \zeta^{\mathfrak{m}}\left( 1, 13 \atop 1, \xi \right) + 5 \zeta^{\mathfrak{m}}\left( 8, 6 \atop 1, \xi \right) -70 \zeta^{\mathfrak{m}}\left( 6, 8 \atop 1, \xi \right)+ 2016 \zeta^{\mathfrak{m}}\left( 4, 10 \atop 1, \xi \right)- 89760 \zeta^{\mathfrak{m}}\left( 2, 12 \atop 1, \xi \right).$
\end{itemize}

\begin{lemm} 
The depth $2$ part of the basis of $\mathcal{F}^{k_{N}/\mathbb{Q},P/1}_{1}\mathcal{H}_{n}$ is for even $n$:
\begin{small}
$$\left\{ \zeta^{\mathfrak{m}}\left(2a+1, 2b+1 \atop  1, \xi\right)- \sum_{k=b+1}^{\frac{n}{2}-1} \alpha^{a,b}_{k} \zeta^{\mathfrak{m}}\left(n-2k, 2k \atop 1, \xi \right), a,b\geq 0, (a,b)\neq(0,0) \right\},$$
\end{small}
For odd $n$, the part in depth $2$ of the basis of $\mathcal{F}^{k_{N}/\mathbb{Q},P/1}_{1}\mathcal{H}_{n}$ is:
\begin{small}
$$\left\{ \zeta^{\mathfrak{m}}\left(x_{1}, x_{2} \atop 1, \xi \right)+ (-1)^{x_{2}+1}  \binom{n-2}{x_{2}-1} \zeta^{\mathfrak{m}}\left(1, n-1 \atop 1, \xi \right), x_{1},x_{2} >1, \text{ one even, the other odd } \right\}.$$
\end{small}
\end{lemm}
\begin{proof}
\begin{itemize}
\item[$\cdot$] For even $n$, we need to cancel $D_{2r}$ (else $D_{2s}\circ D_{2r}(\cdot) \neq 0$), so we substract the same linear combination than in the previous lemma.
\item[$\cdot$] For odd $n$, we need to cancel $D_{1}\circ D_{2r}$. Since $D_{1}\circ D_{2r}(Z)= (-1)^{x_{2}} \binom{n-2}{x_{2}-1}$, we substract $(-1)^{x_{2}+1}  \binom{n-2}{x_{2}-1} \zeta^{\mathfrak{m}}(1, \overline{n-1})$.
\end{itemize}
\end{proof}

\begin{lemm} 
The depth $2$ part of the basis of $\mathcal{F}^{k_{N}/\mathbb{Q},P/P}_{0}\mathcal{H}_{n}$ ($=\mathcal{H}_{n}^{\mathcal{MT}_{2}}$ if $N=4$) is:
\begin{small}
$$\left\{ \zeta^{\mathfrak{m}}\left(2a+1, 2b+1 \atop 1, \xi \right)- \sum_{k=b+1}^{\frac{n}{2}-1} \alpha^{a,b}_{k} \zeta^{\mathfrak{m}}\left(n-2k, 2k \atop 1, \xi \right), a,b \geq 0 \right\}.$$
\end{small}
\end{lemm}
\begin{proof}
To cancel $D_{2r}$, we substract the same linear combination than above. 
\end{proof}

\begin{lemm} 
The depth $2$ part of the basis of $\mathcal{F}^{k_{N}/\mathbb{Q},P/P}_{1}\mathcal{H}_{n}$ is for even $n$:
\begin{small}
$$\left\{ \zeta^{\mathfrak{m}}\left(2a+1, 2b+1 \atop 1, \xi \right)- \sum_{k=b+1}^{\frac{n}{2}-1} \alpha^{a,b}_{k} \zeta^{\mathfrak{m}}\left(n-2k, 2k \atop 1, \xi \right), a,b\geq 0, \right\},$$
\end{small}
And for odd $n$, the part in depth $2$ of the basis of $\mathcal{F}^{k_{N}/\mathbb{Q},P/P}_{1}\mathcal{H}_{n}$ is:
\begin{small}
$$\left\{ \zeta^{\mathfrak{m}}\left( x_{1}, x_{2} \atop 1, \xi \right), x_{1},x_{2} \geq 1 \text{, one even, the other odd } \right\}.$$
\end{small}
\end{lemm}
\begin{proof}
If $n$ is even, to cancel $\lbrace D_{2r}\rbrace$, we use the same linear combination than above.\\
If $n$ is odd, we already have $\zeta^{\mathfrak{m}}(x_{1}, x_{2};  1, \xi)\in \mathcal{F}^{k_{N}/\mathbb{Q},P/P}_{1}\mathcal{H}_{n}$.\\
\end{proof}

\subsection{$\boldsymbol{N=8}$: Depth $\boldsymbol{2}$}

Let us illustrate the results for the depth 2; proofs being similar (albeit longer) as in the previous sections are left to the reader; same notations than the previous case. 

\begin{lemm} 
\begin{itemize}
	\item[$\cdot$] The depth $2$ part of the basis of MMZV$_{\mu_{4}}$ is:
{\small $$\left\{ \zeta^{\mathfrak{m}}\left(x_{1}, x_{2} \atop  1, \xi\right)+ \zeta^{\mathfrak{m}}\left(x_{1},x_{2} \atop -1, -\xi\right)+ \zeta^{\mathfrak{m}}\left(x_{1},x_{2} \atop 1, -\xi\right)+ \zeta^{\mathfrak{m}}\left(x_{1},x_{2} \atop -1, \xi\right), x_{i} \geq 1 \right\}.$$}
	\item[$\cdot$] The depth $2$ part of the basis of motivic Euler sums is:
\begin{small}
$$\left\{ \zeta^{\mathfrak{m}}\left(2a+1, 2b+1 \atop  1, \xi\right)+ \zeta^{\mathfrak{m}}\left(2a+1, 2b+1 \atop  -1, -\xi\right) + \zeta^{\mathfrak{m}}\left(2a+1, 2b+1 \atop  1, -\xi\right)+ \zeta^{\mathfrak{m}}\left(2a+1, 2b+1 \atop  -1, \xi\right) \right. $$
$$\left. - \sum_{k=b+1}^{\frac{n}{2}-1} \alpha^{a,b}_{k} \left( \zeta^{\mathfrak{m}}\left(n-2k, 2k \atop 1, \xi\right) + \zeta^{\mathfrak{m}}\left(n-2k, 2k \atop -1, -\xi\right) + \zeta^{\mathfrak{m}}\left(n-2k, 2k \atop 1, -\xi\right) + \zeta^{\mathfrak{m}}\left(n-2k, 2k \atop -1, \xi\right)\right)\right\}_{a,b \geq 0}$$
\end{small}
	\item[$\cdot$] The depth $2$ part of the basis of MMZV is:
	\begin{small}
$$\left\{ \zeta^{\mathfrak{m}}\left(2a+1, 2b+1\atop  1, \xi\right) + \zeta^{\mathfrak{m}}\left(2a+1, 2b+1\atop  -1, -\xi\right)  + \zeta^{\mathfrak{m}}\left(2a+1, 2b+1\atop  1, -\xi\right)  + \zeta^{\mathfrak{m}}\left(2a+1, 2b+1\atop  -1, \xi\right) \right.$$
$$ \left. - \sum_{k=b+1}^{\frac{n}{2}-1} \alpha^{a,b}_{k} \left( \zeta^{\mathfrak{m}}\left(n-2k, 2k\atop 1, \xi\right)+ \zeta^{\mathfrak{m}}\left(n-2k, 2k\atop -1,-\xi\right)  + \zeta^{\mathfrak{m}}\left(n-2k, 2k\atop 1, -\xi\right)  + \zeta^{\mathfrak{m}}\left(n-2k, 2k\atop -1, \xi\right)   \right) \right.$$
$$\left. - \beta^{a,b} \left( \zeta^{\mathfrak{m}}\left(1,n-1\atop  1, \xi\right)+ \zeta^{\mathfrak{m}}\left(1,n-1\atop  -1, \xi\right)+ \zeta^{\mathfrak{m}}\left(1,n-1\atop  1, -\xi\right)+ \zeta^{\mathfrak{m}}\left(1,n-1\atop  -1, -\xi\right) \right),  a,b> 0 \right\} $$ 
\end{small}
\end{itemize}
\end{lemm}

\begin{lemm} 
\begin{itemize}
	\item[$\cdot$] The depth $2$ part of the basis of $\mathcal{F}^{k_{8}/k_{4},2/2}_{1}\mathcal{H}_{n}$ is, for even $n$:
	\begin{small}
$$\left\{ \zeta^{\mathfrak{m}}\left( x_{1}, x_{2}\atop   1, \xi\right)+ \zeta^{\mathfrak{m}}\left( x_{1},x_{2} \atop  -1, -\xi\right), \zeta^{\mathfrak{m}}\left( x_{1},x_{2} \atop  1, -\xi\right)- \zeta^{\mathfrak{m}}\left( x_{1},x_{2} \atop  -1, -\xi\right), \zeta^{\mathfrak{m}}\left( x_{1},x_{2} \atop  -1, \xi\right)+ \zeta^{\mathfrak{m}}\left( x_{1},x_{2} \atop  -1, -\xi\right),  x_{i} \geq 1 \right\}.$$ 
\end{small}
	\item[$\cdot$] The depth $2$ part of the basis of $\mathcal{F}^{k_{8}/\mathbb{Q},2/2}_{1}\mathcal{H}_{n}$ is for odd $n$:
	\begin{small}
$$\left\{ \zeta^{\mathfrak{m}}\left( x_{1}, x_{2}\atop   1, \xi\right)+ \zeta^{\mathfrak{m}}\left( x_{1}, x_{2}\atop   -1, -\xi\right) + \zeta^{\mathfrak{m}}\left( x_{1}, x_{2}\atop   1, -\xi\right)+ \zeta^{\mathfrak{m}}\left( x_{1}, x_{2}\atop   -1, \xi\right) , \text{ exactly one even } x_{i} \right\}.$$ 
\end{small}
The depth $2$ part of the basis of $\mathcal{F}^{k_{8}/\mathbb{Q},2/2}_{1}\mathcal{H}_{n}$ is for even $n$:
\begin{small}
\begin{multline}\nonumber
\left\{ \zeta^{\mathfrak{m}}\left( 2a+1, 2b+1\atop   -1, \xi\right)+ \zeta^{\mathfrak{m}}\left( 2a+1, 2b+1\atop   -1, -\xi\right) - \sum_{k=b+1}^{\frac{n}{2}-1} \alpha^{a,b}_{k} \left( \zeta^{\mathfrak{m}}\left( n-2k, 2k\atop  -1, \xi\right) + \zeta^{\mathfrak{m}}\left( n-2k, 2k\atop  -1, -\xi\right) \right)\right\}_{a,b\geq 0}\\
\cup \left\{ \zeta^{\mathfrak{m}}\left( 2a+1, 2b+1\atop   1, -\xi\right)- \zeta^{\mathfrak{m}}\left( 2a+1, 2b+1\atop   -1, -\xi\right)- \sum_{k=b+1}^{\frac{n}{2}-1} \alpha^{a,b}_{k} \left( \zeta^{\mathfrak{m}}\left( n-2k, 2k\atop  1, -\xi\right) - \zeta^{\mathfrak{m}}\left( n-2k, 2k\atop  -1, -\xi\right) \right)\right\}_{ a,b\geq 0}.
\end{multline}
\end{small}
	\item[$\cdot$]  The depth $2$ part of the basis of $\mathcal{F}^{k_{8}/\mathbb{Q},2/1}_{1}\mathcal{H}_{n}$ is for odd $n$:
	\begin{small}
\begin{multline}\nonumber
	\left\{ \zeta^{\mathfrak{m}}\left( x_{1}, x_{2}\atop   1, \xi\right)+ \zeta^{\mathfrak{m}}\left( x_{1}, x_{2}\atop   -1, -\xi\right)+ \zeta^{\mathfrak{m}}\left( x_{1}, x_{2}\atop   1, -\xi\right) + \zeta^{\mathfrak{m}}\left( x_{1}, x_{2}\atop   -1, \xi\right) \right. \\
 \left. - \gamma^{x_{1},x_{2}} \left( \zeta^{\mathfrak{m}}\left( 1, n-1\atop   1, \xi\right) + \zeta^{\mathfrak{m}}\left( 1, n-1\atop   -1, -\xi\right)+ \zeta^{\mathfrak{m}}\left( 1, n-1\atop   -1, \xi\right)+ \zeta^{\mathfrak{m}}\left( 1, n-1\atop   1, -\xi\right) \right), \text{ exactly one even } x_{i} \right\}.
 \end{multline}
\end{small}
In even weight $n$, depth $2$ part of the basis of $\mathcal{F}^{k_{8}/\mathbb{Q},2/1}_{1}\mathcal{H}_{n}$ is:
\begin{small}
\begin{multline}\nonumber
\left\{ \zeta^{\mathfrak{m}}\left( 1, n-1\atop   1, \xi\right)+ \zeta^{\mathfrak{m}}\left( 1, n-1\atop   -1,-\xi \right)+ \zeta^{\mathfrak{m}}\left( 1, n-1\atop   1,-\xi \right) + \zeta^{\mathfrak{m}}\left( 1, n-1\atop   -1,\xi \right) \right\}  \\
\cup \left\{ \zeta^{\mathfrak{m}}\left(  n-1,1\atop   1, \xi\right)+ \zeta^{\mathfrak{m}}\left(  n-1,1\atop   -1,-\xi \right)+ \zeta^{\mathfrak{m}}\left( n-1,1\atop   1,-\xi \right) + \zeta^{\mathfrak{m}}\left( n-1,1\atop   -1,\xi \right) + \right.\\
\left. -\sum_{k=1}^{\frac{n}{2}-1} \alpha^{0,\frac{n}{2}-1}_{k} \left( \zeta^{\mathfrak{m}}\left(  n-2k, 2k\atop   1, \xi\right)  + \zeta^{\mathfrak{m}}\left(  n-2k, 2k\atop   -1, -\xi\right) + \zeta^{\mathfrak{m}}\left(  n-2k, 2k\atop   1, -\xi\right)+ \zeta^{\mathfrak{m}}\left(  n-2k, 2k\atop   -1, \xi\right)  \right)\right\}\\
\cup \left\{ \zeta^{\mathfrak{m}}\left( 2a+1, 2b+1\atop   \epsilon_{1}, \epsilon_{2}\xi\right)+ \epsilon_{2} \zeta^{\mathfrak{m}}\left( 2a+1,2b+1\atop   -1, -\xi\right) \right. - \beta^{a,b} \left( \zeta^{\mathfrak{m}}\left( 1, n-1\atop   \epsilon_{1}, \epsilon_{2} \xi\right) + \epsilon_{2} \zeta^{\mathfrak{m}}\left( 1,n-1\atop   -1, -\xi\right)  \right)\\
\left. -\sum_{k=b+1}^{\frac{n}{2}-1} \alpha^{a,b}_{k} \left( \zeta^{\mathfrak{m}}\left( n-2k, 2k\atop   \epsilon_{1}, \epsilon_{2}\xi \right) + \epsilon_{2} \zeta^{\mathfrak{m}}\left( n-2k, 2k\atop   -1, -\xi \right)  \right), a,b >0 , \epsilon_{i}\in\left\{\pm 1\right\}, \epsilon_{1}=- \epsilon_{2}  \right\} .
\end{multline}
\end{small}
Where $\gamma^{x_{1},x_{2}}=(-1)^{x_{2}} \binom{2r-1}{2r-x_{2}}$.
\end{itemize}
\end{lemm}

\subsection{$\boldsymbol{N=\mlq 6 \mrq}$: Depth $\boldsymbol{2}$}

In depth 2, coefficients are explicit as previously:
      
\begin{lemm} 
The depth $2$ part of the basis of MMZV, for even weight $n$ is:
$$\left\{ \zeta^{\mathfrak{m}}\left(2a+1, 2b+1 \atop 1, \xi \right)-  \sum_{k=b+1}^{\frac{n}{2}-1} \alpha^{a,b}_{k} \zeta^{\mathfrak{m}}\left(n-2k, 2k\atop 1, \xi\right), a,b> 0 \right\},$$
\end{lemm}
\begin{proof}
Proof being similar than the cases $N=3,4$ is left to the reader.\\
\\
\end{proof}

\section{Homographies of $\boldsymbol{\mathbb{P}^{1}\diagdown \lbrace 0, \mu_{N}, \infty\rbrace}$}

The homographies of the projective line $\mathbb{P}^{1}$ which permutes $\lbrace 0, \mu_{N}, \infty \rbrace$, induce automorphisms $\mathbb{P}^{1}\diagdown \lbrace 0, \mu_{N}, \infty\rbrace \rightarrow \mathbb{P}^{1}\diagdown \lbrace 0, \mu_{N}, \infty\rbrace$. The projective space $\mathbb{P}^{1} \diagdown \lbrace 0, \mu_{N}, \infty \rbrace$ has a dihedral symmetry, the dihedral group $Di_{N}= \mathbb{Z}\diagup 2 \mathbb{Z} \ltimes \mu_{N}$ acting with $x \mapsto x^{-1}$ and $x\mapsto \eta x$. In the special case of $N=1,2,4$, and for these only, the group of homographies is bigger than the dihedral group, due to particular symmetries of the points $\mu_{N}\cup \lbrace 0, \infty\rbrace$ on the Riemann sphere. Let specify these cases:
\begin{itemize}
\item[For $N=1:$] The homography group is the anharmonic group generated by $z\mapsto\frac{1}{z}$ and $z \mapsto 1-z$, and corresponds to the permutation group $\mathfrak{S}_{3}$. Precisely, projective transformations of $\mathbb{P}^{1}\diagdown \lbrace 0, 1, \infty\rbrace$ are:
$$\begin{array}{lll}\label{homography1}
\phi_{\tau}: & t \mapsto 1-t : &  \left\lbrace \begin{array}{l} 
(0,1,\infty)\mapsto (1,0,\infty)\\
(\omega_{0},\omega_{1}, \omega_{\star}, \omega_{\sharp}) \mapsto (\omega_{1},\omega_{0}, -\omega_{\star}, \omega_{0}-\omega_{\star}).
\end{array} \right. \\
\phi_{c}: &  t \mapsto \frac{1}{1-t}  :&  \left\lbrace \begin{array}{l} 
0\mapsto 1 \mapsto \infty \mapsto 0\\
(\omega_{0},\omega_{1}, \omega_{\star}, \omega_{\sharp}) \mapsto (\omega_{\star},-\omega_{0}, -\omega_{1}, -\omega_{0}-\omega_{1})
\end{array} \right. \\
\phi_{\tau c} : &  t \mapsto \frac{t}{t-1} :& \left\lbrace \begin{array}{l}
 (0,1,\infty)\mapsto (0,\infty,1)\\
 (\omega_{0},\omega_{1}, \omega_{\star}) \mapsto (-\omega_{\star},-\omega_{1}, -\omega_{0})
\end{array} \right.\\
\phi_{c\tau}: &  t \mapsto \frac{1}{t} : & \left\lbrace \begin{array}{l} 
(0,1,\infty)\mapsto (\infty,1,0)\\
(\omega_{0},\omega_{1}, \omega_{\star},\omega_{\sharp}) \mapsto (-\omega_{0},\omega_{\star}, \omega_{1}, \omega_{\sharp})
\end{array} \right. \\
 \phi_{c^{2}}: &  t \mapsto \frac{t-1}{t}  : & \left\lbrace \begin{array}{l} 
0\mapsto \infty \mapsto 1 \mapsto 0\\
(\omega_{0},\omega_{1}, \omega_{\star}) \mapsto (-\omega_{1},-\omega_{\star}, \omega_{0})
\end{array} \right. \\
\end{array}$$
Remark that hexagon relation ($\ref{fig:hexagon}$) corresponds to a cycle $c$ whereas the reflection relation corresponds to a transposition $\tau$, and :
$$\mathfrak{S}_{3}= \langle c, \tau \mid c^{3}=id, \tau^{2},c\tau c =\tau \rangle= \lbrace 1, c, c^{2}, \tau, \tau c, c\tau\rbrace.$$
\item[For $N=2:$] Here, $(0,\infty, 1,-1)$ has a cross ratio $-1$ (harmonic conjugates) and there are $8$ permutations of $(0,\infty, 1,-1)$ preserving its cross ratio. The homography group corresponds indeed to the group of automorphisms of a square with consecutive vertices $(0, 1, \infty, -1)$, i.e. the dihedral group of degree four $Di_{4}$ defined by the presentation $\langle \sigma, \tau \mid \sigma^{4}= \tau^{2}=id, \sigma\tau \sigma= \tau \rangle$:
$$\begin{array}{lll}\label{homography2}
\phi_{\tau}: & t \mapsto \frac{1}{t} :& \left\lbrace \begin{array}{l} 
\pm 1\mapsto \pm 1  \quad 0 \leftrightarrow \infty\\
(\omega_{0},\omega_{1}, \omega_{\star},\omega_{-1}, \omega_{-\star}, \omega_{\pm\sharp}) \mapsto (-\omega_{0},\omega_{\star}, \omega_{1},\omega_{-\star},\omega_{-1}, \omega_{\pm\sharp})
\end{array} \right.\\
\\
 \phi_{\sigma}: &  t \mapsto \frac{1+t}{1-t}  : & \left\lbrace \begin{array}{l} 
-1\mapsto 0\mapsto 1\mapsto \infty\mapsto -1\\
(\omega_{0},\omega_{1},\omega_{\star},\omega_{-1}, \omega_{-\star}) \mapsto (\omega_{-1}- \omega_{1}, -\omega_{-1}, - \omega_{1}, - \omega_{-\star}, - \omega_{\star})\\
(\omega_{\sharp}, \omega_{-\sharp})  \mapsto  (-\omega_{1}-\omega_{-1}, -\omega_{\star}-\omega_{-\star})
\end{array} \right. \\
\\
 \phi_{\sigma^{2}\tau}: &  t \mapsto -t:& \left\lbrace \begin{array}{l} 
-1 \leftrightarrow 1 \\
(\omega_{0},\omega_{ 1}, \omega_{-1}, \omega_{ \pm \ast}, \omega_{\pm \sharp}) \mapsto (\omega_{0},\omega_{-1}, \omega_{1},\omega_{\mp \ast}, \omega_{\mp \sharp})
\end{array} \right. \\
\\
\phi_{\sigma^{2}}: & t \mapsto \frac{-1}{t} : &  \left\lbrace \begin{array}{l}
0 \leftrightarrow \infty \quad -1 \leftrightarrow 1\\
(\omega_{0},\omega_{1},\omega_{\star},\omega_{-1}, \omega_{-\star}, \omega_{\pm \sharp}) \mapsto (-\omega_{0}, \omega_{-\star}, - \omega_{-1}, \omega_{\star}, \omega_{1}, \omega_{\mp\sharp})
\end{array} \right.\\
\\
\phi_{\sigma^{-1}}: & t \mapsto \frac{t-1}{1+t} : & \left\lbrace \begin{array}{l}
0 \mapsto -1 \mapsto \infty \mapsto 1 \mapsto 0 \\
(\omega_{0}, \omega_{1}, \omega_{-1}, \omega_{\star}, \omega_{-\star}) \mapsto (\omega_{-1}-\omega_{1}, - \omega_{\star}, -\omega_{1}, -\omega_{-\star}, -\omega_{-1}) \\
(\omega_{\sharp}, \omega_{-\sharp})  \mapsto  ( -\omega_{\star}-\omega_{-\star}, -\omega_{1}-\omega_{-1})
\end{array} \right.\\
\\
\phi_{\tau \sigma}: &  t \mapsto \frac{1-t}{1+t} : & \left\lbrace \begin{array}{l}
-1 \leftrightarrow \infty \quad 0 \leftrightarrow 1 \\
(\omega_{0},\omega_{1},\omega_{\star},\omega_{-1}, \omega_{-\star}) \mapsto (\omega_{1}-\omega_{-1},-\omega_{-\star},-\omega_{\star},-\omega_{-1}, -\omega_{1}) \\
(\omega_{\sharp}, \omega_{-\sharp})  \mapsto  ( -\omega_{\star}-\omega_{-\star}, -\omega_{1}-\omega_{-1})
\end{array} \right.\\
\\
\phi_{\sigma \tau}: & t \mapsto \frac{1+t}{t-1} : & \left\lbrace \begin{array}{l}
-1 \leftrightarrow 0  \quad 1 \leftrightarrow \infty\\
(\omega_{0},\omega_{1},\omega_{\star},\omega_{-1}, \omega_{-\star}) \mapsto  (\omega_{1}-\omega_{-1},-\omega_{1},-\omega_{-1},-\omega_{\star}, -\omega_{-\star}) \\
(\omega_{\sharp}, \omega_{-\sharp})  \mapsto  ( -\omega_{1}-\omega_{-1}, -\omega_{\star}-\omega_{-\star})
\end{array} \right.
\end{array} $$
Remark that the octagon relation ($\ref{fig:octagon}$) comes from the cycle $\sigma$ of order $4$; the other permutations above could also leads to relations.
\item[For $N=4:$] $\mathbb{P}^{1}\diagdown \lbrace 0, 1,-1,i,-i, \infty\rbrace$ has an octahedral symmetry, and the homography group is the group of automorphisms of this octahedron placed on the Riemann sphere of vertices $(0,1,i,-1, -i,\infty)$.\footnote{Zhao showed this octahedral symmetry allows to reach the \say{non standard} relations which appeared in weight $3$, $4$ for $N=4$; non standard relations are these which do not come from distribution, conjugation, and regularised double shuffle relation, cf. $\cite{Zh1}$.} It is composed by 48 transformations, corresponding to 24 rotational symmetries, and a reflection.
\end{itemize}
We could also look at other projective transformations: $\mathbb{P}^{1}\diagdown \lbrace 0, \mu_{N}, \infty\rbrace \rightarrow \mathbb{P}^{1}\diagdown \lbrace 0, \mu_{N'}, \infty\rbrace  _{N'\mid N}$.\\
\\
\texttt{Examples:}
\begin{itemize}
\item[$\cdot$] $\mathbb{P}^{1}\diagdown \lbrace 0, -1, \infty\rbrace \rightarrow \mathbb{P}^{1}\diagdown \lbrace 0, +1, \infty\rbrace  \text{  , }  t\mapsto 1+t $. 
\item[$\cdot$] $\mathbb{P}^{1}\diagdown \lbrace 0, -1, \infty\rbrace \rightarrow \mathbb{P}^{1}\diagdown \lbrace 0, +1, \infty\rbrace  \text{  , }  t\mapsto \frac{1}{1+t} $.
\item[$\cdot$] $\mathbb{P}^{1}\diagdown \lbrace 0, \pm 1, \infty\rbrace \rightarrow \mathbb{P}^{1}\diagdown \lbrace 0, 1, \infty \rbrace  \text{  , }  t\mapsto t^{2} $.  
\end{itemize}

\newpage
\section{Hybrid relation for MMZV}

The commutative polynomial setting is briefly introduced in $\S 6.1.1$. \\
Let consider the following involution, which represents the Antipode $\shuffle$ as seen in $\S 4.2.1$:\\
\begin{equation} \label{eq:sigma} 
\boldsymbol{\sigma}: \quad \mathbb{Q} \langle Y\rangle  \rightarrow \mathbb{Q} \langle Y\rangle \text{  ,   } \quad f(y_{0},y_{1},\cdots, y_{p}) \mapsto (-1)^{w} f(y_{p},y_{p-1},\cdots, y_{1}),
\end{equation}
with $w$ the weight, equal to the degree of $f$ plus $p$. In particular, for $f\in\rho(\mathfrak{g}^{\mathfrak{m}})$:
\begin{equation} \label{eq:fantipodesh} \textsc{ Antipode } \shuffle\text{ :  } f+\sigma(f)=0.
\end{equation}
Note that $f^{(p)}$ denotes the part of $f$ involving $y_{0},\cdots, y_{p}$. We can also consider:
\begin{equation} \label{eq:tau} \boldsymbol{\tau}: \quad \mathbb{Q} \langle X\rangle  \rightarrow \mathbb{Q} \langle X\rangle\text{ ,    }\quad  \overline{f}^{(p)}(x_{1},\cdots, x_{p}) \mapsto (-1)^{p}\overline{f}^{(p)}(x_{p}\cdots, x_{1}).
\end{equation}
The Antipode stuffle corresponds to $\tau(\overline{f}^{\star})$, where $\overline{f}^{\star}$ is defined by:
\begin{equation} \label{eq:fstar}  \overline{f}^{\star}(x_{1}, \ldots, x_{p})\mathrel{\mathop:}= \sum_{s \leq p, i_{k} \atop p=\sum i_{k}} f(\lbrace x_{1} \rbrace ^{i_{1}}, \ldots, \lbrace x_{s} \rbrace ^{i_{s}}) (-1)^{d-1} \prod_{k=1}^{s} x^{i_{k}-1}_{k} .
\end{equation}
It corresponds naturally to the Euler sums $\star$ version. Then, for $\overline{f}\in \overline{\rho} (\mathfrak{g}^{\mathfrak{m}})$:
\begin{equation} \label{eq:fantipodest}  \textsc{ Antipode  } \ast\text{: } \overline{f}+ \tau(\overline{f}^{\star})=0 .
\end{equation}
\\
The \textit{hybrid relation} (Theorem $\ref{hybrid}$) for motivic multiple zeta values is equivalent to, in this setting of commutative polynomials to the following, already in some notes of F. Brown:
\begin{theom}[\textsc{F. Brown}]
For $\overline{f}\in \overline{\rho} (\mathfrak{g}^{\mathfrak{m}})$, the 6 terms relation holds:
\begin{multline}\nonumber
\overline{f}^{(p)} (x_{1},\cdots, x_{p}) +  \frac{\overline{f}^{(p-1)} (x_{2}-x_{1}, \ldots, x_{p}-x_{1}) - \overline{f}^{(d-1)} (x_{2},\cdots, x_{p})}{x_{1}} \\
=(-1)^{w+1} \left(  \overline{f}^{(p)} (x_{p},\cdots, x_{1}) +  \frac{\overline{f}^{(p-1)} (x_{p-1}-x_{p}, \ldots, x_{1}-x_{p}) - \overline{f}^{(p-1)} (x_{p-1},\cdots, x_{1})}{x_{p}} \right).
\end{multline}
\end{theom}
\noindent
Before giving the proof, to be convinced these statements are equivalent, let just write $\overline{f}$ as:
$$\overline{f}=\sum \alpha_{n_{1}, \ldots, n_{k}} x_{1}^{n_{1}-1} \cdots x_{k}^{n_{k}-1}.$$
Then:
\begin{flushleft}
$\frac{\overline{f}^{(p-1)} (x_{2}-x_{1}, \ldots, x_{p}-x_{1}) - \overline{f}^{(p-1)} (x_{2},\cdots, x_{p})}{x_{1}} $
\end{flushleft}
\begin{align*}
\quad \quad = & \sum \alpha_{n_{1}, \ldots, n_{p-1}}  \frac{(x_{2}-x_{1})^{n_{1}-1} \cdots (x_{p}-x_{1})^{n_{p}-1} -x_{2}^{n_{1}-1} \cdots x_{p}^{n_{p}-1} }{x_{1}}  \\
\quad\quad = & \sum \alpha_{n_{1}, \ldots, n_{p-1}}  \sum_{1 \leq k_{i} \leq n_{i} \atop k\mathrel{\mathop:}=\sum n_{i}-k_{i}>0}  (-1)^{k}x_{1}^{k-1}  \prod_{i=1}^{d-1}\left(  \binom{n_{i}-1}{k_{i}-1} x_{i+1}^{k_{i}-1} \right)  \\
 \quad \quad = & \sum_{\sum i_{j}=k}  \alpha_{k_{1}+i_{1}, \ldots, k_{p-1}+i_{p-1}} \binom{k_{1}+i_{1}-1}{k_{1}-1} \cdots \binom{k_{p-1}+i_{p-1}-1}{k_{p-1}-1}  (-1)^{k}x_{1}^{k-1} x_{2}^{k_{1}-1} \cdots x_{p}^{k_{p-1}-1} 
\end{align*}
This, according to the shuffle regularization $(\ref{eq:shufflereg})$, matches exactly with the definition of $\zeta^{\mathfrak{m}}_{k}(k_{1}, \ldots, k_{p-1})$.

\begin{proof}[Proof of the previous theorem]
The proof combines the shuffle relation (using that $f$ is translation invariant notably), the linearized stuffle relation (giving a relation between depth $p$ and depth $p-1$) and the antipode $\shuffle$. \\
Let's take $f$ in $\rho (\mathfrak{g}^{\mathfrak{m}})$ and let consider the difference
\begin{equation} \label{eq:if} I(y_{0},y_{1},\cdots, y_{p})\mathrel{\mathop:}=  f^{(p)} (y_{0}, y_{1},\cdots, y_{p})+(-1)^{w} f^{(p)} (y_{0}, y_{p},\cdots, y_{1}) 
\end{equation}
$$= f^{(p)} (y_{0}, y_{1},\cdots, y_{p})- f^{(p)} (y_{1},\cdots, y_{p}, y_{0}). $$
Consider also the relation given by the linearized stuffle relation (in $\mathcal{L})$, between depth $p$ and depth $p-1$, defining $St$:
\begin{equation} \label{eq:stf0}
St(y_{0}, y_{1} \shuffle y_{2} \cdots y_{p})\mathrel{\mathop:}= f^{(p)} (y_{0}, y_{1} \shuffle y_{2}\cdots y_{p}),
\end{equation}
Where $St$ can then be expressed by $f^{(p-1)}$ using stuffle:
\begin{equation} \label{eq:stf} St(y_{0}, y_{1} \shuffle y_{2} \cdots y_{p})= \sum \frac{1}{y_{i}-y_{1}} \left(f^{(p-1)} (y_{0},y_{2},,\cdots y_{i-1},y_{1},y_{i+1}, \ldots, y_{p}) - f^{(p-1)} (y_{0}, y_{2},\cdots, y_{p})\right).
\end{equation}
The theorem is then equivalent to the following identity
$$ (\Join) \text{  } I(y_{0},y_{1},\cdots, y_{p}) = (-1)^{w+1}St(y_{p}, y_{0} \shuffle y_{p-1} \cdots y_{1})- St(y_{1}, y_{0} \shuffle y_{2} \cdots y_{p}).$$
Indeed, looking at the previous definition ($\ref{eq:stf}$ ), most of the terms of $St$ in the right side of $(\Join)$ get simplified together, and it remains only:
$$ (-1)^{w+1} \frac{f^{(p-1)} (y_{p}, \ldots, y_{2}, y_{0}) - f^{(p-1)} (y_{p},\cdots, y_{1})}{y_{1}-y_{0}}  -  \frac{f^{(p-1)} (y_{1}, \ldots, y_{p-1}, y_{0}) - f^{(p-1)} (y_{1},\cdots, y_{p})}{y_{p}-y_{0}}. $$
Passing to the $x_{i}$ variables, we conclude that $ (\Join)$ is equivalent to the theorem's statement; let now prove $(\Join) $. By definition: 
$$St(y_{1},y_{0}\shuffle y_{2} \cdots y_{p})= f^{(p)}(y_{1},y_{0}\shuffle y_{2} \cdots y_{p})= f^{(p)}(y_{1},y_{0}\shuffle y_{2} \cdots y_{p-1}, y_{p})+ f^{(p)}(y_{1}, y_{2} , \ldots, y_{p}, y_{0}) .$$
Doing a right shift, using the definition of $I$:
\begin{equation} \label{eq:a} 
St(y_{1},y_{0}\shuffle y_{2} \cdots y_{p})
\end{equation}
$$=f^{(p)}(y_{p}, y_{1},y_{0}\shuffle y_{2} \cdots y_{p-1}) - I(y_{p}, y_{1},y_{0}\shuffle y_{2} \cdots y_{p-1}) + f^{(p)}(y_{0},y_{1}, y_{2} \cdots, y_{p}) -I(y_{0},y_{1}, y_{2} \cdots, y_{p}).$$
Since:
\begin{align*}
f^{(p)}(y_{p}, y_{1},y_{0}\shuffle y_{2} \cdots y_{p-1}) & = St(y_{p}, y_{0}\shuffle y_{1} y_{2} \cdots y_{p-1})- f^{(p)}(y_{p}, y_{0}, y_{1}, y_{2}, \ldots, y_{p-1}) \\
f^{(p)}(y_{0},y_{1}, y_{2} \cdots, y_{p}) & = -I (y_{p}, y_{0}, y_{1}, y_{2}, \ldots, y_{p-1})+ f^{(p)}(y_{p}, y_{0}, y_{1}, y_{2}, \ldots, y_{p-1}).
\end{align*}
Then, $\eqref{eq:a}$ becomes:
\begin{multline}\nonumber
St(y_{1},y_{0}\shuffle y_{2} \cdots y_{p})- St(y_{p}, y_{0}\shuffle y_{1} y_{2} \cdots y_{p-1}) \\
=-I (y_{p}, y_{0}, y_{1}, y_{2}, \ldots, y_{p-1}) - I(y_{p}, y_{1},y_{0}\shuffle y_{2} \cdots y_{p-1})  -I(y_{0},y_{1}, y_{2} \cdots, y_{p}).
\end{multline}
The sum of the first two $I$ is $I(y_{p}, y_{0}\shuffle y_{1} \cdots y_{p-1})$ which gives:
\begin{multline} \label{eq:b} 
I(y_{0},y_{1}, y_{2} \cdots, y_{p})= - St(y_{1},y_{0}\shuffle y_{2} \cdots y_{p})+ St(y_{p}, y_{0}\shuffle y_{1} y_{2} \cdots y_{p-1})-I(y_{p}, y_{0}\shuffle y_{1} \cdots y_{p-1}) \\
=- St(y_{1},y_{0}\shuffle y_{2} \cdots y_{p}) + (-1)^{w+1}St(y_{p}, y_{0} \shuffle y_{p-1} \cdots y_{1}).
\end{multline}
The identity $(\Join)$ holds, and the identity of the theorem follows.
\end{proof}

\newpage

 \printnomenclature

\end{document}